\documentclass[10.9pt,a4paper]{amsart}
\usepackage[english]{babel}
\usepackage[normalem]{ulem}

\usepackage[square,sort,comma,numbers]{natbib}

\usepackage[all]{xy}
\usepackage{pstricks}
\usepackage{enumerate}
\usepackage{amsfonts,amssymb,amsmath,pinlabel,array,hhline}
\usepackage{slashed}
\usepackage{tabulary}
\usepackage{fancyhdr}
\usepackage{color}
\usepackage{a4wide}
\usepackage{bbm}
\usepackage[position=b]{subcaption}
\usepackage[colorlinks,
    linkcolor={blue!50!black},
    citecolor={blue!50!black},
    urlcolor={blue!80!black}]{hyperref}

    \usepackage{tikz}
\usetikzlibrary{
  cd,
  calc,
  positioning,
  fit,
  arrows,
  decorations.pathreplacing,
  decorations.markings,
  shapes.geometric,
  backgrounds,
  bending
}
\usepackage{tikzsymbols}
\newtheorem{innercustomthm}{Theorem}
\newenvironment{customthm}[1]
  {\renewcommand\theinnercustomthm{#1}\innercustomthm}
  {\endinnercustomthm}

\newtheorem{theorem}{Theorem}[section]

\newtheorem{question*}{Question}

\newtheorem{lemma}[theorem]{Lemma}
\newtheorem{proposition}[theorem]{Proposition}

\newtheorem{claim}[theorem]{Claim}
\newtheorem*{claim*}{Claim}
\theoremstyle{definition}
\newtheorem{definition}[theorem]{Definition}
\newtheorem{construction}[theorem]{Construction}
\newtheorem{notation}[theorem]{Notation}
\newtheorem*{notation*}{Notation}
\newtheorem{convention}[theorem]{Convention}

\newtheorem{example}[theorem]{Example}

\newtheorem{remark}[theorem]{Remark}
\newtheorem*{remark*}{Remark}
\newcommand{\tmfrac}[2]{\mbox{\large$\frac{#1}{#2}$}} % tiny medium frac
\definecolor{bettergreen}{rgb}{0.0, 0.5, 0.0}

\newcommand{\Z}{\mathbb{Z}}
\newcommand{\Q}{\mathbb{Q}}
\newcommand{\R}{\mathbb{R}}
\newcommand{\C}{\mathbb{C}}

\newcommand{\bsm}{\left(\begin{smallmatrix}}
\newcommand{\esm}{\end{smallmatrix}\right)}
\newcommand{\id}{\operatorname{id}}
\newcommand{\Bl}{\operatorname{Bl}}

\newcommand{\coker}{\operatorname{coker}}
\newcommand{\proj}{\operatorname{proj}}
\newcommand{\Iso}{\operatorname{Iso}}

\newcommand{\Imm}{\operatorname{Imm}}
\newcommand{\Surf}{\operatorname{Surf}}
\newcommand{\Homeo}{\operatorname{Homeo}}
\newcommand{\Aut}{\operatorname{Aut}}
\newcommand{\Ext}{\operatorname{Ext}}

\newcommand{\im}{\operatorname{im}}
\newcommand{\intt}{\operatorname{int}}
\newcommand{\fr}{\operatorname{fr}}
\newcommand{\Hom}{\operatorname{Hom}}
\newcommand{\shrink}{\operatorname{shr}}

\newcommand{\ks}{\operatorname{ks}}
\newcommand{\pt}{\operatorname{pt}}
\newcommand{\sm}{\setminus}
\newcommand{\ol}{\overline}

\newcommand{\Span}{\operatorname{span}}

\newcommand{\unaryminus}{\scalebox{0.75}[1.0]{\( - \)}}

\DeclareSymbolFont{EulerScript}{U}{eus}{m}{n}
\DeclareSymbolFontAlphabet\mathscr{EulerScript}

\begin{document}

\title{Immersed surfaces with knot group~$\Z$}

\author[A.~Conway]{Anthony Conway}
\address{The University of Texas at Austin, Austin TX}
%%%
\email{anthony.conway@austin.utexas.edu}
\author[A.~N.~Miller]{Allison N.~Miller}
\address{Swarthmore College, Swarthmore PA}
\email{amille11@swarthmore.edu }

\begin{abstract}
This article is concerned with locally flatly immersed surfaces in simply-connected~$4$-manifolds where the complement of the surface has fundamental group~$\Z$.
Once the genus and number of double points are fixed,  we classify such immersed surfaces in terms of the equivariant intersection form of their exterior and a secondary invariant.
Applications include criteria for deciding when an immersed $\Z$-surface in~$S^4$ is isotopic to the standard immersed surface that is obtained from an unknotted surface by adding local double points.
As another application,  we enumerate~$\Z$-disks in $D^4$ with a single double point and boundary a given knot; we prove that the number of such disks may be infinite.
We also prove that a knot bounds a~$\Z$-disk in $D^4$ with $c_+$ positive double points and~$c_-$ negative double points if and only if it can be converted into an Alexander polynomial one knot via changing $c_+$ positive crossings and $c_-$ negative crossings.
In $4$-manifolds other than~$D^4$ and~$S^4$,  applications include measuring the extent to which immersed~$\Z$-surfaces are determined by the equivariant intersection form of their exterior.
Along the way,  we prove that any two~$\Z^2$-concordances between the Hopf link and an Alexander polynomial one link $L$ are homeomorphic rel. boundary.
\end{abstract}
\maketitle

\section{Introduction}
\label{sec:Intro}

This article is concerned with locally flatly immersed surfaces in compact, simply-connected topological~$4$-manifolds where the complement of the surface has fundamental group~$\Z$.
For brevity, we refer to these as \emph{immersed~$\Z$-surfaces}.
We say that two such surfaces are \emph{equivalent} (rel.\ boundary) if there is an orientation-preserving (rel.\ boundary) homeomorphism of the~$4$-manifold taking one surface to the other,  and \emph{isotopic} (rel.\ boundary) if there exists such a homeomorphism that is isotopic (rel.\ boundary) to the identity.
Whereas 
%equivalence classes of 
embedded~$\Z$-surfaces are well understood~\cite{ConwayPowell, ConwayPiccirilloPowell}, little is known about the immersed case.

Our main results, stated in Theorems~\ref{thm:SurfacesClosedIntro} and~\ref{thm:SurfacesRelBoundaryIntro} formulate necessary and sufficient conditions for the existence of an immersed~$\Z$-surface with a prescribed genus and number of double points,  and establish bijections between sets of equivalence classes of immersed~$\Z$-surfaces whose exteriors have a prescribed equivariant intersection form and sets of isometries of certain Blanchfield forms.

Here is a non-exhaustive list of applications of these results.
Theorems~\ref{thm:UnknottingAllSameSignIntro} and~\ref{thm:UniqueSpherec=1Intro} state criteria for an immersed~$\Z$-surface in~$S^4$ to be topologically isotopic to the  \emph{standard} immersed~$\Z$-surface obtained from an unknotted surface in~$S^4$ by adding local positive and negative double points.
Theorems~\ref{thm:UnknottingAllSameSignIntro},~\ref{thm:UniqueSpherec=1Intro} and~\ref{thm:g=0c=1BoundaryD4Intro} describe various classification results of immersed~$\Z$-surfaces in~$D^4$.
In particular, we show that there are knots that bound infinitely many immersed~$\Z$-disks in~$D^4$ with a single double point.
Theorem~\ref{thm:Unknotting} shows that a knot $K$ bounds a $\Z$-disk in $D^4$ with~$c_+$ positive double points and~$c_-$ negative double points if and only if it can be converted into an Alexander polynomial one knot by changing $c_+$ positive crossings and $c_-$ negative crossings.
Theorem~\ref{thm:Other4ManifoldsSpheresIntro} contains applications to immersed~$\Z$-surfaces in~$4$-manifolds other than~$D^4$ and~$S^4$. 

Independently of these results, Theorem~\ref{thm:HopfConcordancesEquivalent} shows that any two~$\Z^2$-concordances between the Hopf link and an Alexander polynomial one link $L$ are homeomorphic rel.\ boundary.

\medskip

In what follows,  a~$4$-manifold is understood to mean a compact, connected, oriented, topological~$4$-manifold,  and knots and surfaces are assumed to be oriented. 
Immersions are understood to be locally flat,  and are assumed to have 
 singular sets that consist only of transverse double points within the interior of the surface.

\subsection{Results in~$D^4$ and~$S^4$}
\label{sub:IntroD4S4}

The \emph{standard} genus~$g$ immersed~$\Z$-surface in~$S^4$ with $c_+$ positive double points and~$c_-$ negative double points is 
%the one 
obtained from the unknotted genus $g$ surface by locally adding $c_+$ positive double points and $c_-$ negative double points.
%%Independent of where you point the double points.
An Alexander polynomial one knot $K$ bounds an embedded $\Z$-disk~\cite{FreedmanQuinn} which is unique up to isotopy rel.\ boundary~\cite{ConwayPowellDiscs}, 
%a unique embedded $\Z$-disk~\cite{FreedmanQuinn,ConwayPowellDiscs} 
and the \emph{standard} genus~$g$ immersed~$\Z$-surface in $D^4$ with boundary~$K$ is
% the one 
 obtained by taking the connected sum of this disk with the corresponding standard immersed $\Z$-surface.

Our first result provides a necessary and sufficient condition for an immersed $\Z$-surface  to be isotopic to the standard one.

\begin{theorem}
\label{thm:UnknottingAllSameSignIntro}
Assume that~$S$ is either a closed immersed~$\Z$-surface in~$S^4$ or a immersed~$\Z$-surface in~$D^4$ with boundary an Alexander polynomial one knot~$K$.
The following assertions are equivalent:
\begin{itemize}
\item $S$ is isotopic to the standard genus $g$ surface with~$c_+$ positive double points and~$c_-$ negative double points.  In the non-closed case,  this isotopy can be taken to be rel.\ boundary. 
\item The equivariant intersection form of the exterior of~$S$ is isometric to
$$\lambda_{c_+,c_-} \oplus \mathcal{H}_2^{\oplus g}:=((t-1)(t^{-1}-1))^{\oplus c_+} \oplus (-(t-1)(t^{-1}-1))^{\oplus c_-} \oplus \bsm 0&t-1\\ t^{-1}-1&0 \esm^{\oplus g}.$$
\end{itemize}
\end{theorem}

It is conceivable that the surface exterior $X_S$ could necessarily have equivariant intersection form isometric to~$\lambda_{c_+,c_-} \oplus \mathcal{H}_2^{\oplus g}$ but we are not able to prove this in general (see Theorem~\ref{thm:UniqueSpherec=1Intro} for spheres and disks with a single double point). 
In the embedded case ($c_+=0=c_-)$,  it is known that for~$g \geq 3$,  the equivariant intersection form of $X_S$ is necessarily~$\mathcal{H}_2^{\oplus g}$~\cite[Theorem 7.4]{ConwayPowell}.

When the immersed surfaces are genus zero and have a single double point, we argue that only the standard equivariant intersection form can arise, allowing us to strengthen Theorem~\ref{thm:UnknottingAllSameSignIntro}.

\begin{theorem}
\label{thm:UniqueSpherec=1Intro}
~
\begin{itemize}
\item  Every~$\Z$-sphere in~$S^4$ with a single double point is isotopic to the standard ~$\Z$-sphere with a single double point of the same sign.
\item Given an Alexander polynomial one knot~$K$,  every~$\Z$-disk in~$D^4$ with boundary~$K$ and a single double point is isotopic rel.\ boundary to the standard~$\Z$-disk with boundary~$K$ and a single double point of the same sign.
\end{itemize}
\end{theorem}

Next we introduce some terminology needed to generalize this result to 
 knots with nontrivial Alexander polynomial.
Given a polynomial~$\Delta$, consider the following group of \emph{unitary units}:
$$U(\Delta)=\left\lbrace x(t) \in \Z[t^{\pm 1}]/\Delta \ \Big| \ x(t)x(t^{-1})=1   \right\rbrace.$$
A rapid verification shows that~$\{t^k \}_{k \in \Z}$ acts on this group by multiplication.

Our classification of~$\Z$-disks with a single double point now reads as follows.

\begin{theorem}
\label{thm:g=0c=1BoundaryD4Intro}
Let~$K$ be a knot and $\varepsilon \in \{\pm 1 \}$. 
%T and  let~$(c_+,c_-) \in \{(0,1),(0,1)\}$. 
The following assertions are equivalent:
\begin{enumerate}
\item the Blanchfield form $\Bl_K$ is presented by $(\varepsilon \Delta_K)$;
%\item the Blanchfield form~$\Bl_{P_{K,0}(c_+,c_-)}$ is presented by~$\varepsilon(t-1)(t^{-1}-1) \Delta_K$;
\item the set of rel.\ boundary isotopy classes of~$\Z$-disks in $D^4$ with boundary~$K$ and a single double point of sign~$\varepsilon$  is nonempty and corresponds bijectively to~$U(\Delta_K)/\{ t^k \}_{k \in \Z}.$
\end{enumerate}
\end{theorem}

Section~\ref{sub:BlanchfieldIntro} contains some background on Blanchfield forms, but presently we simply note that~$\Bl_K$ is \emph{presented} by a size $n$ hermitian matrix~$A$ with~$\det(A) \neq 0$ and entries in~$\Z[t^{\pm 1}]$ if~$\Bl_K$ is isometric to the following linking form:
\begin{align}
\label{eq:presents}
\Z[t^{\pm 1}]^n/A^T\Z[t^{\pm 1}]^n \times \Z[t^{\pm 1}]^n/A^T\Z[t^{\pm 1}]^n  &\to \Q(t)/\Z[t^{\pm 1}] \\
([x],[y]) &\mapsto -x^TA^{-1}\overline{y}. \nonumber
\end{align}
Also,  in Theorem~\ref{thm:g=0c=1BoundaryD4Intro} and in what follows,  $\Delta_K$ denotes the symmetric representative of the Alexander polynomial of $K$ with~$\Delta_K(1)=1$.

\begin{remark}
\label{rem:InfinitebAut}
The group~$U(\Delta_K)/\{ t^k \}_{k \in \Z}$ can be infinite.
Indeed as proved in~\cite[Theorem~1.7]{ConwayDaiMiller}, if the polynomial~$\Delta_K$ is quadratic and so
% up to multiplication by $\pm t^\ell$
of the form~$\Delta_n:=-nt+(2n+1)-nt^{-1}$ for some $n \in \Z$,
%AC: In CDM, we had written $\Delta_n:=nt-(2n+1)+nt^{-1}$ so that the index matched the leading term but it  doesn't matter for the calculation because  their  related by multiplication by $-1.$ 
 the group~$U(\Delta_n)/\{ t^k \}_{k \in \Z}$ is finite if and only if~$n = 2, 1, 0, -1$, or~$-p^k$ for a prime~$p$ and, in these cases, 
$$
U(\Delta_n)/\{t^k\}_{k \in \Z} 
\cong \begin{cases}
\lbrace 1 \rbrace & \quad \text{ if $n=-1,0$},  \\
\Z/2\Z & \quad \text{ if $n=1,2$ or $n=-p^k$ with $k$ odd},  \\
\Z/4\Z & \quad \text{ if $n=-p^k$ with $k$ even}. \\
\end{cases}
$$
If $n=0$, then $\Delta_0=1$ and so Theorem~\ref{thm:g=0c=1BoundaryD4Intro} accords with Theorem~\ref{thm:UniqueSpherec=1Intro}.
Theorem~\ref{thm:g=0c=1BoundaryD4Intro} applied to the case~$n=-1$ implies that the right-handed trefoil knot $K_{-1}$ bounds a unique immersed~$\Z$-disk in~$D^4$ with a single positive double point. 
This disk can be taken to be smoothly immersed and arises from the fact that $K_{-1}$ can be unknotted by changing a single positive crossing.
We refer to~\cite[end of Section~6]{ChaOrrPowell} for a discussion relating unknotting sequences to smooth $\Z$-disks.
%%Don't delete just yet.
%{; this disk can be realized smoothly by using crossing changes, see e.g.  the discussion in~\cite[End of Section~6]{ChaOrrPowell}.}
In general the polynomial~$-nt+(2n+1)-nt^{-1}$ arises as the Alexander polynomial of the twist knot~$K_n$ illustrated in~\cite[Figure 2]{ConwayDaiMiller}.
Each $K_n$ can be unknotted by changing a single positive crossing and therefore bounds an immersed $\Z$-disk in $D^4$ with a single positive double point.
Thus, for example,  Theorem~\ref{thm:g=0c=1BoundaryD4Intro} implies that the knot~$K_6$ bounds infinitely many immersed $\Z$-disks in~$D^4$ with a single positive double point.
This result should be contrasted with the embedded case: every knot bounds at most one embedded~$\Z$-disk in $D^4$~\cite{ConwayPowellDiscs}.
One can always remove double points from immersed discs by blowing up with $\C P^2$. 
In fact,  the number of immersed $\Z$-disks in~$D^4$ with boundary $K$ and a single positive double point obtained in Theorem~\ref{thm:g=0c=1BoundaryD4Intro} is identical to the number of embedded $\Z$-disks in $\C P^2 \setminus \intt (B^4)$ with boundary~$K$~\cite{ConwayDaiMiller},  showing that, at least in a simple case,  this operation gives a bijective correspondence.
\end{remark}

In the case of immersed surfaces with two or more double points, our (non-) uniqueness results are in general less definitive because we are not able to control the equivariant intersection form of the surface exterior.
We nevertheless obtain the following existence result.

\begin{theorem}
\label{thm:Unknotting}
Let $c_+,c_- \geq 0$ be integers and set $c:=c_++c_-.$
Given a knot $K$,  the following assertions are equivalent:
\begin{enumerate}
\item $K$ bounds a $\Z$-disk in $D^4$ with $c_+$ positive double points and $c_-$ negative double points;
\item $K$ can be converted into an Alexander polynomial one knot via $c_+$ positive-to-negative crossing changes and $c_-$ negative-to-positive crossing changes;
\item $K$ bounds an embedded $\Z$-disk in $(\#_{c_+}  \C P^2 \#_{c_-}  \overline{\C P^2})\setminus \intt(B^4)$;
\item $\Bl_K$ is presented by a size $c$ nondegenerate hermitian matrix $A$ with~$A(1)=(1)^{\oplus c_+} \oplus~(-1)^{\oplus c_-}$.
\end{enumerate}
\end{theorem}

%Here, a positive (resp. negative) crossing change consists of switching a positive (resp. negative) crossing to a negative (resp. positive) one.
The equivalence $(4) \Leftrightarrow (2)$ is due to Borodzik-Friedl~\cite{BorodzikFriedlLinking},  and is seen without too much difficulty to imply $(4) \Leftrightarrow (3) \Leftrightarrow (2)$.
Our contribution is the addition of $(1)$. 

When $c_+=c_-$,  Feller and Lewark have proved that that the last three conditions of Theorem~\ref{thm:Unknotting} are equivalent to~$K$ bounding an embedded $\Z$-surface in $D^4$ of genus $2c_+$~\cite{FellerLewarkBalanced}.
Informally speaking,  while pairs of double points can always be resolved at the price of genus,  Theorem~\ref{thm:Unknotting} shows that for~$\Z$-surfaces, the converse is true as well: genus can be converted into double points.

\subsection{Results in other~$4$-manifolds}
\label{sub:Other4ManifoldsIntro}

Our theory applies to surfaces in~$4$-manifolds other than~$S^4$ and~$D^4$.
For example,  we obtain the following result for immersed $\Z$-surfaces in arbitrary simply connected $4$-manifolds.

\begin{theorem}
\label{thm:Other4ManifoldsSpheresIntro}
Let~$X$ be a closed simply-connected~$4$-manifold, and let~$c_+,c_-$ be positive integers.
The number of equivalence classes of closed~$\Z$-surfaces in~$X$ with~$c_+$~positive double points, ~$c_-$ negative double points and whose exteriors have the same equivariant intersection form is
\begin{itemize}
\item one if~$(c_+,c_-) \in \{(1,0),(0,1)\}$,
\item at most two if~$(c_+,c_-)=(1,1)$,
\item finite if~$(c_+,c_-) \in \{(c,0),(0,c)\}$.
\end{itemize}
For $N$ a simply-connected $4$-manifold with boundary $S^3$, the same statement holds for rel.\ boundary equivalence classes of $\Z$-surfaces in $N$ with boundary a fixed Alexander polynomial one knot.
\end{theorem}

We restate the first statement for emphasis: closed immersed~$\Z$-surfaces in~$X$ with a single double point of a given sign are equivalent if and only if their exteriors have isometric equivariant intersection forms; the analogous statement holds for equivalence rel boundary of~$\Z$-surfaces with boundary an Alexander polynomial one knot.
We note that in the proof of the third item of Theorem~\ref{thm:Other4ManifoldsSpheresIntro}, we obtain an explicit upper bound for the number of~$\Z$-surfaces
% in~$X$ 
with~$c_+$ positive double points, ~$c_-$ double points and whose exteriors have the same equivariant intersection form when $(c_+,c_-) \in \{(c,0),(0,c)\}$,  though we do not expect that bound to be sharp.
Remark~\ref{rem:InfinitebAut} illustrates that the finiteness results of Theorem~\ref{thm:Other4ManifoldsSpheresIntro} fail in general when~$K$ has nontrivial Alexander polynomial,  even when $c=1$.

\begin{remark}
\label{rem:DeterminedByIntForm}
Whereas closed embedded $\Z$-surfaces in closed simply-connected $4$-manifolds are determined by the equivariant intersection form of their exteriors~\cite[Theorem 1.4]{ConwayPowell}, we have not been able to prove this in general in the presence of double points.
We discuss some of the challenges in Remark~\ref{rem:WhyItsHard} below. 

Despite the lack of a general uniqueness result, Theorem~\ref{thm:Other4ManifoldsSpheresIntro} could nonetheless in principle be used to produce exotica: this would require constructing large families of smoothly distinct surfaces whose exteriors all have the same equivariant intersection form,  where here `large' means more than specified in (the proof of) Theorem~\ref{thm:Other4ManifoldsSpheresIntro}.
\end{remark}

We note two additional applications.
Firstly,  Theorem~\ref{thm:StandardNotS4} records an analogue of Theorem~\ref{thm:UnknottingAllSameSignIntro} for arbitrary simply-connected~$4$-manifolds that are either closed or have boundary $S^3$: the equivariant intersection form detects whether an immersed $\Z$-surface is equivalent to a local copy of the standard immersed $\Z$-surface.
Secondly, on the topic of disks with a single double point, we note that Theorem~\ref{thm:g=0c=1Boundary} contains a generalization of Theorem~\ref{thm:g=0c=1BoundaryD4Intro} to arbitrary simply connected~$4$-manifolds with boundary $S^3$. 

\subsection{Main classification results}
\label{sub:ClassificationIntro}

We state our main classification theorems and discuss how, together with some algebra,  they lead to the results stated up to this point.

Let~$N$ be a simply-connected~$4$-manifold with boundary~$\partial N \cong S^3$, and let~$K \subset \partial N$ be a knot.
Our goal is to describe the set of rel.\ boundary equivalence classes of immersed~$\Z$-surfaces in~$N$ with boundary~$K$ whose exteriors have a fixed equivariant intersection form.
In other words, we wish to understand the following set:
\begin{align*}
\operatorname{Surf}^0_\lambda(g;c_+,c_-)(K,N) :=
\frac{
\{ \Z\text{-surface } S \subset N  \mid \partial S=K, \ g(S)=g, \ \text{$(c_+,c_-)$ double points, } \lambda_{N_S}\cong \lambda \}}
{\text{equivalence rel.\ boundary}}.
\end{align*}
Here $g(S)$ refers to the genus of $S$ and we say that $S$ has $(c_+,c_-)$ double points if it has $c_+$ positive and $c_-$ negative double points. 
Also,  $N_S$ denotes the exterior of~$S$ and $\lambda_{N_S}$ refers to its equivariant intersection form.\

In the closed case,  we fix a simply-connected closed~$4$-manifold~$X$ and consider
\begin{align*}
\operatorname{Surf}_\lambda(g;c_+,c_-)(X)
&:=
\frac{ \{\text{closed } \Z\text{-surface } S \subset X \mid  g(S)=g, \ \text{$(c_+,c_-)$-double points, }  \lambda_{X_S}\cong \lambda\}}{\text{equivalence}}.
\end{align*}
Roughly speaking, our description of these sets of surfaces will be in terms of quotients of the automorphism group of the Blanchfield pairing of 
a 3-manifold that is homeomorphic to the boundary of an immersed surface exterior.  Let~$P_g(c_+,c_-)$ be the $3$-manifold obtained from~$\Sigma_{g,1} \times S^1$ by~$c_+$ positive plumbings and~$c_-$ negative plumbings and, given a knot $K$,  let~$P_{K,g}(c_+,c_-)$ be the closed~$3$-manifold obtained as the union of $P_g(c_+,c_-)$ with the knot exterior~$E_K:=S^3 \setminus \nu(K)$ (for more details, see Section~\ref{sec:PlumbedManifolds}).
If~$K$ bounds an immersed genus~$g$ surface in~$N$ with~$c_+$ positive double points and~$c_-$ negative double points, then the boundary of its exterior is homeomorphic to~$P_{K,g}(c_+,c_-)$.
%Recall that~$P_g(c_+,c_-)$ denotes the $3$-manifold obtained from~$\Sigma_{g,1} \times S^1$ by~$c_+$ positive plumbings and~$c_-$ negative plumbings and that, given a knot $K$,  we write~$P_{K,g}(c_+,c_-)$ for the closed~$3$-manifold obtained as the union of~$P_g(c_+,c_-)$ with the knot exterior~$E_K$.
%When~$K=U$ is the unknot,~$P_{U,g}(c_+,c_-)$ can equivalently be described as the closed~$3$-manifold obtained from~$\Sigma_g \times S^1$ by~$c_+$ positive plumbings and~$c_-$ negative plumbings.
%AC: I don't think it should be "$c_-$ negative"? (aka no hyphen)
We write
$$\Bl_{P_{K,g}(c_+,c_-)} \colon H_1(P_{K,g}(c_+,c_-);\Z[t^{\pm 1}]) \times H_1(P_{K,g}(c_+,c_-);\Z[t^{\pm 1}])\to \Q(t)/\Z[t^{\pm 1}]$$
 for the Blanchfield form of~$P_{K,g}(c_+,c_-)$,  deferring details to
Section~\ref{sub:BlanchfieldIntro}.

%\begin{notation}
%\label{notation:Intro}
 When there is no risk for confusion we write 
 $$P:=P_g(c_+,c_-) \ \ \quad \text{and} \quad \ \ P_K:=P_{K,g}(c_+,c_-)$$
but, in theorem statements, we use the full notation to emphasize the dependence on $c_+,c_-$ and~$g$.
 
%\end{notation}

The statements of our classifications differ depending on whether the surfaces are closed or have a knot as their boundary,  but much of the underlying notation is similar.
Indeed in both cases, we will use that if a $\Z[t^{\pm 1}]$-valued hermitian form $\lambda$ presents the Blanchfield form $\Bl_{P_K}$, then the group $\Aut(\lambda)$ of isometries of $\lambda$ acts on $\Aut(\Bl_{P_K})$, the group of isometries of $\Bl_{P_K}$.
This action, which involves the boundary linking form $\partial \lambda$ and played a prominent r\^ole in~\cite{ConwayPowell,ConwayPiccirilloPowell},  will be described in detail in Section~\ref{sec:Classification}. 
There is another action on $\Aut(\Bl_{P_K})$ by a certain subgroup of $\Homeo^{+}(\Sigma_g)$ (or $\Homeo(\Sigma_{g,1}, \partial)$ in the rel.\  boundary case).  While we defer the details until Section~\ref{sec:Classification},  this subgroup,  which we will denote $\Homeo_{\alpha}(\Sigma_g)$ (or $\Homeo_\alpha(\Sigma_{g,1}, \partial)$  in the rel.\  boundary case),  essentially consists of those homeomorphisms that pointwise fix prescribed neighborhoods of chosen points in $\Sigma$ that will become transverse double points under our immersions.  
%Roughly,  
These homeomorphisms extend to homeomorphisms of $P_U$
(or $P_K= P \cup_{\partial} E_K$ in the rel.\ boundary case) that lift to the infinite cyclic cover,  where their action on first homology gives the desired isometry.

One last piece of notation: if $(H,\lambda)$ is a hermitian form on a finitely generated free~$\Z[t^{\pm 1}]$-module, then~$\lambda(1)$ denotes the symmetric bilinear form~$(H,\lambda) \otimes_{\Z[t^{\pm 1}]} \Z[t^{\pm1}]/(t-1)$.
More concretely, if $\lambda$ is represented by the matrix $A(t)$, then $\lambda(1)$ is represented by $A(1).$

Our main classification of closed immersed $\Z$-surfaces,  which is a combination of Theorem~\ref{thm:SurfacesClosed} and Proposition~\ref{prop:SimplifyAlgebraClosed}, now reads as follows,  where $Q_X$ denotes the intersection form of  $X$.

\begin{theorem}
\label{thm:SurfacesClosedIntro}
Let~$X$ be a closed simply-connected~$4$-manifold, and let $c_+,c_-$ and $g$ be non-negative integers.
Given a nondegenerate hermitian form~$\lambda$ over~$\Z[t^{\pm 1}]$, the following assertions are equivalent:
\begin{enumerate}
\item
the hermitian form~$\lambda$ presents~$\Bl_{P_{U,g}(c_+,c_-)}$ and satisfies~$\lambda(1)\cong Q_X \oplus  (0)^{\oplus 2g+c_++c_-}$;
\item the set~$\operatorname{Surf}_\lambda(g;c_+,c_-)(X)$ is nonempty and there is a bijection
$$\operatorname{Surf}_\lambda(g;c_+,c_-)(X) \approx
\frac{\Aut(\Bl_{P_{U,g}(c_+,c_-)})}{\Aut(\lambda) \times \Homeo_\alpha(\Sigma_g)}.
$$
\end{enumerate}
\end{theorem}

The results stated in Theorem~\ref{thm:UnknottingAllSameSignIntro},~\ref{thm:UniqueSpherec=1Intro} and~\ref{thm:Other4ManifoldsSpheresIntro} are obtained by analyzing the algebra involved in Theorem~\ref{thm:SurfacesClosedIntro}.
This involves analyzing the Blanchfield form $P_U$ and deciding which of its isometries arise from surface homeomorphisms and isometries of $\lambda$.
These results in $X=S^4$ involve ambient isotopies instead of equivalences thanks to Quinn's result that every orientation-preserving homeomorphism of $S^4$ is isotopic to the identity~\cite{Quinn}.

Before moving on to surfaces with nonempty boundary, we make a note on the embedded case i.e.  on the case where~$c_+=0=c_-$.

\begin{remark}
\label{rem:ClosedEmbedded}
If $c_+=0=c_-$, then $P_{U,g}(c_+,c_-)=\Sigma_g \times S^1$ and $\Homeo_\alpha(\Sigma_g)=\Homeo^+(\Sigma_g)$.
% and the embedding $\alpha$ disappears.
Thus,  if the set of embedded $\Z$-surfaces, which in this case we denote $\Surf_\lambda(g)(X)$, is nonempty then the bijection of Theorem~\ref{thm:SurfacesClosedIntro} takes the form
$$\Surf_\lambda(g)(X) \approx \frac{\Aut(\Bl_{\Sigma_g \times S^1})}{\Aut(\lambda) \times \Homeo^{+}(\Sigma_g)}.$$
The target of this bijection is trivial~\cite[Proposition 5.6]{ConwayPowell} so, in the embedded case, Theorem~\ref{thm:SurfacesClosedIntro} reduces to~\cite[Theorem 6.10]{ConwayPiccirilloPowell}: a nondegenerate hermitian form $\lambda$ presents~$\Bl_{\Sigma_g \times S^1}$ and satisfies~$\lambda(1) \cong Q_X \oplus (0)^{\oplus 2g}$ if and only if $\Surf_\lambda(g)(X)$ is a singleton.
In particular,  closed embedded~$\Z$-surfaces are determined by the equivariant intersection forms of their exteriors~\cite[Theorem~1.4]{ConwayPowell}.
\end{remark}

\begin{remark}
\label{rem:WhyItsHard}
In Remark~\ref{rem:DeterminedByIntForm}, we noted that we were not able to prove that closed immersed~$\Z$-surfaces are determined by the equivariant intersection forms of their exteriors.
Theorem~\ref{thm:SurfacesClosedIntro} encapsulates the challenge: deciding whether or not~$\Aut(\Bl_{P_{U,g}(c_+,c_-)})/(\Aut(\lambda) \times \Homeo_\alpha(\Sigma_g))$ is trivial for \emph{every} hermitian form $\lambda$ that presents $\Bl_{P_{U,g}(c_+,c_-)}$.
In the embedded case, this was possible because~\cite[Proposition 5.6]{ConwayPowell} actually shows that~$\Aut(\Bl_{\Sigma_g \times S^1})/\Homeo^{+}(\Sigma_g)$ is trivial.
%During the proof of Theorem~\ref{thm:Other4ManifoldsSpheresIntro},  we show
We show in Remark~\ref{rem:WhyItsHardMain} that the corresponding result does not generally hold in the presence of double points: for $g=0$ and $(c_+,c_-)=(1,0)$, 
$$\Bigg| \frac{\Aut(\Bl_{P_{U,0}(1,0)})}{\Homeo_\alpha(\Sigma_0)} \Bigg|=2.$$
In fact,  we conjecture that this set is trivial if and only if $c_++c_-=0$. 
We also conjecture that for $(c_+,c_-)=\{(0,c),(c,0)\}$ it contains $2^cc!$ elements; for $(c_+,c_-)=(1,1)$ it contains 2 elements; and otherwise it is infinite.
\end{remark}

We move on to surfaces with boundary.
Proposition~\ref{prop:AutBlPK} shows that the inclusions $P \subset P_K$ and~$E_K \subset P_K$ induce an isometry $\Aut(\Bl_P) \oplus \Aut(\Bl_K) \cong \Aut(\Bl_{P_K})$,  where $\Bl_K$ denotes the Blanchfield form of $E_K$.
We will see that the action of~$\Homeo_\alpha(\Sigma_{g,1},\partial)$ on~$\Aut(\Bl_{P_K})$ is trivial on $\Aut(\Bl_K)$ and so restricts to an action on $\Aut(\Bl_P)$.

Our main classification of immersed $\Z$-surfaces with boundary,  which is a combination of Theorem~\ref{thm:SurfacesRelBoundary} and Proposition~\ref{prop:SimplifyAlgebraBoundary},  now reads as follows.

\begin{theorem}
\label{thm:SurfacesRelBoundaryIntro}
Let~$N$ be a simply-connected~$4$-manifold with boundary~$\partial N \cong S^3$,  let~$K \subset S^3$ be a knot, and let $c_+,c_-$ and $g$ be non-negative integers.
%Set $c:=c_++c_-$.
Given a nondegenerate hermitian form~$\lambda$ over~$\Z[t^{\pm 1}]$, the following assertions are equivalent:
\begin{enumerate}
\item
the hermitian form~$\lambda$ presents~$\Bl_{P_{K,g}(c_+,c_-)}$ and satisfies~$\lambda(1)\cong Q_N \oplus  (0)^{\oplus 2g+c_++c_-}$;
\item the set~$\operatorname{Surf}_\lambda^0(g;c_+,c_-)(N,K)$ is nonempty and there is a bijection
$$\operatorname{Surf}_\lambda^0(g;c_+,c_-)(N,K) \approx
  \frac{(\Aut(\Bl_{P_g(c_+,c_-)})/\Homeo_\alpha(\Sigma_{g,1},\partial))\times \Aut(\Bl_K)}{\Aut(\lambda)}.
  $$
\end{enumerate}
\end{theorem}

The results stated in Theorem~\ref{thm:UnknottingAllSameSignIntro},~\ref{thm:UniqueSpherec=1Intro} and~\ref{thm:g=0c=1BoundaryD4Intro} are obtained by analyzing the target of the bijection from Theorem~\ref{thm:SurfacesRelBoundaryIntro}.
The Alexander trick is what makes it possible for these results to involve ambient isotopies instead of equivalences when $N=D^4$.
The key step for obtaining the~$(1) \Leftrightarrow (4)$ equivalence of Theorem~\ref{thm:Unknotting} from Theorem~\ref{thm:SurfacesRelBoundaryIntro} is to show that~$\Bl_{P_{K,0}(c_+,c_-)}$ is presented by a size~$c:=c_++c_-$ matrix~$B$ with~$B(1)=(0)^{\oplus c}$ if and only if~$\Bl_K$ is presented by a size~$c$ matrix~$A$ with~$A(1)=(1)^{\oplus c_+} \oplus (1)^{\oplus c_-}$.
In particular,  when~$g=0$,  the condition in Theorem~\ref{thm:SurfacesRelBoundaryIntro} ensuring the existence of a~$\Z$-surface (with the condition on the specific~$\lambda$ dropped) can be expressed solely in terms of~$\Bl_K$; the same is known to be true when~$c=0$~\cite{ConwayPiccirilloPowell,FellerLewarkBalanced}.
We expect the same to be true for arbitrary~$c_+,c_-$ and~$g$.
% but will not pursue this further.

We also note that when $c_+=0=c_-$, Theorem~\ref{thm:SurfacesRelBoundaryIntro} specializes to the classification of embedded~$\Z$-surfaces stated in~\cite[Theorem~6.2 and Remark~6.3]{ConwayPiccirilloPowell}.
It also is possible to generalize Theorem~\ref{thm:SurfacesRelBoundaryIntro} leading to a result with the rel.\ boundary clause dropped.
%but we will not purse this further.

\begin{remark}
The bijections from Theorems~\ref{thm:SurfacesClosedIntro} and~\ref{thm:SurfacesRelBoundaryIntro} are not canonical.
Indeed, as explained in Section~\ref{sub:Simplifications}, they depend on the choice of an isometry of the boundary linking form~$\partial \lambda$ with the relevant Blanchfield form.
Sections~\ref{sub:SurfacesRelBoundary} and~\ref{sub:SurfacesClosed} state our classification results in a more canonical way,  where the bijection is explicit, but as the drawback,  the algebra becomes more challenging to parse and the connection to the results stated in Sections~\ref{sub:IntroD4S4} and~\ref{sub:Other4ManifoldsIntro} becomes less clear.
\end{remark}

The strategy underlying the proofs of Theorems~\ref{thm:SurfacesClosedIntro} and~\ref{thm:SurfacesRelBoundaryIntro} is described in Section~\ref{sec:ProofStrategy} but here is a brief outline in the case with boundary when for simplicity we assume $\lambda$ is even.
Very roughly, the idea is to show that mapping a surface to its exterior determines a bijection from~$\operatorname{Surf}_\lambda^0(g;c_+,c_-)(N,K)$ to a set $\mathcal{V}_\lambda^0(P_{K,g}(c_+,c_-))$ of $\Z$-fillings of $P_{K,g}(c_+,c_-)$ and to then apply~\cite[Theorem 1.1]{ConwayPiccirilloPowell} to relate this latter set to isometries of the Blanchfield form~$\Bl_{P_{K,g}(c_+,c_-)}$.
The idea is modeled on~\cite[Section 5]{ConwayPiccirilloPowell}, which treats the embedded case.  The generalization to immersed surfaces requires several additional ideas as well as careful bookkeeping of the choices involved in defining the aforementioned bijection.
The main challenge stems from the fact that, contrarily to the embedded case,  the data of $g,c_+,c_-$, and~$K$ does not uniquely determine~$P_{K,g}(c_+,c_-)$: specifying the location of the plumbings requires we fix an embedding~$\alpha \colon \bigsqcup_{2(c_+ + c_-)} D^2 \hookrightarrow \intt(\Sigma_{g,1})$.
We  then need to 
%simultaneously
consider immersions that are suitably compatible with $\alpha$ while simultaneously ensuring that we obtain a bijection with~$\operatorname{Surf}_\lambda^0(g;c_+,c_-)(N,K)$, which, as a set,  is independent of~$\alpha.$

Along the way, we prove the following result that might be of independent interest.
\begin{theorem}
\label{thm:HopfConcordancesEquivalent}
Let~$L$ be a~$2$-component link with multivariable Alexander polynomial~$1$.
Any two~$\Z^2$-concordances between~$L$ and the Hopf link $H$ are equivalent rel.\  boundary.
\end{theorem}
For context, Davis proved that any~$2$-component link with multivariable Alexander polynomial~$1$ is concordant to the Hopf link~\cite{Davis},  which should be thought of as a $\Z^2$-analogue of Freedman's theorem that Alexander polynomial one knots are $\Z$-slice~\cite{Freedman}.
Theorem~\ref{thm:HopfConcordancesEquivalent} is then a $\Z^2$-analogue of the fact that Alexander polynomial one knots bound a single $\Z$-disk up to isotopy rel.\ boundary~\cite{ConwayPowellDiscs}.
Our result is stated up to equivalence instead of isotopy because we do not know whether a given $\Z^2$-concordance from $L$ to $H$ is isotopic rel.\ boundary to the concordance obtained by applying a Dehn twist to $S^3 \times [0,1]$.
The proof of Theorem~\ref{thm:HopfConcordancesEquivalent}, which is given in the appendix, is similar to the proofs of the results in~\cite{ConwayPowellDiscs,ConwayDiscs}.
\color{black}

\begin{remark}
We conclude with a note concerning the isotopy classification.
Relying on the work of Quinn~\cite{Quinn} and Orson-Powell~\cite{OrsonPowell}, a criterion for improving our results from equivalence to isotopy can be formulated in a similar way to~\cite[Theorem 1.4]{ConwayPowell} and~\cite[Theorem~G]{OrsonPowell}.
This will not be pursued further,  nor will be the question of developing the algebra needed to describe the isotopy analogues of the sets $\operatorname{Surf}_\lambda^0(g;c_+,c_-)(N,K)$ and $\operatorname{Surf}_\lambda(g;c_+,c_-)(X)$.
\end{remark}

\subsection*{Organization}

This article decomposes into two main parts.
Sections~\ref{sec:PlumbedManifolds}-\ref{sec:Unknotting} set up the preliminaries needed to state our classification results and apply them to obtain the theorems stated in Sections~\ref{sub:IntroD4S4} and~\ref{sub:Other4ManifoldsIntro}.
Sections~\ref{sec:ProofStrategy}-\ref{sec:ProofClassifications} focus on the proofs of the classification results.

More specifically,  Section~\ref{sec:PlumbedManifolds} introduces the manifolds $P$ and $P_K$ in more detail,  and Section~\ref{sec:HomologyZ} calculates their homology. 
Section~\ref{sec:IsoBlPk} is concerned with the Blanchfield form of $P_K$.
Section~\ref{sec:Classification} precisely states our main classification results,  though we defer their proofs to Section~\ref{sec:ProofClassifications}.
Section~\ref{sec:EquivariantIntersectionForm} focuses on equivariant intersection forms, while Sections~\ref{sec:BlP} and~\ref{sec:BlPHomeomorphisms} center on the isometries of $\Bl_P$.
Section~\ref{sec:Enumerating} applies our classification results to obtain the theorems stated in Section~\ref{sub:IntroD4S4} and~\ref{sub:Other4ManifoldsIntro},  and Section~\ref{sec:Unknotting} proves Theorem~\ref{thm:Unknotting}.

Section~\ref{sec:ProofStrategy} outlines the main steps of the proofs to come.
Section~\ref{sec:ImmersionsAndNormalBundles} discusses immersions and their normal bundles.
Section~\ref{sec:BoundarIdentifications} describes how to identify the boundary of the exterior of an immersed surface with $P_K$.
Section~\ref{sec:MainStatement} establishes a relation between the sets of immersed~$\Z$-surfaces and sets of $\Z$-fillings of $P_K$.
Section~\ref{sec:ProofClassifications} proves our main classification results, namely Theorems~\ref{thm:SurfacesClosedIntro} and~\ref{thm:SurfacesRelBoundaryIntro}.

Finally, Section~\ref{sec:OpenQuestions} lists some open questions, and the appendix proves Theorem~\ref{thm:HopfConcordancesEquivalent}.

\subsection*{Acknowledgments}
AC was partially supported by the NSF grant DMS~2303674.  ANM was partially supported by an AMS Simons Research Enhancement Grant for PUI Faculty,  and thanks the Department of Mathematics at the University of Texas at Austin for their hospitality during her visit of fall 2024, when this paper was completed. 

\subsection*{Conventions}

We work in the topological category with locally flat immersions.
All manifolds are assumed to be compact, connected, based,  and oriented; if a manifold has a nonempty, connected boundary, then its basepoint is assumed to be in the boundary.
The singular sets of our immersions are assumed to consist only of transverse double points in the interior of the surface.

We write~$\Sigma_g$ for the closed genus~$g$ surface and~$\Sigma_{g,1}$ for the compact genus~$g$ surface with one nonempty boundary component.
Unless mentioned otherwise, the normal bundle of an immersion~$e$ is denoted~$\nu(e)$ and its disk bundle is~$\overline{\nu}(e)$.
A tubular neighborhood of an immersed surface $S$ is denoted~$\overline{\nu}(S)$.
%If~$S$ is an immersed surface, we write~$\overline{\nu}(S)$ for a tubular neighborhood of~$S$ and~$\intt(\overline{\nu}(S))$ for its interior.}
Finally,  all hermitian forms will be defined on finitely generated free modules.

%\newpage
%\tableofcontents

\section{Plumbed manifolds}
\label{sec:PlumbedManifolds}

%Given an immersed surface $S \subset N$ with boundary $K$,  
%%AC: Commented this because it made it seems like the manifolds depended on $E,W$ and $P.$

In this section, we construct several manifolds:~$E^4$, a model for the normal bundle of an immersed surface with boundary~$S$;~$W^4$,  a model for a tubular neighborhood of~$S$; and~$P^3$,  a model such that the boundary of the exterior of~$S$ is homeomorphic to~$P \cup E_K$, where~$K=\partial S.$
While~$P$ is quickly described as the result of plumbing the $3$-manifold~$\Sigma_{g,1} \times S^1$, several of our later calculations require a more granular approach.
%As noted in the introduction,   each of the manifolds we define
% ($E$, $W$, $P:=P_g(c_+,c_-)$,  and~$P_K:=P_{K,g}(c_+,c_-)$) 
%depend on non-negative integers~$c_+,c_-$, and~$g$, but we will avoid including these indices 
%in the notation 
%as much as possible.

\begin{notation}\label{notation:abstractmodels}
In this section,  we fix the following: 
\begin{itemize}
\item a genus~$g$ surface $\Sigma:=\Sigma_{g,1}$ with a single nonempty boundary component;
\item integers $c_+,c_- \geq 0$ and $c:=c_++c_-$;
\item an embedding $\alpha =\bigsqcup \alpha_k \colon \bigsqcup_{2c} D^2 \hookrightarrow \operatorname{int}(\Sigma)$.  
\end{itemize}
We omit this data from the notation, but the reader should keep in mind these dependencies.  

We define $B_k:= \alpha_k(D^2)$ for $k=1, \dots, 2c$,  and set $\Sigma^{\circ}= \Sigma \smallsetminus \intt \bigcup_{k=1}^{2c} B_k$.  
%The boundary~$\partial \Sigma^\circ$ consists of~$2c+1$ copies of~$S^1$.
 We think of~$B_1,  \dots B_{2c_+}$ as positively labelled disks  and $B_{2c_++1}, \dots, B_{2c}$ as negatively labelled disks. 
\end{notation}

We begin by describing a model for 
%the Euler number $2(c_+-c_-)$
a disk bundle over $\Sigma$.

\begin{construction}[A model for a~$D^2$-bundle over~$\Sigma$]
\label{cons:E}
Consider the manifold
%manifold
$$ E:=\left(\bigsqcup_{k=1}^{2c}\left( D^2 \times D^2 \right)_k \sqcup \Sigma^\circ \times D^2\right)\Bigg/\sim_E$$
where~$\sim_E$ identifies~$x \in (S^1 \times D^2)_k$ with its image under the composition
\[ (S^1 \times D^2)_k \xrightarrow{\eta^{\pm}} S^1 \times D^2 \xrightarrow{\alpha_k|_{S^1} \times \id_{D^2}} \partial B_k \times D^2\subset \Sigma^\circ \times D^2,\]
where   $\eta^{\pm}(\omega,(r,  \theta)):= (\omega, (r,  \theta \pm \omega))$, and we use $\eta^{\pm}$ depending on whether~$B_k$ is a positively or negatively labelled disk. Note that $\eta^{\pm}$ sends~$S^1 \times \{ \pt \} \subset S^1 \times \partial D^2$ to the~$(1,\pm 1)$ curve. 

The assignment
\[(x,y) \mapsto
 \begin{cases}
  \alpha_k(x)  \quad & \text{if } (x,y) \in (D^2 \times D^2)_k \\ 
  x \quad & \text{if } (x,y) \in \Sigma^\circ \times D^2
  \end{cases}\]
 induces a map~$\pi_E \colon E \to \Sigma$ that is the projection map of a~$D^2$-bundle over~$\Sigma$.
 %%Don't delete.
 %(x,y) \sim_E (\alpha_k \times \id) \circ \eta both get mapped to \alpha_k(x).  Indeed use the definition of \eta which leaves the first compo untouched.
When~$\Sigma:=\Sigma_g$ is closed, this construction can be repeated and yields the total space of the Euler number~$2(c_+-c_-)$ disk bundle over $\Sigma_g$.
\end{construction}

Next we describe a manifold homeomorphic to a tubular neighbhorhood of an immersed surface.
\begin{construction}[The plumbed $4$-manifold $W$]
\label{cons:W}
Write~$W$ for the~$4$-manifold obtained from the~$D^2$-bundle~$E$ by~$c_+$ positive plumbings and~$c_-$ negative plumbings along the~$B_i$. 
More precisely, the plumbed~$4$-manifold~$W$ is defined as
$$W:=E/\sim_W,$$
where the only non-reflexive identifications of $\sim_W$ are that for 
 $(z_1,w_1) \in(D^2 \times D^2)_{2i-1}$ and $(z_2, w_2) \in (D^2 \times D^2)_{2i}$ we have
\begin{align}
\label{eq:Sim}
(z_1,w_1) 
 \sim_W
 (z_2,w_2) 
\quad  &\text{ iff }  \quad 
\shrink_{2i-1}(z_1,w_1)=\shrink_{2i}(z_2,w_2),
% \left(z_1,\frac{w_1}{2} \right)= \left( \frac{w_2}{2}, z_2\right) &\quad \text{ if~$B_{2i-1},B_{2i}$ are positive,} \\
%  \left(z_1,\frac{w_1}{2} \right)=- \left( \frac{w_2}{2}, z_2\right) &\quad \text{ if~$B_{2i-1},B_{2i}$ are negative,} \\
% \end{cases}
\end{align}
where $\shrink_k \colon (D^2 \times D^2)_k \to D^2 \times D^2$ is defined by
\begin{align*}
\shrink_k (z,w):=
\begin{cases} (z,  w/2)  & \text{if }k=2i-1 \\
\pm (w/2, z)  &\text{if } k=2i, \end{cases}
\end{align*}
and the $\pm$ sign is the same as the (shared) sign of $B_{2i-1}$ and $B_{2i}$.  This is illustrated in Figure~\ref{fig:plumbingneighborhoods}. 
%%%%%
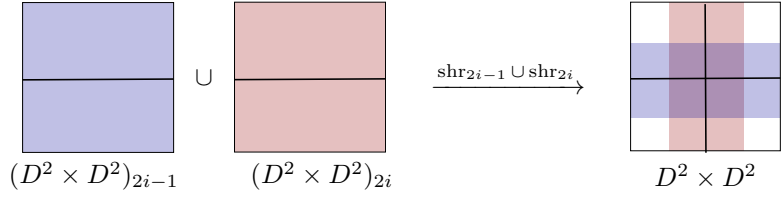
\begin{figure}[htbp!]
\centering
\begin{tikzpicture}
%\draw[step=1cm,color=gray] (0,0) grid (9,2);Uncomment this to get some helpful grid lines
\node[anchor=south west,inner sep=0] at (0,0){\includegraphics[height=2cm]{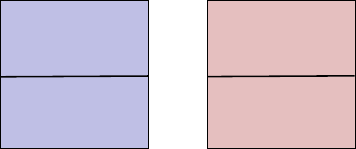}};
\node[anchor=south west,inner sep=0] at (8,0){\includegraphics[height=2cm]{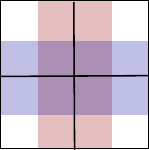}};
\node at (.95,.-.3){$(D^2 \times D^2)_{2i-1}$};
\node at (3.95,.-.3){$(D^2 \times D^2)_{2i}$};
\node at (2.4,  1){$\cup$};
\node at (6.4,  1){$\xrightarrow{\shrink_{2i-1} \cup \shrink_{2i}}$};
\node at (9, -.3){$D^2 \times D^2$};
\end{tikzpicture}
\caption{Quotienting by $\sim_W$ identifies points which have the same image under the shrink maps,  which has the effect of plumbing.}
\label{fig:plumbingneighborhoods}
\end{figure}

In particular there is a quotient map
$\operatorname{proj} \colon E \to W.$
\end{construction}

If $N$ is a $4$-manifold with $\partial N \cong S^3$ and $S \subset N$ is an immersed genus $g$ surface with $c_+$ positive double points and $c_-$ negative double points, then $W$ is homeomorphic to a tubular neighbhorhood~$\overline{\nu}(S)$ of $S$.

Next we describe a manifold that makes up most of the boundary of a tubular neighborhood of a immersed surface.

\begin{construction}[The plumbed $3$-manifold $P$]
Define the $3$-manifold
$$P:=\partial W \setminus \operatorname{proj} (E|_{\partial \Sigma}),$$
so that~$\partial W= \operatorname{proj} (E|_{\partial \Sigma}) \cup_\partial P$.

Note that $\operatorname{proj} (E|_{\partial \Sigma}) \cong \partial \Sigma \times D^2$ is homeomorphic to a solid torus.
We write~$\partial P=\partial \Sigma \times S^1$, omitting the projection map $\operatorname{proj} \colon E \to W$ from the notation.
We also write $\mu \subset \partial P$ for a loop representing a curve of the form $\{ \pt \} \times S^1$.
\end{construction}

We outline how this definition of $P$ aligns with its more familiar description as a manifold obtained by self-plumbing $\Sigma \times S^1.$
\begin{remark}[$P$ as a plumbed $3$-manifold]
%%Don't delete. 
%See the file "AM P as a quotient of Delta x S1 Annotated AC.pdf"
\label{rem:PasaQuotient}
By analyzing $\partial W$, one verifies that 
\begin{align*}
P 
%&=\partial W \setminus \operatorname{proj} (E|_{\partial \Sigma})
= \left( ( \Sigma^\circ \times S^1 ) \sqcup \bigsqcup_{k=1}^{2c} \left\lbrace(z,w) \in (D^2 \times S^1)_k \ \Big| \ |z| \geq \frac{1}{2}\right\rbrace   \right)\Bigg/ \sim,
 \end{align*}
 where $\sim$ identifies 
 \begin{itemize}
\item each~$\partial B_k \times S^1$ with $(S^1 \times S^1)_k$ via $(\alpha_k \times \id) \circ \eta^{\pm}$ for~$k=1, \dots, 2c$,
\item each~$\left\lbrace(z,w) \in (D^2 \times S^1)_{2i-1} \mid |z|=\tmfrac{1}{2}\right\rbrace$ with $\left\lbrace(z,w) \in (D^2 \times S^1)_{2i} \mid |z|=\tmfrac{1}{2}\right\rbrace$ via the homeomorphism~$(z,w) \leftrightarrow \pm(w/2, 2z)$ for $i=1, \dots, c$. 
%each~$\{(z,w) \in (D^2 \times S^1)_{2i}  \mid  |z|=\frac{1}{2}\}$ with $\{(z,w) \in (D^2 \times S^1)_{2i-1}  \mid  |z|=\frac{1}{2}\}$ via the homeomorphism~$\pm(w/2, 2z) \leftrightarrow (z,w) $ for $i=1, \dots, c$. 
%via $(z,w) \leftrightarrow \pm(w/2, 2z)$ for $i=1, \dots, c$ via
\end{itemize}
One then verifies that there is a homeomorphism
$$
\left(\overbrace{\left\lbrace(z,w) \in (D^2 \times S^1)_{2i-1} \  \Big| \  |z| \geq 1/2 \right\rbrace}^{\cong T^2 \times [\frac{1}{2},1]}
\cup \overbrace{\left\lbrace(z,w) \in (D^2 \times S^1)_{2i} \ \Big| \  |z| \geq 1/2\right\rbrace}^{\cong T^2 \times [\frac{1}{2},1]}\right)
\Bigg/ \sim  \ \xrightarrow{\cong} T^2 \times I.$$
Tracing through the identifications one obtains that 
\begin{align*}
P 
&\cong 
\left((\Sigma^\circ \times S^1) \cup \bigcup_{i=1}^c T^2 \times I  \right)\Bigg/ \sim\,  \, \cong (\Sigma^\circ \times S^1)/ \sim_P,  \nonumber
\end{align*}
where for $i=1, \dots,  c$ the equivalence relation $\sim_P$ is defined as
\begin{align}\label{eqn:simP}
 \partial B_{2i-1} \times S^1 \ni (\alpha_{2i-1} \times \id) \circ \eta^{\pm} (z,w) \sim_P (\alpha_{2i} \times \id) \circ \eta^{\pm}(\pm w,  \pm z) \in \partial B_{2i} \times S^1.
 \end{align}
Throughout the paper, we will move back and forth between thinking of $P$ as a quotient of $\Sigma^\circ \times S^1$ and as a quotient of $(\Sigma^\circ \times S^1) \cup \bigcup_{i=1}^c (T^2 \times I)$.
\end{remark}

Finally we describe a manifold homeomorphic to the boundary of an immersed surface exterior when the surface has a connected nonempty boundary.
\begin{construction}[The closed $3$-manifold $P_K$]
\label{cons:PK}
For a knot $K\subset S^3$,  fix a homeomorphism
$$d_K \colon \partial E_K \to \partial \Sigma \times S^1$$
 with $d_K(\mu_K)=\{ \pt \} \times S^1$ and~$d_K(\lambda_K)=\partial \Sigma \times \lbrace \pt \rbrace$, where $\mu_K$ and $\lambda_K$ respectively denote a fixed meridian and Seifert longitude for $K$.
 
 By recognizing $\partial P= \partial \Sigma \times S^1$,  we can define the closed $3$-manifold
$$P_K:=E_K \cup_{d_K} P.$$
\end{construction}

If $N$ is a $4$-manifold with $\partial N \cong S^3$ and $S \subset N$ is an immersed genus $g$ surface with $c_+$ positive double points,  $c_-$ negative double points, and $\partial S=K$,  
then $\partial N_S$, the boundary of the exterior of $S$,  is homeomorphic to $P_K$,  while $\partial \overline{\nu}(S)$,  the boundary of a tubular neighborhood of $S$,  is homeomorphic to $\overline{\nu}(K) \cup_\partial P$.
When $S$ is a $\Z$-surface, an explicit identification of $P_K$ with $\partial N_S$ will be described in Section~\ref{sec:BoundarIdentifications}.

\begin{remark}
\label{rem:SigmaInW}
We conclude this section by constructing an immersion
$$[\times \lbrace 0 \rbrace] \colon \Sigma \looparrowright W,$$
by first constructing an embedding $\Sigma \hookrightarrow E$ and then post-composing with $\proj_W$. 
Consider the map 
\begin{align*}
\Sigma & \to \bigsqcup_{k=1}^{2c}\left( D^2 \times D^2 \right)_k \sqcup \Sigma^\circ \times D^2 \\
x &\mapsto \begin{cases}
(\alpha_k^{-1}(x), 0) \in (D^2 \times D^2)_k & \text{ if } x \in B_k,\\
(x,0) \in \Sigma^\circ \times D^2 & \text{ if } x \in \Sigma^{\circ}.
 \end{cases}
\end{align*}
%\color{teal}
For this map to descend to a
% well-defined 
map $\Sigma \to E$,  we need to show that for any $x \in \partial B_k$,  we have that 
\[ (D^2 \times D^2)_k \ni (\alpha_k^{-1}(x), 0) \sim_E (x,0) \in \Sigma^{\circ} \times D^2.\]
But this is immediate, since ~$\sim_E$ identifies~$x \in (S^1 \times D^2)_k$ with its image under the composition
%(\alpha_k|_{S^1} \times \id_{D^2}) \circ \eta^\pm 
\[ (S^1 \times D^2)_k \xrightarrow{\eta^{\pm}} S^1 \times D^2 \xrightarrow{\alpha_k|_{S^1} \times \id_{D^2}} \partial B_k \times D^2\subset \Sigma^\circ \times D^2\]
and $\eta^{\pm}$ is the identity on $S^1 \times \{0\}$. 
\color{black}
The immersion $[\times \lbrace 0 \rbrace] \colon \Sigma \looparrowright W$ is now obtained by composition of the resulting embedding $\Sigma \hookrightarrow E$ with the projection $\proj_W \colon E \to W.$
\end{remark}

\begin{remark}
\label{rem:ClosedModels}
It is possible to repeat all the constructions of this section in the case where~$\Sigma$ is closed instead of having a boundary.
Repeating the construction of~$P$ with a closed surface yields a closed~$3$-manifold homeomorphic to~$P_U$,  the manifold~$P_K$ from Construction~\ref{cons:PK} with~$K=U$ the unknot.
For this reason,  when~$\Sigma$ is closed, we write~$E_U$ and~$W_U$ for the analogues of~$E$ and~$W$ and note that~$\partial W_U=P_U$. 
The same construction as in Remark~\ref{rem:SigmaInW} similarly yields an immersion~$[\times \lbrace 0 \rbrace] \colon \Sigma \looparrowright W_U$.

\end{remark}

In order to avoid stating all our results and constructions in both the closed and with-boundary cases, we will instead occasionally conclude sections with remarks such as this one listing the changes that occur when $\Sigma$ is closed.

\color{black}

\section{Integral homology calculations}
\label{sec:HomologyZ}

This section is concerned with the homology of the spaces $W,P$, and $P_K$ as well as the homology of exteriors of immersed $\Z$-surfaces.
These homological calculations are carried out in Section~\ref{sub:HomologyZ} whereas Section~\ref{sub:CoefficientSystems} uses these calculations to describe coverings of these spaces that will be used throughout this article. 
We note that the aforementioned homology groups could instead readily be calculated from handle diagrams, but our approach will be useful for later computations.  

\begin{notation*}
We remind the reader that the manifolds $W,P$ and $P_K$ depend on the genus $g$ of the surface~$\Sigma:=\Sigma_{g,1}$, integers $c_+,c_- \geq 0$ and a fixed embedding $\alpha \colon \bigsqcup_{2c} D^2 \hookrightarrow \operatorname{int}(\Sigma).$
\end{notation*}

\subsection{Homology calculations}
\label{sub:HomologyZ}

This section describes the homology groups of $P,W,P_K$ and the exterior of immersed $\Z$-surfaces.

\begin{construction}
\label{cons:PlumbingCurvesNotCellComplex}
We introduce curves and arcs in~$\Sigma^\circ \times \{1\}$ that will be helpful to describe generators of~$H_1(P)$,  thinking of $P$ as a quotient of $\Sigma^\circ \times S^1$.  
Our basepoint for $P$ will be the image of $\{\operatorname{pt}\} \times \{1\}$ for $\operatorname{pt} \in \partial \Sigma$. 
%AC: * was used again below for T^2 xI.
Fix points~$z_k \in\partial B_k \times \{1\} $ that satisfy~$z_{2i-1} = z_{2i} \in P$ for~$i=1,\ldots ,c$, using the description of $P$ as a quotient of $\Sigma^\circ \times S^1$ from Remark~\ref{rem:PasaQuotient}. 
Additionally,  for $i=1, \dots, c$ fix paths~$\omega_i$ connecting~$z_{2i-1}$ to~$z_{2i}$,  all disjoint except at their endpoints.  
Finally,  choose~$a_1,b_1,\ldots,a_g,b_g$  to be a standard symplectic basis of curves in $\Sigma^\circ \times \{1\}$ that are disjoint from the $\omega_i$ arcs.  This is all illustrated in Figure~\ref{fig:surfacewithcurves}. 

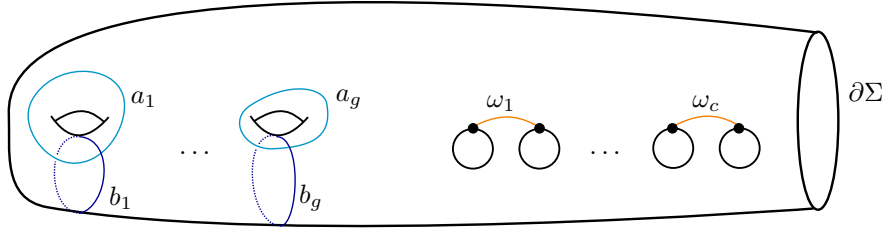
\begin{figure}[htbp!]
\begin{centering}
\begin{tikzpicture}
\node[anchor=south west,inner sep=0] at (0,0)
{\includegraphics[height=3cm]{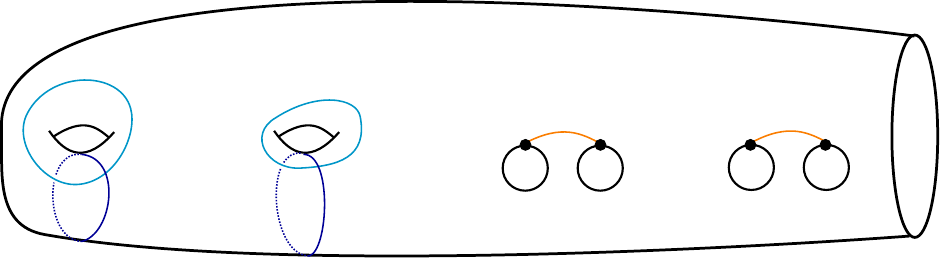}};
\node at (1.5, .4){$b_1$}; \node at (1.8, 1.7){$a_1$}; \node at (4, .4){$b_g$}; \node at (4.5, 1.7){$a_g$};
%\node at (5.5,1.7){$\xi_1$}; \node at (7.7, 1.9){$\xi_c$}; 
\node at (6.5, 1.65){$\omega_1$}; \node at (9.2, 1.65){$\omega_c$};
\node at (2.5, 1){$\dots$};\node at (7.9,1){$\dots$};
\node at (11.3, 1.8){$\partial \Sigma$};
\end{tikzpicture}
\end{centering}
\caption{
Certain curves and arcs on $\Sigma^\circ \times \{1\}$ that become loops in $P$.}
\label{fig:surfacewithcurves}
\end{figure}

Every~$(z,x) \in \Sigma^\circ \times D^2$ determines an point in~$E$ and if~$x \in \partial D^2$,  $(z,x)$ is in~$P$.
If~$\eta \subset \Sigma^\circ$ is a curve and~$x \in \partial D^2$, then we write~$\eta \times \{x\}$ for the corresponding loop in~$P$.
Different choices of $x \in \partial D^2$ lead to isotopic curves in $P$ and so,  when no confusion will occur, we slightly abuse notation by dropping the~`$ \times \{x\}$' from the notation and referring to the curve in $P$ as coming from~$\eta$ in $\Sigma^\circ$.

\begin{itemize}
\item We call the images of the loops $a_1,b_1,\ldots,a_g,b_g$ in $P$ \emph{genus loops}, and continue to refer to them as $a_1,b_1,\ldots,a_g,b_g$.
\item For each $i=1,\dots, c$ the points $z_{2i-1}$ and $z_{2i}$ are identified in $P$,  so the arc $\omega_i \subset \Sigma^\circ \times \{1\}$ becomes a loop  in $P$ based at $z_{2i-1}=z_{2i}$,  which we continue to refer to as $\omega_i$ and will call a \emph{plumbing loop}. 
%Then~$\xi_i \cdot \omega_i \cdot \overline{\xi_i}$ is a loop in $P$ based at $z$,  which we will call the \emph{plumbing curve} $\gamma_i$. 
\item The \emph{$S^1$-fibre} (or \emph{meridian}) of $P$ refers to the image $\mu\subset P$ of~$\lbrace z \rbrace \times S^1 \subset \Sigma^\circ \times S^1$ in~$P$.
\end{itemize}
When we think of $P$ as a quotient of $(\Sigma^\circ \times S^1) \cup \bigcup_{i=1}^c (T^2 \times I)$,  the plumbing loops are obtained as $ \omega_i \cdot (*_i  \times I)$,  where $*_i$ is a point in the $i$-th copy of $T^2$.
\end{construction}

The next proposition uses these generators to describe the homology of the plumbed $3$-manifold~$P$.

\begin{proposition}
\label{prop:HomologyP}
The homology of the plumbed~$3$-manifold~$P$ is given by
$$
H_i(P) \cong 
\begin{cases}
\Z &\quad \text{ for } i=0, \\
\Z^{2g+c+1}&\quad \text{ for } i=1,\\
\Z^{2g+c} &\quad \text{ for } i=2,\\
0 &\quad \text{ otherwise.}
\end{cases}
$$
More precisely,  the~$S^1$-fibre  together with the genus and plumbings curves give an isomorphism
$$ H_1(P) \cong \Z \mu \oplus \bigoplus_{i=1}^g \left( \Z a_i \oplus \Z b_i \right) \oplus \bigoplus_{i=1}^c \Z\omega_i .$$
Additionally,  the inclusion induced map $H_1(\partial P) \to H_1(P)$ maps $[\partial \Sigma]$ to zero, and is the identity on the $S^1$-fibre.
\end{proposition}
\begin{proof}
The calculations of $H_0$ and $H_3$ follow from the fact that $P$ is a compact, connected~$3$-manifold with boundary.
The calculation of $H_1(P)$ will be described in Proposition~\ref{prop:H1P}, as will the sentence concerning $H_1(\partial P) \to H_1(P).$
The result for $H_2$ then follows from Poincar\'e duality,  and the universal coefficient theorem.
%%Don't delete.
%The relative H_1 is free, it just has one less generator than the absolute one; namely the meridian got killed. and the fact that $P$ has torus boundary.
\end{proof}

The next proposition describes the homology of the plumbed $4$-manifold $W$.

\begin{proposition}
\label{prop:HomologyW}
%Let~$\Sigma$ be a surface with a connected nonempty boundary component.
%%AC: Commented because it's in the notation up top.
The homology of the plumbed~$4$-manifold~$W$ is given by
$$
H_i(W)\cong 
\begin{cases}
\Z &\quad \text{ for } i=0, \\
\Z^{2g+c} &\quad \text{ for } i=1,\\
0 &\quad \text{ otherwise.}
\end{cases}
$$
Additionally, the inclusion induces an isomorphism
$$H_1(P)/\Z\mu \xrightarrow{\cong} H_1(W).$$
\end{proposition}
\begin{proof}
This follows because $W$ deformation retracts onto the image of the immersion~$[\times \{0\}]$,  which further deformation retracts onto the wedge of the genus and plumbing loops of~$P$.
\end{proof}

The next proposition deduces $H_*(P_K)$ from the calculation of $H_*(P).$

\begin{proposition}
\label{prop:HomologyPK}
%Let~$\Sigma$ be a surface with a connected nonempty boundary component.
%%AC: Commented because it's in the notation up top.
The homology of~$P_K$ is 
$$
H_i(P_K)\cong 
\begin{cases}
\Z &\quad \text{ for } i=0,3, \\
\Z^{2g+c+1} &\quad \text{ for } i=1,2, \\
0 &\quad \text{ otherwise.}
\end{cases}
$$
More precisely, the inclusions~$E_K \hookrightarrow P_K$ and~$P \hookrightarrow P_K$ give rise to a short exact sequence
$$0 \to  \Z \xrightarrow{j} H_1(E_K) \oplus H_1(P) \to H_1(P_K) \to 0~$$
where~$j(1)=(\mu_K,\mu)$
%Splitting this sequence leads 
%%AC: No splitting needed
and therefore to an isomorphism
$$H_1(P_K) \cong H_1(P)/\Z\langle \mu \rangle \oplus \Z\langle \mu_K \rangle.$$
Thus,  $H_1(P_K)$ is freely generated by the meridian of $K$, the genus loops and the plumbing loops.
\end{proposition}
\begin{proof}
The calculation of $H_0$ and $H_3$ follows from the fact that $P_K$ is a closed, connected~$3$-manifold.
The calculation of $H_2$ follows from that of $H_1$ together with Poincar\'e duality and the universal coefficient theorem.
The proposition therefore reduces to calculating~$H_1(P_K)$.

Consider the Mayer-Vietoris sequence for~$P_K=E_K \cup_\partial P$:
$$\ldots  \to H_1(S^1 \times S^1) \xrightarrow{f} H_1(P) \oplus H_1(E_K) \to H_1(P_K) \to 0.$$
Write $e_1,e_2$ for the generators of $H_1(S^1 \times S^1)=\Z^2$.
The definition of the $3$-manifold~$P_K$ implies that~$f(e_2)=(\mu_K,\mu)$ and~$f(e_1)=(\lambda_K,\partial \Sigma  \times \lbrace \operatorname{pt} \rbrace)$.
The Seifert longitude vanishes in $H_1(E_K)$ as does $[\partial \Sigma]$ in $H_1(P)$ as mentioned in Proposition~\ref{prop:HomologyP}.
The conclusion follows.
\end{proof}

Next we describe the homology of $\Z$-surface exteriors.

\begin{proposition}
\label{prop:HomologyZExterior}
The homology of the exterior of a genus~$g$ immersed $\Z$-surface~$S \subset N$ with~$c$ double points~is
%In $D^4$ one doesn't need the $\Z$-surface assumption. Outside, it's not clear the surface is nullhomologous.
$$
H_i(N_S)\cong
\begin{cases}
\Z &\quad \text{ for } i=0,1, \\
\Z^{2g+c+b_2(N)} &\quad \text{ for } i=2,\\
0 &\quad \text{ otherwise.}
\end{cases}
$$
Furthermore, 
\begin{itemize}
\item the group~$H_1(N_S)$ is generated by a meridian~$\mu_S$ of~$S$;
\item the map~$\pi_1(\partial N_S) \to \pi_1(N_S)$ is surjective;
\item the intersection form of $N_S$ is isometric to~$Q_N \oplus (0)^{2g+c}$.
\end{itemize}
The same assertions hold for closed immersed $\Z$-surfaces in closed $4$-manifolds.
\end{proposition}
\begin{proof}
Use~$S^\circ$ to denote the surface obtained by removing small open $4$-balls~$B_1,\ldots,B_c$ in $N$ around the immersion points.
Note that~$S^\circ$ has~$2c+1$ boundary components and is embedded in~$N^\circ:=N \setminus \bigcup_{i=1}^c B_i$.
%%Don't delete.
%Note that this~$S^\circ$ is not the same as~$\Sigma_{g,1}^\circ$ which has~$2c+1$-boundary components.
%\color{teal}
For~$i>0$, the Mayer-Vietoris sequence for~$N^\circ=N_S \cup \overline{\nu}(S^\circ)$ now yields
\begin{equation}
\label{eq:MVExterior}
H_2(N^\circ) \to  \overbrace{H_1(S^\circ \times S^1)}^{\cong \Z^{2g+2c+1}} \to \overbrace{H_1(N_S)}^{\cong \Z} \oplus \overbrace{H_1(\overline{\nu}(S^\circ))}^{\cong \Z^{2g+2c}} \to \overbrace{H_1(N^\circ)}^{=0}.
\end{equation}
The middle map is an isomorphism because it is a surjection between free abelian groups of the same~rank.
Since the~$H_1(S^\circ)$-summand of $H_1(S^\circ \times S^1) \cong \Z\mu_S \oplus  H_1(S^\circ)$ maps isomorphically onto~$H_1(\overline{\nu}(S^\circ))$, it follows that
$$H_1(N_S)\cong \Z \mu_S .$$
This concludes the proof of the first assertion.
The second assertion follows from 
Proposition~\ref{prop:HomologyPK} which shows that $\mu_S$ lies in the image of~$H_1(P_K) \cong H_1(\partial N_S) \to H_1(N_S)\cong \pi_1(N_S)$.
This also proves that $H_1(N_S,\partial N_S)=0$ and therefore
$$ H_3(N_S) \cong H^1(N_S,\partial N_S)=0.$$
%{AC: One can also calculate the $H_2$ using $N^\circ=N_S \cup \overline{\nu}(S^\circ)$.
%Since~$H_i(\overline{\nu}(S^\circ))\cong H_i(S^\circ)=0$ for $i=2,3$ the exact sequence becomes 
%$$ 0 \to \overbrace{H_3(N^\circ)}^{\cong \Z^c} \to \overbrace{H_2(S^\circ \times S^1)}^{\cong \Z^{2g+2c}} \to H_2(N_S) \to \overbrace{H_2(N^\circ)}^{\cong H_2(N)} \to 0. $$
%}
For the remaining statement, the Mayer-Vietoris sequence for $N=N_S \cup_P \overline{\nu}(S)$ provides a convenient route.
Identify $\overline{\nu}(S)$ with $W$ and $\partial \overline{\nu}(S) \setminus (\partial N \setminus \nu(K))$ with $P$ (see Proposition~\ref{prop:IdentifyW} for details on how to do so),  and consider the following portion of the Mayer-Vietoris sequence
\begin{equation}
\label{eq:MVExterior}
 H_2(N) \to \overbrace{H_1(P)}^{\cong \Z^{2g+c+1}} \xrightarrow{f} \overbrace{H_1(N_S)}^{\cong \Z} \oplus \overbrace{H_1(W)}^{\cong \Z^{2g+c}} \to \overbrace{H_1(N)}^{=0}. 
 \end{equation}
%%Don't delete just yet
%We assert $f \colon  \Z \mu \oplus H_1(P)/\Z\mu \to H_1(N_S) \oplus H_1(W)$ is an isomorphism.
% Proposition~\ref{prop:HomologyW} the inclusion~$P \subset W$ induces an isomorphism~$H_1(P)/\Z\mu \xrightarrow{\cong} H_1(W)$.
%Since~$H_1(N_S)\cong \Z\mu_S$,  we deduce that~$f$ is of the form~$\bsm 1& * \\ 0 & A\esm$ with $A$ an isomorphism.
%The assertion follows.
Here,  $f$ is an isomorphism because it is a surjection between free abelian groups of the same~rank.

Since $N$ is simply-connected and has boundary $S^3$, we have~$H_3(N)=0$.
Moving further to the left in the Mayer-Vietoris sequence for $N=N_S \cup_P W$ now yields
$$ 0 \to \overbrace{H_2(P)}^{\cong \Z^{2g+c}} \to H_2(N_S) \to  \overbrace{H_2(N)}^{\cong \Z^{b_2(N)}} \to 0.$$
The rightmost zero is due to the fact that the map $f$ in~\eqref{eq:MVExterior} is an isomorphism.
Since $N$ is simply-connected, $H_2(N)$ is free abelian,  and so the short exact sequence involving $H_2(N_S)$ splits.
This yields the required isomorphism $H_2(N_S) \cong \Z^{2g+c+b_2(N)}$.

We conclude by calculating the intersection form $Q_{N_S}$ of $N_S$.
Since $P \subset \partial N_S$, the restriction of  $Q_{N_S}$ to $H_2(P)$ is trivial.
Since $H_2(P) \cong \Z^{2g+c}$,  the isometry~$Q_{N_S} \cong Q_N \oplus (0)^{2g+c}$ follows.
\color{black}

Finally, we note that the proofs in the closed case are entirely analogous.
\end{proof}

The reader may suspect that being a $\Z$-surface is an unnecessarily strong hypothesis for the previous proposition,  and indeed,  if $H_1(N)=0$,  it is enough to assume that $S$ is null-homologous.   
We have limited ourselves to the result we need, but record for later that $\Z$-surfaces are in fact nullhomologous.

\begin{proposition}
\label{prop:ZSurfaceAreNullhomologous}
Immersed $\Z$-surfaces are nullhomologous.
\end{proposition}
\begin{proof}
%We first prove the case with boundary.
The proof is identical to~\cite[Lemma 5.1]{ConwayPowell} (which treats embeddings) but we recall the argument briefly.
We focus on the case with nonempty boundary as the closed case is similar.
We must show that~$[S,\partial S] =0 \in H_2(N,\partial N)$.
The intersection form~$Q_N$ pairs~$H_2(N,\partial N)$ with~$H_2(N)$ nonsingularly because~$N$ is simply-connected.
Thus, it is sufficient to show that~$Q_N([S,\partial S],x)=0$ for every class~$x \in H_2(N)$.
Represent such an~$x \in H_2(N)$ by a closed surface~$R \subseteq N$ that intersects~$S$ transversely in points~$p_1,\ldots,p_n$, so that~$Q_N([S,\partial S],x)=\sum_{i=1}^n \varepsilon(p_i)$, where~$\varepsilon(p_i)=\pm 1$.
%FNOP Theorem 10.9
Now the intersection~$R \cap N_S$ is an embedded surface in~$N_S$ with oriented boundary (homologous to)~$\sum_{i=1}^n \varepsilon(p_i)\mu_S$.
This implies that~$\sum_{i=1}^n \varepsilon(p_i)\mu_S =0\in H_1(N_S)$.
But now, since as calculated in Proposition~\ref{prop:HomologyZExterior} we have that~$H_1(N_S)=\pi_1(N_S)=\Z\mu_S$ is torsion-free, we deduce that~$Q_N([S,\partial S],x)=\sum_{i=1}^n \varepsilon(p_i)=0$, establishing that~$S$ is nullhomologous.
\end{proof}

\subsection{Coefficient systems}
\label{sub:CoefficientSystems}

This section describes a surjective homomorphism $\varphi \colon \pi_1(P_K) \to \Z$.
Throughout the remainder of the paper,  $P_K^\infty$ will refer to the infinite cyclic cover corresponding to $\ker(\varphi)$.
Similarly,  for $X \in \{P,P_K,E_K\}$, the homology groups~$H_*(X;\Z[t^{\pm 1}]) \cong H_*(X^\infty)$ will always be understood to be taken with respect to $\varphi.$

\begin{construction}[The epimorphism $\varphi \colon \pi_1(P_K) \to \Z$]
\label{cons:CoeffSystPK}
Recall from Proposition~\ref{prop:HomologyPK} that $H_1(P_K)$ is freely generated by the meridian of $K$, the genus loops and the plumbing loops.
Define an epimorphism $\varphi \colon \pi_1(P_K) \to \Z$ by composing the abelianization map with the map $H_1(P_K) \to \Z$ that sends the meridian to $1 \in \Z$, the genus loops to $0 \in \Z$, and the plumbing loops to $0 \in \Z$.
\end{construction}

\begin{remark}
\label{rem:CoeffSystP}
Composing $\varphi \colon \pi_1(P_K) \to \Z$ with the inclusion induced maps gives rise to epimorphisms $\pi_1(P) \to \Z$ and $\pi_1(E_K)\to \Z$.
The former is characterized by the property that it maps the $S^1$-fibre to $1 \in \Z$, the genus loops to $0 \in \Z$, and the plumbing loops to $0 \in \Z$.
The latter is the usual abelianization homomorphism that maps the meridian of $K$ to $1\in \Z$.
\end{remark}

\begin{remark}
\label{rem:Lifts}
Since the composition $\pi_1(\intt \Sigma^\circ \times \{ 1\}) \to \pi_1(P) \xrightarrow{\varphi} \Z$ is trivial,  the surface $\intt \Sigma^\circ \cong \intt \Sigma^\circ  \times \{ 1\}$ lifts to~$P^\infty$.
Therefore,   when we refer to the lift of a loop $\eta \subset \intt \Sigma^\circ $ to~$P^\infty$,  we will always mean the unique lift of $\eta$ that lies within a preferred lift of $\intt \Sigma^\circ \times\{1\}$ within $P^{\infty}$.  
When we refer to the lift of a plumbing loop $\omega$ to $P^\infty$,  we will always mean the unique lift of $\omega$ such that the interior of the arc in $\Sigma^\circ$ that gives rise to $\omega$ lies within our preferred lift of $\intt \Sigma^\circ \times\{1\}$.  
\end{remark}

\section{Decomposing the Blanchfield form of~$P_K$ and its isometries}
\label{sec:IsoBlPk}
The organization of this section is as follows: in Section~\ref{sub:BlanchfieldIntro}, we collect some facts about Blanchfield forms, and in Section~\ref{sub:HomologyPKZZ}, we establish that~$H_1(P_K;\Z[t^{\pm 1}]) \cong H_1(P;\Z[t^{\pm 1}]) \oplus H_1(E_K;\Z[t^{\pm 1}])$.  
In Section~\ref{sub:BlanchfieldPK}, we prove that the Blanchfield form of $P_K$ decomposes as the direct sum of the Blanchfield forms of $P$ and of $E_K$, i.e.~$\Bl_{P_K}\cong \Bl_P \oplus \Bl_K$.
Finally,  in Section~\ref{sub:IsometriesBlanchfieldPK}, we prove that every isometry of $\Bl_{P_K}$ decomposes as the direct sum of an isometry of $\Bl_P$ with an isometry of~$\Bl_K$, i.e. $\Aut(\Bl_{P_K}) \cong \Aut(\Bl_P) \oplus \Aut(\Bl_K).$

\subsection{Some facts about Blanchfield forms}
\label{sub:BlanchfieldIntro}

This section sets up some notation and facts concerning twisted homology and Blanchfield forms.
This is not aimed to be an expansive treatment of the topic, and we refer to e.g.~\cite{FriedlPowell} for further details.

\medbreak
We start with twisted homology.
In what follows,  spaces are assumed to have the homotopy type of a finite CW complex.
Given a space~$X$ together with an epimorphism~$\varphi \colon \pi_1(X) \twoheadrightarrow \Z$, we write~$p\colon X^\infty \to X$ for the infinite cyclic cover corresponding to~$\ker(\varphi)$.
If~$A \subset X$ is a subspace, then we set~$A^\infty :=p^{-1}(A)$ and often write~$H_*(X,A;\Z[t^{\pm 1}])$ instead of~$H_*(X^\infty,A^\infty)$.
Note that these are finitely generated~$\Z[t^{\pm 1}]$-modules because $X$ and $A$ have the homotopy type of finite~CW-complexes and $\Z[t^{\pm 1}]$ is Noetherian, see e.g.~\cite[Proposition A.9]{FriedlNagelOrsonPowell}.
In the sequel, we will also write~$H^*(X,A;\Z[t^{\pm 1}])$ for the homology of $\Hom_{\Z[t^{\pm 1}]}(\overline{C_*(X^\infty,A^\infty)},\Z[t^{\pm 1}])$. 
%and remind the reader that~$H^*(X,A;\Z[t^{\pm 1}])$ is isomorphic to the cohomology with compact support of the pair~$(X^\infty,A^\infty)$,  not to $H^*(X^\infty,A^\infty).$
Here,  given a~$\Z[t^{\pm 1}]$-module $V$,  we write $\overline{V}$ for the $\Z[t^{\pm 1}]$-module whose underlying group agrees with that of~$V$ but with the $\Z[t^{\pm 1}]$-module structure induced by $t \cdot v=t^{-1}v$ for $v \in V$.

We move on to Blanchfield forms.
Given a $3$-manifold $Y$ and an epimorphism $\varphi \colon \pi_1(Y) \twoheadrightarrow \Z$ such that the Alexander module~$H_1(Y;\Z[t^{\pm 1}])$ is torsion, the \emph{Blanchfield form} is a pairing
$$
\Bl_Y \colon H_1(Y;\Z[t^{\pm 1}]) \times H_1(Y;\Z[t^{\pm 1}]) \to \Q(t)/\Z[t^{\pm 1}].$$
We will not need the definition of $\Bl_Y$, only its properties, and so we focus on the latter.
First of all, the Blanchfield form is sesquilinear and hermitian; see e.g.~\cite{PowellBlanchfield}.
Additionally,  if $Y$ is closed, then some homological algebra shows that~$\Bl_Y$ is nonsingular.

Next we introduce the algebra that relates the equivariant intersection form of a $4$-manifold with $\pi_1\cong \Z$ to the Blanchfield form of its boundary.
A lengthier exposition can be found in~\cite[Sections 2 and 3]{ConwayPowell}.

\begin{construction}[The boundary linking form $\partial \lambda$]
\label{cons:BoundaryLinkingForm}
If~$\lambda \colon H \times H \to \Z[t^{\pm 1}]$ is a nondegenerate hermitian form on a finitely generated free $\Z[t^{\pm 1}]$-module,  then we write~$\widehat{\lambda} \colon H \to H^*$ for the linear map~$z \mapsto \lambda(-,z)$, and
there is a short exact sequence
$$ 0 \to H \xrightarrow{\widehat{\lambda}} H^* \xrightarrow{} \coker(\widehat{\lambda}) \to 0.$$
Such a presentation induces a \emph{boundary linking form}~$\partial \lambda$ on~$\coker(\widehat{\lambda})$ in the following manner.
For~$[x] \in \coker(\widehat{\lambda})$ with~$x \in H^*$, since $\coker(\widehat{\lambda})$ is $\Z[t^{\pm 1}]$-torsion,  there exist elements~$z\in H$ and~$p\in\Z[t^{\pm 1}] \sm \{0\}$ such that~$\lambda(-,z)=px\in H^*$.
For~$[x],[y]\in \coker(\widehat{\lambda})$ with~$x,y\in H^*$, we define
$$\partial\lambda([x],[y]):=\frac{y(z)}{p}\in\Q(t)/\Z[t^{\pm 1}].$$
One can check that~$\partial \lambda$ is independent of the choices of~$p$ and $z$.
Furthermore, as noted in~\cite[Remark 2.4]{ConwayPowell}, if the hermitian form~$\lambda$ is represented by a matrix $A$, then~$(\coker(\widehat{\lambda}),\partial \lambda) = (\coker(A^T),\ell_A)$, where~$\ell_A([x],[y])=x^TA^{-1}y.$
\end{construction}

The next definition introduces terminology associated to the notion of a boundary linking form.

\begin{definition}
\label{def:presentation}
For~$T$ a torsion~$\Z[t^{\pm 1}]$-module with a linking form~$\ell \colon T \times T \to \Q(t)/\Z[t^{\pm 1}]$,  a nondegenerate hermitian form~$\lambda$ \textit{presents}~$(T,\ell)$ if there is an isomorphism~$h\colon\coker(\widehat{\lambda})\to T$ such that~$\partial\lambda(x,y)=\unaryminus \ell(h(x),h(y))$.
Such an isomorphism~$h$ is called an \emph{isometry} of the forms,  and the set of isometries from $\partial \lambda$ to $-\ell$ is denoted~$\Iso(\partial\lambda,\unaryminus \ell)$.
\end{definition}

%If $\lambda$ is represented by a matrix $A$, then it is common to say that the nondegenerate hermitian matrix~$A$ \emph{presents} $(T,\ell)$.
If a nondegenerate hermitian form~$\lambda$ presents $(T,\ell)$ and is represented by a matrix $A$, then it is common to say that~$A$ itself \emph{presents} $(T,\ell)$.
We warn the reader that in this notation,  if $A$ presents~$(T,\ell)$, then~$(T,\ell) \cong (\coker(A^T),\unaryminus \ell_A)$,  so when considered just as a module $T$ is presented by~$A^T$,  not by~$A$.

\begin{remark}
\label{rem:MinusPresentation}
It is also common to say that a nondegenerate hermitian form $\lambda$ presents a linking form $\ell$ if $\partial \lambda \cong \ell$ instead of our convention that~$\partial \lambda \cong \unaryminus \ell$.
The reason for our choice is that several of our statements involve a presentation for~$\unaryminus \Bl_{P_K}$, and so we find it convenient to make this part of the definition.
This is also consistent with the conventions in~\cite{ConwayPiccirilloPowell,ConwayDaiMiller}.
\end{remark}

The next proposition relates the equivariant intersection form of a $4$-manifold with $\pi_1\cong \Z$ to the Blanchfield form of its boundary.
The result is known, see e.g.~\cite[Proposition 3.5]{ConwayPowell}.
% for a proof.

\begin{proposition}
\label{prop:PresentsBlanchfield}
Let~$W$ be a~$4$-manifold with~$\pi_1(W) \cong \Z$.
%If~$W$ is a~$4$-manifold with~$\pi_1(W) \cong \Z$, 
If the inclusion induced homomorphism~$\varphi \colon \pi_1(\partial W) \to \pi_1(W)$ is surjective and~$H_1(\partial W;\Z[t^{\pm 1}])$ is torsion, then there is an isometry
%~\cite[Proposition 3.5]{ConwayPowell} defines a canonical isometry 
$$D_W \colon \partial \lambda_W \cong \unaryminus \Bl_{\partial W},$$
where~$\lambda_W$ denotes the equivariant intersection form of~$W$ and~$\Bl_{\partial W}$ the Blanchfield form of~$(\partial W,\varphi)$.
\end{proposition}

We conclude with a fact that be useful to state our classifications of immersed~$\Z$-surfaces.
\begin{remark}
\label{rem:InducedIsometry}
Let $(H_0,\lambda_0)$ and $(H_1,\lambda_1)$ be hermitian forms over $\Z[t^{\pm 1}].$
A $\Z[t^{\pm 1}]$-linear isomorphism~$\varpi \colon H_0 \to H_1$ induces an isomorphism~$(\varpi^*)^{-1} \colon H_0^* \to H_1^*$.
If additionally, the isomorphism~$\varpi$ is an isometry, then~$(\varpi^*)^{-1}$ descends to an isomorphism
%An isometry~$\varpi \in \Iso(\lambda_0,\lambda_1)$ of hermitian forms induces an isomorphism~$\varpi^{-*} \colon H_0^* \to H_1^*$ which descends to an isomorphism
$$\partial \varpi:=(\varpi^*)^{-1} \colon \coker(\widehat{\lambda}_0)\to \coker(\widehat{\lambda}_1).$$
One can then verify that~$\partial \varpi$ is an isometry of the boundary linking forms.
\end{remark}

\subsection{The~$\Z[t^{\pm 1}]$-homology of~$P_K$}
\label{sub:HomologyPKZZ}

The goal of this section is to describe the~$\Z[t^{\pm 1}]$-homology of the~$3$-manifold~$P$ and to prove that the inclusions~$E_K \subset P_K$ and~$P \subset P$ induce a~$\Z[t^{\pm 1}]$-isomorphism~$H_1(P;\Z[t^{\pm 1}])\oplus H_1(E_K;\Z[t^{\pm 1}])\xrightarrow{\cong} H_1(P_K;\Z[t^{\pm 1}]).$

\begin{notation}
In what follows,  we write $\Z_\varepsilon$ to denote the $\Z[t^{\pm 1}]$-module with underlying group~$\Z$ and where the module structure is obtained by linearly extending the trivial $\Z$-action~$t \cdot x=x$.
Note that the augmentation $\Z[t^{\pm 1}] \to \Z$ induces an isomorphism $\Z[t^{\pm 1}]/(t-1) \cong \Z_\varepsilon$.
\end{notation}

We begin with the calculation of~$H_1(P;\Z[t^{\pm 1}]).$

\begin{proposition}
\label{prop:HomologyPZZ}
The~$\Z[t^{\pm 1}]$-homology of~$P$ is given by
$$
H_i(P;\Z[t^{\pm 1}])=
\begin{cases}
\Z_\varepsilon &\quad \text{ for } i=0, \\
\Z_\varepsilon^{2g} \oplus \left( \Z[t^{\pm 1}]/(t-1)^2\right)^{\oplus c} 
&\quad \text{ for } i=1,\\
0 &\quad \text{ otherwise.}
\end{cases}
$$
The first summand of~$H_1(P;\Z[t^{\pm 1}])$ is generated by lifts of the genus loops,  the second by lifts of the plumbing loops.

Additionally, the inclusion induced map~$H_1(\partial P;\Z[t^{\pm 1}]) \to H_1(P;\Z[t^{\pm 1}])$ is the zero map.
\end{proposition}
\begin{proof}
The infinite cyclic cover of~$P$ determined by the epimorphism~$\varphi \colon H_1(P) \to \Z$ is connected and noncompact.
We deduce that $H_0(P;\Z[t^{\pm 1}])=\Z_\varepsilon$  and~$H_3(P;\Z[t^{\pm 1}])=0$.
The calculation of $H_1(P;\Z[t^{\pm 1}])$ will be described in Proposition~\ref{prop:H1PZZ}, as will the sentence concerning the homomorphism~$H_1(\partial P;\Z[t^{\pm 1}]) \to H_1(P;\Z[t^{\pm 1}]).$

It follows that $H_1(P,\partial P;\Z[t^{\pm 1}]) \cong H_1(P;\Z[t^{\pm 1}])$ is torsion and, since $H_0(P,\partial P;\Z[t^{\pm 1}])=0$,  a calculation involving the universal coefficient spectral sequence (see~\cite[Theorem 2.3]{LevineKnotModules}) shows 
that~$H_2(P;\Z[t^{\pm 1}])=H^1(P,\partial P;\Z[t^{\pm 1}])=0$.
%%Don't delete: Because Hom(H_1(P,\partial P;\Z[t^{\pm 1}]))=0.
\end{proof}

Next, we describe $H_1(P_K;\Z[t^{\pm 1}]).$

\begin{proposition}
\label{prop:HomologyPKZZ}
The~$\Z[t^{\pm 1}]$-homology of~$P_K$ satisfies
$$
H_i(P_K;\Z[t^{\pm 1}]) \cong
\begin{cases}
\Z_\varepsilon &\quad \text{ for } i=0,2,\\
H_1(E_K;\Z[t^{\pm 1}]) \oplus H_1(P;\Z[t^{\pm 1}]) &\quad \text{ for } i=1, \\
0 &\quad \text{ otherwise.}
\end{cases}
$$
The isomorphism on~$H_1(-;\Z[t^{\pm 1}])$ is induced by the inclusions $E_K \subset P_K$ and $P \subset P_K$.
\end{proposition}
\begin{proof}
%\color{teal}
The infinite cyclic cover of~$P_K$ determined by the epimorphism~$\varphi \colon H_1(P_K) \to \Z$ is connected and noncompact.
We deduce that~$H_0(P_K;\Z[t^{\pm 1}]) \cong \Z_\varepsilon$ and~$H_3(P_K;\Z[t^{\pm 1}])=0$.
Consider the Mayer-Vietoris sequence for the decomposition~$P_K=E_K \cup_\partial P$ with~$\Z[t^{\pm 1}]$ coefficients:
\begin{equation}
\label{eq:MVPgcK}
 \cdots \to H_1(S^1 \times S^1;\Z[t^{\pm 1}]) \xrightarrow{j} H_1(E_K;\Z[t^{\pm 1}])\oplus H_1(P;\Z[t^{\pm 1}]) \to H_1(P_K;\Z[t^{\pm 1}])\to 0.
 \end{equation}
As~$H_1(E_K;\Z[t^{\pm 1}])$ and~$H_1(P;\Z[t^{\pm 1}])$ are both~$\Z[t^{\pm 1}]$-torsion, we deduce that~$H_1(P_K;\Z[t^{\pm 1}])$ is also~$\Z[t^{\pm 1}]$-torsion.
Since~$P_K$ is closed, this implies that~$H_2(P_K;\Z[t^{\pm 1}]) \cong \Z_\varepsilon$ (see e.g.~\cite[Lemma 3.2]{ConwayPowell}) and the lemma therefore reduces to proving that the map~$j$ in~\eqref{eq:MVPgcK} is the zero map.

We first study the restriction of the coefficient system to~$S^1 \times S^1$.
Under the inclusion~$\partial P \to P$,  Proposition~\ref{prop:HomologyP} ensures that,  in homology, the second factor is mapped to the~$S^1$-fibre of~$P$ and the first factor is mapped to~$[\partial \Sigma \times \lbrace \operatorname{pt} \rbrace]$.
The coefficient system maps the former to $1 \in \Z$ and the latter, say $\delta$,  to $0 \in \Z$.
We deduce that 
$$ H_1(S^1 \times S^1;\Z[t^{\pm 1}]) \cong \Z[t^{\pm 1}]/(t-1)\langle \widetilde{\delta} \rangle.$$
We now return to showing that $j$ is the zero map.
The map~$j$ has two components, one mapping into~$H_1(E_K;\Z[t^{\pm 1}])$ and one into~$H_1(P;\Z[t^{\pm 1}]$). 
We argue that each of them is the zero map.
The fact that~$H_1(S^1 \times S^1;\Z[t^{\pm 1}]) \to H_1(P;\Z[t^{\pm 1}])$ is the zero map was stated in Proposition~\ref{prop:HomologyPZZ} and so we focus on the knot complement.
For the knot complement, this follows from the fact that $j$ maps~$\widetilde{\delta}$ to a lift of a Seifert longitude.
The latter vanishes in $H_1(E_K;\Z[t^{\pm 1}])$, the nullhomology being provided by a lift of a Seifert surface.
Thus $j$ is the zero map and the proof is concluded.
\end{proof}

\subsection{The Blanchfield form of~$P_K$}
\label{sub:BlanchfieldPK}

The goal of this short section is to establish the isometry~$\Bl_{P_K} \cong \Bl_P \oplus \Bl_K$.
This relies on the following result.

\begin{theorem}[{Friedl-Leidy-Nagel-Powell~\cite[Theorem 1.1]{FLNP}}]
\label{thm:FLNP}
Let~$Y$ be a~$3$-manifold with empty or toroidal boundary and let~$Y=A \cup_T B$ be a decomposition of~$Y$ along a torus~$T$ into two~$3$-manifolds~$A$ and~$B$.
Let~$\phi \colon \Z[\pi_1(Y)]\to \Z$ be a homomorphism such that~$H_1(T;\Z[t^{\pm 1}])$ and~$H_1(Y;\Z[t^{\pm 1}])$ are torsion.
Then~$H_1(A;\Z[t^{\pm 1}])$ and~$H_1(B;\Z[t^{\pm 1}])$ are torsion and the inclusion induced maps~$i_A \colon A \to Y$ and~$i_B \colon B \to Y$ induce a morphism of linking forms
$$ i_A +i_B \colon \Bl_A \oplus \Bl_B \xrightarrow{}  \Bl_Y.$$
In particular, if~$i_A + i_B$ is an isomorphism, then it is in fact an isometry.
\end{theorem}

The next proposition is the main result of this short section.

\begin{proposition}
\label{prop:BlPK}
The inclusions~$P \subset P_K$ and~$E_K \subset P_K$ induce an isometry
$$\Bl_K \oplus \Bl_{P} \cong \Bl_{P_K}.$$
In particular the inclusion~$P \hookrightarrow P_U$ induces an isometry
$$\Bl_{P} \cong \Bl_{P_U}.$$
\end{proposition}
\begin{proof}
Proposition~\ref{prop:HomologyPKZZ} implies that the inclusions~$P \subset P_K$ and~$E_K \subset P_K$ induce an isomorphism~$ H_1(E_K;\Z[t^{\pm 1}]) \oplus H_1(P;\Z[t^{\pm 1}]) \to H_1(P_K;\Z[t^{\pm 1}]).$
 Since~$H_1(\partial P;\Z[t^{\pm 1}])=\Z_\varepsilon$ is torsion,  Theorem~\ref{thm:FLNP} implies that this isomorphism is in fact an isometry.
 The last assertion follows by taking~$K=U$ to be the unknot and noting that the Alexander module of the unknot is trivial.
\end{proof}

\subsection{Isometries of the Blanchfield form of~$P_K$}
\label{sub:IsometriesBlanchfieldPK}

The goal of this section is to prove that the inclusions~$E_K \subset P_K$ and~$P \subset P$ induce an isometry~$\Aut(\Bl_P) \oplus \Aut(\Bl_K) \xrightarrow{\cong} \Aut(\Bl_{P_K}).$

\medbreak

We start with some preliminary algebraic lemmas.

\begin{lemma}
\label{lem:HomDim}
If~$f \colon T_1 \to T_2$ is a homomorphism between~$\Z[t^{\pm 1}]$-modules where~$T_2$ has projective dimension~$\leq 1$, then~$\im(f)$ admits a square presentation matrix.
\end{lemma}
\begin{proof}
Consider the short exact sequence~$ 0 \to \im(f) \to T_2 \to \coker(f) \to 0$.
Let $V$ be a $\Z[t^{\pm 1}]$-module and consider, for $i \geq 2$, the exact sequence
$$\leftarrow \overbrace{\Ext_{\Z[t^{\pm 1}]}^{i+1}(\coker(f),V)}^{=0} \leftarrow \Ext_{\Z[t^{\pm 1}]}^i(\im(f),V) \leftarrow \overbrace{\Ext_{\Z[t^{\pm 1}]}^i(T_2,V)}^{=0} \leftarrow   $$
Since~$\Z[t^{\pm 1}]$ has global dimension~$2$,  all the modules to the left of this sequence are zero.
The zero on the right is due to $T_2$ having projective dimension~$\leq 1$. We deduce that~$\im(f)$ has projective dimension~$\leq 1$.

Since~$T_2$ is torsion, so is~$\im(f)$.
%E.g. 0 \neq  \Delta_{T_2}=\Delta_{im(f}}\Delta_{coker(f}}
%More simply, if~$x \in \im(f)$ then it is still in~$T_2$ and so is torsion.
The proposition follows: a torsion module over $\Z[t^{\pm 1}]$ is seen to admit a square presentation matrix if and only if it has projective dimension at most one.
\end{proof}

We call two polynomials in~$\Z[t^{\pm 1}]$ \textit{relatively prime} if their only common divisors are units.
\begin{example}
\label{ex:Coprime}
For a knot~$K$ we argue that~$\Delta_K$ and~$(t-1)$ are relatively prime.
In fact,  since~$\Delta(K)$ is a symmetric polynomial with~$\Delta_K(1)=1$,  we can write~$\Delta_K= (t-1)(t^{-1}-1)q +1$ for some symmetric polynomial $q$.  This immediately implies that~$\Delta_K$ and~$(t-1)$ are relatively prime,  since any common divisor must divide $1$ and hence be a unit. 
\end{example}

One last lemma is needed to describe $\Aut(\Bl_{P_K}).$

\begin{lemma}
\label{lem:coprime}
If~$T_1,T_2$ are torsion~$\Z[t^{\pm 1}]$-modules with relatively prime orders, then
$$\Hom_{\Z[t^{\pm 1}]}(T_1,T_2)=0.$$
\end{lemma}
\begin{proof}
Given~$f \in \Hom_{\Z[t^{\pm 1}]}(T_1,T_2)$, we will show that~$\im(f)=0$.
We first prove that~$\Delta_{\im(f)} \doteq 1$.
Here, given a~$\Z[t^{\pm 1}]$-module $M$,  we let $\Delta_{M}$ denote its order; we refer to~\cite[Section 6]{LickorishIntroduction} for details but note that this is an element of $\Z[t^{\pm1}]$ that is well-defined up to multiplication by units.

The short exact sequences
\begin{align*}
0 &\to \ker (f) \to T_1 \to \im(f) \to 0, \\
0 &\to \im (f) \to T_2 \to \coker(f) \to 0
\end{align*}
imply that~$\Delta_{\im(f)}$ divides both~$\Delta_{T_1}$ and~$\Delta_{T_2}$.
Since by assumption these polynomials are relatively prime, we deduce that~$\Delta_{\im(f)}$ is a unit i.e. that~$\Delta_{\im(f)} \doteq 1$.

We conclude by showing that~$\Delta_{\im(f)} \doteq 1$ implies $\im(f)=0$.
Lemma~\ref{lem:HomDim} implies that the~$\Z[t^{\pm 1}]$-module~$\im(f)$~has a square presentation matrix~$A$. 
It follows that~$1\doteq\Delta_{\im(f)}=\det(A)$ and therefore~$A$ is an isomorphism.
This implies~$\im(f)=0$, as required.
\end{proof}

We are now able to prove the main result of this section.

\begin{proposition}
\label{prop:AutBlPK}
The inclusions~$P \hookrightarrow P_K$ and~$E_K \hookrightarrow P_K$ induce an isomorphism
$$\Aut(\Bl_K) \oplus \Aut(\Bl_{P}) \cong \Aut(\Bl_{P_K}).$$
\end{proposition}
\begin{proof}
Proposition~\ref{prop:BlPK} implies that the inclusions induce an isometry~$\Bl_K \oplus \Bl_{P}\cong \Bl_{P_K}$.
The proposition therefore reduces to establishing the equality~$\Aut(\Bl_K \oplus \Bl_{P})=\Aut(\Bl_{P}) \oplus \Aut(\Bl_K)$.
The $\supset$ inclusion is clear and so we focus on the reverse inclusion.
An isometry~$f \in \Aut(\Bl_K \oplus \Bl_{P})$ can be written as~$f=\bsm f_{11}&f_{12} \\ f_{21}&f_{22} \esm$.
The order of~$H_1(E_K;\Z[t^{\pm 1}])$ is~$\Delta_K$ and the order of~$H_1(P;\Z[t^{\pm 1}])$ is~$(t-1)^{2g+2c}$ (by Proposition~\ref{prop:HomologyPZZ}).
Since these are relatively prime (see Example~\ref{ex:Coprime}), Lemma~\ref{lem:coprime}, implies that~$f_{12}=f_{21}=0$.
This concludes the proof of the proposition.
\end{proof}

%For $c_+=0=c_-$,  Proposition~\ref{prop:AutBlPK} was used implicitly in~\cite{ConwayPowell,ConwayPiccirilloPowell}.

\section{Statements of classifications of immersed $\Z$-surfaces}
\label{sec:Classification}

The goal of this section is to state the classification theorems from the introduction (Theorems~\ref{thm:SurfacesClosedIntro} and~\ref{thm:SurfacesRelBoundaryIntro}) in more detail. 
Let $N$ be a simply-connected $4$-manifold with boundary~$\partial N \cong S^3$, and let $K \subset \partial N$ be a knot.
Recall our notation for the set of immersed~$\Z$-surfaces in~$N$ with boundary~$K$ whose exterior has a fixed equivariant intersection form:
\begin{align*}
\operatorname{Surf}^0_\lambda(g;c_+,c_-)(K,N)&=
\frac{
\{ \Z\text{-surface } S \subset N \mid\partial S=K,  g(S)=g, (c_+,c_-)\text{-double points}, \lambda_{N_S}\cong \lambda \}
}
{\text{equivalence rel. boundary}}.
\end{align*}
Here,  $S$ having ``$(c_+,c_-)$-double points" is an abbreviation for $S$ having $c_+$ positive double points and $c_-$ negative double points.
Also,  $g(S)$ is the genus of $S$,  whereas~$N_S$ denotes the exterior of~$S$ and $\lambda_{N_S}$ refers to its equivariant intersection form.

In the closed case,  we instead fix a closed simply-connected $4$-manifold $X$ and consider
$$\operatorname{Surf}_\lambda(g;c_+,c_-)(X)=
\frac{
\{ \text{closed } \Z\text{-surface } S \subset X \mid g(S)=g, (c_+,c_-)\text{-double points}
 \text{ and } \lambda_{X_S}\cong \lambda\}}{\text{equivalence}}.
$$
In Section~\ref{sub:SurfacesRelBoundary} we study surfaces up to equivalence rel.\ boundary and in Section~\ref{sub:SurfacesClosed}, we focus on closed surfaces.
The statements of these classifications involve isometries of linking forms and differ slightly from those from the introduction.
 Section~\ref{sub:Simplifications} reconciles these statements.
The proofs of these classifications are lengthy and are delayed to Sections~\ref{sec:ImmersionsAndNormalBundles}-\ref{sec:ProofClassifications}.

\begin{notation*}
We recall that the manifolds $P:=P_g(c_+,c_-)$ and $P_K:=P_{K,g}(c_+,c_-)$ depend on the genus $g$ of the surface~$\Sigma:=\Sigma_{g,1}$, integers $c_+,c_- \geq 0$ and the embedding~$\alpha \colon \bigsqcup_{2c} D^2 \hookrightarrow \Sigma.$
\end{notation*}

\subsection{Immersed $\Z$-surfaces up to homeomorphism rel.\ boundary}
\label{sub:SurfacesRelBoundary}

This section is concerned with the classification of immersed $\Z$-surfaces up to equivalence rel.\ boundary and, more specifically, with the set~$\operatorname{Surf}^0_\lambda(g;c_+,c_-)(K,N)$.
In order to obtain a bijection from this set to an orbit set of $\Iso(\partial \lambda,\unaryminus \Bl_{P_K})$, we describe two actions on this set of isometries. 

\begin{construction}[The action of $\Aut(\lambda)$ on $\Iso(\partial \lambda,\unaryminus \Bl_{P_K})$]
\label{cons:AutlambdaAction}
Recall from Construction~\ref{cons:BoundaryLinkingForm} that, associated to a non-degenerate hermitian form~$\lambda$, there is a linking form
$$\partial \lambda \colon \coker({\widehat{\lambda}}) \times \coker({\widehat{\lambda}}) \to \Q(t)/\Z[t^{\pm 1}].$$
Additionally, recall from Remark~\ref{rem:InducedIsometry} that any isometry~$\varpi \colon \lambda \cong \lambda'$ induces an isometry
$$\partial \varpi \colon \partial \lambda \cong \partial \lambda'.$$
In particular,  given a $3$-manifold $Y$, an epimorphism $\psi \colon \pi_1(Y) \twoheadrightarrow \Z$ and a hermitian form $\lambda$ presenting $\Bl_Y$,  the group~$\Aut(\lambda)$ acts on~$\Iso(\partial \lambda,\unaryminus \Bl_Y)$ by
$$\varpi \cdot \phi:=\phi \circ \partial \varpi^{-1}.$$
We will be particularly concerned with the case $(Y,\psi)=(P_K,\varphi).$
\end{construction}

In what follows,  $\Sigma:=\Sigma_{g,1}$ denotes a surface with a single nonempty boundary component.

\begin{construction}[The action of $\Homeo_\alpha(\Sigma,\partial)$ on $\Iso(\partial \lambda,\unaryminus \Bl_{P_K})$]
\label{cons:HomeoSigmaActionRelBoundary}
Use~$\Homeo_\alpha(\Sigma,\partial)$ to denote the group of rel. boundary homeomorphisms~$\theta \colon \Sigma \to \Sigma$ that satisfy~$\theta \circ \alpha=\theta$.
We first argue that a homeomorphism~$\theta \in \Homeo_\alpha(\Sigma,\partial)$ gives rise to a homeomorphism
$$\widehat{\theta} \colon W \to W.$$
Extend the homeomorphism~$\theta|_{\Sigma^\circ} \times \id_{D^2}$ of~$\Sigma^\circ \times D^2$ over each~$D^2 \times D^2$ by the identity to get a homeomorphism of~$E$. 
%This uses \theta \alpha=\alpha.
Then,  verify that this homeomorphism descends to~$ \widehat{\theta} \colon W \to W$.
%%Don't delete.
%Extension over $E$.
%(x,y)~(\alpha_k(x),\eta^\pm(y))
%The first gets sent to (\theta(x),y). 
%The second gets sent to $(\theta\alpha_k(x),y)$.
%Descending to W is then fine  because doesn't involve \theta.

Slightly abusing notation,  we also write~$\widehat{\theta} \colon P_K \to P_K$ for the homeomorphism obtained by extending~$\widehat{\theta}| \colon P \to P$ by the identity on~$E_K$.
This homeomorphism~$\widehat{\theta} \colon P_K \to P_K$ intertwines the coefficient system~$\varphi$,  meaning that $\varphi \circ \widehat{\theta}_* =\varphi.$
%The surface has trivial coefficient system so it doesn't matter that we applied $\theta$ to it.
It follows that~$\theta$ induces an isometry
$$\widehat{\theta}_* \colon \Bl_{P_K} \xrightarrow{\cong}  \Bl_{P_K}.$$
The required action is now by postcomposition: for~$\phi \in \Iso(\partial \lambda,\unaryminus \Bl_{P_K})$, define
$$\theta \cdot \phi := \widehat{\theta}_* \circ \phi.$$
One verifies that this action combines with the one of Construction~\ref{cons:AutlambdaAction} to give rise to an action of~$\Aut(\lambda) \times \Homeo_\alpha(\Sigma,\partial)$ on~$\Iso(\partial \lambda,\unaryminus \Bl_{P_K})$ by $(\varpi,\theta) \cdot \phi=\widehat{\theta}_* \circ \phi \circ \partial \varpi$.
\end{construction}

Recall from the introduction that given a nondegenerate hermitian form $\lambda \colon H \times H \to \Z[t^{\pm 1}]$, we write $\lambda(1)$ for the symmetric bilinear form $(H,\lambda) \otimes_{\Z[t^{\pm 1}]} \Z_{\varepsilon}.$
As we noted then, if $\lambda$ is represented by a matrix $A(t)$, then $\lambda(1)$ is represented by $A(1).$

We state our main classification result on immersed $\Z$-surfaces up to equivalence rel. boundary,  which (together with Proposition~\ref{prop:SimplifyAlgebraBoundary} below) establishes Theorem~\ref{thm:SurfacesRelBoundaryIntro} from the introduction.
Its proof can be found in Section~\ref{sub:ProofRelBoundary}.

\begin{theorem}
\label{thm:SurfacesRelBoundary}
Let~$N$ be a simply-connected~$4$-manifold with boundary~$\partial N \cong S^3$,  let~$K \subset \partial N$ be a knot, and let $c_+,c_-$ and $g$ be non-negative integers.
Set $c:=c_++c_-$.
Given a nondegenerate hermitian form~$\lambda$ over~$\Z[t^{\pm 1}]$, the following assertions are equivalent:
\begin{enumerate}
\item
the hermitian form~$\lambda$ presents~$\Bl_{P_{K,g}(c_+,c_-)}$ and satisfies~$\lambda(1)\cong Q_N \oplus  (0)^{\oplus 2g+c}$;
\item the set~$\operatorname{Surf}_\lambda^0(g;c_+,c_-)(N,K)$ is nonempty and there is a bijection
$$\operatorname{Surf}_\lambda^0(g;c_+,c_-)(N,K) \approx \frac{\Iso(\partial \lambda,\unaryminus\Bl_{P_{K,g}(c_+,c_-)})}{\Aut(\lambda)\times \Homeo_\alpha(\Sigma_{g,1},\partial)}.$$
\end{enumerate}
\end{theorem}

\subsection{Closed immersed $\Z$-surfaces up to homeomorphism}
\label{sub:SurfacesClosed}

This section focuses on the classification of closed immersed $\Z$-surfaces.
The statement is similar to Theorem~\ref{thm:SurfacesRelBoundary},  so we describe the actions more briskly.
In what follows,  $\Sigma$ denotes a closed surface and $U$ denotes the unknot.
\medbreak

Construction~\ref{cons:AutlambdaAction} shows that if a nondegenerate hermitian form $\lambda$ presents $\Bl_{P_U}$, then $\Aut(\lambda)$ acts on~$\Iso(\partial \lambda,\unaryminus \Bl_{P_U})$.
Use~$\Homeo_\alpha(\Sigma)$ to denote the group of orientation-preserving homeomorphisms~$\theta \colon \Sigma \to \Sigma$ that satisfy~$\theta \circ \alpha=\theta$.
Recall from Remark~\ref{rem:ClosedModels} that the constructions of Section~\ref{sec:PlumbedManifolds} can be repeated for closed surfaces and that we denote the resulting plumbed manifolds by $W_U$ and $P_U$.
Since~$\partial W_U=P_U$, the same procedure as in Construction~\ref{cons:HomeoSigmaActionRelBoundary} leads to an action of~$\Homeo_\alpha(\Sigma)$ on $\Iso(\partial \lambda,\unaryminus \Bl_{P_U})$: extending a homeomorphism~$\theta \in \Homeo_\alpha(\Sigma)$ by the identity on the $D^2$ factor leads to a homeomorphism of $W$ which can then be restricted to the boundary.
Again, these actions combine to form an action of~$\Aut(\lambda) \times \Homeo_\alpha(\Sigma)$ on~$\Iso(\partial \lambda,\unaryminus \Bl_{P_U})$.

The next theorem constitutes our main classification result on closed immersed $\Z$-surfaces,  which (together with Proposition~\ref{prop:SimplifyAlgebraClosed} below) establishes Theorem~\ref{thm:SurfacesClosedIntro} from the introduction.
Its proof can be found in Section~\ref{sub:ProofClosed}.

\begin{theorem}
\label{thm:SurfacesClosed}
Let~$X$ be a closed simply-connected~$4$-manifold, and let $c_+,c_-$ and $g$ be non-negative integers.
Set $c:=c_++c_-$.
Given a nondegenerate hermitian form~$\lambda$ over~$\Z[t^{\pm 1}]$, the following assertions are equivalent:
\begin{enumerate}
\item
the hermitian form~$\lambda$ presents~$\Bl_{P_{U,g}(c_+,c_-)}$ and satisfies~$\lambda(1)\cong Q_X \oplus  (0)^{\oplus 2g+c}$;
\item the set~$\operatorname{Surf}_\lambda(g;c_+,c_-)(X)$ is nonempty and there is a bijection
$$\operatorname{Surf}_\lambda(g;c_+,c_-)(X) \approx \frac{\Iso(\partial \lambda,\unaryminus \Bl_{P_{U,g}(c_+,c_-)})}{\Aut(\lambda) \times \Homeo_\alpha(\Sigma_g)}.$$
\end{enumerate}
\end{theorem}

\subsection{Simplifications of the targets}
\label{sub:Simplifications}

Sections~\ref{sub:SurfacesRelBoundary} and~\ref{sub:SurfacesClosed} described bijections between equivalence classes of surfaces and certain sets of isometries. 
In this section, we recast these sets so as to make them more amenable to calculations.
We continue with the notation of the previous section.

\begin{convention}
\label{conv:TargetSimplify}
Let~$Y$ be a closed~$3$-manifold and let~$\psi \colon \pi_1(Y) \twoheadrightarrow \Z$ be an epimorphism.
If a hermitian form~$\lambda$ presents~$\Bl_Y$, then there are noncanonical bijections
$$\Iso(\partial \lambda,\unaryminus \Bl_Y) \cong \Aut(\Bl_Y).$$
%\quad \text{ and } \quad  \Iso(\partial \lambda,\unaryminus \Bl_Y) \cong \Aut(\partial \lambda).$$
Indeed,  since~$\lambda$ presents~$\Bl_Y$,  the leftmost set is nonempty, so fixing a~$h \in \Iso(\partial \lambda,\unaryminus \Bl_Y)$ (there is no preferred such choice), the bijection is~$g \mapsto g\circ  h^{-1}$.
%first bijection is~$g \mapsto g\circ  h^{-1}$, whereas the second is~$g \mapsto g^{-1} \circ h$.

Throughout this section, we will make such identifications implicitly and use them to transport actions on~$\Iso(\partial \lambda,\unaryminus \Bl_Y)$ to actions on~$\Aut(\Bl_Y)$.
In particular,  if~$\lambda$ presents~$\Bl_{P_K}$, then will we use without any further mention that the set~$\Aut(\Bl_{P_K})$ supports actions of~$\Aut(\lambda)$ and~$\Homeo_\alpha(\Sigma)$.
\end{convention}

\begin{proposition}
\label{prop:SimplifyAlgebraClosed}
Let~$\Sigma$ be a closed genus $g$ surface and set $P_U:=P_{U,g}(c_+,c_-).$
If a nondegenerate hermitian form~$\lambda$ presents~$\Bl_{P_U}$, then there is a bijection
$$ \frac{\Iso(\partial \lambda,\unaryminus \Bl_{P_U})}{\Aut(\lambda) \times \Homeo_\alpha(\Sigma)}\cong 
\frac{\Aut(\Bl_{P_U})}{\Aut(\lambda) \times \Homeo_\alpha(\Sigma)}.$$
\end{proposition}
\begin{proof}
This follows directly from Convention~\ref{conv:TargetSimplify}.
\end{proof}

Next we focus on the target of the bijection from Theorem~\ref{thm:SurfacesRelBoundary} which classifies immersed~$\Z$-surfaces with boundary a knot~$K$.
Assume that~$\lambda$ presents~$\Bl_{P_K}$.
As noted in Convention~\ref{conv:TargetSimplify}, there is a bijection~$ \Iso(\partial \lambda,\unaryminus \Bl_{P_K}) \cong \Aut(\Bl_{P_K})$ and Proposition~\ref{prop:AutBlPK} then implies that the inclusions~$P \subset P_K$ and~$E_K \subset P_K$ induce an isomorphism
\begin{equation}
\label{eq:SplitBlPk}
 \Aut(\Bl_{P}) \oplus \Aut(\Bl_K) \xrightarrow{\cong}  \Aut(\Bl_{P_K}).
 \end{equation}
Note furthermore that~$\Homeo_\alpha(\Sigma)$ acts on~$\Aut(\Bl_P)$ as~$\theta \cdot h=\widehat{\theta}_* \circ h$ and that this agrees with the restriction of the action of~$\Homeo_\alpha(\Sigma)$ on $\Aut(\Bl_{P_K})$ to~$ \Aut(\Bl_P)$.
The fact that the~$\Homeo_\alpha(\Sigma)$-action on~$\Aut(\Bl_{P_K})$ restricts to $\Aut(\Bl_P)$ uses that $\Homeo_\alpha(\Sigma)$ acts trivially on $\Aut(\Bl_K).$

\begin{proposition}
\label{prop:SimplifyAlgebraBoundary}
Let~$\Sigma$ be a genus $g$ surface with one boundary component.
Set~$P_K:=P_{K,g}(c_+,c_-)$ and~$P:=P_g(c_+,c_-).$
If a nondegenerate hermitian form~$\lambda$ presents~$\Bl_{P_K}$, then there is a bijection
\begin{align*}
&\frac{\Iso(\partial \lambda,\unaryminus \Bl_{P_K})}{\Aut(\lambda) \times \Homeo_\alpha(\Sigma,\partial)} \cong
  \frac{(\Aut(\Bl_P)/\Homeo_\alpha(\Sigma,\partial))\times \Aut(\Bl_K)}{\Aut(\lambda)}.
\end{align*}
\end{proposition}
\begin{proof}
Proposition~\ref{prop:AutBlPK} ensures that every isometry~$f \in \Aut(\Bl_{P_K})$ decomposes as~$(f_P,f_K)$ with~$f_P \in \Aut(\Bl_P)$ and~$f_K \in \Aut(\Bl_K)$.
We verify that the assignment~$f \mapsto [([f_P],f_K)]$ gives the required bijection.
First, note that this map does indeed descend to the relevant orbit sets: since the action of~$\Homeo_\alpha(\Sigma,\partial)$ on~$ \Aut(\Bl_{P_K})$ is trivial on~$\Aut(\Bl_K)$,  an element~$\widehat{\theta}_* \circ f \circ \partial \varpi$ with~$\varpi \in \Aut(\lambda)$ and~$\theta \in \Homeo_\alpha(\Sigma,\partial)$ is mapped to~$(\widehat{\theta}_* \circ f_P \circ (\partial \varpi)_P,f_K \circ (\partial \varpi)_K)$ and the latter is equivalent to~$(f_P,f_K).$
The assignment on the orbit sets certainly remains surjective,  so it only remains to prove it is injective.
Assume that~$[f]$ and~$[g]$ are such that~$[([f_P],f_K)]=[([g_P],g_K)]$,  meaning that~$f_P=\widehat{\theta}_* \circ g_P \circ (\partial \varpi)_P$ and~$f_K=g_K \circ (\partial \varpi)_K$ for~$\varpi \in \Aut(\lambda)$ and~$\theta \in \Homeo_\alpha(\Sigma,\partial)$,  but this is precisely the condition that~$f=\widehat{\theta}_* \circ g \circ \partial \varpi$,  establishing injectivity.
\end{proof}

\section{Equivariant intersection forms of immersed~$\Z$-surface exteriors}
\label{sec:EquivariantIntersectionForm}

One take away from Section~\ref{sec:Classification} is the following: in order to understand closed immersed $\Z$-surfaces whose exteriors have a prescribed equivariant intersection form $\lambda$, we need to analyze the orbit set~$\Aut(\Bl_{P_U})/(\Aut(\lambda) \times \Homeo_\alpha(\Sigma))$.
%This is the topic of the next couple of sections.
In this section we begin by collecting facts related to equivariant intersection forms that will be relevant to this analysis.

\medbreak

We begin by describing the $\Z[t^{\pm 1}]$-homology of an immersed $\Z$-surface exterior.
In the next lemma it is understood that if $N$ has boundary $\partial N \cong S^3$ and $S \subset N$ is an immersed surface,  then~$\partial S \subset \partial N$ is nonempty.

\begin{lemma}
\label{lem:HomologyZZExterior}
Let $N$ be a simply-connected $4$-manifold that is either closed or has $\partial N \cong S^3$.
The~$\Z[t^{\pm 1}]$-homology of the exterior of a genus~$g$ immersed~$\Z$-surface~$S \subset N$ with~$c$ double points~is
$$
H_i(N_S;\Z[t^{\pm 1}])\cong
\begin{cases}
\Z_\varepsilon &\quad \text{ for } i=0, \\
\Z[t^{\pm 1}]^{2g+c+b_2(N)} &\quad \text{ for } i=2, \\
0 &\quad \text{ otherwise.}
\end{cases}
$$
\end{lemma}
\begin{proof}
We focus on the case where $\partial N \cong S^3$ and $K:=\partial S$ is a knot,  as the closed case is entirely analogous.
 Proposition~\ref{prop:HomologyPKZZ} proves that~$H_1(\partial N_S;\Z[t^{\pm 1}]) \cong H_1(P_K;\Z[t^{\pm 1}])$ is torsion and Proposition~\ref{prop:HomologyZExterior} shows that~$\pi_1(\partial N_S) \to \pi_1(N_S) \cong \Z$ is surjective.
We now apply~\cite[Lemma~3.2]{ConwayPowell} to deduce that~$H_0(N_S;\Z[t^{\pm 1}]) \cong \Z_\varepsilon,H_i(N_S;\Z[t^{\pm 1}])=0$ for~$i \neq 0,2$ and that~$H_2(N_S;\Z[t^{\pm 1}])$ is free.
An Euler characteristic argument now implies that~$H_2(N_S;\Z[t^{\pm 1}]) \cong \Z[t^{\pm 1}]^{b_2(N_S)}$.
Proposition~\ref{prop:HomologyZExterior} calculated $b_2(N_S)=2g+c+b_2(N)$, thus concluding the proof.
\end{proof}

We list some facts about the equivariant intersection forms of immersed $\Z$-surface exteriors.

\begin{proposition}
\label{prop:NecessaryConditions}
~
\begin{itemize}
\item Let $N$ be a simply-connected $4$-manifold with $\partial N \cong S^3$ and let $K \subset \partial N$ be a knot.  
If~$S \subset N$ is a genus $g$ immersed $\Z$-surface with boundary $K$ and $c$ double points, then 
%its equivariant intersection form 
$\lambda_{N_S}$ is nondegenerate, presents $\Bl_{P_K}$ and satisfies~$\lambda_{N_S}(1) \cong Q_N \oplus (0)^{2g+c}$.
\item Let $X$ be a closed simply-connected $4$-manifold.
If $S \subset X$ is a genus $g$ closed immersed~$\Z$-surface with $c$ double points, then 
%its equivariant intersection form 
$\lambda_{X_S}$ is nondegenerate, presents $\Bl_{P_U}$ and satisfies~$\lambda_{X_S}(1) \cong Q_X \oplus (0)^{2g+c}$.
\end{itemize}
\end{proposition}
\begin{proof}
We focus on the case where $\partial N \cong S^3$ and $K:=\partial S$ is a knot,  as the closed case is entirely analogous.
Throughout this proposition, we use the identification $\partial N_S \cong P_K$ implicitly.
The fact that $\lambda_{N_S}$ is nondegenerate follows from~\cite[Lemma 3.2 item (4)]{ConwayPowell}, whereas the fact that $\lambda_{N_S}$ presents $\Bl_{P_K}$ follows from Proposition~\ref{prop:PresentsBlanchfield}.
Finally, the arguments from~\cite[Lemma 5.10]{ConwayPowell} show that $\lambda_{N_S}(1) \cong Q_{N_S}$ and Proposition~\ref{prop:HomologyZExterior} shows that $Q_{N_S} \cong Q_N \oplus (0)^{2g+c}.$
\end{proof}

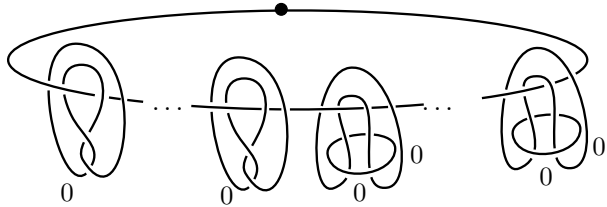
\begin{figure}[!htbp]
\begin{centering}
\begin{tikzpicture}
\node[anchor=south west,inner sep=0] at (0,0)
{\includegraphics[height=2.5cm]{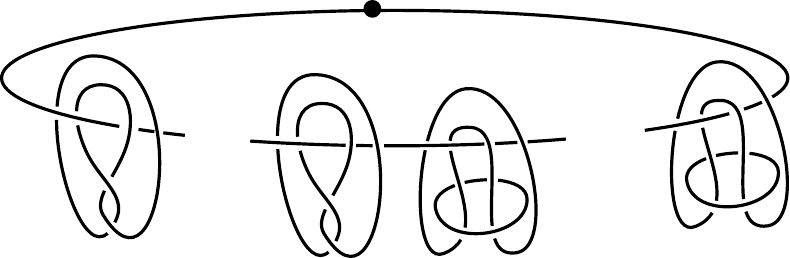}};
%\node at (1.5, 3.2){$z$}; \node at (1.2,1.1){$x$}; \node at (1.8, 1.1){$y$};
%\node at (1.45, 1.7){$\gamma$}; \node at (.8,0){$\zeta$}; 
% \node at (.6,.9){$\alpha$}; \node at (2.3,.9){$\beta$};
% \node at (1.1,  2.5){$\epsilon$}; \node at (1.5,  .5){$\Delta$};
\node at (.8,0){$0$}; \node at (2.9,-.05){$0$}; \node at (4.65,0){$0$}; \node at (7.1,.2){$0$}; 
\node at (5.4,.5){$0$}; \node at (7.8,.6){$0$};
\node at (2.15,1.1){$\dots$}; \node at (5.7,1.1){$\dots$}; 
\end{tikzpicture}
\end{centering}
\caption{A handle diagram for the exterior of the standard genus $g$ immersed surface $S$ in $S^4$ with $c=c_++c_-$ double points,  given by attaching $c+2g$ 2-handles to $S^1 \times D^3$.   Note that the signs of the clasps of the components on the left part of the figure are determined by the signs of the double points--we depict the all negative case. }
\label{fig:exterior}
\end{figure}

\begin{example}
\label{ex:SimpleImmersed}
We introduce examples of immersed $\Z$-surfaces in $S^4$ and $D^4$, that roughly speaking, play the r\^ole of the unknot in the theory of knotted surfaces.
\begin{itemize}
\item The \emph{standard} genus $g$ immersed $\Z$-surface~$S \subset S^4$ with~$c_+$ positive double points and $c_-$ negative double points is the one obtained from the unknotted genus~$g$ surface by adding local transverse double points of the appropriate sign.
Drawing a handle diagram for this surface as in Figure~\ref{fig:exterior} (see~\cite[pages 216-217]{GompfStipsicz})
%Kirby diagrams are built from bottom to top. We got the sign right.
%%%Don't delete
%{AM: Also, check the signs in the figures and the captions and the framings.
%Checked: In the left part of Figure 3,  the 0-framed curve has writhe +2,  which since its lift to the infinite cyclic cover has writhe 0 implies that its lift has framing -2, and we obtain an equiv intersection pairing [$t-2+t^{-1}$]=[$-(t-1)(t^{-1}-1)$] as drawn in Figure 4,  which is consistent with the intersection pairing we give that corresponds to one negative double point. } 
and the infinite cyclic cover of its exterior as in Figure~\ref{fig:exteriorcover},   we see that the~$\Z[t^{\pm 1}]$-intersection form of this surface exterior is represented by the size~$2g+c$ matrix
$$ \overbrace{((t-1)(t^{-1}-1))^{\oplus c_+} \oplus (-(t-1)(t^{-1}-1))^{\oplus c_-}}^{:=\lambda_{c_+,c_-}} \oplus \overbrace{\begin{pmatrix} 0&t-1 \\ t^{-1}-1&0 \end{pmatrix}^{\oplus g}}^{:=\mathcal{H}_2^{\oplus g}}.$$
%Here we use~$\mathcal{H}_2$ as a shorthand for~$\bsm 0&t-1 \\ t^{-1}-1&0 \esm$.
\begin{figure}[!htbp]
\begin{centering}
\begin{tikzpicture}
\node[anchor=south west,inner sep=0] at (0,0)
{\includegraphics[height=3cm]{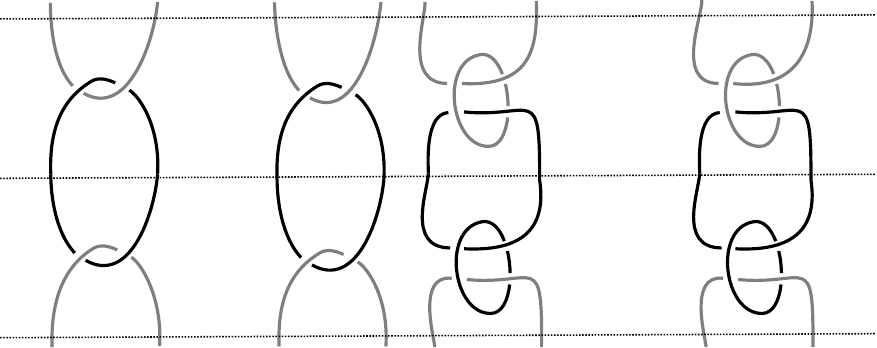}};
%\node at (1.5, 3.2){$z$}; \node at (1.2,1.1){$x$}; \node at (1.8, 1.1){$y$};
%\node at (1.45, 1.7){$\gamma$}; \node at (.8,0){$\zeta$}; 
% \node at (.6,.9){$\alpha$}; \node at (2.3,.9){$\beta$};
% \node at (1.1,  2.5){$\epsilon$}; \node at (1.5,  .5){$\Delta$};
\node at (.1,1.7){$-2$}; \node at (2.1,1.7){$-2$}; \node at (4.9,1.7){$0$}; \node at (3.9,1.2){$0$};  \node at (7.2,1.7){$0$}; \node at (6.3,1.2){$0$};
\node at (1.9,1.25){$\dots$}; \node at (5.5,1.25){$\dots$};  \node at (.8,-.4){$\vdots$}; \node at (2.8,-.4){$\vdots$};  \node at (4.2,-.4){$\vdots$};  \node at (6.6,-.4){$\vdots$};
\node at (.8,3.4){$\vdots$}; \node at (2.8,3.4){$\vdots$};  \node at (4.2,3.4){$\vdots$};  \node at (6.6,3.4){$\vdots$};
\end{tikzpicture}
\end{centering}
\caption{The infinite cyclic cover of the exterior of $S$ is obtained from $\mathbb{R} \times D^3$ by attaching 2-handles.  We draw the preferred lift of each downstairs 2-handle in a darker line,  and note that the diagram is invariant under the $\Z$-action generated by shifting the figure up by one unit. }
\label{fig:exteriorcover}
\end{figure}
%%AC: I commented this because it's in a proposition below
%We remark that this gives an alternate means of calculation for $\Bl_{P_g(c_+,c_-)}$. 
\item An Alexander polynomial one knot $K$ admits a unique~$\Z$-disk up to isotopy rel. boundary~\cite{Freedman,ConwayPowellDiscs}.
The \emph{standard} genus $g$ immersed~$\Z$-surface for $K$ with~$c_+$ positive double points and $c_-$ negative double points is the one obtained by connect summing the aforementioned $\Z$-disk with the corresponding standard genus $g$ immersed surface.
%%Don't delete.
%%AC: As in the embedded case the exteriors get identified along \mu x D^2 so the two meridians become identified. 
%Since each pi_1=\Z is generated by a meridian, the result follows.
Since~$\Z$-disk exteriors in $D^4$ are aspherical~\cite[Lemma 2.1]{ConwayPowell},  a Mayer-Vietoris argument implies that the~$\Z[t^{\pm 1}]$-intersection form of this surface exterior is also $\lambda_{c_+,c_-} \oplus \mathcal{H}_2^{\oplus g}.$
\end{itemize}
\end{example}

\section{The Blanchfield form of~$P$ and its isometries}
\label{sec:BlP}

Section~\ref{sec:Classification} demonstrates that understanding the group~$\Aut(\Bl_P)$ is crucial to the classification of immersed~$\Z$-surfaces,  where as before $P$ is a shorthand for $P_g(c_+,c_-)$. 
The goal of this section is to establish cases in which generators for $\Aut(\Bl_P)$ can be understood.
\medbreak

We begin by describing the isometry type of $\Bl_P$.

\begin{proposition}
\label{prop:BlP}
The Blanchfield forms of~$P$ and $P_U$ are presented by
$$ \lambda_{c_+,c_-} \oplus \mathcal{H}_2^{\oplus g}:=
((t-1)(t^{-1}-1))^{\oplus c_+} \oplus (-(t-1)(t^{-1}-1))^{\oplus c_-}
 \oplus \begin{pmatrix} 0&t-1 \\ t^{-1}-1&0 \end{pmatrix}^{\oplus g}.$$
\end{proposition}
\begin{proof}
Proposition~\ref{prop:BlPK} ensures that the inclusion~$P \hookrightarrow P_U$ induces an isometry~$\Bl_{P} \cong \Bl_{P_U}.$
The exterior of any closed surface $S \subset S^4$ has boundary $P_U$.
If $S$ is a $\Z$-surface, Proposition~\ref{prop:PresentsBlanchfield} therefore ensures that the equivariant intersection form of its exterior presents $\Bl_{P_U}.$
In particular, this applies to the closed $\Z$-surfaces described in Example~\ref{ex:SimpleImmersed} whose equivariant intersection is represented by the required matrix.
\end{proof}

\subsection{Examples of isometries of $\Bl_P$}
\label{sub:IsometriesOfBlP}

We describe several families of isometries of $\Bl_P \cong \Bl_{P_U}$.
In this and the next subsection, for brevity we set~ $c:=c_++c_-$ and
$$\lambda:=\lambda_{c_+,c_-} \oplus \mathcal{H}_2^{\oplus g}.$$
The form~$\lambda$ is defined on $\Z[t^{\pm 1}]^c \oplus \Z[t^{\pm 1}]^{2g}$.
We write~$X_1,\ldots,X_c$ for a basis of the first summand and~$Y_1,\ldots,Y_{2g}$ for a basis of the second.
Finally we use~$x_1,\ldots,x_c$ and~$y_1,\ldots,y_{2g}$ for the projections of the duals of these elements to~$\coker({\widehat{\lambda}})$, which  generate~$\coker({\widehat{\lambda}}).$ 
Finally,  for~$i=1,\dots, c$ we define $\varepsilon_i \in \{ \pm 1\}$ so that $\Bl_P(x_i,x_i)= \frac{\varepsilon_i}{(t-1)(t^{-1}-1)}$.

We now describe examples of isometries of $\Bl_P \cong \partial \lambda$.

\begin{example}[Scaling a generator by $t^k$]
\label{ex:ScaleIsometry}
The isomorphism $\coker({\widehat{\lambda}}) \to \coker({\widehat{\lambda}})$ obtained by multiplying a single generator $x_i$ by $t^k$ (with $k \in \Z$), fixing all other generators and extending~$\Z[t^{\pm 1}]$-linearly,  defines an isometry of~$\partial \lambda$.
\end{example}

\begin{example}[The isometry $p_{i,i'}$]
\label{ex:fij}
For all distinct choices of ~$1 \leq i,i' \leq c$,  the isomorphism $p_{i,i'}$ of~$\coker({\widehat{\lambda}})$ obtained by $\Z[t^{\pm 1}]$-linearly extending the assignment that satisfies
$$
p_{i,i'} \colon 
\begin{cases}
x_i \mapsto & x_i+\varepsilon_{i'} (t-1)x_{i'}  \\
x_{i'} \mapsto& \varepsilon_i(t-1)x_i+x_{i'},
%x_k & \quad \text{ if } k \neq i,i'.
\end{cases}
$$
and fixes all other generators is seen to define an isometry of~$\partial \lambda.$
\end{example}

\begin{example}[The isometries $q_{i,j}$]
\label{ex:Rhoij}
For every $i=1,\ldots,c$ and $k=1,\ldots,g$,  the isomorphism $q_{i,2k-1}$ of $\coker({\widehat{\lambda}})$ obtained by $\Z[t^{\pm 1}]$-linearly extending the assignment that satisfies
$$
q_{i,2k-1}\colon \left\{\begin{array}{rcrrr} x_i  &\mapsto&  x_i &+  y_{2k-1}& \\ 
y_{2k-1} & \mapsto & & y_{2k-1} &\\
y_{2k} & \mapsto& \varepsilon_i (t-1)x_i& - y_{2k-1}&+y_{2k}
\end{array} \right.
$$
and fixes all other generators is seen to define an isometry of~$\partial \lambda$.
The same can be said for
$$
q_{i,2k}\colon \left\{\begin{array}{rcrrr} x_i  \mapsto&  x_i & &+ y_{2k} \\ 
y_{2k-1}  \mapsto& \varepsilon_i (t-1)x_i &+ y_{2k-1}&+y_{2k} \\
y_{2k}  \mapsto& && y_{2k}
\end{array} \right.
$$
where, again, all other generators are fixed.
%%Don't delete
%Could have written  \varepsilon_i (t-1)x_i &*\pm* y_{2k-1}&+y_{2k} and it would still have worked.
%We just picked what was easiest to realize by Dehn twists.
\end{example}

\begin{example}[$\operatorname{Sp}(2g,\Z)$-type isometries]
\label{ex:Sp2gZ}
Note that $\coker(\mathcal{H}_2^{\oplus g}) \cong \Span_\Z(y_1,\ldots,y_{2g})$.
For every~$A \in \operatorname{Sp}(2g,\Z)$,  the isomorphism $r_A$ of $\coker({\widehat{\lambda}})= \coker({\widehat{\lambda}_{c_+,c_-}}) \oplus \coker(\mathcal{H}_2^{\oplus g})$ defined by  
$$
r_A
\begin{pmatrix}
x \\
y
\end{pmatrix}
=
\begin{pmatrix}
x \\
Ay
\end{pmatrix}
$$
defines an isometry of~$\partial \lambda$.
This follows from~\cite[Proposition 5.6]{ConwayPowell} which shows that
$$\Aut(\coker(\mathcal{H}_2^{\oplus g}),\partial \mathcal{H}_2^{\oplus g})
 \cong \Aut(H_1(\Sigma_g \times S^1;\Z[t^{\pm 1}]), \Bl_{\Sigma_g \times S^1})
 \cong \Aut(H_1(\Sigma_g),Q_{\Sigma_g}) \cong \operatorname{Sp}(2g,\Z).$$
The exact same result holds if the surface has one boundary component.
\end{example}

\begin{example}[Scaling a generator by $-1$]
\label{ex:ScaleIsometry-1}
The isomorphism $\coker({\widehat{\lambda}}) \to \coker({\widehat{\lambda}})$ obtained by sending $x_i \mapsto -x_i$ for a single $i$,  leaving all other generators intact, and extending $\Z[t^{\pm 1}]$-linearly defines an isometry of $\partial \lambda$.
\end{example}

\begin{example}[Permutation-induced isometries]
\label{ex:PermutationIsometry}
Given a permutation $\sigma \in S_{c_+}$, the isomorphism $f_\sigma^+ \colon \coker({\widehat{\lambda}}) \to \coker({\widehat{\lambda}})$ obtained by permuting the generators $x_1,\ldots,x_{c_+}$ according to $\sigma$, leaving all other generators intact,   and extending $\Z[t^{\pm 1}]$-linearly defines an isometry of $\partial \lambda$.
One can define~$f_\sigma^-$ analogously for a permutation~$\sigma \in S_{c_-}$.
\end{example}

\subsection{Isometries of $\Bl_P$ modulo isometries of the standard form}
\label{sub:IsomBlP}

This section collects several results on the group~$\Aut(\Bl_P)$.
Combined with the classifications of Section~\ref{sec:Classification}, these form the backbone of the proofs of the theorems from Sections~\ref{sub:IntroD4S4} and~\ref{sub:Other4ManifoldsIntro}.

\begin{proposition}
\label{prop:AutBl/AutStandard}
Every isometry $f \in \Aut(\partial \lambda_{c_+,c_-} \oplus \partial \mathcal{H}^{\oplus 2g})$ can be written as 
$$ f=s \circ \partial \rho$$
for  $s$ a composition of the maps described in Examples~\ref{ex:ScaleIsometry}-\ref{ex:Sp2gZ} (namely the  $t^k$, the $p_{i,i'}$, the $q_{ij}$ and the~$\operatorname{Sp}(2g,\Z)$-isometries) and their inverses, 
and $\rho \in \Aut(\lambda_{c_+,c_-}\oplus \mathcal{H}_2^{\oplus 2g})$.
\end{proposition}
\begin{proof}
For brevity,  we set~$\lambda:=\lambda_{c_+,c_-}\oplus \mathcal{H}_2^{\oplus 2g}$ and $c:=c_++c_-$.
The module underlying this form is~$\Z[t^{\pm 1}]^{\oplus c} \oplus \Z[t^{\pm 1}]^{\oplus 2g}.$
Write $X_1,\ldots,X_c$ for a basis of the first summand and $Y_1,\ldots,Y_{2g}$ for a basis of the second.
We write $x_1,\ldots,x_c$ and $y_1,\ldots,y_{2g}$ for the projections of the duals of these elements to $\coker({\widehat{\lambda}}).$
During this proof, we write~$\Bl:=\partial \lambda$.

Consider the subsets $H_x:=\Span_\Z(x_1,\ldots,x_c)$ and $H_y:=\Span_\Z(y_1,\ldots,y_{2g})$ of $\coker({\widehat{\lambda}})$.
Since elements of $\coker({{\widehat{\lambda}_{c_+,c_-}}})$ have order $(t-1)^2$,  they can be written as $a+b(t-1)$ with $a,b\in H_x$.
Similarly,  since~$\coker(\mathcal{H}_2^{\oplus g}) \cong \Z_\varepsilon^{2g} \cong H_y$, elements of $H_y$ can be written as $\Z$-linear combinations of the $y_i$.  
%%Don't delete.
%Note that we conflated the forms with their matrices in these bases.
We can therefore write
\begin{align}\label{eqn:form_of_fStep1}
f(x_i)&=a_i+(t-1)b_i+c_i,  \\
f(y_j)&=\phantom{a_i+ }(t-1)d_j+e_j  \nonumber
\end{align}
for some $a_i,b_i,d_j \in  H_x$ and $c_i,e_j \in H_y$.

\begin{itemize}
\item Step 1.
We prove that after postcomposing the isometry~$f$ with an isometry of the form~$\partial \rho$ with~$\rho \in \Aut(\lambda)$ we can arrange that~$a_i=x_i$ for all $i=1,\dots, c$. 

We assert that $\Bl(x_i,x_{i'})=\Bl(a_i,a_{i'})$ for all $1 \leq i, i' \leq c$. 
Since the $a_i$ are $\Z$-linear combinations of the $x_i$,  we have that $\Bl(a_i,a_{i'})=n/(t-1)(t^{-1}-1)$ for some $n \in \Z$.
For~$i \neq i'$, we have that for some polynomial $p(t)$,
$$0=\Bl(x_i,x_{i'})=\Bl(f(x_i),f(x_{i'}))=\overbrace{\Bl(a_i,a_{i'})}^{=\frac{n}{(t-1)(t^{-1}-1)}}+\frac{(t-1)p(t)}{(t-1)(t^{-1}-1)}.$$
%%The Bl(c_i,c_i’) term also ends up in that second term.
This implies that $n+(t-1)p(t) \in (t-1)(t^{-1}-1)\Z[t^{\pm 1}]$ which in turn forces~$n=0$.  It follows that~$\Bl(a_i,a_{i'})=0=\Bl(x_i,x_{i'})$ for $i \neq {i'}$.
The same reasoning applied to the case where~$i=i'$ implies that~$\Bl(a_i,a_i)=\pm 1/(t-1)(t^{-1}-1)=\Bl(x_i,x_i)$.
This concludes the proof of the assertion.

Each $a_i$ can be written as a $\Z$-linear combination of the $x_i$: write $A_i$ for the corresponding $\Z$-linear combination of the $X_i$.
We write $H_X$ for the $\Z[t^{\pm 1}]$-span of the $X_i$ and $H_Y$ for the~$\Z$-span of the $Y_j$.

We assert that the  map~$\rho \colon H_X \to H_X$ obtained by $\Z[t^{\pm 1}]$-linearly extending the assignment $\rho(X_i):=A_i$ is an isometry of $\lambda$.
We first note $\rho$ is form-preserving: since $A_i$ is the same $\Z$-linear combination of $\{X_i\}$ as $a_i$ of the $\{x_i\}$, the diagonalisability of~$\lambda|_{H_X}$ and the equality $\Bl(a_i,a_{i'})=\Bl(x_i,x_{i'})$ imply that $\lambda(A_i,A_{i'})=\lambda(X_i,X_{i'}).$
%%Don't delete (see also lambda(AiAj)=lambda(XiXj).jpeg for a cleaner version).
%\Bl(x_i,x_j)=\sum_{kl} \alpha_{ik} \beta_{jl} x_i^TDx_j.
%\Bl(a_i,a_j)=\Bl(A_i^*,A_j^*)=x_i^TDx_j.
%This gives \sum_{kl} \alpha_{ik} \beta_{jl} -1 \in \Z[t^{\pm 1}] .
%Since the $\alpha,\beta$ are integers, this forces \sum_{kl} \alpha_{ik} \beta_{jl}=1.
%Then use this when repeating the calculation for $\lambda(X_i,X_j)$ and $\lambda(A_i,A_j).$
Since~$\lambda|_{H_X}=\lambda_{c_+,c_-}$ is nondegenerate,  this implies that~$\rho$ is injective.
The assertion therefore reduces to proving that~$\rho$ is surjective.
Since the~$X_i$ generate~$H_X$ as a~$\Z[t^{\pm1}]$-module, it suffices to show that the~$X_i$ can be written as a~$\Z$-linear combination of the~$A_i$.  
Further,  since the~$A_i$ are the same~$\Z$-linear combinations of the~$X_i$ as the~$a_i$ are of the~$x_i$,  it is enough for us to show that  the~$x_i$ are~$\Z$-linear combinations of the~$a_i$.
%%Don't delete.
%<a_1,...,a_c>_\Z\cong <A_1,\ldots,A_c>_\Z
%<x_1,...,x_c>_\Z\cong <X_1,\ldots,X_c>_\Z
%So x_i being Z-linear combination of the~$A_i$ would mean~$X_i$ being~$Z$-linear combination of the~$A_i$.
Since~$f$ is surjective and~$x_1,\dots,x_c,y_1,\dots,y_{2g}$ generate~$\coker({\widehat{\lambda}})$, 
,we can write any~$x \in H_x$ as a linear combination of the~$f(x_i)$ and~$f(y_j)$:
$$x=\sum_i (\alpha_i+(t-1)\beta_i)f(x_i)+\sum_j (\gamma_j + (t-1)\delta_j) f(y_j),$$
for some~$\alpha_i,\beta_i,\gamma_j,\delta_j\in \Z$. 
Expanding out this expression and using~\eqref{eqn:form_of_fStep1} we obtain
$$H_x \ni x =\sum_i \alpha_ia_i + (\text{expression in $(t-1)H_x$}) + (\text{expression in $H_y$}).$$
This forces the expressions in $(t-1)H_x$ and $H_y$ to vanish, thus establishing that $x$ is a~$\Z$-linear combination of the $a_i$.
This concludes the proof that $\rho$ is an isometry.
%\\
\medskip

\noindent
Step $1$ follows: the assertion shows that~$\partial \rho \circ f(a_i)=x_i$.
%\\
\medskip

\noindent
At this point,  we can assume that
\begin{align}\label{eqn:form_of_f_2}
f(x_i)&=x_i+(t-1)b_i+c_i,  \\
f(y_j)&=\phantom{a_i+ }(t-1)d_j+e_j, \nonumber
\end{align}
again for $b_i, d_j \in H_x$ and $c_i, e_j \in H_y$,  though the specific values may be different than before.
Our long term goal is to show that this new~$f$ must be a composition of the isometries given in Examples~\ref{ex:ScaleIsometry}-\ref{ex:Sp2gZ},  together with their inverses.

\item Step 2.
We now argue that after postcomposing by a composition of the $q_{i,j}$-isometries given in Example~\ref{ex:Rhoij} we can assume that $c_i=0$ for all $i=1,\dots,c$. 

Consider the $\Z$-linear (and therefore $\Z[t^{\pm 1}]$-linear) assignment $\psi \colon H_y \to H_y$ obtained by $\Z$-linearly extending $\psi(y_j)=e_j$.
We claim that~$\psi$ is an isomorphism.
To show that~$\psi$ is injective,  since $\Bl|_{H_y}=\mathcal{H}_2^{\oplus g}$ is nondegenerate,  it suffices to show that~$\psi$ is form preserving, and this is a direct calculation: 
$$ \Bl(y_j,y_k)=\Bl(f(y_j),f(y_k))=(t-1)(t^{-1}-1)\Bl(d_j,d_k)+\Bl(e_j,e_k)=\Bl(e_j,e_k). $$
Surjectivity amounts to verifying that every element $y$ in the $\Z$-span of the $y_i$ can be written as a $\Z$-linear combination of the $e_j$.
% the~$\Z$-span of the $e_j$ agrees with the $\Z$-span of the $y_j$, which in turn reduces to show that the $y_i$ can be written as a $\Z$-linear combination of the $e_j$.
%Write $y_k \in $
Since $f$ is an isometry, we can write $y$ as a linear combination of the $f(x_i)$ and $f(y_j)$ of the form
$$y=\sum_i (\alpha_i+(t-1)\beta_i)f(x_i)+\sum_j (\gamma_j+(t-1)\delta_j)f(y_j),$$
where $\alpha_i,\beta_i,\gamma_j,\delta_j\in \Z$.
Expanding this expression using Equation~\eqref{eqn:form_of_f_2},  we obtain
$$H_y \ni y =\sum_i \alpha_i x_i + (\text{expression in $(t-1)H_x$}) +  \sum_i \alpha_i c_i +\sum_j \gamma_j e_j .$$
Since $y$ is in $H_y$,  the term $\sum_i \alpha_i x_i  \in H_x$ and the expression in $(t-1)H_x$ must vanish. 
Since the $\alpha_i$ are integers and the $x_i$ are only annihilated by powers of~$(t-1)(t^{-1}-1)$,  it follows that the $\alpha_i$ themselves all vanish.
We deduce that~$ \sum_i \alpha_i c_i =0$ and so $y$ is indeed a $\Z$-linear combination of the~$e_j$.

We now prove Step 2.
We wish to use the~$e_i$ to clear the~$c_i$ via postcomposing with the~$q_{i,j}$ isometries.  
Observe that postcomposing with $q_{i,j}$ has the effect of modifying $c_i$ by adding $e_j$.  Every $c_i$ is a linear combination of the $y_j$ by definition,  and when combined with our just proven assertion that every $y_j$ is a linear combination of the $e_k$,  we have that each $c_i$ can be written as a linear combination of the $e_k$. 
Thus via compositions of the~$q_{i,j}$ isometries and their inverses, we can use the~$e_j$ to eliminate the~$c_i$,  as claimed.

\item Step 3.
At this point, without loss of generality,  we can assume that
\begin{align*}
f(x_i)&=x_i+(t-1)b_i,  \\
f(y_k)&=\phantom{a_i+ }(t-1)d_j+e_j,
\end{align*}
where again $b_i, d_j \in H_x$ and $e_j \in H_y$.
We show that after postcomposing by a $\operatorname{Sp}(2g,\Z)$-isometry (i.e. as in Example~\ref{ex:Sp2gZ}) we can assume that $y_j=e_j$ for all $j=1, \dots, 2g$.

The same reasoning as in Step 2 shows that $\Z$-linearly extending the assignment~$y_j \mapsto e_j$ gives rise to an isometry of~$(H_y,\Bl|_{H_y}) \cong (\coker(\mathcal{H}_2^{\oplus g}),\partial \mathcal{H}_2^{\oplus g}).$
Further extending by the identity on the $x_i$ then gives rise to an isometry of $\Bl$.
We can now postcompose $f$ by a~$\operatorname{Sp}(2g,\Z)$-isometry to arrange that $y_j=e_j$, thus proving Step $3$.

\item Step 4.
At this point, without loss of generality we can assume
\begin{align*}
f(x_i)&=x_i+(t-1)b_i,  \\
f(y_k)&=\phantom{a_i+ }(t-1)d_j+y_j,
\end{align*}
where $b_i, d_j \in H_x$. 
We first show that $(t-1)d_j=0$ for all $j=1, \dots, 2g$ and then show that after postcomposing by a product of the $t^k$-isometries of Example~\ref{ex:ScaleIsometry} and $p_{i,j}$-isometries of Example~\ref{ex:fij}  we can assume that $f$ is the identity.

We start by proving $(t-1)d_j=0$.
First note that for every $i=1,\dots,c$, we have
$$0=\Bl(x_i,y_j)=\Bl(f(x_i),f(y_j))=\Bl(x_i,(t-1)d_j).$$
Since $\Bl|_{H_x \oplus (t-1)H_x}$ is nondegenerate,  it follows that $(t-1)d_j=0$ as claimed.

Finally, we reduce $f$ to the identity by postcomposing with isometries from Examples~\ref{ex:ScaleIsometry} and~\ref{ex:fij}. 
For this it is convenient to expand the $b_i$ so that
\begin{align*}f(x_i)
&=x_i+(t-1)b_i
=x_i+(t-1)\sum_{i' = 1}^c \alpha_{i,i'}x_{i'} \\
&=x_i+(t-1)\alpha_{i,i} x_i +(t-1)\sum_{i' \neq i} \alpha_{i,i'}x_{i'} 
=t^{\alpha_{i,i}}x_i+(t-1)\sum_{i' \neq i} \alpha_{i,i'}x_{i'},
\end{align*}
with $\alpha_{ij} \in \Z.$
For the last equality one needs that  $t^n=1+n(t-1) \in \Z[t^{\pm1}]/ (t-1)(t^{-1}-1)$ for all $n \in \Z$,  as can be quickly verified via induction.
%%Don't delete
%t^2=t\cdot t=(1+(t-1)) \cdot (1+(t-1))=1+2(t-1).
%t^{-1}=1-(t-1) is because $-(t-1)(t^{-1}-1)=t-1+t^{-1}-1$

After postcomposing with isometries from Example~\ref{ex:ScaleIsometry},  we can assume that~$\alpha_{i,i}=0$ for all $i=1,\dots,c$. 
We note that the $\alpha_{i,i'}$ with $i \neq i'$ do not change throughout this process.
% might change but remain integers.
The quickest way to see this is to note that scaling $x_i$ by $t^k$ is the same as replacing $x_i$ by~$x_i+k(t-1)x_i$ and that this does not affect the $\alpha_{i,i'}$. 
%x_i-k(t-1)x_i=t^{-k}x_i

Finally we assert that after postcomposing with isometries from Example~\ref{ex:fij} we can assume that~$\alpha_{i,i'}=0$ for all $i \neq i'$. 
This will conclude Step $4$ and therefore the proof of the proposition.
Expanding $0=\Bl(x_i,x_{i'})=\Bl(f(x_i),f(x_{i'}))$ we obtain the equation 
\begin{align*}
0&=(t-1)\alpha_{i,i'}\Bl(x_{i'},x_{i'})+(t^{-1}-1)\alpha_{i',i}\Bl(x_i,x_i)= \frac{\alpha_{i,i'} \varepsilon_{i'}- \alpha_{i,i} \varepsilon_i}{t^{-1}-1}.
\end{align*}
%%%%For the (t^{-1}-1) it's a little calculation that uses $t^{-1}=1-(t-1)$ (see the previously mentioned induction).
Since the $\alpha_{i,i'}$ are integers we deduce that the $\alpha_{i,i'}$ and $\alpha_{i',i}$ agree up to a sign; this sign is positive if $\varepsilon_i=\varepsilon_{i'}$ and negative if $\varepsilon_i=-\varepsilon_{i'}$.
We can therefore get the $\alpha_{i,i'}$ to be zero,  and hence $f$ to be the identity,  by postcomposing~$f$ with
\[\prod_{1 \leq i < i'  \leq c} p_{i,i'}^{-\varepsilon_i \varepsilon_{i'} \alpha_{i,i'}}.\]
(One can check that the result of this composition of $p_{i,i'}$ maps is unchanged by reordering its terms.) 
\end{itemize}
The combination of these five steps shows that any $f \in \Aut(\Bl)$
%%Don't delete
%We showed $s \circ \partial \varphi \circ f=\id$ so $f=\partial \varphi^{-1} \circ s^{-1}$
can be written as~$f=\partial \rho \circ s$ where~$\rho \in \Aut(\lambda)$ and~$s \in \Aut(\Bl)$ is as in the statement of the proposition.
The statement with~$f=s \circ \partial \rho$ instead of~$\partial \rho \circ s$ can be obtained by applying the previous version to~$f^{-1}$ and then taking inverses.
\end{proof}

In certain cases, the statement of Proposition~\ref{prop:AutBl/AutStandard} can be strengthened.
The outcome wll be useful later for establishing Theorem~\ref{thm:Other4ManifoldsSpheresIntro}.

\begin{proposition}
\label{prop:AutBl/AutStandardImproved}
Assume that $(c_+,c_-) \in \{(c,0),(0,c)\}$.
Any~$f \in \Aut(\partial \lambda_{c_+,c_-} \oplus \partial \mathcal{H}^{\oplus 2g})$ can be written as $f= s \circ s'$  for $s$ a composition of the isometries described in Examples~\ref{ex:ScaleIsometry}-\ref{ex:Sp2gZ} and their inverses,  and $s'$ a composition of the scale-a-generator by $(-1)$ isometries and permutation-induced isometries from Examples~\ref{ex:ScaleIsometry-1} and~\ref{ex:PermutationIsometry}. 
\end{proposition}
\begin{proof}
We consider the case where~$\lambda:= \lambda_{c,0} \oplus \mathcal{H}_2^{\oplus 2g}$; the proof for $\lambda_{0,c} \oplus \mathcal{H}_2^{\oplus 2g}$ is exactly analogous. 
The proof is nearly identical to that of Proposition~\ref{prop:AutBl/AutStandard}: only the first step changes: we need to show that the use of elements of~$\Aut(\Bl)$ can be replaced by a composition of the scale-a-generator by $(-1)$ isometries and permutation-induced isometries.
Here are some details.

Continuing with the notation from the proof of Proposition~\ref{prop:AutBl/AutStandard}, we write
\begin{align}\label{eqn:form_of_f}
f(x_i)&=a_i+(t-1)b_i+c_i,  \\
f(y_j)&=\phantom{a_i+ }(t-1)d_j+e_j  \nonumber
\end{align}
for some $a_i,b_i,d_j \in  H_x$ and $c_i,e_j \in H_y$.

We prove the analogue of Step 1 of Proposition~\ref{prop:AutBl/AutStandard}: after postcomposing the isometry~$f$ with a composition of the scale-a-generator by~$(-1)$ isometries and permutation-induced isometries we can arrange that~$a_i=x_i$ for all $i=1,\dots, c$. 
Writing 
$$a_i=\sum_{j=1}^c n_{i,j}x_j$$
 for $n_{i,j} \in \Z$,  the idea is to show that $n_{i,j}=\pm 1$ if $j=\sigma(i)$ and $0$ otherwise for some~$\sigma\in S_c$.  
The same argument as in Proposition~\ref{prop:AutBl/AutStandard} shows that  $\Bl(x_i, x_{i'})= \Bl(a_i,a_{i'})$.
Expanding our expression for $a_i$ in this case we can quickly observe that 
\begin{align}\label{eqn:allpos}
\sum_{j=1}^c n_{i,j}n_{i',j}= \begin{cases} 1  & i'=i \\ 0 & i' \neq i. \end{cases}
\end{align}
When $i'=i$ this equation implies that there is exactly one value $j=\sigma(i)$ such that $n_{i,j}$ is nonzero, and for that value $n_{i,\sigma(i)}= \pm 1$.  
%Sum of squares.
Analyzing~\eqref{eqn:allpos} for $i' \neq i$ then quickly shows that $\sigma$ is a permutation of $\{1,\dots, c\}$. 
(Alternately,  considering~\eqref{eqn:allpos} across all $i,i'$ can be repackaged as saying that the $c \times c$ integral matrix $A=(n_{i,j})$ has the property that $A \cdot A^T= I$, i.e. that $A$ is an integer orthogonal matrix, i.e. $A \in O(c,\Z).$
It is well known that all such matrices are obtained from the identity matrix by changing the signs of diagonal entries and then permuting rows.)
%https://math.stackexchange.com/questions/1464522/order-of-the-group-of-integer-orthogonal-matrices
\end{proof}

Similarly, we obtain the following for $c_+=1=c_-.$
\begin{proposition}
\label{prop:AutBl/AutStandard11}
Any~$f \in \Aut(\partial \lambda_{1,1} \oplus \partial \mathcal{H}_2^{\oplus g})$ can be written as $f= s \circ s'$  for $s$ a composition of the isometries described in Examples~\ref{ex:ScaleIsometry}-\ref{ex:Sp2gZ} and their inverses,  and $s'$ a composition of the scale-a-generator by $(-1)$ isometries from Example~\ref{ex:ScaleIsometry-1}. 
\end{proposition}
\begin{proof}
The proof is exactly analogous to that of the previous proposition: writing that $a_1= n_{1,1} x_1 + n_{1,2} x_2$ and $a_2= n_{2,1}x_1+n_{2,2} x_2$,  the conditions that $\Bl(a_i, a_j)= \Bl(x_i,x_j)$ quickly imply that $n_{1,2}=n_{2,1}=0$ and $n_{1,1}, n_{2,2} \in \{\pm1\}$.
\end{proof}

\begin{remark}
\label{rem:special_cases_of_isom/homeo}
We collect two remarks concerning Proposition~\ref{prop:AutBl/AutStandardImproved}.
\begin{itemize}
\item When  $(c_+,c_-) \in \{(c,0),(0,c)\}$,  Proposition~\ref{prop:AutBl/AutStandard} is recovered from Proposition~\ref{prop:AutBl/AutStandardImproved} by noting that the isometries from Examples~\ref{ex:ScaleIsometry-1} and~\ref{ex:PermutationIsometry} lie in the image of~$\Aut(\lambda) \to \Aut(\partial \lambda).$
\item In the notation of the proof of Proposition~\ref{prop:AutBl/AutStandardImproved}, for general $c_+$ and $c_-$,  determining the possible values of~$\{a_i\}_{i=1}^c$ is equivalent to enumerating the elements of the indefinite integer orthogonal group
$$ O(c_+,c_-;\Z)=\left\lbrace  A \in \operatorname{GL}_{c_++c_-}(\Z) \ \Bigg| \ A \begin{pmatrix} I_{c_+} & 0 \\ 0 & -I_{c_-} \end{pmatrix} A^T=  \begin{pmatrix} I_{c_+} & 0 \\ 0 & -I_{c_-} \end{pmatrix} \right\rbrace $$
which need not be finite,  see e.g.  the introduction of~\cite{Borcherds}.
%\AC: https://math.stackexchange.com/questions/566116/integer-subgroup-of-indefinite-orthogonal-group}.
%%AM: {For now, https://kconrad.math.uconn.edu/blurbs/linmultialg/descentPythag.pdf, but this is unpublished}
\end{itemize}
\end{remark}

\section{Realising isometries of~$\Bl_P$ by homeomorphisms}
\label{sec:BlPHomeomorphisms}

In order to get a further grasp on immersed $\Z$-surfaces,  the classifications of Section~\ref{sec:Classification} imply that we need to understand the action of $\Homeo_\alpha(\Sigma,\partial)$ on $\Aut(\Bl_{P})$.
In other words,  we need to be able to realize isometries of $\Bl_{P}$ by homeomorphisms of $P$ that arise from surface homeomorphisms.

In order to understand what isometry a given surface homeomorphism induces,  we construct an explicit cellular decomposition of the~$3$-manifold $P=P_g(c_+,c_-).$
This groundwork is laid in Section~\ref{sub:P01} and will be useful both in Section~\ref{sub:HomologyPagain} to establish the homology claims stated in earlier sections and in Section~\ref{sub:RealisingIsometries} to understand the isometries of $\Bl_{P}.$
%Finally, we emphasize the following point. 
The homology calculations could be performed without constructing explicit cell structures, but since the cell structures are useful for describing the action of $\Homeo_\alpha(\Sigma,\partial)$ on $\Aut(\Bl_{P})$, we also use them to compute homology.

\subsection{The first homology of $P_0(1,0)$ using cell complexes}
\label{sub:P01}
Throughout this section we set 
$$Q:=P_0(1,0).$$
Recall that this is the $3$-manifold with torus boundary obtained by performing a single positive plumbing on a solid torus~$D^2 \times S^1$.
%Recall that this is one of the two $3$-manifolds with torus boundary obtained by performing a single plumbing on a solid torus~$D^2 \times S^1$.
The goal of this section is to construct a cell structure for a space $Q'$ that is homotopy equivalent to $Q$ and to describe explicit generators for its first homology group, both with $\Z$ and $\Z[t^{\pm 1}]$ coefficients.

\begin{construction}[The spaces $\Delta'$ and $Q' \simeq Q$]
The $3$-manifold $Q$ arises as a quotient of $\Delta \times S^1$, where~$\Delta$ is a disk with two subdisks removed  as illustrated on the left of Figure~\ref{fig:disccell}; recall Remark~\ref{rem:PasaQuotient}.
Observe that~$\Delta$ deformation retracts onto the 1-complex $\Delta'$ on the right of Figure~\ref{fig:disccell},  and that~$Q$ is homotopy equivalent to the space $Q'$ obtained from $\Delta' \times S^1$ by identifying $\alpha_1 \times S^1$ with $\alpha_2 \times S^1$ using the identification as given in Remark~\ref{rem:PasaQuotient}.   Note that $\alpha_i$ as defined here is $\alpha_i(S^1)$ in the notation of previous sections. 
The notation is unfortunate but we believe that it will not cause any confusion since it is local to this section.
%the letters of the alphabet are limited.}
\end{construction}

\begin{figure}[htbp!]
\begin{centering}
\begin{tikzpicture}
\node[anchor=south west,inner sep=0] at (0,0)
{\includegraphics[height=3cm]{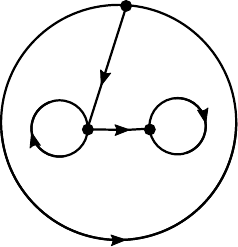}};
\node[anchor=south west,inner sep=0] at (5,0){\includegraphics[height=2.5cm]{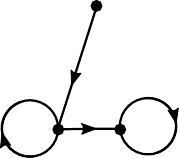}};
%\node at (1.5, 3.2){$z$}; \node at (1.2,1.1){$x$}; \node at (1.8, 1.1){$y$};
%\node at (1.45, 1.7){$\gamma$}; \node at (.8,0){$\zeta$}; 
% \node at (.6,.9){$\alpha$}; \node at (2.3,.9){$\beta$};
% \node at (1.1,  2.5){$\epsilon$}; \node at (1.5,  .5){$\Delta$};
\node at (5.5, -.25){$\alpha_1$};\node at (7.5, -.25){$\alpha_2$};
\node at (6.15,  .25){$z_1$}; \node at (6.7, .25){$z_2$}; \node at (6.4, .7){$\omega$};
\node at (6.8, 2.5){$z$};
\node at (5.9, 1.5){$\xi$};
\end{tikzpicture}
\end{centering}
\caption{A cell structure for $\Delta$ (left) and for its deformation retract $\Delta'$ (right).}
\label{fig:disccell}
\end{figure}
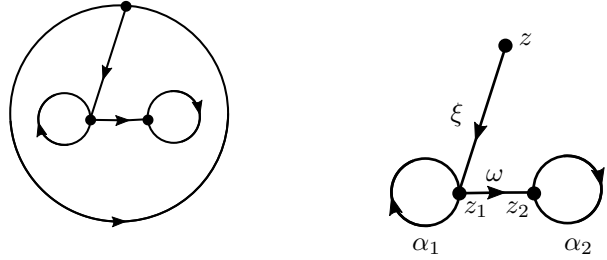

\begin{construction}[A warm-up cell structure for $\Delta' \times S^1$]
\label{cons:WarmUpCellStructure}
Let $\{v,e\}$ be the usual cell-structure on $S^1$.  
This gives a product cell-structure on $\Delta' \times S^1$ with three 0-cells ($z_1v, z_2v \text{ and } zv$),  seven~$1$-cells ($\alpha_1 v, \alpha_2 v, \xi v,  \omega v,  z_1 e, z_2e,  \text{ and }  ze$),  and four 2-cells ($\alpha_1 e,  \alpha_2 e,  \xi e,  \text{ and }  \omega e$),  where throughout we refer to the product cell $c_1 \times c_2$ as $c_1c_2$. 
\end{construction}

This cell structure for $\Delta' \times S^1$ does not give rise to a cell structure for $Q'$ because we cannot accomplish our desired identification of $\alpha_1 \times S^1$ with $\alpha_2 \times S^1$ cellularly.  
We therefore modify the cell structure as follows.
\begin{itemize}
\item The torus $\alpha_1 \times S^1 \subset \Delta' \times S^1$ has cell structure $z_1 v \cup \alpha_1 v \cup z_1 e \cup \alpha_1 e$. 
We add a 1-cell $\beta_1$ representing the $(1,1)$ curve on this torus,  which has the effect of splitting the 2-cell $\alpha_1 e$ into two 2-cells $C_1$ and $C_2$.   
\item The torus $\alpha_2 \times S^1 \subset \Delta' \times S^1$ has cell structure $z_2 v \cup \alpha_2 v \cup z_2 e \cup \alpha_2 e$. 
We add a 1-cell $\beta_2$ representing the $(1,1)$ curve on this torus,  which has the effect of splitting the 2-cell $\alpha_2 e$ into two 2-cells $D_1$ and $D_2$.   
\end{itemize}

This leads to a cellular decomposition of $Q'$, as illustrated in Figure~\ref{fig:pcell}. 
We summarize this discussion as follows.

\begin{construction}(A cell structure for $Q' \simeq Q$.)
\label{cons:CellP'}
We can now accomplish our desired identification of $\alpha_1 \times S^1$ with $\alpha_2 \times S^1$  cellularly. 
Translating the identifications of Remark~\ref{rem:PasaQuotient} into the current notation,  we see that~$Q'$ arises as a quotient of $\Delta' \times S^1$ via an identification of $\alpha_1 \times S^1$ with $\alpha_2 \times S^1$ such that
\[\alpha_1 \times S^1  \ni (\alpha_1 \times \id)\circ \eta^{+}(z,w)\sim (\alpha_2 \times \id) \circ \eta^{+}( w, z) \in \alpha_2 \times S^1.\]
Tracing through this identification, we see that
\begin{itemize}
\item On the level of the~$0$-cells, $z_1v$ is identified with $z_2v$ in an orientation preserving way. 
\item On the level of the~$1$-cells, $\alpha_1 v$ is identified with $\alpha_2 v$ in an orientation reversing way,  $z_1 e$ is identified with $\beta_2$ in an orientation preserving way,  and $\beta_1$ is identified with $z_2 e$ in an orientation preserving way.
\item On the level of the $2$-cells, $C_i$ is identified with $D_i$ in an orientation preserving way for~$i=1,2$ (this last step may require re-indexing and orientation change of $D_1$, $D_2$). 
\end{itemize}
We therefore have a cell structure for $Q'$ with two $0$-cells,  six $1$-cells, and four $2$-cells:
\begin{align*}
\text{0-cells}&: z_1 v \sim z_2 v,  zv,  \\
\text{ 1-cells}&: \alpha_1 v \sim \overline{\alpha_2 v},  \beta_1 \sim z_2 e,  z_1 e \sim \beta_2,  \xi v,  \omega v,   ze,  \\
\text{ 2-cells}&: C_1 \sim D_1,  C_2 \sim D_2, \xi e,  \omega e.
\end{align*}
%two 0-cells  ($xv \sim yv,  zv$), 
%six 1-cells ($\alpha v \sim \overline{\beta v},  a \sim ye,  xe \sim b,  \gamma_\alpha v,  \gamma_\beta v,   ze$),  and four 2-cells ($C_1 \sim D_1,  C_2 \sim D_2, \gamma_ \alpha e,  \gamma_\beta e$. )
The attaching maps of the 2-cells are
%{AC: Do we need to list $\partial(D_i)$?  AM: No, because $D_i$ *is* $C_i$.}
\begin{align*}
 \partial(C_1)&= (z_1 e)(\alpha_1 v) \overline{\beta_1},\\
 \partial (C_2) &= (\alpha_1 v) (z_1 e) \overline{\beta_1},\\
 \partial (\xi e)&= (\xi v) (z_1 e) (\overline{\xi v}) (\overline{ze}),\\
 \partial (\omega e)&= (\omega v) (z_2 e) (\overline{\omega v}) (\overline{z_1e})=
 (\omega v) (\beta_1) (\overline{\omega v}) (\overline{z_1 e}).
 \end{align*}
Observe that under the deformation retraction of $Q$ onto $Q'$, the $S^1$-fibre $\mu \subset Q$ and plumbing loop $\omega \subset Q$ are respectively taken to~$ze \subset Q'$ and~$\omega v \subset Q'$.
We occasionally refer to $ze$ and $\omega v$ as the \emph{$S^1$-fibre} and \emph{plumbing loop} respectively.
\end{construction}

 %%Don't delete.
 %See "cell complex homology computation.pdf"

 \begin{lemma}
 \label{lem:H1P01}
The integral homology of $Q$ is freely generated by the $S^1$-fibre and the plumbing loop:
$$H_1(Q) \cong\Z \mu \oplus \Z \omega.$$
The inclusion induced map $H_1(\partial Q) \to H_1(Q)$ maps $[\partial \Sigma]$ to zero, and is the identity on the $S^1$-fibre.
  \end{lemma}
  \begin{proof}
The first assertion follows by using the cell structure from Construction~\ref{cons:CellP'} to note that
  $$H_1(Q') \cong \Z \langle ze \rangle \oplus \Z \langle \omega v \rangle.$$
To see this,  one uses the attaching maps described in Construction~\ref{cons:CellP'} to note that both~$z_1e \sim \beta_2$ and $z_2 e \sim \beta_1$ are homologous to $ze$ and that~$\alpha_1 v$ (and hence also~$\alpha_2 v$) is null-homologous.  
  
The map~$H_1(\partial Q) \to H_1(Q)$ can similarly studied using $Q'$.
Under the deformation retraction of~$Q$ to~$Q'$ guided by the deformation retraction of $\Delta$ to $\Delta'$,  the outer boundary of $\Delta$ times the point $v$  is sent to the curve~$(\xi v)(\overline{\alpha_1 v})(\omega v)(\overline{\alpha_2 v})(\overline{\omega v}) (\overline{\xi v})$.
The fact that $\alpha_1 v$ and $\alpha_2 v$ are nullhomologous in $H_1(Q')$ ensures that so is this loop.
% which is trivial in~$H_1(P')$.
  \end{proof}

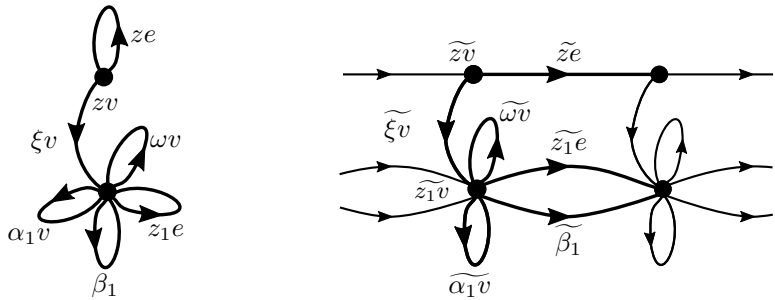
\begin{figure}[htbp!]
\begin{centering}
\begin{tikzpicture}
\node[anchor=south west,inner sep=0] at (0,0)
{\includegraphics[height=3.5cm]{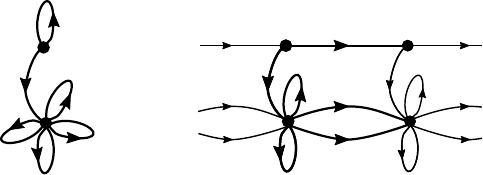}};
\node at (-.1, .45){$\alpha_1 v$}; \node at (.9, -.2){$\beta_1$}; \node at (1.7, .5) {$z_1e$};
%\node at (.9, 1.4){$\small z_1v$}; 
\node at (.9, 2.2){$zv$};
\node at (.1,  1.7) {$\xi v$}; \node at (1.7,  1.7) {$\omega v$}; 
\node at (1.4,3.1){$ze$};
\node at (5.7, -.2){$\widetilde{\alpha_1 v}$};\node at (7, .4){$\widetilde{\beta_1}$}; \node at (7, 1.7){$\widetilde{z_1 e}$};
\node at (4.75, 1.9){$\widetilde{\xi v}$}; \node at (6.3,2.1){$\widetilde{\omega v}$}; \node at (7, 2.9){$\widetilde{ze}$};
\node at (5.2, 1.05){\small $\widetilde{z_1 v}$}; \node at (5.6, 2.9){$\widetilde{zv}$};
\end{tikzpicture}
\end{centering}
\caption{The 1-skeleton of $Q'$ (left) and the 1-skeleton of $(Q')^\infty$ (right).}
\label{fig:pcell}
\end{figure}

We now calculate the $\Z[t^{\pm 1}]$-homology of $Q$ with respect to the epimorphism $\varphi \colon \pi_1(Q,zv) \to \Z$
from Construction~\ref{cons:CoeffSystPK} that maps the $S^1$-fibre to $1 \in \Z$ and the plumbing loop to $0 \in \Z$.
We achieve this by calculating the~$\Z[t^{\pm 1}]$-homology of $Q'$ with respect to the epimorphism $\varphi' \colon H_1(P') \to \Z$ with~$\varphi'(ze)=1$ and~$\varphi'((\xi v)(\omega v) (\overline{\xi v}))=0$.
Indeed, the infinite cyclic cover $\widetilde{Q}'$ corresponding to~$\ker(\varphi')$ deformation retracts onto the infinite cyclic cover of $Q$ corresponding to~$\ker(\varphi)$.
In the sequel,  we use~$\widetilde{\omega} \subset Q^\infty$ to denote our preferred lift of the plumbing loop~$\omega \subset Q$ as in Remark~\ref{rem:Lifts}.

 \begin{lemma}
 \label{lem:H1P01ZZ}
 The $\Z[t^{\pm1}]$-homology of $Q$ is generated by a lift of the plumbing loop:
\[ H_1(Q;\Z[t^{\pm1}]) \cong \left( \Z[t^{\pm}]/ (t-1)^2\right) \langle \widetilde{\omega} \rangle. \]
Additionally, the inclusion induced map~$H_1(\partial Q;\Z[t^{\pm 1}]) \to H_1(Q;\Z[t^{\pm 1}])$ is the zero map.
\end{lemma}
\begin{proof}
The cell structure for $Q'$ from Construction~\ref{cons:CellP'} lifts to an equivariant cell-structure for~$(Q')^\infty$,  the 1-skeleton of which is drawn on the right of Figure~\ref{fig:pcell}.  
%%AC: infty notation better accords with the rest of the paper.
The first assertion will follow by using this cell structure to show that 
\[ H_1(Q';\Z[t^{\pm1}]) \cong \left( \Z[t^{\pm}]/ (t-1)^2\right) \langle \widetilde{\omega v} \rangle. \]
It is not overly challenging to compute the differentials~$\partial_i \colon C_{i}(Q'; \Z[t^{\pm1}]) \to C_{i-1}(Q'; \Z[t^{\pm 1}])$ for~$i=1,2$ and obtain that 
\begin{align*}
\ker(\partial_1) &= \langle \overbrace{\widetilde{\alpha_1 v}}^{:=h_1}, \overbrace{\widetilde{\omega v}}^{:=h_2},  \overbrace{\widetilde{z_1 e} - \widetilde{\beta_1}}^{:=h_3},  \overbrace{(1-t)\widetilde{\xi v} + \widetilde{z_1 e}- \widetilde{z e}}^{:=h_4} \rangle,\\ 
\text{Im}(\partial_2)&= \langle\partial(\widetilde{C_1})= t\cdot h_1 + h_3,  \partial(\widetilde{C_2})= h_1+h_3, \partial(\widetilde{\xi e})=h_4,  \partial(\widetilde{\omega e})= (1-t) h_2- h_3 \rangle.
\end{align*}
Some algebraic simplifications then yield the calculation of $H_1(Q'; \Z[t^{\pm1}])$,  with $h_2$ as a generator.
%%Don't delete.
 %See "cell complex homology computation.pdf"

Next, we study~$H_1(\partial Q;\Z[t^{\pm 1}]) \to H_1(Q;\Z[t^{\pm 1}])$.
Recall that under the deformation retraction of $Q$ onto $Q'$,  the outer boundary of $\Delta$ times $v$ is sent to the curve~$(\xi v)(\overline{\alpha_1 v})(\omega v)(\overline{\alpha_2 v})(\overline{\omega v}) (\overline{\xi v})$.
One verifies that this latter loop lifts to 
\[\widetilde{\xi v} + \widetilde{\alpha_1 v} +\widetilde{\omega v} - \widetilde{\alpha_1 v} - \widetilde{\omega v} - \widetilde{\xi v}=0 \in C_1(Q', \Z[t^{\pm1}])
\]
and so is certainly zero in~$H_1(Q'; \Z[t^{\pm1}])$. 
\end{proof}

\begin{remark}\label{rem:onenegativedp}
An extremely similar cell complex computation shows that the $\Z[t^{\pm1}]$-first homology of $Q_{-}= P_0(0,1)$ is generated by a lift of the plumbing loop,  which is order $(t-1)(t^{-1}-1)$,  and that the inclusion induced map $H_1(\partial Q_{-}; \Z[t^{\pm1}]) \to H_1(Q_{-}; \Z[t^{\pm1}])$ is the zero map.  

However, there is one difference among the relations,  which will play a role in later computations:
in $H_1(Q; \Z[t^{\pm1}]$ we have that $$[\widetilde{\alpha_1 v}]= (t-1)[ \widetilde{\omega}]= -[\widetilde{\alpha_2 v}],$$while in $H_1(Q_-; \Z[t^{\pm1}])$,  we have that $$[\widetilde{\alpha_1 v}]= -(t-1)[ \widetilde{\omega}]= -[\widetilde{\alpha_2 v}].$$ 
This sign difference comes from the differing identifications between $\alpha_1 \times S^1$ and $\alpha_2 \times S^1$ depending on the sign of the double point,  as described in Equation~\eqref{eqn:simP}.
\end{remark}
%{AM: See cell complex for one negative dp.pdf for the details.} 

\subsection{The first homology of $P_g(c_+,c_-)$ via cell complexes}
\label{sub:HomologyPagain}

This section calculates the first homology~of
$$P:=P_g(c_+,c_-)$$
 both with integral and~$\Z[t^{\pm 1}]$ coefficients.
To do so, we decompose~$P$ as a union of~$\Sigma_{g,c+1}$, the genus~$g$ surface with~$c+1$ boundary components, with $c_+$ copies of~$P_0(1,0)$ and $c_-$ copies of~$P_0(0,1)$ and apply Mayer-Vietoris.

\begin{notation}
%Set $c:=c_++c_-.$
For $i=1,\ldots,c$, we use the following notation.
\begin{itemize}
\item Write~$Q_i$ for the aforementioned copies of $P_0(1,0)$ and $P_0(0,1).$
The signs need not be distinguished in the statements of this section.
Recall that~$Q_i$ is a quotient of $\Delta_i \times S^1$, where~$\Delta_i$ is a disk with two open balls removed.
Write $\partial_{\operatorname{out}}\Delta_i$ for the outer boundary of~$\Delta_i$.
\item Both the plumbing curve of $Q_i$ and its image in $P$ are denoted $\omega_i$.
\item Write~$\mu_i  \subset Q_i$ for the~$S^1$-fibre of~$Q_i$ and~$\mu$ for the~$S^1$-fibre of~$\Sigma_{g,c+1} \times S^1$.
\item As illustrated in Figure~\ref{fig:surfacewithcurves},  $\partial  \Sigma_{g,c+1}$ decomposes into an outer boundary curve,  $\partial_{\text{out}} \Sigma_{g,c+1}$ and interior boundary curves $\partial_i$ that arise from the removal of the open balls.
These latter curves freely generate the~$\Z^c$ summand of~$H_1(\Sigma_{g,c+1})=\Z^{2g} \oplus \Z^c.$
When $P$ is formed from~$\Sigma_{g,c+1} \times S^1$ by gluing the $Q_i$, the $\partial_i \subset \Sigma_{g,c+1}$ are identified with the $\partial_{\operatorname{out}}\Delta_i \subset Q_i$,  and the~$S^1$-fibre~$\mu \subset \Sigma_{g,c+1} \times S^1$ with the~$S^1$-fibres $\mu_i \subset Q_i $.
\item  Fix a standard symplectic basis of $2g$ curves in the interior of~$\Sigma_{g,c+1}$.
The \emph{genus {loops}} refer to the images of these curves both in~$\Sigma_{g,c+1} \times S^1$ and in~$P$.
We use the notation~$a_1,b_1,\ldots,a_g,b_g$ for these curves.
\end{itemize}
\end{notation}

The result of the next proposition was stated (but not proved) in Proposition~\ref{prop:HomologyP}.
\begin{proposition}
\label{prop:H1P}
The group~$H_1(P)$ is freely generated by the genus {loops}, the~$S^1$-fibre and the plumbing {loops}:
\[ H_1(P) \cong \Z \mu  \oplus  \bigoplus_{i=1}^g \left( \Z a_i \oplus \Z b_i \right) \oplus \bigoplus_{i=1}^c \Z  \omega_i.
\]
The map~$\Z^2 \cong H_1(\partial P) \to H_1(P)$ sends the~$S^1$-fibre to~$\mu$ and~$\partial_{\text{out}} \Sigma_{g,c+1}$ to zero.
\end{proposition}
\begin{proof}
This follows quickly from Lemma~\ref{lem:H1P01} by using the Mayer-Vietoris sequence for~$P= ( \Sigma_{g,c+1}\times S^1) \cup \bigcup_{i=1}^c Q_i$.  We leave the details to the reader.
%\color{teal}
%The Mayer-Vietoris sequence for~$P= ( \Sigma_{g,c+1}\times S^1) \cup \bigcup_{i=1}^c Q_i$ gives
%$$
%\bigoplus_{i=1}^c \overbrace{H_1(S^1 \times S^1)}^{\cong \Z^2} \xrightarrow{j_1} \overbrace{H_1( \Sigma_{g,c+1}\times S^1)}^{\cong \Z^{2g+c}\oplus \Z}  \oplus  \bigoplus_{i=1}^c H_1(Q_i) \to H_1(P)  \to \bigoplus_{i=1}^c \overbrace{H_1(S^1 \times S^1)}^{\cong \Z^2} \xrightarrow{j_0} \ldots
%$$
%The rightmost map, namely~$j_0$, is seen to be injective.
%We deduce
%\begin{align*}
%H_1(P)
%&\cong \left(  H_1( \Sigma_{g,c+1}\times S^1)  \oplus  \bigoplus_{i=1}^c H_1(Q_i) \right)\Bigg/ \im (j_1) \\
%&\cong \left(  \Z\mu \oplus  \bigoplus_{i=1}^g \left( \Z a_i \oplus \Z b_i\right) \oplus\bigoplus_{i=1}^c \Z \partial_i  \oplus \bigoplus_{i=1}^c    \left(\Z\omega_i \oplus \Z \mu_i\right) \right) \Bigg/ \im (j_1).
%\end{align*}
%The map~$j_1$ identifies~$\partial_i \in H_1(\Sigma_{g,c+1})$ with~$\partial_{\text{out}} \Delta_i \in H_1(Q_i)$ and the $S^1$-fibre~$\mu \in H_1(\Sigma_{g,c+1})$ with~$\mu_i \in H_1(Q_i).$
%Lemma~\ref{lem:H1P01} implies that $\partial_{\text{out}} \Delta_i =0\in H_1(Q_i)$.
%As required,  we obtain
%$$ H_1(P)=\Z \mu \oplus \bigoplus_{i=1}^g (\Z a_i \oplus \Z b_i ) \oplus \bigoplus_{i=1}^c \Z \omega_i.$$
%The inclusion induced map~$H_1(\partial P) \to H_1(P)$ certainly sends the $S^1$-fibre to itself and the equality~$[\partial_{\text{out}}  \Sigma_{g,c+1}]=-\sum_{i=1}^c[\partial_i] \in H_1(\Sigma_{g,c+1})$, ensures that it sends~$[\partial_{\text{out}}  \Sigma_{g,c+1}]$ to $0 \in H_1(P)$. 
\color{black}
\end{proof}

The result of the next proposition was stated (but not proved) in Proposition~\ref{prop:HomologyPZZ}. 
Recall that homology with $\Z[t^{\pm 1}]$-coefficients is taken with respect to the homomorphism $\varphi \colon H_1(P) \to \Z$ that sends the $S^1$-fibre to $1$ and the genus and plumbing {loops} to $0$.

\begin{proposition}
\label{prop:H1PZZ}
The first $\Z[t^{\pm 1}]$-homology module of the plumbed~$3$-manifold $P$ is given by
$$
H_1(P;\Z[t^{\pm 1}]) \cong 
\Z_\varepsilon^{2g} \oplus \left( \Z[t^{\pm 1}]/(t-1)^2\right)^{\oplus c}.
$$
The summands are respectively generated by lifts of the genus and plumbing {loops}.

Additionally,  the inclusion induced map $H_1(\partial P;\Z[t^{\pm 1}]) \to H_1(P;\Z[t^{\pm 1}])$ is the zero map.
\end{proposition}
\begin{proof}
%\color{teal}
Consider the Mayer-Vietoris sequence for~$P= ( \Sigma_{g,c+1}\times S^1)  \cup \bigcup_{i=1}^c Q_i$:
$$
 \bigoplus_{i=1}^c H_1(S^1 \times S^1;\Z[t^{\pm 1}]) \xrightarrow{j_1} H_1( \Sigma_{g,c+1}\times S^1;\Z[t^{\pm 1}])  \oplus  \bigoplus_{i=1}^c H_1(Q_i;\Z[t^{\pm 1}]) \to H_1(P;\Z[t^{\pm 1}])  \xrightarrow{j_0} $$
We study the restriction of the coefficient system on $P$ to the~$Q_i$, the $S^1 \times S^1$ and to~$\Sigma_{g,c+1}$.
\begin{itemize}
\item  We begin with the~$Q_i$.
Under the inclusion induced map~$H_1(Q_i) \to H_1(P)$, the $S^1$-fibre~$\mu_i \subset Q_i$ is mapped to the~$S^1$-fibre $\mu \subset P$ and the plumbing loop~$\omega_i \subset Q_i$ is mapped to the corresponding plumbing loop in $P$.
It follows that~$\varphi(\mu_i)=1$ and~$\varphi(\omega_i)=0$ so that by Lemma~\ref{lem:H1P01ZZ} we have
$$H_1(Q_i;\Z[t^{\pm 1}]) \cong \Z[t^{\pm 1}]/(t-1)^2,$$
generated by a lift of a plumbing loop.
\item We move on to~$\Sigma_{g,c+1}$.
Under the inclusion induced map, the genus {loops} (resp.\ $S^1$-fibre) of $\Sigma_{g,c+1} \times S^1$ map to the genus {loops} (resp.\ $S^1$-fibre) of $P$,  and the $\partial_i$ map to~$[\partial_{\text{out}} \Delta_i]=0 \in H_1(P)$. 
It follows that the cover of $\Sigma_{g,c+1} \times S^1$ induced by $\varphi$ is just~$\Sigma_{g,c+1} \times \R$, and therefore
\[H_1(\Sigma_{g,1} \times S^1; \Z[t^{\pm 1}]) \cong \Z^{2g}_\varepsilon \oplus \Z_\varepsilon^c\] is generated by lifts of the genus {loops} as well as by lifts $\widetilde{\partial}_i$ of the $\partial_i$ for $i=1,\ldots c$.
\item Finally we consider~$S^1 \times S^1\cong \partial_{\text{out}} \Delta_i \times S^1$.
Under the inclusion induced map, the first factor is mapped to $\partial_i \sim \partial_{\text{out}} \Delta_i$,  and the second factor is mapped to the $S^1$-fibre.
Thus the coefficient system maps the first factor to $0$ and second factor to $1$. 
It follows that 
$$H_1(S^1 \times S^1;\Z[t^{\pm 1}]) \cong \Z_\varepsilon$$
 is generated by a lift of a $(1,0)$-curve.
\end{itemize}
In particular, since none of the coefficients systems is trivial, $j_0$ is injective and thus
\begin{align*}
H_1(P; \Z[t^{\pm1}])
& \cong \left(  H_1( \Sigma_{g,c+1}\times S^1; \Z[t^{\pm1}])  \oplus  \bigoplus_{i=1}^c H_1(Q_i; \Z[t^{\pm1}]) \right)\Bigg/ \im (j_1).
%&\cong  \Z_
\end{align*}
The first summand is generated by lifts of the genus {loops} and lifts $\widetilde{\partial}_i$ of the $\partial_i$, whereas the second summand is generated by lifts of the plumbing loops.

The $i$-th copy of the lift of the aforementioned $(1,0)$-curve in $H_1(S^1 \times S^1;\Z[t^{\pm 1}])$ maps under~$j_1$ to~$[\widetilde{\partial}_i] \in H_1(\Sigma_{g,c+1} \times S^1;\Z[t^{\pm 1}])$ and~$[\partial_{\text{out}} \Delta_i]\in H_1(Q_i;\Z[t^{\pm 1}])$.
 Lemma~\ref{lem:H1P01ZZ} ensures that this latter homology class vanishes.
It follows that 
\begin{align*}
H_1(P; \Z[t^{\pm1}])
%& \cong \left(  H_1( \Sigma_{g,c+1}\times S^1; \Z[t^{\pm1}])  \oplus  \bigoplus_{i=1}^c H_1(Q_i; \Z[t^{\pm1}]) \right)\Bigg/ \im (j_1)  \\
& \cong \Z_\epsilon^{2g} \oplus \bigoplus_{i=1}^c \Z[t^{\pm1}]/(t-1)^2.
\end{align*}
The summands are respectively generated by lifts of the genus and plumbing {loops}.

It remains to show that~$H_1(\partial P;\Z[t^{\pm 1}]) \to H_1(P;\Z[t^{\pm 1}])$ is the zero map.
We first study the restriction of the coefficient system to $\partial P \cong \partial_{\text{out}} \Sigma_{g,c+1} \times S^1$.
Under the inclusion map, the second factor is mapped to the $S^1$-fibre of $P$.
Since $\partial_{\text{out}} \Sigma_{g,c+1}\times \lbrace \operatorname{pt} \rbrace$ is homologous to~$\cup_{i=1}^c \partial_i$ in~$\Sigma_{g,c+1}$, the first factor $\delta$ is mapped to~$[\partial_{\text{out}} \Sigma_{g,c+1} \times \lbrace \operatorname{pt} \rbrace]=-\sum_{i=1}^c[\partial_i]=0\in H_1(P)$.
We deduce that
$$ H_1(\partial P;\Z[t^{\pm 1}]) \cong \Z[t^{\pm 1}]/(t-1)\langle \widetilde{\delta} \rangle.$$
The assertion thus reduces to proving that $\widetilde{\delta}$ is trivial in $H_1(P;\Z[t^{\pm 1}])$.  
Note that~$\delta$ cobounds a surface with~$\cup_{i=1}^c \partial_i$ within $\Sigma_{g,c+1}$.
%%AC: Technically it's \partial_{\text{out}} \Sigma_{g,c+1}\times \lbrace \operatorname{pt} \rbrace. But this is probably fine.
Since the coefficient is trivial on $\Sigma_{g,c+1}$, this cobordism lifts to the cover.
It follows from our description of $H_1(\Sigma_{g,c+1} \times S^1; \Z[t^{\pm 1}])$ above that 
$$\widetilde{\delta}=-\sum_{i=1}^c [\widetilde{\partial_i}] \in H_1(\Sigma_{g,c+1};\Z[t^{\pm 1}]) \cong \Z^{2g}_\varepsilon \oplus \Z^c_\varepsilon. $$
Since the~$\Z^c_\varepsilon$ summand of $H_1(\Sigma_{g,c+1};\Z[t^{\pm 1}])$ maps to zero in~$H_1(P;\Z[t^{\pm 1}])$, the assertion follows.
This concludes the proof of the lemma.
\color{black}
\end{proof}

\subsection{Realising isometries  by surface homeomorphisms}
\label{sub:RealisingIsometries}

In this section,  we show that each of the isometries of Examples~\ref{ex:ScaleIsometry}-\ref{ex:Sp2gZ} is induced by $\theta \times \id_{S^1}$ for some~$\theta \in \Homeo_\alpha(\Sigma, \partial)= \Homeo(\Sigma^\circ, \partial \Sigma^\circ)$.  
%AC: This equality is near immediate: It’s because for \alpha(d) \in \partial B_k, \theta(\alpha(d))=\alpha(d).
In fact, each of these realizations results also holds if $\Sigma$ is closed: all the constructions take place far from the boundary and Proposition~\ref{prop:BlPK} shows that the inclusion induces an isometry~$\Bl_P \cong \Bl_{P_U}$.
%%Look at realizing isometries sept 3.pdf

\medbreak

We begin by observing that~$\intt(\Sigma^\circ) \times \{1\} \subset P=P_g(c_+,c_-)$ has the property that the following composition is the zero map:
\[\pi_1(\intt(\Sigma^\circ) \times \{1\}) \to \pi_1(P) \xrightarrow{\varphi} \Z.\]
Therefore~$\intt \Sigma^\circ \times \{1\}$ lifts to~$P^\infty$,  the ~$\varphi$-induced infinite cyclic cover of~$P$.  
In particular,  Dehn twisting along any curve~$\beta$  in~$\intt \Sigma^\circ \times \{1\}$  will induce a map of~$\Sigma^{\circ} \times S^1$ of the form~$f \times \id_{S^1}$,  where~$f$ is supported on~$\mathcal{N}_{\Sigma^\circ}(\beta)$,  a neighborhood of $\beta$ in $\Sigma^\circ$.
The induced homeomorphism of~$P$,  thought as a quotient of~$\Sigma^{\circ} \times S^1$,  will then lift to a map that is supported on~$\mathcal{N}_{\Sigma^\circ}(\beta) \times \R \subset P^\infty$,  where it is of the form~$f \times \id_{\R}$.  
We will understand how~$f_*$ acts on~$H_1(P; \Z[t^{\pm1}])$  by analyzing~$f$'s action on curves in~$\Sigma^\circ \times \{1\}$,  thought of as a subset of~$\Sigma^ \circ \times \R \subset P^\infty$.  

Recall from Proposition~\ref{prop:HomologyPZZ} that our preferred generators for~$H_1(P; \Z[t^{\pm 1}])$ consist of~$y_1, \ldots y_{2g}$,  which are order~$(t-1)$ elements obtained by lifting the genus {loops}~$a_1,b_1,\dots,a_g,b_g$ to~$\Sigma^\circ \times \{1\} \subset P^\infty$,   and~$x_1, \dots, x_c$,  which are order~$(t-1)(t^{-1}-1)$ elements obtained by lifting~$\omega_1,\dots, \omega_c$ to start (and end) within~$\Sigma^\circ \times \{1\}  \subset P^\infty$.  
Here, recall that the $\omega_i$ are the arcs that project to the plumbing loops in $P$, and that we also denote the latter by $\omega_i$.
We will repeatedly use the observation from Remark~\ref{rem:onenegativedp} that in~$H_1(P; \Z[t^{\pm1}])$,  
\[ [\widetilde{\alpha}_{2i-1}]= \varepsilon_i (t-1) [\widetilde{\omega}_i]\]
where  
$\widetilde{\alpha}_{2i-1}$ is the lift of~$\alpha_{2i-1}$  to within~$\Sigma^\circ \times \{1\}$, ~$\widetilde{\omega}_i$ is the lift of~$\omega_{i}$ starting (and ending) in~$\Sigma^{\circ} \times \{1\}$, 
 and~$\varepsilon_i= \pm 1$ according to the sign of the~$i$-th plumbing. 

%%Look at realizing isometries sept 3.pdf
\begin{construction}
For~$i=1,\dots,c$,  let~$\zeta_i$ be a small push-off of the boundary component~$\alpha_{2i-1}$ of~$\Sigma^{\circ}$ into the interior.  
For~$i \neq i'$,  let~$\sigma_{i,i'}$ be a tubing of~$\zeta_i$ to~$\zeta_{i'}$,  as illustrated in Figure~\ref{fig:surfacewithmorecurvesf}. 
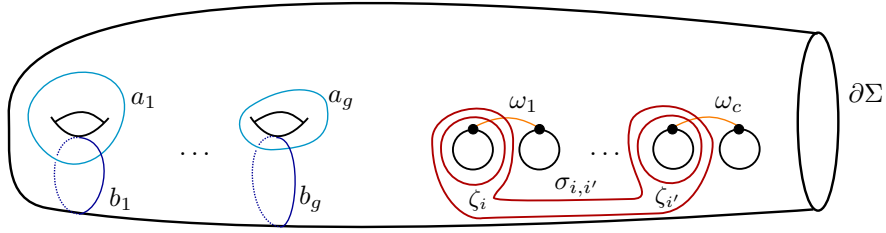
\begin{figure}[htbp!]
\begin{centering}
\begin{tikzpicture}
\node[anchor=south west,inner sep=0] at (0,0)
{\includegraphics[height=3cm]{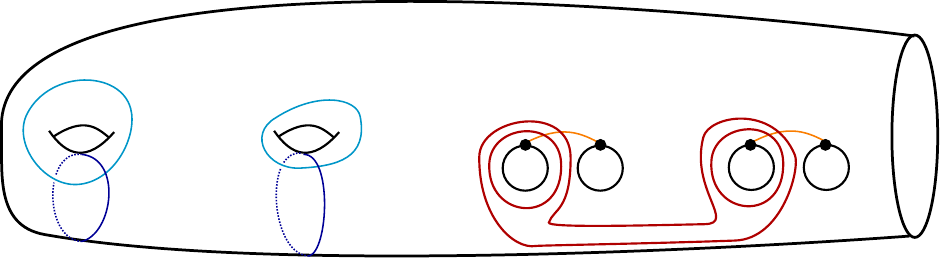}};
%\node at (1.5, 3.2){$z$}; \node at (1.2,1.1){$x$}; \node at (1.8, 1.1){$y$};
%\node at (1.45, 1.7){$\gamma$}; \node at (.8,0){$\zeta$}; 
% \node at (.6,.9){$\alpha$}; \node at (2.3,.9){$\beta$};
% \node at (1.1,  2.5){$\epsilon$}; \node at (1.5,  .5){$\Delta$};
\node at (1.5, .4){$b_1$}; \node at (1.8, 1.7){$a_1$}; \node at (4, .4){$b_g$}; \node at (4.4, 1.7){$a_g$};
\node at (6.2,.38){\small $\zeta_i$};\node at (8.7,.42){\small $\zeta_{i'}$};\node at(7.5, .55){$\sigma_{i,i'}$};
\node at (6.8, 1.65){$\omega_1$}; \node at (9.5, 1.65){$\omega_c$};
\node at (2.5, 1){$\dots$};\node at (7.9,1){$\dots$};
\node at (11.3, 1.8){$\partial \Sigma$};
\end{tikzpicture}
\end{centering}
\caption{The curves $\zeta_i$ and $\sigma_{i,i'}$.}
\label{fig:surfacewithmorecurvesf}
\end{figure}
For~$i=1,\dots,c$ and $j=1,\dots, 2g$,  let $\tau_{i,j}$ be a tubing of $\zeta_i$ to a parallel push-off of $a_k$ (if~$j=2k-1$) or of $b_k$ (if~$j=2k$),  as illustrated in Figure~\ref{fig:surfacewithmorecurvesrho}.
\begin{figure}[hbtp!]
\begin{centering}
\begin{tikzpicture}
\node[anchor=south west,inner sep=0] at (0,0)
{\includegraphics[height=3cm]{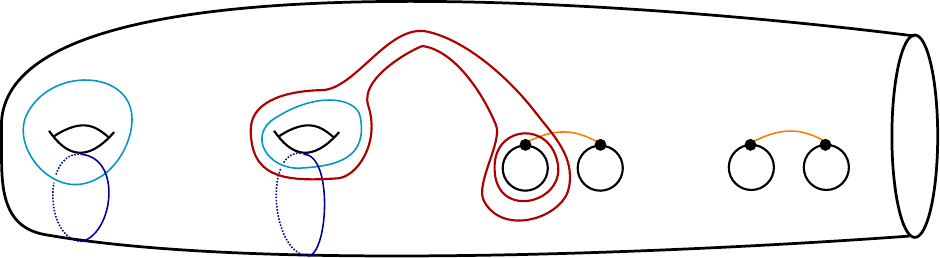}};
%\node at (1.5, 3.2){$z$}; \node at (1.2,1.1){$x$}; \node at (1.8, 1.1){$y$};
%\node at (1.45, 1.7){$\gamma$}; \node at (.8,0){$\zeta$}; 
% \node at (.6,.9){$\alpha$}; \node at (2.3,.9){$\beta$};
% \node at (1.1,  2.5){$\epsilon$}; \node at (1.5,  .5){$\Delta$};
\node at (1.5, .4){$b_1$}; \node at (1.8, 1.7){$a_1$}; \node at (4, .4){$b_g$}; \node at (5,2){$\tau_{i,j}$};
\node at (6,1.65){$\zeta_1$};
\node at (6.8, 1.65){$\omega_1$}; \node at (9.5, 1.65){$\omega_c$};
\node at (2.5, 1){$\dots$};\node at (7.9,1){$\dots$};
\node at (11.3, 1.8){$\partial \Sigma$};
\end{tikzpicture}
\end{centering}
\caption{The curves $\tau_{i,j}$,  illustrated for $i=1$ and $j=2g-1$.}
\label{fig:surfacewithmorecurvesrho}
\end{figure}
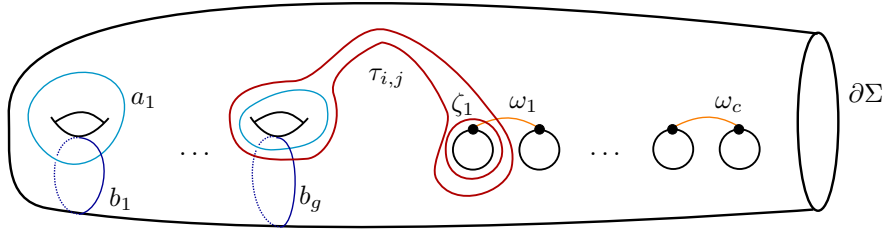
\end{construction}

%%Look at realizing isometries sept 3.pdf
\begin{example}[Realizing the `scale a generator by~$t^k$'-isometries]
\label{ex:ScaleIsometryrealize}
For any $1 \leq i \leq c$,  left handed Dehn twisting along~$\zeta_i$ realizes the isometry of~$\Bl_P$ that sends
\[x_i= [\widetilde{\omega}_i] \mapsto
[\widetilde{\alpha}_{2i-1}] + [\widetilde{\omega}_i]=\varepsilon_i (t-1)[\widetilde{\omega}_i]+[\widetilde{\omega}_i]= \begin{cases} t  [\widetilde{\omega}_i]& \varepsilon_i=1 \\
t^{-1} [\widetilde{\omega}_i] & \varepsilon_i=-1 \end{cases}
=\begin{cases} t  x_i& \varepsilon_i=1 \\
t^{-1}  x_i & \varepsilon_i=-1 \end{cases}
\] while fixing all other generators. 
%(When~$\varepsilon_i=-1$,  we use  above that~$[\widetilde{\omega_i}]$ is order~$(t-1)(t^{-1}-1)$,  so~$(2-t)[\widetilde{\omega_i}]= t^{-1} [\widetilde{\omega_i}]$.)
Compositions of this map and its inverse allow us to realize all `scale~$x_i$ by~$t^k$' isometries of Example~\ref{ex:ScaleIsometry}. 
\end{example}

%%Look at realizing isometries sept 3.pdf
\begin{example}[Realizing~$p_{i,i'}$]\label{ex:fijrealize}
For any distinct $1 \leq i, i' \leq c$,  simultaneous left-handed Dehn twisting along~$\sigma_{i,i'}$  and right-handed Dehn twisting along both~$\zeta_i$ and~$\zeta_i'$ realizes the isometry of~$\Bl_P$ that sends 
\begin{align*}
x_i=[\widetilde{\omega}_i] \mapsto \ & [\widetilde{\omega}_i] + [\widetilde{\alpha}_{2i'-1}]=  [\widetilde{\omega}_i] + \varepsilon_{i'}(t-1) [\widetilde{\omega}_{i'}]= x_i + \varepsilon_{i'}(t-1) x_{i'}, \\
x_{i'}=[\widetilde{\omega}_{i'}] \mapsto \ &  [\widetilde{\alpha}_{2i-1}] + [\widetilde{\omega}_{i'}] = \varepsilon_i (t-1) [\widetilde{\omega}_i] + [\widetilde{\omega}_{i'}]= \varepsilon_i (t-1)x_i + x_{i'},
\end{align*}
while fixing all other generators; that is, it realizes the~$p_{i,i'}$ isometry of Example~\ref{ex:fij}. 
\end{example}

%%Look at realizing isometries sept 3.pdf
\begin{example}[Realizing~$q_{i,j}$]
\label{ex:Rhoijrealize}
For any $1 \leq i \leq c$ and $1 \leq k \leq g$,  simultaneous left-handed Dehn twisting along~$\tau_{i,2k-1}$ and right-handed twisting along~$\zeta_i$ realizes the isometry of~$\Bl_P$ that sends 
\begin{align*}
x_i=[\widetilde{\omega}_i]  & \mapsto\phantom{abc} [\widetilde{\omega}_i] +  [\widetilde{a}_{k}]  \phantom{ \ + [\widetilde{b}_{k}]}=   \phantom{\varepsilon_i (t-1)}[\widetilde{\omega}_i] +  [\widetilde{a}_{k}]  \phantom{ +[\widetilde{b}_{k}]} = \phantom{ \varepsilon_i(t-1)}x_i + y_{2k-1}, \phantom{ \ + y_{2k}}
 \\
y_{2k-1}=[\widetilde{a}_{k}] & \mapsto \phantom{[\widetilde{\alpha}_{2i-1}] + \ }   [\widetilde{a}_k] \phantom{\ + [\widetilde{b}_k]}= \phantom{\varepsilon_i (t-1) [\widetilde{\omega}_i] -\ \ } [\widetilde{a}_{k}]  \phantom{\ + [\widetilde{b}_{k}]}= \phantom{ \varepsilon_i(t-1)x_i+ \ } y_{2k-1},  \phantom{ \ + y_{2k}}
\\
y_{2k}=[\widetilde{b}_k] & \mapsto  [\widetilde{\alpha}_{2i-1}]- [\widetilde{a}_{k}] + [\widetilde{b}_{k}]=  \varepsilon_i (t-1) [\widetilde{\omega}_i]-[\widetilde{a}_{k}] + [\widetilde{b}_{k}]= \varepsilon_i(t-1)x_i - y_{2k-1}+y_{2k},
\end{align*}
while fixing all other generators; that is, it realizes the~$q_{i,2k-1}$ isometry of Example~\ref{ex:Rhoij}.  A similar computation shows that simultaneous left-handed Dehn twisting along~$\tau_{i,2k}$ and right handed Dehn twisting along~$\zeta_i$ realizes the isometry~$q_{i,2k}$. 
\end{example}

%%Look at realizing isometries sept 3.pdf
\begin{example}[Realizing $\operatorname{Sp}(2g,\Z)$-type isometries]
\label{ex:Sp2gZrealize}
Realising the $\operatorname{Sp}(2g,\Z)$-isometries from Example~\ref{ex:Sp2gZ} amounts to the fact that isometries of
$$\Aut(\coker(\mathcal{H}_2^{\oplus g}),\partial \mathcal{H}_2^{\oplus g}) \cong \Aut(H_1(\Sigma_{g,1}),  Q_{\Sigma_{g,1}}) \cong \operatorname{Sp}(2g,\Z)$$ are realized by homeomorphisms of $\Sigma_{g,1}$ that fix the boundary pointwise~\cite[Section~2.1 and the discussion following Theorem 6.4]{FarbMargalit}.
\end{example}

Each of these constructions work if $\Sigma=\Sigma_g$ is closed and result in isometries of $\Bl_{P_U} \cong \Bl_P$.
Indeed everything takes place far from the boundary of $\Sigma_{g,1}$.
We summarize the above work in the following proposition. 

\begin{proposition}~\label{prop:realizability}
Let $s \in \Aut(\Bl_P)$ be a composition of the `scale by $t^k$',   $p_{i,i'}$,  $q_{i,j}$, and~$\operatorname{Sp}(2g,\Z)$-isometries from Examples~\ref{ex:ScaleIsometry}-\ref{ex:Sp2gZ},  and their inverses. 
There exists $\theta \in \Homeo_\alpha(\Sigma, \partial)$ such that 
\[ s= \widehat{\theta}_*,\]
where $\widehat{\theta}$ is the homeomorphism of $P$ determined by $\theta$ as described in Construction~\ref{cons:HomeoSigmaActionRelBoundary}. 

The same result holds if $\Sigma$ is closed,  with $P_U$ in place of $P$.
\end{proposition}

\section{Enumerating immersed~$\Z$-surfaces}
\label{sec:Enumerating}

The goal of this section is to prove the results stated in Sections~\ref{sub:IntroD4S4} and~\ref{sub:Other4ManifoldsIntro} of the introduction.

\subsection{Immersed surfaces in $D^4$ and $S^4$.}
\label{sub:ProofofIntroD4S4}

We prove the results stated in Section~\ref{sub:IntroD4S4}.

\begin{customthm}
{\ref{thm:UnknottingAllSameSignIntro}}
Assume that~$S$ is either a closed genus~$g$ immersed~$\Z$-surface in~$S^4$ or a genus~$g$ immersed~$\Z$-surface in~$D^4$ with boundary an Alexander polynomial one knot~$K$.
The following assertions are equivalent:
\begin{enumerate}
\item $S$ is isotopic to the standard genus $g$ surface with~$c_+$ positive double points and~$c_-$ negative double points; in the non-closed case,  this isotopy can be taken to be rel. boundary. 
\item the equivariant intersection form of the exterior of~$S$ is isometric to~$\lambda:=\lambda_{c_+,c_-} \oplus\mathcal{H}_2^{\oplus g}.$
%$$\lambda:=((t-1)(t^{-1}-1))^{\oplus c_+} \oplus (-(t-1)(t^{-1}-1))^{\oplus c_-} \oplus \mathcal{H}_2^{\oplus g}.$$
\end{enumerate}
\end{customthm}
\begin{proof}
We begin with the closed case.
The $(1) \Rightarrow (2)$ direction follows from Example~\ref{ex:SimpleImmersed}, so we focus on the $(2)\Rightarrow (1)$ direction.
We write $P_U:=P_{U,g}(c_+,c_-)$.
Thanks to the classification stated in Theorem~\ref{thm:SurfacesClosed} and the simplification from Proposition~\ref{prop:SimplifyAlgebraClosed},  the theorem amounts to proving~$\Aut(\Bl_{P_U})/(\Aut(\lambda) \times \Homeo_\alpha(\Sigma_g))$ is trivial.
It suffices to show that every  $f \in \Aut(\Bl_{P_U})$ can be written as
\[f= \widehat{\theta}_* \circ \partial \rho\]
for some $\theta \in \Homeo_{\alpha}(\Sigma_g)$ and $\rho \in \Aut(\lambda)$.  
By Proposition~\ref{prop:AutBl/AutStandard},  we can write 
\[f= s \circ \partial \rho\]
%AC: no that that prop applies both P and P_U
for  $s$ a composition of isometries of $\Bl_{P_U}$ as given in Examples~\ref{ex:ScaleIsometry}-\ref{ex:Sp2gZ} and $\rho \in \Aut(\lambda)$.  But by Proposition~\ref{prop:realizability},  such an $s$ can be realized as the induced map of some $\theta \in \Homeo_{\alpha}(\Sigma_g)$,  and the result up to equivalence follows.

To obtain the result up to isotopy instead of up to equivalence, we apply work of Quinn~\cite[Theorem 1.1]{Quinn} (see also~\cite[Theorem~10.1]{FreedmanQuinn}) that ensures that every homeomorphism of $S^4$ is topologically isotopic to the identity.

The proof for surfaces with boundary is analogous.  
The only differences are that this time we appeal to Theorem~\ref{thm:SurfacesRelBoundary} and to the simplification from Proposition~\ref{prop:SimplifyAlgebraBoundary}, and that the upgrade to isotopy relies on the Alexander trick rather than on Quinn's deep work.
\end{proof}

\begin{customthm}
{\ref{thm:UniqueSpherec=1Intro}}
~
\begin{itemize}
\item  Every~$\Z$-sphere~$S \subset S^4$ with a single double point is isotopic to the standard~$\Z$-sphere with a single double point of the same sign.
\item Given an Alexander polynomial one knot~$K$,  every~$\Z$-disk in~$D^4$ with boundary~$K$ and a single double point is isotopic rel. boundary to the standard~$\Z$-disk with boundary~$K$ and a single double point of the same sign.
\end{itemize}
\end{customthm}
\begin{proof}
We begin with the spherical case.
Set $P_U:=P_{U,0}(c_+,c_-)$ and let $\varepsilon \in \{-1,1\}$ be the sign of the double point.
There is a unique size one matrix presenting $\Bl_{P_U}$, namely $(\varepsilon (t-1)(t^{-1}-1))$.
Recall from Example~\ref{ex:SimpleImmersed} that this is the equivariant intersection form of the exterior of the standard immersed sphere with a single double point of sign $\varepsilon$.
Theorem~\ref{thm:UnknottingAllSameSignIntro} therefore implies that $S$ is isotopic to said standard immersed sphere.
The proof for disks is entirely analogous.
\end{proof}

\begin{customthm}
{\ref{thm:g=0c=1BoundaryD4Intro}}
\label{thm:g=0c=1BoundaryD4IntroProof}
Let~$K$ be a knot,  let~$(c_+,c_-) \in \{(0,1),(0,1)\}$,  and let~$\varepsilon\in \{-1,1\}$ be $1$ if $c_+=1$ and $-1$ if $c_-=1$. 
The following assertions are equivalent:
\begin{enumerate}
\item the Blanchfield form $\Bl_K$ is presented by $(\varepsilon \Delta_K)$. 
\item the set of rel. boundary isotopy classes of~$\Z$-disks with boundary~$K$ and a single double point of sign~$\varepsilon$  is nonempty and corresponds bijectively to~$U(\Delta_K)/\{ t^k \}_{k \in \Z}.$
\end{enumerate}
\end{customthm}
\begin{proof}
There is a unique size one nondegenerate hermitian form presenting~$\Bl_{P_{K,0}(c_+,c_-)}$, namely the form~$\lambda:=\varepsilon \Delta_K(t-1)(t^{-1}-1)$.
Proposition~\ref{prop:BlKBlPK} below implies that $\Bl_{P_{K,0}(c_+,c_-)}$ is presented by $\lambda$ if and only if $\Bl_K$ is presented by $(\varepsilon\Delta_K)$.
By the classification stated in Theorem~\ref{thm:SurfacesRelBoundary} and the simplification from Proposition~\ref{prop:SimplifyAlgebraBoundary},  the first condition is therefore equivalent to the set of surfaces under consideration being nonempty and corresponding to
$$  \frac{(\Aut(\Bl_P)/\Homeo_\alpha(D^2,\partial))\times \Aut(\Bl_K)}{\Aut(\lambda)}.$$
Here we use the notation~$P:=P_0(c_+,c_-)$.
Proposition~\ref{prop:AutBl/AutStandardImproved} implies~$\Aut(\Bl_{P})=\lbrace \pm t^k \rbrace_{k \in \Z}$.
The action of~$\Homeo_\alpha(D^2,\partial)$ factors through the action of the mapping class group $\operatorname{MCG}_\alpha(D^2,\partial)$ which is the mapping class group of the pair of pants.  
It is well-known (see for example \cite[Section 3.6.4]{FarbMargalit}) that this group is isomorphic to $\Z^3$, generated by Dehn twisting around push-offs of each of the three boundary components.  The Dehn twist along either of the two `inner' boundary components scales a generator by $t^{\pm1}$ (recall Example~\ref{ex:ScaleIsometryrealize}),  while the Dehn twist along the outer boundary $\partial D^2$ induces the identity map.  
Therefore,  ~$\Aut(\Bl_P)/\Homeo_\alpha(D^2,\partial)=\{ \pm \id\}$.
As~$\lambda$ has size one, we have~$\Aut(\lambda)=\{ \pm t^k \}_{k \in \Z}$.
We deduce that
$$ \frac{(\Aut(\Bl_P)/\Homeo_\alpha(D^2,\partial))\times \Aut(\Bl_K)}{\Aut(\lambda)}
\cong 
 \frac{ \{ \pm \id \} \times \Aut(\Bl_K)}{ \{ \pm t^k \}_{k \in \Z} } 
 \cong 
 \frac{\Aut(\Bl_K)}{ \{t^k \}_{k \in \Z} }.
$$
\color{black}
%$$\frac{\Aut(\Bl_K)}{\Aut(\lambda)} \times \frac{\Aut(\Bl_P)}{\Aut(\lambda) \times \Homeo_\alpha(D^2,\partial)}.$$
%Here we use the notation~$P:=P_0(c_+,c_-)$.
%%and $\Sigma:=\Sigma_{g,1}.$
%Since Proposition~\ref{prop:AutBl/AutStandard} implies that~$\Aut(\Bl_P)/(\Aut(\lambda) \times \Homeo_\alpha(D^2,\partial))$ is trivial,  the aforementioned product reduces to~$\Aut(\Bl_K)/\Aut(\lambda)$.
As noted in the introduction, 
% since $\lambda$ has size one, 
 this set corresponds bijectively to~$U(\Delta_K)/\{ t^k \}_{k \in \Z}.$
\end{proof}

\subsection{Immersed surfaces in other $4$-manifolds}
\label{sub:ProofofOther4ManifoldsIntro}

We prove the results stated in Section~\ref{sub:Other4ManifoldsIntro}.

\begin{customthm}
{\ref{thm:Other4ManifoldsSpheresIntro}}
Let~$X$ be a closed simply-connected~$4$-manifold, and let~$c_+,c_-$ be positive integers.
The number of equivalence classes of~$\Z$-surfaces in~$X$ with~$c_+$ positive double points, ~$c_-$ double points and whose exteriors have the same equivariant intersection form $\lambda$ is
\begin{itemize}
\item one if~$(c_+,c_-) \in \{(1,0),(0,1)\}$.
\item at most two if~$(c_+,c_-)=(1,1)$,
\item finite if~$(c_+,c_-) \in \{(c,0),(0,c)\}$,
\end{itemize}
For $N$ a simply-connected $4$-manifold with boundary $S^3$, the same statement holds for rel. boundary equivalence classes of $\Z$-surfaces in $N$ with boundary an Alexander polynomial one knot.
\end{customthm}
\begin{proof}
We begin with closed surfaces in closed manifolds.
Set $P_U:=P_{U,g}(c_+,c_-)$.
Thanks to the combination of Theorem~\ref{thm:SurfacesClosed} and Proposition~\ref{prop:SimplifyAlgebraClosed}, it suffices to estimate the cardinality of~$\Aut(\Bl_{P_U})/(\Aut(\lambda) \times \Homeo_\alpha(\Sigma_g))$.

Applying Proposition~\ref{prop:AutBl/AutStandardImproved} under the assumption that~$(c_+,c_-) \in \{(c,0),(0,c)\}$ ,  we have that any~$f \in \Aut(\Bl_{P_U})$ can be written as $f= s \circ s'$ for some $s$ a composition of the isometries in Examples~\ref{ex:ScaleIsometry}-\ref{ex:Sp2gZ} (which,  by Proposition~\ref{prop:realizability} can be realized by a~$\theta \in \Homeo_{\alpha}(\Sigma_g)$) and $s'$ a composition of the `scale a generator by $-1$' and permutation-induced isometries from Examples~\ref{ex:ScaleIsometry-1} and~\ref{ex:PermutationIsometry}.  We therefore immediately obtain the first item: 
\begin{equation*}
%\label{eq:Estimate}
\Bigg|\frac{\Aut(\Bl_{P_U})}{(\Aut(\lambda) \times \Homeo_\alpha(\Sigma_g))}\Bigg| \leq
\Bigg|\frac{\Aut(\Bl_{P_U})}{\Homeo_\alpha(\Sigma_g)}\Bigg| \leq  c! \cdot 2^c.
\end{equation*}
Now suppose that~$c=1$ and let~$f \in \Aut(\Bl_{P_U})$. 
 Using the notation from the proof of Proposition~\ref{prop:AutBl/AutStandard},  write $f(x_1)= a_1 + (t-1) b_1 + c_1$.  If $a_1=x_1$,  then by Steps (2)-(4) of the proof of Proposition~\ref{prop:AutBl/AutStandard}, 
 $f$ is a composition of isometries from Examples~\ref{ex:ScaleIsometry}-\ref{ex:Sp2gZ} and their inverses hence, by Proposition~\ref{prop:realizability},  is related to the identity by the action of $\Homeo_\alpha(\Sigma_g)$.  
 If $a_1 \neq x_1$,  then the proof of Proposition~\ref{prop:AutBl/AutStandardImproved} shows that $a_1=-x_1$ (specifically, see~\eqref{eqn:allpos} with $c=1$).
  Let $\rho$ be the isometry of $\lambda$ that scales every element of the equivariant second homology by $-1$.  We then have that $f \circ \partial \rho$ has the property that $(f \circ \partial \rho)(x_1)= x_1 -(t-1)b_1-c_1$, and as explained above is related to the identity by the action of $\Homeo_\alpha(\Sigma_g)$.  
Thus, in this case, the set~$\frac{\Aut(\Bl_{P_U})}{(\Aut(\lambda) \times \Homeo_\alpha(\Sigma_g))}$ is a singleton for every $\lambda$.
 
The argument is similar for $c_+=1=c_-$.  
First,  the proof of Proposition~\ref{prop:AutBl/AutStandard11} tells us that we may have~$a_1= \pm x_1$ and $a_2 = \pm x_2$.  Then,  the `scale everything by $-1$' isometry of $\lambda$ allows us to assume that either $a_1=x_1,  a_2=x_2$ or $a_1=x_1, a_2=-x_2$.  Any map  $f$ as in the first case is related to the identity via the action of $\Homeo_\alpha(\Sigma_g)$,  and any two maps $f,f'$ as in the second case are related via the action of $\Homeo_\alpha(\Sigma_g)$,  because the composition $f' \circ f^{-1}$ is as in the first case.
Thus, in this case, the set~$\frac{\Aut(\Bl_{P_U})}{(\Aut(\lambda) \times \Homeo_\alpha(\Sigma_g))}$ has at most two elements for every $\lambda$.
%%Don't delete.
%%Again use  {Steps (2)-(4) of} the proof of Proposition~\ref{prop:AutBl/AutStandard} and Proposition~\ref{prop:realizability},

We move on to surfaces with boundary.
Set $P:=P_g(c_+,c_-)$.
Since knots with Alexander polynomial one have trivial Alexander module, thanks to the combination of Theorem~\ref{thm:SurfacesRelBoundary} and Proposition~\ref{prop:SimplifyAlgebraBoundary}, it suffices to estimate the cardinality of the set~$\Aut(\Bl_P)/(\Aut(\lambda) \times \Homeo_\alpha(\Sigma_g,\partial))$. 
The arguments for that are exactly as in the closed case
\end{proof}

\begin{remark}
\label{rem:WhyItsHardMain}
As announced in Remark~\ref{rem:WhyItsHard}, the argument in this proof can be modified as follows to show that for $g=0$ and $(c_+,c_-)=(1,0)$,  we have
 $$\Bigg| \frac{\Aut(\Bl_{P_U})}{\Homeo_\alpha(\Sigma_g)} \Bigg|=2.$$
First note that~$\Aut(\Bl_{P_U})=\lbrace \pm t^k \rbrace_{k \in \Z}$.
%%Don't delete.
%%E.g. by Proposition~\ref{prop:AutBl/AutStandardImproved}.
Next, note that the action of~$\Homeo_\alpha(\Sigma_g)$ factors through the action of the mapping class group.
In this case, $\Sigma^\circ$ is an annulus and the generator of~$\operatorname{MCG}_\alpha(\Sigma_0)=\operatorname{MCG}(\Sigma^\circ,\partial \Sigma^\circ) \cong \Z$ acts by $t$, so $\frac{\Aut(\Bl_{P_U})}{\Homeo_\alpha(\Sigma_g)}=\lbrace \pm \id \rbrace.$
%%Don't delete.
%%By Example~\ref{ex:ScaleIsometryrealize}.
\end{remark}

We record the following generalization of Theorem~\ref{thm:UnknottingAllSameSignIntro} to other simply-connected $4$-manifolds.
\begin{theorem}
\label{thm:StandardNotS4}
Let~$X$ be a closed simply-connected~$4$-manifold,  let~$c_+,c_-$ be positive integers, and let~$S \subset X$ be a closed genus $g$ immersed~$\Z$-surface.
The following assertions are equivalent:
\begin{itemize}
\item $S$ is equivalent to (a local copy of) the standard genus $g$ surface in $X$ with~$c_+$ positive double points and~$c_-$ negative double points;
\item the equivariant intersection form of the exterior of~$S$ is isometric to~$Q_X \oplus \lambda_{c_+,c_-} \oplus \mathcal{H}_2^{\oplus g}.$
\end{itemize}
For $N$ a simply-connected $4$-manifold with boundary $S^3$, the same statement holds for $\Z$-surfaces in $N$ with boundary an Alexander polynomial one knot, and the equivalence can be taken to be rel. boundary.
\end{theorem}
\begin{proof}
The proof is identical to that of Theorem~\ref{thm:UnknottingAllSameSignIntro}: since $Q_N$ is nonsingular, the boundary linking form of $\lambda:=Q_X \oplus \lambda_{c_+,c_-} \oplus \mathcal{H}_2^{\oplus g}$ is still $\Bl_{P_U}\cong \partial(\lambda_{c_+,c_-} \oplus \mathcal{H}_2^{\oplus g})$ and the isometries of~$\lambda_{c_+,c_-} \oplus \mathcal{H}_2^{\oplus g}$ suffices to ensure that~$\Aut(\Bl_{P_U})/(\Aut(\lambda) \times \Homeo_\alpha(\Sigma_g))$ is a singleton.
The same reasoning applies when there is a boundary.
\end{proof}

We also record the following generalization of Theorem~\ref{thm:g=0c=1BoundaryD4IntroProof} to other simply-connected $4$-manifolds.

\begin{theorem}
\label{thm:g=0c=1Boundary}
Let $N$ be a simply-connected $4$-manifold with $\partial N \cong S^3$,  let $K \subset \partial N$ be a knot, and let $\lambda$ be a nondegenerate hermitian form.
The following assertions are equivalent:
\begin{enumerate}
\item the hermitian form $\lambda$ presents $\Bl_{P_K}$ and satisfies~$\lambda(1) \cong Q_N \oplus (0)$;
\item the sets $\Surf_\lambda^0(0;0,1)(K,N)$ and $\Surf_\lambda^0(0;1,0)(K,N)$ are nonempty and correspond bijectively to $\frac{ \{ \pm \id \} \times \Aut(\Bl_K)}{  \Aut(\lambda)}.$
\end{enumerate}
\end{theorem}
\begin{proof}
The proof is identical to that of Theorem~\ref{thm:g=0c=1BoundaryD4IntroProof} with its first sentence removed and without the simplifications afforded by the explicit knowledge of $\Aut(\lambda)$.
%More explicitly,  the set of disks corresponds bijectively
%$$ \frac{(\Aut(\Bl_P)/\Homeo_\alpha(D^2,\partial))\times \Aut(\Bl_K)}{\Aut(\lambda)}
%\cong 
% \frac{ \{ \pm \id \} \times \Aut(\Bl_K)}{  \Aut(\lambda)  } .
%$$
%This concludes the proof of the theorem.
\color{black}
\end{proof}

\section{Immersed $\Z$-disks and unknotting sequences}
\label{sec:Unknotting}

We prove Theorem~\ref{thm:Unknotting} from the introduction.
%(where the ordering of the statements was different).

\begin{customthm}{\ref{thm:Unknotting}}
\label{thm:UnknottingMain}
Let $c_+,c_- \geq 0$ be integers and set $c:=c_++c_-.$
Given a knot $K$,  the following assertions are equivalent:
\begin{enumerate}
\item $K$ bounds a $\Z$-disk in $D^4$ with $c_+$ positive double points and $c_-$ negative double points;
\item $\Bl_{P_{K,0}(c_+,c_-)}$ is presented by a size $c$ nondegenerate hermitian matrix $B$ with~$B(1)=(0)^{\oplus c}$;
\item $\Bl_K$ is presented by a size $c$ nondegenerate hermitian matrix $A$ with~$A(1)=(1)^{\oplus c_+} \oplus~(-1)^{\oplus c_-}$;
\item $K$ can be converted into an Alexander polynomial one knot via $c_+$ positive-to-negative crossing changes and $c_-$ negative-to-positive crossing changes;
\item $K$ bounds an embedded $\Z$-disk in $(\#_{c_+}  \C P^2 \#_{c_-}  \overline{\C P^2})\setminus \intt(B^4)$.
\end{enumerate}
\end{customthm}
\begin{proof}
The equivalence $(1) \Leftrightarrow (2)$ follows from Theorem~\ref{thm:SurfacesRelBoundary}.
The equivalence $(3) \Leftrightarrow (4)$ is due to Borodzik-Friedl~\cite{BorodzikFriedlLinking}.
%%Don't delete.
%Technically, BF has A(1) congruent to that thing, but since it's an existence statement, both statements are equivalent.
% take A with A(1) congruent to I, then modify A with the same matrix used to make A(1) congruent to A, and you have A’ with A’(1) equal to I.
The implication $(4) \Rightarrow (5)$ is well known (see e.g.~\cite[Theorem~1.1]{ConwayDaiMiller} for the argument when $c=1$).
The implication $(5) \Rightarrow (3)$ follows because the equivariant intersection form of the exterior of a $\Z$-disk in $\#_{c_+}  \C P^2 \#_{c_-}  \overline{\C P^2}$ presents $\Bl_{P_{K,0}(0,0)}=\Bl_K$ and its evaluation at $1$ is $Q_{\#_{c_+}  \C P^2 \#_{c_-}  \overline{\C P^2}}=(1)^{\oplus c_+} \oplus (-1)^{\oplus c_-}.$
Thus $(3) \Leftrightarrow (4) \Leftrightarrow (5)$.
For the implication~$(4) \Rightarrow (1)$,  the immersed~$\Z$-disk in $D^4$ is obtained by using a $\Z$-disk to cap off the trace of the homotopy arising from the crossing changes and arguing that the resulting disk has knot group~$\Z$; see e.g.~\cite[Proof of Lemma 6.2]{ChaOrrPowell} (alternatively one could show $(3) \Rightarrow (2)$ algebraically).
%%Don't delete.
%%Technically speaking their argument is smooth but it probably shows that the union of the exterior of the trace with the exterior of the Freedman-$\Z$-disk has $\pi_1=\Z$. 
%%Yeah, this seems good to me. } 
%% To confirm: if $K'=\partial D'$ is the Alexander polynomial one knot,  then the disk exterior  for $K$ can be built from  $N_{D'}$ by adding $0$-framed $2$-handles along these unknots and the outcome will have $\pi_1=\Z$?}}
%%Concretly,  draw a relative handle decomposition for the immersed disk exterior by taking the knot $K$ putting the 0-framed circles around the crossings and putting <0> on the knot.
Finally,  the implication $(2) \Rightarrow (3)$ is the content of Proposition~\ref{prop:BlKBlPK} below.
\end{proof}

It remains to prove the following.
\begin{proposition}
\label{prop:BlKBlPK}
Let $K$ be a knot, let $c_+,c_-\geq 0$ be integers and set $c=c_++c_-$.
If $\Bl_{P_{K,0}(c_+,c_-)}$ is presented by a nondegenerate hermitian matrix $B$ with $B(1)=0^{\oplus c}$,  then $\Bl_K$ is presented by a hermitian matrix $A$ with~$A(1) = (1)^{\oplus c_+} \oplus (-1)^{\oplus c_-}$. 
\end{proposition}

\begin{notation}
Let $K$ be a knot, let $c_+,c_-\geq 0$ be integers and set $c=c_++c_-$.
In what follows,  we set $P_K:=P_{K,0}(c_+,c_-)$ and fix the following notation for the remainder of this section.
\begin{itemize}
\item The boundary linking form of a size $c$ nondegenerate hermitian form $C$ over $\Z[t^{\pm 1}]$ is
\begin{align*}
\ell_C \colon \Z[t^{\pm 1}]^c/{C^T} \times \Z[t^{\pm 1}]/{C^T} &\to \Q(t)/\Z[t^{\pm 1}] \\
(x,y)& \mapsto x^TC^{-1}\overline{y}.
\end{align*}
\item Set $z:=(t-1)(t^{-1}-1).$
\item Let~$\Delta_K$ be the symmetric representive of the Alexander polynomial with $\Delta_K(1)=1$,  and write $\Delta_K=qz+1$ for $q \in \Z[t^{\pm 1}]$ a symmetric polynomial as in Example~\ref{ex:Coprime}. 
\item Let $\iota \colon \Bl_K \oplus \Bl_P  \xrightarrow{\cong}  \Bl_{P_K}$ be the inclusion induced isometry from Proposition~\ref{prop:BlPK}.
\item Fix a presentation matrix~$B$ for $\Bl_{P_K}$ that satisfies $B(1)=(0)^{\oplus c}$ and an isometry
\[\delta \colon (H_1(P_K;\Z[t^{\pm 1}]),{\unaryminus\Bl_{P_K}}) \to  (\Z[t^{\pm 1}]^c/{B^T}, \ell_B).\]
\end{itemize}
\end{notation}

We extract from $B$ what will be our presentation matrix for $\Bl_K$.
\begin{lemma}
\label{lem:B=zA}
If $\Bl_{P_K}$ is presented by a nondegenerate hermitian matrix $B$ with $B(1)=0^{\oplus c}$, then~$B=zA$ for some nondegenerate hermitian matrix $A$.
\end{lemma}
 \begin{proof}
Let $\vec{e}_1,\dots,\vec{e}_c$ be the standard basis vectors for $\Q[t^{\pm1}]^c$,  so 
$H_1(P_K; \Z[t^{\pm 1}]) \otimes \Q $ is generated by their images ${e}_1,\dots,{e}_c$ in the quotient ${\Q[t^{\pm1}]^c/B^T}.$
%of~$\Q[t^{\pm1}]^c$ induced by the rows of $B$.   
Suppose for a contradiction that some entry of $B$ is not divisible by $z$ over $\Q[t^{\pm1}]$; without loss of generality, say it is the
 $(i,1)$-entry.  
%%AC: We have 0=B^Tf_i picks out the i-th column of B^T, i.e.  the i-th row of $B.$
The $i$-th column of $B^T$
% $i$-th row of $B$ 
 imposes the relation
\begin{align*}
\left(zp+ (t-1)a \right){e}_1= \sum_{j=2}^n p_j {e}_j,
\end{align*}
for some $p, p_j \in \Q[t^{\pm 1}]$ and $a$ a nonzero element of $\Q$.  
Here we use that since $B(1)$ is the zero matrix,  every entry of $B$ and in particular the $(i,1)$-entry is divisible by $(t-1)$.

Multiplying this relation by~$\Delta_K$,  and using that every element of  $H_1(P_K; \Z[t^{\pm 1}]) \otimes \Q$ is annihilated by $\Delta_Kz$, we obtain the following equality in~$H_1(P_K; \Z[t^{\pm 1}]) \otimes \Q$:
\begin{align*}
(t-1)\Delta_K {e}_1 = \frac{-p}{a} \Delta_Kz {e}_1
+ \sum_{j=2}^c \frac{1}{a} \Delta_K p_j {e}_j
= \sum_{j=2}^c \frac{1}{a} \Delta_K p_j {e}_j.
\end{align*}
We deduce that ${e_2}, \dots, {e_c}$ generate $(t-1)\Delta_K\left(H_1(P_K; \Z[t^{\pm 1}]) \otimes \Q\right).$

Using successively Proposition~\ref{prop:HomologyPKZZ},  the fact that $\Delta_K$ annihilates $H_1(E_K; \Q[t^{\pm1})$ and the equality~$\Delta_K= zq+1$,  we deduce that ${e_2}, \dots, {e_c}$ generate 
\begin{align*}
(t-1)\Delta_K\left(H_1(P_K; \Z[t^{\pm 1}]) \otimes \Q\right) & \cong 
(t-1)\Delta_K \left( (\Q[t^{\pm1}]/z)^c \oplus H_1(E_K; \Q[t^{\pm1}) \right) \\
& \cong (t-1)\Delta_K(\Q[t^{\pm1}]/z)^c \\
& \cong(t-1)(\Q[t^{\pm1}]/z)^c \\
& \cong (\Q[t^{\pm1}]/(t^{-1}-1))^c. 
\end{align*}
This contradiction shows that that every entry of $B$ must be divisible by $z$ over $\Q[t^{\pm1}]$.  
However, since $z=(t-1)(t^{-1}-1)$ is monic,  the result holds over $\Z[t^{\pm1}]$ as well. 
\end{proof}

Next, we determine the image of the restriction of~$\delta \circ \iota\colon \unaryminus\left(\Bl_K \oplus \Bl_P\right) \xrightarrow{\cong} \ell_B$ to $H_1(E_K;\Z[t^{\pm 1}])$.

\begin{lemma}
\label{lem:Restrictphiiota}
The isometry~$\delta \circ \iota \colon \Bl_K \oplus \Bl_P \xrightarrow{\cong} \ell_B$ restricts to an isometry
\[ (\delta \circ \iota)|  \colon (H_1(E_K;\Z[t^{\pm 1}]),{\unaryminus\Bl_K}) \to \left( z \cdot \Z[t^{\pm 1}]^c/{B^T},  \ell_B|\right)\]
\end{lemma}
\begin{proof}
%%Don't delete: Because $\delta \circ \iota$ is an isometry...but also we say so again after the proof of the claim.
It suffices to prove that $\delta \circ \iota(H_1(E_K;\Z[t^{\pm 1}]))=z \cdot \Z[t^{\pm 1}]^c/ {B^T}$.
We will in fact show more.
\begin{claim*}
\label{Claim:EquivzLambda/B}
The following equalities hold:
$$\delta \circ \iota(H_1(E_K;\Z[t^{\pm 1}]))\stackrel{(1)}{=}\{ x \in \Z[t^{\pm 1}]/B \mid \Delta_Kx=0\}\stackrel{(2)}{=}z \cdot \Z[t^{\pm 1}]/B^T.$$
\end{claim*}
\begin{proof}
The $\subset$ inclusion of $(1)$ is immediate,  because $\Delta_K$ annihilates $H_1(E_K;\Z[t^{\pm 1}])$,  as is the $\supset$ inclusion of $(2)$, because $\Z[t^{\pm 1}]^c/{B^T} \cong H_1(P_K;\Z[t^{\pm 1}])$ is annihilated by $z\Delta_K$. 
We will therefore show that if $x \in \Z[t^{\pm 1}]^c/{B^T}$ satisfies $\Delta_Kx=0$, then it belongs to the two other sets.

Since $\delta \circ \iota$ is surjective,  we can write $x= (\delta \circ \iota)(x_K, x_P)$ for some $x_K \in H_1(E_K;\Z[t^{\pm 1}])$ and~$x_P \in H_1(P;\Z[t^{\pm 1}])$. 
Since $\delta \circ \iota$ is injective,  $0=\Delta_Kx_P$.
Using that~$z$ annihilates $H_1(P;\Z[t^{\pm 1}])$, we deduce that
\[x_P=\overbrace{(\Delta_K-qz)}^{=1} x_P=\Delta_Kx_P=0.\]
Similarly,  since $\Delta_K$ annihilates $H_1(E_K;\Z[t^{\pm 1}])$,  we have 
\[ x_K= \overbrace{(\Delta_K- qz)}^{=1} x_K=  \unaryminus qz x_K.\]
Putting these two equalities together, we obtain
\[x= (\delta \circ \iota)(x_K,x_P)= z(\delta \circ \iota)(-q x_K,0).\]
We therefore observe that $x$ is in both $z\cdot \Z[t^{\pm 1}]^c/{B^T}$ and $ (\delta \circ \iota)(H_1(E_K;\Z[t^{\pm 1}]))$ and are done. 
\end{proof}
Since $(\delta \circ \iota)$ is an isometry, so is its restriction to~$H_1(E_K;\Z[t^{\pm 1}])$.
Claim 1 implies that the image of this isometry is~$ \left( z \cdot \Z[t^{\pm 1}]^c/{B^T},  \ell_B|\right)$.
\end{proof}

The next lemma determines the isometry type of~$\left( z \cdot \Z[t^{\pm 1}]^c/{B^T},  \ell_B|\right)$ in terms of the matrix~$A$ from Lemma~\ref{lem:B=zA} satisfying~$B=zA$.

\begin{lemma}
Multiplication by $-q(t^{-1}-1)z$ induces an isometry
\begin{align*}
\psi \colon \left(\Z[t^{\pm 1}]^c/{A^T},\ell_A\right) \xrightarrow{\cong} \left(z \cdot \Z[t^{\pm 1}]^c/{B^T},  \ell_B|\right).
\end{align*}
\end{lemma}
\begin{proof}
The equality $B=zA$ ensures that~$\psi$ is well defined, and~$\psi$ is form preserving because
\begin{align*}
\ell_B(\psi(x),\psi(y))
=q^2z^2z\ell_B(x,y)
=q^2z^2z\left(\frac{1}{z}\ell_A(x,y)\right)
=(\Delta_K-1)^2\ell_A(x,y)
=\ell_A(x,y).
\end{align*}
Here we used that $\Delta_K$ annihilates $\Z[t^{\pm 1}]^c/A^T$: indeed we have~$\det(A)=\Delta_K$ as can be deduced from the equalities $z^c\Delta_K=\det(B)=\det(zA)=z^c\det(A)$.

Since $\ell_A$ is nondegenerate and $\psi$ is form-preserving, we deduce that~$\psi$ is injective.
For surjectivity, note that $zx \in \Z[t^{\pm 1}]^c/{B^T}$ is annihilated by $\Delta_K$ so 
$$zx=(1-\Delta_K)zx=-qz^2x=\psi((t-1)x).\qedhere$$
\end{proof}

The combination of Lemmas~\ref{lem:B=zA} and~\ref{lem:Restrictphiiota} shows that if $\Bl_{P_K}$ is presented by a hermitian matrix~$B$ with~$B(1)=0^{\oplus c}$, then $B=zA$ and $A$ presents~$\Bl_K$.
Only one step remains.

\begin{lemma}
\label{lem:A(1)}
If $\Bl_{P_K}$ is presented by $B=zA$,
then $A(1)$ is congruent to $(1)^{\oplus c_+} \oplus (-1)^{\oplus c_-}$.
%$$A(1) \equiv I_{c_+,c-}.$$
 \end{lemma}
\begin{proof}
An argument exactly analogous to that of Claim~\ref{Claim:EquivzLambda/B} shows that
\[ x \in \Delta_K \cdot \Z[t^{\pm 1}]^c/{B^T} \iff z x =0 \iff x \in (\delta \circ \iota)(H_1(P;\Z[t^{\pm 1}])).\]
In particular $\delta \circ \iota \colon \unaryminus(\Bl_K \oplus \Bl_P) \to \ell_B$ restricts to an isometry $\unaryminus\Bl_P \to \ell_B|_{\Delta_K \cdot \coker({B^T})}.$

Use $e_1,\dots,e_c$ to denote the images by the canonical projection of the standard basis of $\Z[t^{\pm 1}]^c$ to~$\Z[t^{\pm 1}]^c/{B^T}$.
Since the~$\Delta_K e_i$ generate~$\Delta_K\cdot \Z[t^{\pm 1}]^c/{B^T}$,  the equivalence above ensures that generators of $H_1(P;\Z[t^{\pm 1}])$ are
\[v_i:= (\delta \circ \iota)^{-1}(\Delta_K e_i) \in H_1(P;\Z[t^{\pm 1}]).\]
The strategy of the proof is as follows.
Calculating $\Bl_P(v_i,v_i)$ in two different ways will lead to the equation $A^{-1}(1)=NI_{c_+,c_-}N^T$ for some integral matrix $N$, where $I_{c_+,c_-}:=(1)^{\oplus c_+} \oplus (-1)^{\oplus c_-}.$
We will then show that $N$ is invertible which will imply that  $A^{-1}$ is congruent to~$I_{c_+,c_-}$,  and hence that~$A(1)$ is congruent to~$I_{c_+,c_-}.$

The first calculation of~$\Bl_P(v_i,v_i)$ involves the definition of the~$v_i$.
Since $A$ presents $H_1(E_K;\Z[t^{\pm 1}])$,   we have $\det(A)\doteq \Delta_K$.
The adjoint formula for a matrix inverse implies there are $\alpha_{i,j} \in \Z[t^{\pm 1}]$ such that $A_{i,j}^{-1}=\left(\frac{\alpha_{i,j}}{\Delta_K}\right).$
Using the isometry $(\delta \circ \iota)| \colon \unaryminus\Bl_P \xrightarrow{\cong} \ell_B|_{\Delta_K \cdot \coker({B^T})}$, we obtain
\begin{align}\label{eqn:blpofv}
\Bl_P(v_i,v_j)
%= \Bl_M(\Delta_K e_i, \Delta_K e_j)
= \Delta_K^2 B^{-1}_{i,j}
= \Delta_K^2 \frac{1}{z}A^{-1}_{i,j} 
= \Delta_K^2 \frac{1}{z}\frac{\alpha_{i,j}}{\Delta_K}
= \frac{ \Delta_K(qz+1)\alpha_{i,j}}{z\Delta_K}
%\nonumber 
= \frac{\alpha_{i,j}}{z}.
\end{align}
The second calculation of~$\Bl_P(v_i,v_i)$ uses that $\Bl_P$ is presented by $z I_{c_+,c-}$, so that there are generators $u_1,\dots,u_c$ of $H_1(P;\Z[t^{\pm 1}])$ that satisfy
\[ \unaryminus\Bl_P(u_i,u_j)= 
\begin{cases} 
0 \quad & \text{if } i \neq j \\
 \frac{ \varepsilon_i}{z} \quad & \text{if } i=j. \end{cases}
\]
Since the $u_i$ generate $H_1(P;\Z[t^{\pm 1}])$,  there are $m_{i,k} \in  \Z$ and $n_{i,k} \in \Z$ such that
\begin{align}\label{eqn:usgeneratevs}
v_i = \sum_{k=1}^c (m_{i,k}(t-1)+n_{i,k})u_k.
\end{align}
%%%Necessary of this form because $(t-1)^2$ annihilates.
Using Equations~\eqref{eqn:blpofv} and~\eqref{eqn:usgeneratevs}, we obtain that
\begin{align*}
\frac{\alpha_{i,j}}{z}= \unaryminus\Bl_P(v_i,v_j)
= \frac{1}{z} \sum_{k=1}^c (m_{i,k}(t-1)+n_{i,k})((m_{j,k}(t^{-1}-1)+n_{j,k})\varepsilon_k.
\end{align*}
Expanding and consolidating terms, we obtain that there are $f_{i,j} \in \Z[t^{\pm 1}]$ such that 
\begin{align*}
\alpha_{i,j}= zf_{i,j}(t)+(t-1) \left(\sum_{k=1}^c (m_{i,k}n_{j,k}-n_{i,k}m_{j,k}) \right) + \sum_{k=1}^c n_{i,k} n_{j,k} \varepsilon_k. 
\end{align*}
Evaluating this equation at $t=1$,  we obtain
\begin{align*}
\alpha_{i,j}(1)= \sum_{k=1}^c n_{i,k} n_{j,k} \varepsilon_k.
\end{align*}
Now recall that $A^{-1}=(\frac{\alpha_{i,j}}{\Delta_K})$.
Since $\Delta_K(1)=1$, we deduce that $A^{-1}(1)=( \alpha_{i,j}(1)).$
By considering~$n_{i,j}$ as a matrix over $\Z$,  we have therefore obtained the following equality in $M_c(\Z[t^{\pm 1}])$:
\begin{align*}
\begin{pmatrix} \alpha_{i,j}(1) \end{pmatrix}= \begin{pmatrix} n_{i,j} \end{pmatrix} I_{c_+,c_-} \begin{pmatrix} n_{i,j} \end{pmatrix}^T. 
\end{align*}
It therefore remains only to show that the matrix~$\begin{pmatrix} n_{i,j} \end{pmatrix}$ is invertible. 

Using that the $v_i$ generate $H_1(P;\Z[t^{\pm 1}])$, we can write the $u_i \in H_1(P;\Z[t^{\pm 1}])$ as
\begin{align*}
u_i= \sum_{k=1}^c (m'_{i,k}(t-1)+n'_{i,k}) v_k. 
\end{align*}
for some~$m'_{i,k} \in \Z$ and $n'_{i,k} \in \Z$.
Using~\eqref{eqn:usgeneratevs} and~$(t-1)^2 \cdot H_1(P;\Z[t^{\pm 1}])=0$,
% the fact that~$(t-1)^2$ annihilates~$H_1(P;\Z[t^{\pm 1}])$, 
we obtain
\begin{align*}
u_i&= \sum_{\ell=1}^c \sum_{k=1}^c (m'_{i, \ell}(t-1) + n'_{i,\ell})(m_{\ell,k}(t-1)+n_{\ell, k}) u_k \\
&=\sum_{k=1}^c \left( (t-1) \sum_{\ell=1}^c(m'_{i,\ell} n_{\ell, k}+ n'_{i, \ell} m_{\ell,k}) + \sum_{\ell=1}^c n'_{i,\ell} n_{\ell,k} \right)u_k.
\end{align*}
It follows that 
\[ \sum_{\ell=1}^c n'_{i,\ell}n_{\ell,k}=
 \begin{cases}
  1  \quad & \text{if }i=k \\ 
  0 \quad & \text{otherwise.}
 \end{cases}\]
We have therefore established that $\begin{pmatrix} n_{i,j} \end{pmatrix}$ and $\begin{pmatrix} n'_{i,j} \end{pmatrix}$ are inverses. 
As explained above,  this implies that $A^{-1}(1)$ is congruent to~$I_{c_+,c_-}$ and therefore that $A(1)$ is congruent to~$I_{c_+,c_-}.$
%AC: As a reader (or as myself in the future, I like being reminded of why we're done).
%But if it is really annoying, we can comment it!!
\end{proof}

We now prove Proposition~\ref{prop:BlKBlPK} whose statement is that if $\Bl_{P_{K,0}(c_+,c_-)}$ is presented by a nondegenerate hermitian matrix $B$ with $B(1)=0^{\oplus c}$,  then $\Bl_K$ is presented by a nondegenerate hermitian matrix $A$ such that $A(1) = I_{c_+,c_-}$,  where $I_{c_+,c_-}:=(1)^{\oplus c_+} \oplus (-1)^{\oplus c_-}$.
\begin{proof}[Proof of Proposition~\ref{prop:BlKBlPK}]
The combination of Lemmas~\ref{lem:B=zA} and~\ref{lem:Restrictphiiota} shows that if $\Bl_{P_K}$ is presented by a nondegenerate hermitian matrix~$B$ with~$B(1)=0^{\oplus c}$, then $B=zA$ with $A$ a nondegenerate hermitian form that presents~$\Bl_K$.
Lemma~\ref{lem:A(1)} ensures that~$A(1)$ is congruent to~$I_{c_+,c_-}$, say $A(1)=NI_{c_+,c_-}N^T$.
Now taking $A'=N^{-1}A(N^{-1})^T$ gives us our desired presentation matrix for $\Bl_K$ with $A'(1)=I(c_+,c_-)$. 
%AC: Note to self: It's important that it's an integral matrix, so the change of basis is automatically unitary.
\end{proof}

\section{Proof strategy of the classification results}
\label{sec:ProofStrategy}

Sections~\ref{sec:ImmersionsAndNormalBundles}-\ref{sec:ProofClassifications} are concerned with the proof of the classifications stated in Section~\ref{sec:Classification}.
Since the proofs are lengthy, we devote this section to outlining their main ideas.
We focus on the rel. boundary classification as the proof in the closed case follows a similar strategy.

\medbreak

Fix a simply-connected~$4$-manifold $N$ with boundary~$\partial N \cong S^3$,  a knot~$K \subset \partial N$,  and non-negative integers $c_+,c_-$ and $g$.
Let $c=c_++c_-$ and~$P_K=P_{K,g}(c_+,c_-).$
The difficult part of Theorem~\ref{thm:SurfacesRelBoundary} consists of showing that if a nondegenerate hermitian form~$\lambda$ over~$\Z[t^{\pm 1}]$ presents~$\Bl_{P_K}$ and satisfies~$\lambda(1)\cong Q_N \oplus  (0)^{\oplus 2g+c}$, then the set~$\operatorname{Surf}_\lambda^0(g;c_+,c_-)(N,K)$ is nonempty and there is a bijection
$$\operatorname{Surf}_\lambda^0(g;c_+,c_-)(N,K) \approx \frac{\Iso(\partial \lambda,\unaryminus\Bl_{P_K})}{\Aut(\lambda)\times \Homeo_\alpha(\Sigma_{g,1},\partial)}.$$
For simplicity,  we focus on the case where $\lambda$ is even: this avoids discussing some additional complications related to the Kirby-Siebenmann invariant.
Additionally, for brevity, we focus on the strategy involved in establishing the bijection, leaving aside the nonemptiness question.
What follows are the steps necessary to establish this bijection; we are slightly informal, delaying a detailed discussion to the next four sections.

\begin{itemize}
\item First we prove that the set of surfaces $\operatorname{Surf}_\lambda^0(g;c_+,c_-)(N,K)$ corresponds bijectively to an analogously defined set of immersions, which we denote~$\Imm(g;c_+,c_-)^0_\lambda(N,K)$, modulo the action of $\Homeo_\alpha(\Sigma_{g,1},\partial)$ by precomposition.
Mapping an immersion to its image is seen without too much difficulty to induce a bijection 
$$\Imm(g;c_+,c_-)^0_\lambda(N,K)/\Homeo(\Sigma_{g,1},\partial) \to \operatorname{Surf}_\lambda^0(g;c_+,c_-)(N,K).$$
Working with $\Homeo_\alpha(\Sigma_{g,1},\partial)$ instead of $\Homeo(\Sigma_{g,1},\partial)$ requires some additional effort.
This is the content of Section~\ref{sec:ImmersionsAndNormalBundles}.
\item Next we prove that mapping an immersion $\Imm(g;c_+,c_-)^0_\lambda(N,K)$ to the exterior of its image defines a bijection $\Theta$ to a set $\mathcal{V}^0_\lambda(P_K)$ which consists, roughly speaking, of homeomorphism classes of pairs $(V,f)$, where $V$ is a $4$-manifold with $\pi_1(V) \cong \Z$ and equivariant intersection form isometric to $\lambda$,  and $f \colon \partial V \to P_K$ is an orientation-preserving homeomorphism.
%not saying ribbon boundary for simplicity.
Given an immersion $e \colon \Sigma_{g,1} \looparrowright N$, defining a homeomorphism 
$$f \colon \partial (N\setminus \intt(\overline{\nu}(e(\Sigma_{g,1}))) \to P_K$$
 is arduous as is verifying that $\Theta$ is indeed a bijection.
 The former is the topic of Section~\ref{sec:BoundarIdentifications}, whereas the latter is the topic of Section~\ref{sec:MainStatement} and is by far the longest part of the proof.
Modding out by the $\Homeo_\alpha(\Sigma_{g,1},\partial)$-actions, then defines a bijection from~$\operatorname{Surf}_\lambda^0(g;c_+,c_-)(N,K)$ to $\mathcal{V}^0_\lambda(P_K)/\Homeo_\alpha(\Sigma_{g,1},\partial)$.
\item The last step essentially consists in verifying that a bijection 
$$\mathcal{V}^0_\lambda(P_K) \to \Iso(\partial \lambda,\unaryminus\Bl_{P_K})/(\Aut(\lambda)$$ from~\cite{ConwayPiccirilloPowell} respects the $\Homeo_\alpha(\Sigma_{g,1},\partial)$-actions. 
This is the topic of Section~\ref{sec:ProofClassifications}.
\end{itemize}

The nature of the bijection we are going to construct requires we pay close attention to all the choices involved in each of our constructions.
For this reason, the level of detail in the second part of this paper will be significantly higher than in the first, to a level that might at first seem overly fastidious.
For example, it will be important that we record the choices made to define the exterior of a surface and this is the reason we devote some time to fixing our conventions concerning ``standard" notions such as immersions and vector bundles.

\section{Immersions and their normal bundles}
\label{sec:ImmersionsAndNormalBundles}

The first goal of this section is to recall the definitions of immersions (Section~\ref{sub:immersions}) and normal bundles (Section~\ref{sub:NormalBundles}) at the level of detail needed to construct the bijection $\Theta$ mentioned in the outline of the previous section.
%This is the topic of Section~\ref{sub:immersions} and~\ref{sub:NormalBundles}.
Additionally,  in Section~\ref{sub:SurfBijectionSetUp}, we identify the set of immersed surfaces with a quotient 
%an orbit set AC: Changed my mind: the word ``set" appears to many times.
of a set of immersions.

\subsection{Immersions}
\label{sub:immersions}

This section recalls the definition of a generic immersion and introduces various sets of immersed surfaces that are relevant to the proofs of our classification results.

\begin{definition}\label{def:DoublePointCharts}
An \emph{immersion} is a locally flat map $e \colon \Sigma \looparrowright N$ that is locally an embedding.
An immersion $e \colon \Sigma \looparrowright N$ is \emph{generic} if its singular set is a closed, discrete subset of~$N$ consisting only of transverse double points each of whose preimages lies in the interior of~$\Sigma$.
\end{definition}

Definition~\ref{def:DoublePointCharts} encompasses both the case where $\Sigma$ is closed and the case where it has a boundary.
We refer to~\cite[Definition 2.1]{KasprowskiPowellRayTeichnerEmbedding} for further details on how to precisely make sense of $e$ being locally an embedding when the boundary of $\Sigma$ or $N$ is involved.

\begin{convention}
In what follows, all immersions are assumed to be generic,~$\Sigma$ denotes an oriented genus~$g$ surface with a single boundary component, $N$ denotes a simply-connected~$4$-manifold with boundary homeomorphic to~$S^3$, and~$K \subset \partial N$ denotes a knot.
\end{convention}

The following definition is inspired by~\cite[Definition 2.3]{KasprowskiPowellRayTeichnerEmbedding}.

\begin{definition}
\label{def:Compatible}
%Hard to fix the line break :(
An immersion $e \colon \Sigma \looparrowright N$ with double points~$e(x_{2i-1})=y_i=e(x_{2i})$ for~$i=1,\ldots,c$ is \emph{compatible} with an embedding~$\alpha=(\alpha_k)_{k=1}^{2c}  \colon \bigsqcup_{k=1}^{2c} D^2 \hookrightarrow \Sigma$,  if for~$i=1,\ldots,c$, there exists a neighborhood~$y_i \in U_i \subset N$,  and a homeomorphism~$\psi_i \colon U_i \to D^2 \times D^2$  such that
\begin{enumerate}
\item \label{item:Compatible1} the $U_i$ are pairwise disjoint,
\item \label{item:Compatible2}  the neighborhood $U_i$ satisfies
$$e(\Sigma) \cap U_i= e(\alpha_{2i-1}(D^2)) \cup e(\alpha_{2i}(D^2)),$$
\item \label{item:Compatible3} the composition $\psi_i \circ e \circ \alpha_k$ equals 
\begin{align*}
\begin{cases} \id \colon D^2 \to D^2 \times \{0\} & \text{ if } k=2i-1 \\ 
\pm \id \colon D^2 \to \{0\} \times D^2 & \text{ if }  k=2i, \end{cases}
\end{align*}
where the $\pm$-sign is according to the sign of $y_i$,  the $i$-th self-intersection of $e(\Sigma)$. 
\end{enumerate}
We call $\mathcal{U}=(U_i, \psi_i)_{i=1}^{c}$ \emph{a family of double point charts for $(e,\alpha)$.}
\end{definition}

\begin{remark}
Here is some motivation and some observations concerning this definition.
\begin{itemize}
\item The reason for this somewhat cumbersome definition is that we need to control the behavior of our immersions near double points.  In the absence of transversality results coming from smoothness,  we instead require that near double points are immersions are `standard' in the sense described in Item~\ref{item:Compatible3} of Definition~\ref{def:Compatible}. 
\item Every generic immersion is compatible with \emph{some} $\alpha$.
The reason for introducing the terminology is that, in the next sections, we will fix $\alpha$ once and for all and require our immersions to be compatible with that $\alpha$. 
\item If $\mathcal{U}$ is a family of double point charts for both $(e,\alpha)$ and $(e,  \alpha')$,   then $\alpha=\alpha'$, since condition~\eqref{item:Compatible3} can be rewritten to express~$\alpha_k$ in terms of $\psi_i$ and $e$.  On the other hand,  given~$(e,\alpha)$ there will generally be many choices for $\mathcal{U}$. 
\end{itemize}
\end{remark}

We introduce some further terminology related to immersions.
\begin{definition}
~
\begin{itemize}
\item Two immersions~$e,  e' \colon \Sigma \looparrowright N$ with $e(\partial \Sigma)=e'(\partial \Sigma)$ are \emph{equivalent} if there is a homeomorphism $F \colon N \to N$ that satisfies $F \circ e=e'$ and \emph{equivalent rel. boundary} if $F$ is additionally required to be the identity map on $\partial N$. 
\item An \emph{immersed surface} refers to the image of an immersion of a surface in a $4$-manifold.
\item Two immersed surfaces $S,S'\subset N$ with $\partial S= \partial S'$ are \emph{equivalent} if there is a homeomorphism $F \colon N \to N$ that satisfies $F(S)=S'$ and \emph{equivalent rel. boundary} if $F$ is additionally required to be the identity map on $\partial N$. 
\item An \emph{immersed $\Z$-surface} is an immersed surface $S \subset N$ with $\pi_1(N \setminus S) \cong \Z$.
A~\emph{$\Z$-immersion} is an immersion $e \colon \Sigma \looparrowright N$ whose image is a $\Z$-surface.
\item An immersed surface $S$ \emph{has $(c_+,c_-)$-double points} if it has $c_+$ positive double points and~$c_-$ negative double points.
An immersion \emph{has $(c_+,c_-)$-double points} if its image has \emph{has $(c_+,c_-)$-double points}.
\end{itemize}
\end{definition}

Next we introduce sets of immersed $\Z$-surfaces with prescribed genus and intersection data.
\begin{notation}
\label{not:Imm0Surf0}
Given integers $g,c_+,c_- \geq 0$, we use the following notation:
\begin{align*}
\operatorname{Imm}^0(g;c_+,c_-)(K,N)&=
\frac{
\{ \Z\text{-immersion } e \colon \Sigma \looparrowright N \mid g(\Sigma)=g,  \  \partial e(\Sigma)=K  \text{ and } (c_+,c_-)\text{-double points}\}
}
{\text{equivalence rel. boundary}}, \\
%%%%
\operatorname{Surf}^0(g;c_+,c_-)(K,N)&=
\frac{
\{ \Z\text{-surface } S \subset N \mid g(S)=g, \partial S=K  \text{ and } (c_+,c_-)\text{-double points}\}
}
{\text{equivalence rel. boundary}}.
\end{align*}
Set $c:=c_++c_-$.
Given an embedding $\alpha\colon  \bigsqcup_{i=1}^{2c} D^2 \hookrightarrow \operatorname{int}(\Sigma)$, we note that compatibility with~$\alpha$ is preserved by equivalence rel. boundary and set
$$
\operatorname{Imm}_\alpha(g;c_+,c_-)^0(K,N):=\{ [e] \in \operatorname{Imm}_\alpha(g;c_+,c_-)^0(K,N)] \mid e  \text{ is compatible with $\alpha$}\}.
$$
Finally, note that the assignment~$\theta \cdot e := e \circ \theta^{-1}$ defines left actions 
\begin{align*}
&\Homeo(\Sigma,\partial) \curvearrowright \operatorname{Imm}(g;c_+,c_-)^0(K,N), \\
&\Homeo_\alpha(\Sigma,\partial) \curvearrowright \operatorname{Imm}_\alpha(g;c_+,c_-)^0(K,N).
\end{align*}
\end{notation}
Finally,  we note that the definitions in this section can be adapted to immersions of closed surfaces in closed $4$-manifolds.

\subsection{Relating sets of immersions and sets of immersed surfaces}
\label{sub:SurfBijectionSetUp}

The goal of this section is to relate the set $\operatorname{Surf}^0(g;c_+,c_-)(K,N)$ of $\Z$-surfaces with the set $\operatorname{Imm}^0(g;c_+,c_-)(K,N)$ of $\Z$-immersions.
Later on however we will want to associate to each immersed surface $S$ its exterior~$N_{S}$ and a homeomorphism $\partial N_S\to P_K$.
%In order to do this,  we will need 
This will require working with the set~$\operatorname{Imm}^0_\alpha(g;c_+,c_-)(K,N)$ of~$\Z$-immersions that are compatible with $\alpha$.

\begin{notation*}
We write $\Sigma:=\Sigma_{g,1}$ for the genus $g$ surface with a single boundary component.
\end{notation*}

The following lemma serves as a first step in relating $\operatorname{Surf}^0(g;c_+,c_-)(K,N)$ to $\operatorname{Imm}^0(g;c_+,c_-)(K,N)$.

\begin{lemma}
\label{lem:EquivalenceImmersionvsSurface}
Let~$e, e' \colon \Sigma \looparrowright N$ be immersions, and suppose 
$F \colon N \to N$  is a rel.~boundary homeomorphism
with the property that~$F(e(\Sigma))= e'(\Sigma)$. 
Then there is a rel. boundary homeomorphism~$\theta \colon \Sigma \to \Sigma$ with the property that~$F \circ e = e' \circ \theta$.
\end{lemma}
\begin{proof}
Let~$x_1, \dots, x_{2c} \in \Sigma$ be the double points of~$e$,  labelled so that~$e(x_{2i-1})= e(x_{2i})$ for each~$i=1,\ldots,c$.  
Let~$U_1, \dots, U_c$ be disjoint balls in~$\operatorname{int}(N)$ containing~$e(x_1), \dots e(x_{2c})$,  chosen to be small enough that~$e^{-1}(U_i)  \subset \Sigma$ consists of two disjoint disks:
~$B_{2i-1}$ containing~$x_{2i-1}$ and~$B_{2i}$ containing~$x_{2i}$.
Removing these disks, we set
$$\Sigma_0:= \Sigma \setminus \intt \bigcup_{i=1}^c  e^{-1}(U_i) \quad \text{ and } \quad \Sigma_0':=  \Sigma \setminus \intt \bigcup_{i=1}^c (e')^{-1}(F (U_i)).$$
Since~$F$ sends the double points of~$e(\Sigma)$ to the double points of~$e'(\Sigma)$, we obtain a homeomorphism
\[\theta_{\Sigma_0}:= (e'|_{\Sigma_0'})^{-1} \circ F  \circ (e|_{\Sigma_0}) \colon \Sigma_0 \to \Sigma_0'.\]
By definition of~$\theta_{\Sigma_0}$, it is certainly the case that~$F(e(s))= e'(\theta_{\Sigma_0}(s))$ for every~$s \in \Sigma_0$. 

We now wish to extend~$\theta_{\Sigma_0}$ across all of~$\Sigma$.  We will do so across one disk at a time.  Note that for each~$1 \leq k \leq 2c$,  the map~$\theta_{\Sigma_0}$ sends~$\partial B_k$ to some boundary curve of~$\Sigma_0'$. 
Since~$e'$ is one-to-one except on~$\bigcup_{i=1}^c (e')^{-1}(F(e(x_i)))$,  we have that~$ \bigcup_{i=1}^c (e')^{-1}(F(U_i))$ is a disjoint union of~$2c$ sub-disks of~$\Sigma$. 
The map~$\theta_{\Sigma_0}$ must send~$\partial B_k$ to the boundary of one of these disks--call it~$B_k'$.  
The map~$(e')|_{B_k'} \colon B_k' \to e'(B_k')= F(e(B_k))$ is a homeomorphism,  so we obtain a well-defined homeomorphism
\[ \theta_{B_k}:= (e'|_{B_k'})^{-1} \circ F \circ (e|_{B_k}) \colon B_k \to B_k'.\]
By construction~$\theta_{B_k}$ agrees with~$\theta_{\Sigma_0}$ on~$\partial B_k$,  and so we define the
%we obtain a well defined 
homeomorphism $\theta \colon \Sigma \to \Sigma$ as
\begin{align*}
\theta(x):=
 \begin{cases}
  \theta_{B_k}(x) & x \in B_k \text{ for } 1 \leq k \leq 2c\\ 
  \theta_{\Sigma_0}(x) & x \in \Sigma_0.
 \end{cases}
\end{align*}
By construction this homeomorphism satisfies~$F(e(s))= e'(\theta(s))$ for all~$s \in \Sigma$, as required.  We remark that, while it was convenient for us to choose balls $U_1, \dots, U_c$ and corresponding disks~$B_1, \dots, B_{2c}$ to verify continuity,  the map $\theta$ is independent of these choices,  since for any point~$x \notin \{x_1, \dots, x_{2c}\}$ we have that $\theta(x)$ is the unique point in $\Sigma$ mapped to $F(e(x))$ by $e'$. 
\end{proof}

\begin{proposition}
\label{prop:SurfQuotientOfImm}
Mapping an immersion to its image induces a bijection
$$\operatorname{Imm}^0(g;c_+,c_-)(K,N)/\Homeo(\Sigma,\partial) \to \operatorname{Surf}^0(g;c_+,c_-)(K,N).$$
\end{proposition}
\begin{proof}
Well-definedness and surjectivity are immediate, so
assume that $S=e(\Sigma)$ and $S'=e'(\Sigma)$ agree in $\operatorname{Surf}^0(g;c_+,c_-)(K,N)$,
meaning that there is a rel.\ boundary homeomorphism $F \colon N \to N$ such that $F \circ e =e'$.  
Lemma~\ref{lem:EquivalenceImmersionvsSurface} implies that there is a self-homeomorphism $\theta \colon \Sigma \to \Sigma$ such that~$F \circ e=e' \circ \theta$,  and so up to homeomorphism $e$ and $e' \circ \theta$ differ by the action of $\operatorname{Homeo}(\Sigma,\partial)$.
%, i.e. $e$ and $e'$ agree in the orbit set.
\end{proof}

The main result of this section is the following.

\begin{proposition}\label{prop:ImmIsImmalpha}
The inclusion of  $\operatorname{Imm}_{\alpha}(g;c_+,c_-)^0(K,N)$ in $
\operatorname{Imm}(g;c_+,c_-)^0(K,N)$
induces a bijection
$$\Phi \colon \operatorname{Imm}_{\alpha}(g;c_+,c_-)^0(K,N)/\Homeo_\alpha(\Sigma,\partial) \to
\operatorname{Imm}(g;c_+,c_-)^0(K,N)/\Homeo(\Sigma,\partial) 
$$
In particular, mapping an immersion to its image induces a bijection
$$ \operatorname{Imm}_{\alpha}(g;c_+,c_-)^0(K,N)/\Homeo_\alpha(\Sigma,\partial) \cong \operatorname{Surf}^0(g;c_+,c_-)(K,N). $$
%Consequenence of the previous proposition.
\end{proposition}

The proposition relies on the following lemma which, informally, states that for a given $\Z$-immersion~$e$, one can always find a homeomorphism preserving $e(\Sigma)$ that takes any one family of double point charts to any another.

\begin{lemma}
\label{lem:Homeo}
Let $e \colon \Sigma \looparrowright N$ be a $\Z$-immersion. 
Suppose that $\alpha, \alpha' \colon  \bigsqcup_{k=1}^{2c} D^2 \hookrightarrow \Sigma$ are compatible with $e$,  and  $\mathcal{U}:=\{ (U_i,\psi_i)\}$ and $\mathcal{U}':=\{ (U_i',\psi_i')\}$  are families of double point charts for $(e,\alpha)$ and $(e, \alpha')$ respectively.
%%Don't delete.
%{AC: Do we ever apply this for $\alpha \neq \alpha'$? I don't mind keeping the more general result, just trying to keep things straight. 
%Also, in the eventuality that we only use $\alpha=\alpha'$, does that still require the surgery theory result?
%AM: I think 'no' and 'yes', in that order. ..
%AC: Does the proof simplify enough for $\alpha=\alpha'$ that we should change something or should we leave it as is?
%%We discussed this. For $\alpha=\alpha'$ we have $C_i=C_i'$ as subsets of $N$ but the $X_i,X_i'$ remain different so there didn't appear to be a clear shortcut.
%}
Then there exists a rel. boundary homeomorphism $F \colon N \to N$ satisfying
$$ F (e(\Sigma)) =e(\Sigma) \ \ \ \text{ and } \ \ \ F|_{U_p}=(\psi_i')^{-1} \circ \psi_i \colon U_i \xrightarrow{\cong} U_i'.$$
\end{lemma}
%ANM: For the record,  I'm pretty sure $F$ is isotopic to the identity,  since it *is* the identity outside some balls. 
\begin{proof}
Note that given any $\alpha, \alpha'$ and $\mathcal{U}, \mathcal{U}'$ as above, there exists an embedding $\alpha'' \colon  \bigsqcup_{k=1}^{2c} D^2 \hookrightarrow \Sigma$ compatible with $e$ and  family of double point charts $\{ (U_i'', \psi_i'') \}_{i=1}^c$  for $(e, \alpha'')$ with~$U_i'' \subset U_i \cap U_i'$ for all~$i=1, \dots, c$. 
%%Don't delete.
% {Choose a metric $d$ on $N$.  For sufficiently small $\epsilon>0$,  the set
 %\[U^\epsilon_i:= \{x \in N: d(x, e(D_{2i-1}) \cup e(D_{2i}))< \epsilon\} \cap (U_i \cap U_i')\]
%satisfies our desired properties with respect to some homeomorphism $\psi_i^\epsilon \colon U^(\epsilon)_i \to D^2 \times D^2$ sending $e(D_{2i-1})$ to $D^2 \times \{0\}$ and $e(D_{2i})$ to $\{0\} \times D^2$. }
It therefore suffices to prove this lemma under the additional hypothesis that~$U_i' \subseteq U_i$, which we now do. 
%transitivity.

For each $i=1, \dots, c$,  consider a pair $(U_i^+,\psi_i^+)$, where $U_i^+$ is a small enlargement of $U_i$,  and~$\psi_i^+ \colon U_i^+ \to D^2_{1+\delta} \times D^2_{1+\delta}$ is a homeomorphism with $\psi_i^+|_{U_i}=\psi_i$ for some $\delta>0$ and such that 
$\psi_i^+(e(\Sigma) \cap U_i^+))=D^2_{1+\delta} \times \{0\} \cup \{0\} \times D^2_{1+\delta}$.  
Note that $\mathcal{U}^+$ is a family of double point charts for $(e,  \alpha^+)$,  where $\alpha^+$
is the unique embedding that satisfies Item~\ref{item:Compatible3} of Definition~\ref{def:Compatible} with respect to~$\psi_i^+$ and $e$.
%%{AC: Does this mean that $\alpha^+$ is the embedding determined by $\psi_i^+$ and $e$ via that item? AM: Yes exactly. }
We can (and will) arrange that the $U_i^+$ are pairwise disjoint.
This is illustrated in Figure~\ref{fig:twodoublepointcharts}, as are the spaces
$$X_i = U_i^+ \smallsetminus \intt \, U_i,
\quad
X_i'= U_i^+ \smallsetminus \intt \, U_i',
\quad 
C_i:=e(\Sigma) \cap X_i, \quad \text{ and }
 \quad C_i':=e(\Sigma) \cap X_i'.$$
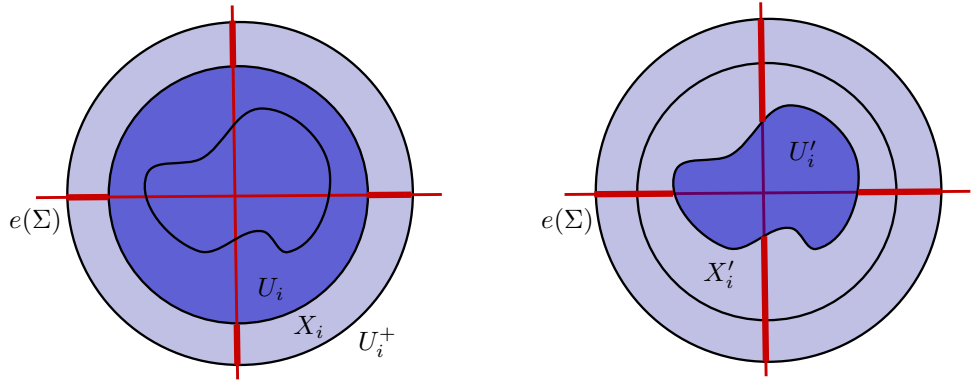
\begin{figure}[h!]
\centering
\begin{tikzpicture}
%\draw[step=1cm,color=gray] (0,0) grid (12,5);Uncomment this to get some helpful grid lines
\node[anchor=south west,inner sep=0] at (0,0){\includegraphics[height=5cm]{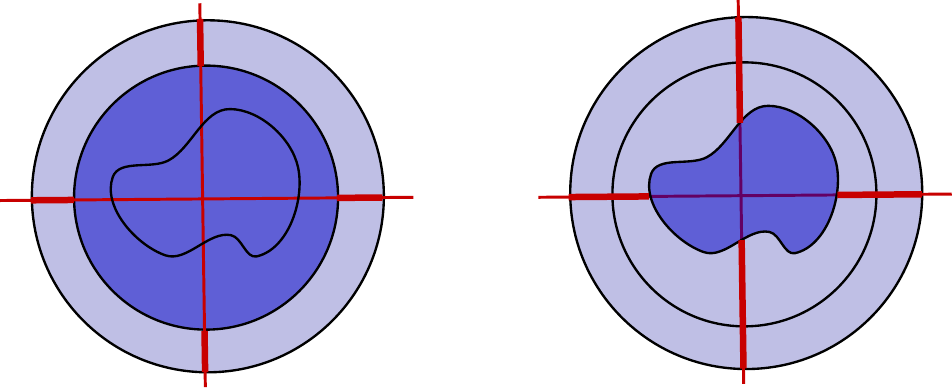}};
\node at (0, 2.1){$e(\Sigma)$};
\node at (7, 2.1){$e(\Sigma)$};
\node at (3.1,1.2) {$U_i$};
\node at (3.6, .7){$X_i$};
\node at (10.1,3) {$U_i'$};
\node at (9,1.4) {$X_i'$};
\node at (4.5,0.5) {$U_i^+$};
\end{tikzpicture}
\caption{The decompositions $U_i^+= X_i \cup U_i$ (left) and $U_i^+= X_i' \cup U_i'$ (right).  The parts of $e(\Sigma)$ comprising $C_i \subset X_i$ and $C_i' \subset X_i'$ are bolded on the left and right respectively.  } 
\label{fig:twodoublepointcharts}
\end{figure}

We have two decompositions of $N$:
\[\left( N \setminus \bigsqcup_{i=1}^c \intt U_i^+ \right) \cup \left(\bigcup_{i=1}^c X_i \right) \cup\left( \bigcup _{i=1}^c U_i\right)= N=\left( N \setminus \bigsqcup_{i=1}^c  \intt U_i^+\right) \cup\left( \bigcup_{i=1}^c X_i' \right) \cup \left(\bigcup _{i=1}^c U_i'\right).\]
We define the required homeomorphism $N \to N$ to be the identity on $N \smallsetminus \sqcup_{i=1}^c \intt U_i^+$,  and the homeomorphism~$(\psi_i')^{-1} \circ \psi_i \colon U_i \to U_i'$ on $U_i$.
% are the same subsets of $N$--we refer to them differently to emphasize their natural inclusions into $X_i$ and $X_i'$. 
In order to extend these assignments to a homeomorphism~$F \colon N \to N$, it therefore suffices to show that there are homeomorphisms
 $F_i \colon X_i \to X_i'$ 
 that extend  $\id \colon \partial U_i^+ \to \partial U_i^+$ and $(\psi_i')^{-1} \circ \psi_i \colon \partial U_i \to \partial U_i'$
 while sending $C_i$ to $C_i'$. 
Indeed $F$ will then automatically satisfy the two properties listed in the lemma.
%Say that trace of isotopy.

Both $X_i$ and $X_i'$ are homeomorphic to $S^3 \times I$ and every homeomorphism of $S^3$ is isotopic to the identity, so there exists a homeomorphism $\zeta_i \colon X_i \to X_i'$ that extends $\id \colon \partial U_i^+ \to \partial U_i^+$ and~$(\psi_i')^{-1} \circ \psi_i \colon \partial U_i \to \partial U_i'$. 
%%Don't delete.
%trace of isotopy
Thus $C_i$ and~$\zeta_i^{-1}(C_i')$ are two concordances in~$X_i\cong S^3 \times I$ going from the Hopf link~$\partial e(\Sigma) \cap \partial U_i^+$ in~$\partial U_i^+ \cong S^3 \times \{1\}$ to the Hopf link $\partial e(\Sigma) \cap \partial U_i$ in~$\partial U_i \cong S^3 \times \{0\}$.
%%Don't delete
%\partial C_i=e(\Sigma) \cap \partial (U_i \cup \partial U_i^+)
%\partial C_i'=e(\Sigma) \cap \partial (U_i' \cup \partial U_i^+)
%\zeta^{-1} takes this to  e(\Sigma) \cap \partial (U_i \cup \partial U_i^+)
%because it extends \psi_i^{-1} \circ \psi_i' and id
In what follows, we write $\mathcal{H} \subset S^3$ for the Hopf link.

\begin{claim*}
We claim that both $C_i \subset X_i$ and $C_i' \subset X_i'$ (and hence $\zeta_i^{-1}(C_i') \subset X_i$) have exteriors with fundamental group isomorphic to $\Z^2$. 
\end{claim*}
\begin{proof}
The homeomorphism~$\psi_i^+ \colon U_i^+ \to D^2_{1+\delta} \times D^2_{1+\delta}$ restricts to a homeomorphism of pairs
\begin{align*}
X_i &\to (D^2_{1+\delta} \times D^2_{1+ \delta}) \smallsetminus (D^2 \times D^2) \\
%%%
C_i &\to  (D^2_{1+ \delta} \smallsetminus D^2) \times \{0\} \cup \{0\} \times (D^2_{1+ \delta} \smallsetminus D^2).
 \end{align*}
Thus $(X_i, C_i)$ is homeomorphic to the trivial concordance~$(S^3 \times I, \mathcal{H} \times I)$,  and so $\pi_1(X_i \setminus C_i) \cong \Z^2$.
% as required.

Next we turn to the concordance exterior $X_i'\setminus C_i'$ and calculate $\pi_1(X_i' \setminus C_i').$
Choose $\epsilon>0$ such that $U_i^{\epsilon}:=(\psi_i)^{-1}(D^2_\epsilon \times D^2_ \epsilon)$ is contained within~$U_i'$ and set
$$X_i^{\epsilon}:= U_i' \smallsetminus \intt\,  U_i^{\epsilon} \quad \text{ and } \quad C_i^\epsilon:= e(\Sigma) \cap X_i^{\epsilon},$$
as illustrated in Figure~\ref{fig:gluingtwoconcordances}.  
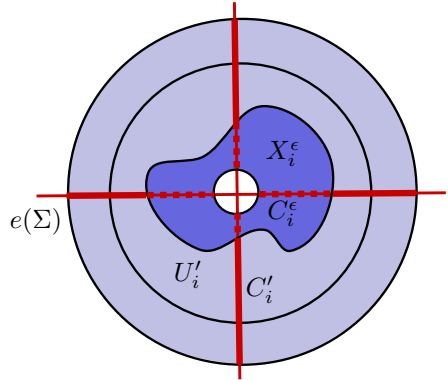
\begin{figure}[h!]
\centering
\begin{tikzpicture}
%\draw[step=1cm,color=gray] (0,0) grid (12,5);Uncomment this to get some helpful grid lines
\node[anchor=south west,inner sep=0] at (0,0){\includegraphics[height=5cm]{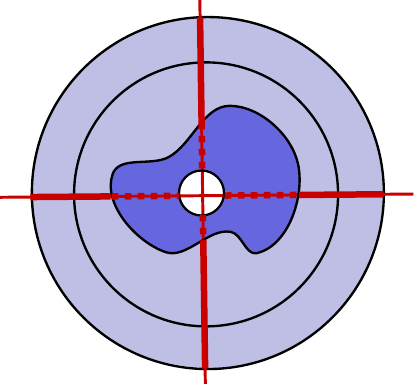}};
\node at (0, 2.1){$e(\Sigma)$};
\node at (2.95,1.2) {$C_i'$};
\node at (2,1.4) {$U_i'$};
\node at (3.25, 2.2){$C_i^{\epsilon}$};
\node at (3.25, 3){$X_i^{\epsilon}$};
\end{tikzpicture}
\caption{The concordances $C_i' \subset X_i'$ (drawn in a heavier line) and $C_i^\epsilon \subset X_i^{\epsilon}$ (drawn in a heavier dashed line). } 
\label{fig:gluingtwoconcordances}
\end{figure}

Observe that $C_i^\epsilon \subset X_i^{\epsilon} \cong S^3 \times I$ is a concordance, and an analogous argument to the one given above for $(X_i,C_i)$ shows that 
\begin{equation}
\label{eq:StackConcordanceExterior}
(X_i' \cup X_i^\epsilon, C_i' \cup C_i^\epsilon)= (U_i^+ \smallsetminus \intt  U_i^{\epsilon},  e(\Sigma) \cap (U_i^+ \smallsetminus \intt U_i^{\epsilon})) \cong (S^3 \times I, \mathcal{H} \times I).
\end{equation}
Apply the van Kampen theorem to the decomposition
 \[( X_i' \cup X_i^\epsilon \smallsetminus \nu( C_i' \cup C_i^\epsilon))
 = (X_i' \smallsetminus \nu(C_i')) \cup_{ \partial U_i' \smallsetminus \nu ( e(\Sigma) \cap \partial U_i')}
 (X_i^{\epsilon} \smallsetminus \nu(C_i^\epsilon))\]
to obtain a push-out diagram 
 \[\begin{tikzcd}
 &  \pi_1(X_i' \smallsetminus \nu(C_i'))  \arrow{rd}{j_1} \\
\pi_1(\partial U_i' \smallsetminus \nu ( e(\Sigma) \cap \partial U_i')) \cong\Z^2 \arrow{ru}{i_1}  \arrow{rd}{i_2} & & \pi_1( X_i' \cup X_i^\epsilon \smallsetminus \nu( C_i' \cup C_i^\epsilon))  \cong \Z^2. \\
&\pi_1(X_i^{\epsilon} \smallsetminus \nu(C_i^\epsilon)) \arrow{ru}{j_2}
\end{tikzcd}
\]
Since $\partial U_i' \smallsetminus \nu ( e(\Sigma) \cap \partial U_i')$ is a Hopf link exterior,~$\pi_1( \partial U_i' \smallsetminus \nu ( e(\Sigma) \cap \partial U_i'))$ is generated by meridians of~$e(\Sigma) \cap \partial U_i'$.
Since we noted in~\eqref{eq:StackConcordanceExterior} that $(X_i' \cup X_i^\epsilon, C_i' \cup C_i^\epsilon)$ is the exterior of a trivial concordance from the Hopf link to itself, we deduce that~$\pi_1( X_i' \cup X_\epsilon \smallsetminus \nu( C_i' \cup C_i^\epsilon))$ is generated by meridians of~$C_i' \cup C_i^{\epsilon}$,  and thus the maps~$j_1 \circ i_1$ and~$j_2 \circ i_2$ are isomorphisms. 
It follows that~$i_1$ and~$i_2$ are injective and since we have a free product with amalgamation,~$j_1$ and~$j_2$ are injective as well~\cite[Example~1B.12]{Hatcher}.
Since the $j_i$ take meridians to meridians, they are also surjective and hence isomorphisms.
This concludes the proof of the claim.
\end{proof}
Since $C_i,\zeta_i^{-1}(C_i') \subset X_i \cong S^3 \times [0,1]$ are both $\Z^2$-concordances,  from the Hopf link~$\partial e(\Sigma) \cap \partial U_i^+$ in~$\partial U_i^+ \cong S^3 \times \{1\}$ to the Hopf link $\partial e(\Sigma) \cap \partial U_i$ in~$\partial U_i \cong S^3 \times \{0\}$,  Theorem~\ref{thm:HopfConcordancesEquivalent} ensures that there is a homeomorphism $\rho_i \colon (X_i, C_i) \to (X_i, \zeta_i^{-1}(C_i'))$ which is the identity on $\partial X_i$.  
%{AC: Technically I suspect that theorem applies in a fixed copy of $S^3$ so we should perhaps add a few words.}
It follows that the homeomorphism~$F_i:= \zeta_i \circ \rho_i \colon X_i \to X_i'$ extends $\id \colon \partial U_i^+ \to \partial U_i^+$ and $(\psi_i')^{-1} \circ \psi_i \colon \partial U_i \to \partial U_i'$ and satisfies~$F_i(C_i)=C_i'$. 
As described above, combining the $F_i$ with the $\psi_i' \circ \psi_i$ and $\id_{N \smallsetminus \cup_{i=1}^c \intt U_i^+}$ leads to the required homeomorphism $F \colon N \to N.$
\end{proof}

We now prove the main result of this section, Proposition~\ref{prop:ImmIsImmalpha},  which states that the inclusion of  $\operatorname{Imm}_{\alpha}(g;c_+,c_-)^0(K,N)$ in $
\operatorname{Imm}(g;c_+,c_-)^0(K,N)$ induces a bijection 
$$\Phi \colon \operatorname{Imm}_{\alpha}(g;c_+,c_-)^0(K,N)/\Homeo_\alpha(\Sigma,\partial) \to
\operatorname{Imm}(g;c_+,c_-)^0(K,N)/\Homeo(\Sigma,\partial).$$
Here, recall that $\alpha \colon \bigsqcup_{2c} D^2 \hookrightarrow \Sigma$ is a fixed embedding.

\begin{proof}[Proof of Proposition~\ref{prop:ImmIsImmalpha}]
The well-definedness of $\Phi$ is immediate; the work is to show bijectivity. 

We begin by showing that $\Phi$ is surjective.  
Let $e' \colon \Sigma \looparrowright N$ be an immersion representing an element of $\operatorname{Imm}(g;c_+,c_-)^0(K,N)$,  which is compatible with some $\alpha' \colon \bigsqcup_{k=1}^{2c} D^2 \hookrightarrow \Sigma$.  We will define a homeomorphism $\theta \colon \Sigma \to \Sigma$ such that $e' \circ \theta$ is compatible with $\alpha$.  
First, we have a homeomorphism~$\alpha_k' \circ (\alpha_k)^{-1}  \colon \alpha_k(D^2) \to \alpha_k'(D^2)$ for all $k=1,  \dots,  2c$.  Now consider the punctured surfaces~$\Sigma^\circ= \Sigma \smallsetminus \bigcup_{k=1}^{2c} \intt \alpha_k(D^2)$ and $(\Sigma^\circ)'=  \Sigma \smallsetminus \bigcup_{k=1}^{2c} \intt \alpha_k'(D^2)$. 
It suffices to show that there exists a homeomorphism $\theta^\circ \colon \Sigma^ \circ \to (\Sigma^\circ)'$ extending the~$\alpha_k' \circ (\alpha_k)^{-1} \colon \partial \alpha_k(D^2)  \to \partial \alpha_k'(D^2)$ and the identity map on $\partial \Sigma$.  But this follows from the classification of surfaces and the fact that, up to isotopy,  there is only one orientation preserving homeomorphism of  $S^1$. 

We now show that $\Phi$ is injective.  
Let $e, e' \colon \Sigma \looparrowright N$ be elements of $\operatorname{Imm}_{\alpha}(g;c_+,c_-)^0(K,N)$ with the same image under $\Phi$.  That is,  both $e$ and $e'$ are compatible with $\alpha$ and there exist rel. boundary homeomorphisms $F \colon N \to N$ and  $\theta \colon \Sigma \to \Sigma$ that satisfy $F \circ e' = e \circ \theta$.  
We need to show that there exist rel. boundary homeomorphisms $F' \colon N \to N$ and~$\theta' \colon \Sigma \to \Sigma$ such that~$F'  \circ e' = e \circ \theta'$ and $\theta' \circ \alpha= \alpha$.

Fix double point charts $\{ (U_p,\psi_p) \}_p$ for $(e,\alpha)$ and $\{ (U_p',\psi_p')\}_p$ for $(e',\alpha)$.
\begin{claim*}
\label{claim:F'}
There is a homeomorphism $F' \colon N \to N$ such that 
$$F'(e'(\Sigma))=e(\Sigma) \quad \text{ and } \quad \psi_p'=\psi_p \circ F'|_{U_p'}.$$
\end{claim*}
\begin{proof}
Verify that $\{ (F(U_p'), \psi_p' \circ F^{-1}|_{F(U_p')})\}_p$ is a family of double point charts for $(e,\alpha)$.
%$$(\psi_p' \circ F^{-1}) \circ e \circ \alpha_{2j-1}
%=\psi_p' \circ F) \circ e' \circ \alpha_{2j-1}
%=\id$$
Apply Lemma~\ref{lem:Homeo} to $\{ (U_p,\psi_p) \}_p$ and $\{ (F(U_p'), \psi_p' \circ F^{-1}) \}_p$ to obtain a homeomorphism $H \colon N \to N$ with $H(e(\Sigma))=e(\Sigma)$ and 
$H|_{F(U_p')}
=\psi_p^{-1}  \circ (\psi_p' \circ F^{-1}|_{F(U_p')}).$
%=(\psi_p' \circ F^{-1})^{-1}\circ \psi_p
%=F|_{U_p'} \circ (\psi_p')^{-1} \circ \psi_p.$$
We can now complete the proof of the claim by verifying that the homeomorphism $F':=H \circ F \colon N \to N$ satisfies the required properties: namely~$F'( e'(\Sigma))=H( F( e'(\Sigma)))=H (e(\Sigma))=e(\Sigma)$ and~$\psi_p \circ F'|_{U_p'}
=\psi_p \circ H \circ F|_{U_p'}
=\psi_p'.$
%This concludes the proof of the claim.
\end{proof}

Apply Proposition~\ref{prop:SurfQuotientOfImm} to obtain a homeomorphism $\theta' \colon \Sigma \to \Sigma$ with $F' \circ e'=e\circ \theta'.$

\begin{claim*}
The homeomorphism $\theta'$ satisfies~$\alpha=\theta' \circ \alpha$.
\end{claim*}
\begin{proof}
%\color{teal}
The equality $F' \circ e'=e \circ \theta'$ implies that for each disc $D^2$ in $\bigsqcup_{k=1}^{2c} D^2$ we have
 \[\psi_p \circ F' \circ e' \circ \alpha|_{D^2}=\psi_p \circ e \circ \theta' \circ \alpha|_{D^2}.\]
Again using~$F' \circ e'=e\circ \theta',$ one verifies that $\alpha(D^2) = \theta'(\alpha(D^2))$, 
%%Don't delete.
%\theta'(\alpha(D^2))=e^{-1} F'e'(\alpha(D^2))
%Because F'(e'(\Sigma))=e(\Sigma) this gives
%e^{-1}e(\alpha(D^2))=\alpha(D^2).
and so we rewrite the above as
$$ (\psi_p \circ F' \circ (\psi_p')^{-1}) \circ (\psi_p' \circ e' \circ \alpha|_{D^2})=(\psi_p \circ e \circ \alpha|_{D^2}) \circ (\alpha^{-1}|_{\alpha(D^2)} \circ \theta'\circ \alpha|_{D^2}).$$
Here we implicitly use on the lefthand side that  $e' \circ \alpha|_{D^2}$ takes values in~$U_p'$.
%%AC: Used to be $(\psi_p')^{-1}|_{U_p'}$; hopefully this is better.
%%AM: Inserted a prime so that the two things cancel.  Though maybe you would prefer $(\psi_p')|_{U_p'}$ instead of just $\psi_p'$?
%%AC: No strong feelings: it was more about justifying being able to take inverses.
%%AM: Hmm,  isn't $\psi_p$ itself invertible? 

The previous claim tells us that $\psi_p \circ F' \circ (\psi_p')^{-1}=\id_{D^2 \times D^2}$,  and combining this with the definition of a double point chart applied to $(e, \psi_p)$ and $(e', \psi_p')$ lets us rewrite our equation as
%
%The previous claim combined with the definition of a double point chart makes it possible to rewrite the right hand side as~$ \psi_p^{-1}|_{D^2} \circ \pm \id_{D^2}$
%%{meaning one of the two factors.}
%so that
%$$e \circ \theta' \circ \alpha|_{D^2}= \psi_p^{-1}|_{D^2} \circ \pm \id_{D^2}.$$
%The definition of a double point chart then leads to simplifying the left hand side so that
$$\pm \id_{D^2}=(\pm \id_{D^2}) \circ (\alpha^{-1}|_{\alpha(D^2)} \circ \theta' \circ \alpha|_{D^2}).$$
So $\alpha|_{D^2}= \theta' \circ \alpha|_{D^2}$ and we have established the claim.
%Since $\alpha(D^2) = \theta'(\alpha(D^2))$, we can rewrite this as $\alpha|_{D^2}=\theta' \circ \alpha|_{D^2}$,  thus proving the claim.
\end{proof}
The existence of the homeomorphisms~$F' \colon N \to N $ and $\theta' \colon \Sigma \to \Sigma$ with~$F' \circ e'= e \circ \theta'$ and~$\alpha=\theta' \circ \alpha$ ensures that~$e$ and~$e'$ agree in~$\operatorname{Imm}^0_\alpha(g;c_+,c_-)(K,N)/\Homeo_\alpha(\Sigma,\partial)$.
\end{proof}

\begin{remark}
\label{rem:ClosedImmSurfAdapt}
We record some variants on the content of this chapter. 
\begin{itemize}
\item 
We remark for later that, should we restrict $\operatorname{Surf}$,  $\operatorname{Imm}$,  and $\operatorname{Imm}_\alpha$ to  immersed/ immersions of surfaces whose complements have certain properties,  such as having a certain equivariant intersection form,  the proofs of Propositions~\ref{prop:SurfQuotientOfImm} and~\ref{prop:ImmIsImmalpha} go through verbatim to establish the corresponding bijections. 
\item We also note that the proofs of Propositions~\ref{prop:SurfQuotientOfImm} and~\ref{prop:ImmIsImmalpha} can be adapted to closed surfaces.
Given a closed, simply-connected $4$-manifold $X$,  one proceeds as in Notation~\ref{not:Imm0Surf0} to define the sets $\operatorname{Imm}_\alpha(g;c_+,c_-)(X)$ and $\operatorname{Surf}(g;c_+,c_-)(X)$ so that mapping an immersion to its image yields a bijection
$$\operatorname{Imm}_\alpha(g;c_+,c_-)(X)/\Homeo_\alpha(\Sigma) \xrightarrow{\cong} \operatorname{Surf}(g;c_+,c_-)(X).$$ 
Again, this can be refined to include the constraint that the exteriors of the immersed surfaces have a fixed equivariant intersection form.
\end{itemize}
\end{remark}

\subsection{Normal bundles}
\label{sub:NormalBundles}

This short section fixes some conventions concerning normal bundles of immersions and tubular neighborhouds of immersed surfaces.

\begin{notation*}
For the first part of this section,  $\Sigma$ denotes a genus $g$ surface that is either closed or has a single boundary component.
\end{notation*}

We use the definition of a normal bundle of an immersion from~\cite[Definition~2.2]{KasprowskiPowellRayTeichnerEmbedding}.
\begin{definition}
\label{def:NormalBundle}
A \emph{normal bundle} of an immersion $e \colon \Sigma \looparrowright N$ is a pair $(\nu(e),\iota)$ consisting of a rank two vector bundle~$\pi \colon \nu(e) \to \Sigma$ together with an immersion $\iota \colon \nu(e) \looparrowright N$ that restricts to $e$ on the zero section $s_0$ i.e. $\iota \circ s_0=e$ and such that each point $p \in \Sigma$ has a neighborhood $U$ such that~$\iota|_{\pi^{-1}(U)}$ is an embedding.
%We refer to $\nu(S):=\iota(\nu(e))$ as an (open) \emph{tubular neighbhorhood} of the immersed surface $S:=e(\Sigma).$
\end{definition}

As is customary, we will refer to the vector bundle~$\nu(e)$ as the normal bundle of~$e$ but it should be understood that an immersion is also present in the background.
We also write~$\overline{\nu}(e)$ for the disk bundle of~$e$ and~$\overline{\nu}(e(\Sigma)):=\iota(\overline{\nu}(e))$ for the closed tubular neighbhorhood of $e(\Sigma).$
The normal bundle~$\nu(e)$ is homeomorphic to~$\iota(\nu(e))$ if and only if~$e$ is an embedding.
In general,  neighbhorhoods of immersed surfaces are obtained by plumbing the disk bundle over the surface.
We will describe the outcome in more detail in Section~\ref{sec:BoundarIdentifications} but we introduce some relevant terminology here.

%%%Don't delete
%AC: We say that this is a normal bundle but then define a disk bundle.
%We can probably define a normal bundle by using the same conditions and replacing the $D^2$s by $\R^2$s.
%Note that $s_0(x):=(x,0)$ is a $0$-section (i.e. $\iota \circ s_0=e$) by using the definition of compatibility.
%%%Don't delete.
%%Since shrink take (z,0) to (z,0) we get $\iota \circ s_0=\psi_^{-1} \circ \shrink^\pm \circ (\alpha_k^{-1} \times \id)(x,0)=\psi_i^{-1} \circ \alpha_k^{-1}=\pm e(x)$.
%This comment holds throughout this document: $D^2$-bundles are actually linear.
%}
\begin{definition}
\label{def:UAdapted}
Let $\alpha \colon \bigsqcup_{2c} D^2 \hookrightarrow \Sigma$ be an embedding,  let $e \colon \Sigma \looparrowright N$ be an $\alpha$-compatible immersion,  and let~$\mathcal{U}:=\{(U_i,\psi_i)\}_{i=1}^c$ be a family of double point charts for $(e,\alpha)$. 
A normal bundle~$(\nu(e),\iota \colon \nu(e) \looparrowright N)$ is \emph{$\mathcal{U}$-adapted} if it satisfies the following conditions:
\begin{enumerate}
\item over the $B_k:= \alpha_k(D^2) \subset \Sigma$, the total space of the vector bundle is 
%disk
%equal to $B_k \times D^2$
$$ \nu(e|_{B_k})=B_k \times \R^2$$
and the bundle-projection map $\nu(e|_{B_k})=B_k \times \R^2 \to B_k$ is the projection~$(x,y) \mapsto x.$
% projection onto the first component;
\item for each $i=1, \dots, c$,  the immersion $\iota \colon \overline{\nu}(e) \looparrowright N$ restricts to
%immersion
%$$\iota \colon  \nu(e|_{\sqcup B_i})=\bigsqcup B_i \times D^2 \looparrowright \bigsqcup U_i $$
embeddings 
%of the disk bundle
\begin{align*}
\iota \colon  \overline{\nu}(e|_{B_{2i-1}})&=B_{2i-1} \times D^2 \hookrightarrow  U_i \\
\iota \colon  \overline{\nu}(e|_{B_{2i}})&=B_{2i} \times D^2 \hookrightarrow  U_i
\end{align*}
that,  for any $(z,w) \in D^2 \times D^2$, satisfy
\begin{align}\label{eqn:shrink}
\left( \psi_i \circ \iota \circ (\alpha_k \times \id_{D^2})\right)(z,w)
&= \shrink_k (z,w)=
\begin{cases} (z,  w/2)  & \text{if }k=2i-1 \\
\pm (w/2, z)  &\text{if } k=2i, \end{cases}
\end{align}
where the $\pm$ is determined by the sign of the $i$-th self-intersection of $e(\Sigma)$. 
\end{enumerate}
\end{definition}

Just as we required our immersions to be 	`standard' in neighborhoods of double points in Definition~\ref{def:Compatible},  the two conditions of Definition~\ref{def:UAdapted} ensure that we only work with normal bundles that are given by a standard model in the neighborhood of double points.  Away from double points,  $e$ is an embedding and so we will have control over its normal bundle coming from uniqueness results for normal bundles of embeddings.

The next proposition ensures that one can always choose a normal bundle $\nu(e)$ and an immersion~$\iota \colon \nu(e) \looparrowright N$ that are suitably compatible with the double point charts.

\begin{proposition}
\label{prop:NormalBundle}
Let $\alpha \colon \bigsqcup_{2c} D^2 \hookrightarrow \Sigma$ be an embedding,  let $e \colon \Sigma \looparrowright N$ be an $\alpha$-compatible immersion, and let~$\mathcal{U}:=\{(U_i,\psi_i)\}_{i=1}^c$ be a family of double point charts for $(e,\alpha)$.
When $\partial \Sigma \neq \emptyset$,  let $(V,  \iota_V\colon V \hookrightarrow \partial N)$ be an arbitrary normal bundle for $e|_{\partial \Sigma}$.
\begin{enumerate}
\item There exists a normal bundle~$(\nu(e),\iota \colon \nu(e) \looparrowright N)$  that is $\mathcal{U}$-adapted and extends $(V,\iota_V)$. 
\item  If $(\nu(e),\iota)$ is an $\mathcal{U}$-adapted normal bundle, then pairs of consecutive disks~$B_{2i-1} \cup B_{2i}$ are taken under~$\iota$ to the plumbing of two disks in $N$: for~$x=(x_1,x_2) \in \overline{\nu}(e|_{B_{2i-1}})$ and~$y=(y_1,y_2) \in \overline{\nu}(e|_{B_{2i}}),$
\begin{align*}
\iota(x)=\iota(y) \ \ &\text{ iff } \ \ 
\shrink_{2i-1} \circ(\alpha_{2i-1}^{-1} \times \id_{D^2})(x)=\shrink_{2i}\circ(\alpha_{2i}^{-1} \times \id_{D^2})(y)\\
&\text{ iff } \ \ 
 \left(\alpha_{2i-1}^{-1}(x_1),  \frac{x_2}{2}\right)=\pm \left(\frac{y_2}{2},  \alpha_{2i}^{-1}(y_1)\right),
\end{align*}
where as usual the $\pm$ sign is determined by the sign of the $i$-th intersection point of $e(\Sigma)$. 
The notation $\alpha_k^{-1}$ is legitimate since $\overline{\nu}(e|_{B_k})=B_k\times D^2$ and $\alpha_k \colon D^2 \to B_k \subset \intt(\Sigma)$ is an embedding.

In particular,  the centers of the $B_k$ are carried to the double points of~$e$.
\end{enumerate}
\end{proposition}
\begin{proof}
Since $\shrink_k$ and the~$\psi_i$ are homeomorphisms and $\alpha_{2i-1}$,  and $\alpha_{2i}$ are homeomorphisms onto their images, for $k=2i-1,2i$ there exists a unique embedding~$\iota_k \colon B_k \times D^2 \hookrightarrow U_i$  such that~\eqref{eqn:shrink} is satisfied. 
The image of $\psi_i \circ \iota_k \circ (\alpha_k \times \id_{D^2})$ is $D^2 \times D^2_{r \leq \frac{1}{2}}$ for $k=2i-1$ and $D^2_{r \leq \frac{1}{2}} \times D^2$ for~$k=2i$,  as illustrated on the left of Figure~\ref{fig:extendingnbhd}.  
In either case one can extend $\psi_i \circ \iota_k \circ (\alpha_k \times \id_{D^2})$ to an embedding $j_k$ of $D^2 \times \R^2$ into $D^2 \times D^2$,  as illustrated on the right of Figure~\ref{fig:extendingnbhd}.
\begin{figure}[h!]
\centering
\begin{tikzpicture}
%\draw[step=1cm,color=gray] (0,0) grid (12,5);Uncomment this to get some helpful grid lines
\node[anchor=south west,inner sep=0] at (0,0){\includegraphics[height=2.5cm]{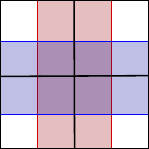}};
\node[anchor=south west,inner sep=0] at (4,0){\includegraphics[height=2.5cm]{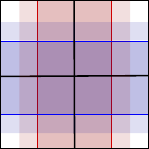}};
\end{tikzpicture}
\caption{ The embedding $\psi_i \circ \iota_k \circ (\alpha_k \times \id_{D^2})$ of $D^2 \times D^2$ into $D^2 \times D^2$ can be extended to an embedding $j_k$ of $D^2 \times \R^2$ into $D^2 \times D^2$.} 
\label{fig:extendingnbhd}
\end{figure}
%%%Don't delete.
%AM: I could say more? For $|w|>1$,  write $\psi_i \circ \iota_k \circ (\alpha_k \times \id_{D^2})(z, w/|w|)=(z', w')$  where and let $j_k(z,w)= (z',  \theta_{w'},  s(r_{w}))$,  where $s\colon [1, \infty) \to [1/2,3/4)$ is a homeomorphism.  (At least, this is the formula for $k=2i-1$--appropriately flip it for $k=2i$. )}
The embedding~$B_k \times \mathbb{R}^2 \hookrightarrow U_i \subset N$ given by $\psi_i^{-1} \circ j_k \circ (\alpha_k \times \id_{D^2})^{-1}$ gives rise to a normal bundle for $B_k \subset \Sigma$.
As explained in~\cite{KasprowskiPowellRayTeichnerEmbedding},  work of Freedman-Quinn~\cite[Theorem~9.3]{FreedmanQuinn} ensures that we can extend these normal bundles (together with our chosen normal bundle for $e|_{\partial \Sigma}$,  when $\partial\Sigma \neq \emptyset$) to
a normal bundle over the whole of $\Sigma$.
By construction, this normal bundle is $\mathcal{U}$-adapted.

We now prove the second item. 
Observe that $\iota(x)= \iota(y)$ if and only if $\psi_i(\iota(x))= \psi_i(\iota(y))$,  since $\psi_i$ is a homeomorphism. 
We have from item (2)  of Definition~\ref{def:UAdapted} that
\begin{align*}
\psi_i(\iota(x))&= (\psi_i \circ \iota \circ (\alpha_{2i-1}\times \id_{D^2}))( \alpha_{2i-1}^{-1} (x_1),  x_2)=\shrink_{2i-1}( \alpha_{2i-1}^{-1} (x_1),  x_2)= ( \alpha_{2i-1}^{-1}(x_1),  x_2/2),\\
\psi_i(\iota(y))&= (\psi_i \circ \iota \circ (\alpha_{2i} \times \id_{D^2}))( \alpha_{2i}^{-1} (y_1),  y_2)=
\shrink_{2i}(\alpha_{2i}^{-1}(y_1), y_2)=\pm (y_2/2,  \alpha_{2i}^{-1}(y_1)).
\end{align*}
So $\psi_i(\iota(x_1,x_2))= \psi_i(\iota(y_1, y_2))$ if and only if $(\alpha_{2i-1}^{-1} (x_1),  x_2/2)= \pm (y_2/2,  \alpha_{2i}^{-1}(y_1))$,  as desired. 
\end{proof}

\begin{notation}
From now on we assume that $\Sigma$ has a single nonempty boundary component,  in order to fix a normal bundle for $e|_{\partial \Sigma}$. 
\end{notation}

%Finally,  

\begin{notation}[Bundle data related to $K \subset S^3$]\label{not:nbhdforK}
Given a knot $K \subset S^3$,  choose a normal bundle~$\nu(K)$ and,  as is customary in knot theory, identify it with its image in $S^3$.
This defines a tubular neighborhood~$K \subset \overline{\nu}(K) \subset S^3$ 
%with a bundle projection $\pi_K \colon \nu(K) \to K$ with the property that~$\pi_K(x)=x$ for all $x \in K$ as well as 
and an exterior $E_K= S^3 \setminus \intt(\overline{\nu}(K))$. 
The Seifert framing~$\fr_K \colon \nu(K) \to \partial \Sigma \times \R^2$ restricts to a homeomorphism
 $$\fr_K \colon \overline{\nu}(K) \to \partial \Sigma \times D^2$$
that sends $K$ to $\partial \Sigma \times \{0\}$,  $\mu_K$ to $\{\text{pt}\} \times \partial D^2$ and the Seifert longitude to $\partial \Sigma \times \{\text{pt}\}$. 
By restricting to the boundary,  we obtain 
\[ d_K \colon \partial E_K = \partial \overline{\nu}(K) \to \partial \Sigma \times S^1.\]
\end{notation}

\begin{construction}[A normal bundle for~$e|_{\partial \Sigma}$]\label{cons:normalbundleforboundary}
Fix a homeomorphism $h \colon \partial N \to S^3$. 
 By an abuse of notation,  we usually do not distinguish between  $K \subset S^3$ and $h^{-1}(K)= e(\partial \Sigma) \subset N$,  but we briefly do so here. 
We obtain a normal bundle for $e|_{\partial \Sigma}$ by taking $\nu(e|_{\partial \Sigma})= \partial \Sigma \times \R^2$,  and 
\[ \iota|_{\partial \Sigma} \colon \nu(e|_{\partial \Sigma})= \partial \Sigma \times \R^2 \xrightarrow{\fr_{K}^{-1}} \nu(K) \subset S^3 \xrightarrow{h^{-1}} \partial N.\]
%This gives a normal bundle for $e|_{\partial \Sigma}$. 
%The bundle projection map 
%\[ \pi_{h^{-1}(K)} \colon \nu(K) \xrightarrow{\pi_K} K \xrightarrow{h^{-1}} h^{-1}(K)= e(\partial \Sigma) \xrightarrow{e^{-1}} \partial \Sigma\]
%together with the embedding 
%\(h^{-1} \colon \nu(K) \hookrightarrow N\)
%gives a normal bundle $(\nu(K),  h^{-1})$ for $e|_{\partial \Sigma}$.  
%%%Don't delete.
%{AC: A note. I think that according to KPRT, technically,  the condition~$\iota \circ s_0=e$ is required where $s_0$ is the $0$-section.
%AM: Can we not just take $s_0 \colon \partial \Sigma \to \overline{\nu}(K)$ by $s_0(x)= h(e(x))$? (We are really thinking of $K \subset \overline{\nu}(K) \subset S^3$,  because $\overline{\nu}(K)$ is a subset of $S^3$ rather than the abstract normal bundle of an embedding.)
%}
\end{construction}

\section{Boundary identifications}
\label{sec:BoundarIdentifications}

Given an immersion $e \colon \Sigma \looparrowright N$ with $\partial e(\Sigma)=K$ together with some additional data that we will soon introduce,  Section~\ref{sub:Identifications} constructs a homeomorphism $W \cong \overline{\nu}(e|_{\Sigma^\circ})$.
This homeomorphism depends on a choice of framing of $\nu(e|_{\Sigma^\circ})$ and Section~\ref{sub:GoodFramings} describes the particular type of framing that we will use.

\begin{notation*}
We remind the reader that the manifolds $W,P,$ and $P_K$ depend on the genus $g$ of the surface~$\Sigma:=\Sigma_{g,1}$, integers $c_+,c_- \geq 0$ and a fixed embedding $\alpha \colon \bigsqcup_{2c} D^2 \hookrightarrow \operatorname{int}(\Sigma).$
\end{notation*}

\subsection{Identifications}
\label{sub:Identifications}

The goal is to show that the data from the previous section together with certain framings of $\nu(e|_{\Sigma^\circ})$ canonically determine homeomorphisms~$\overline{\nu}(e) \cong E$ and~$\overline{\nu}(e(\Sigma)) \cong W$.
%Not all framings: they have to extend the the \alpha \times \id.

\begin{notation}
\label{not:PuncturedEmbeds}
In addition to fixing $\Sigma$ and $\alpha$,  we fix the following. 
\begin{itemize}
\item An immersion $e \colon \Sigma \looparrowright N$ with $c=c_++c_-$ double points that is compatible with $\alpha$.
\item A family of double point charts $\mathcal{U}:=\{ (U_i,\psi_i) \}_{i=1}^c$ for $(e,\alpha)$ as in Definition~\ref{def:Compatible}.  
%According to whether~$\Sigma$ is closed or has a boundary component, 
\item A normal bundle~$(\nu(e),\iota \colon \nu(e) \looparrowright N)$ that satisfies the properties listed in Proposition~\ref{prop:NormalBundle} with respect to the normal bundle for $\nu(e|_{\partial \Sigma})$ described in Construction~\ref{cons:normalbundleforboundary}. 
\end{itemize}
Set $N^\circ:=N\setminus \bigsqcup_{i=1}^c  \intt U_i$.  
The boundary $\partial N^\circ$ consists of~$c+1$ copies of~$S^3$.
We write
$$S:=e(\Sigma), \ \ \ \ S^\circ:=e(\Sigma^\circ), \ \ \ \ \overline{\nu}(S):=\iota(\overline{\nu}(e)), \ \ \ \
\overline{\nu}(S^\circ):=\iota(\overline{\nu}(e|_{\Sigma^\circ})).$$
A framing of~$\nu(e|_{\Sigma^\circ})$ determines a framing of~$\nu(S^\circ)$ and vice versa.
The same applies to sections.
\end{notation} 

%\begin{convention}
%We will henceforth make no distinction between sections/framings of a vector bundle and the induced sections/framings of its linear disk bundle.
%A \emph{linear section} (resp. \emph{linear framing} of a disk bundle of a vector bundle,  is one that is induced from a section (resp. framing) of the vector bundle.
%%Really a linear bundle so we keep track
%\end{convention}

\begin{remark}
We record two remarks on our set-up.
\begin{itemize}
\item The immersion~$\iota \colon \nu(e) \looparrowright N$ restricts to an embedding~$\nu(e|_{\Sigma^\circ}) \hookrightarrow N^\circ.$
%of $\overline{\nu}(e|_{\Sigma^\circ})$ in $N^\circ.$
\item The link~$e(\partial B_{2i-1} \cup \partial B_{2i})$ is a Hopf link in (a component of)~$\partial N^\circ$ whose sign agrees with the sign of the disks.
\end{itemize}
\end{remark}

We will consider framings of $\nu(e|_{\Sigma^\circ})$ that extend the following framing of the boundary.

\begin{construction}[The Seifert framing on $e|_{\partial \Sigma^\circ}$]
%~$\overline{\nu}(e|_{\partial \Sigma^\circ})$]
\label{cons:SeifertFraming}
We now describe a framing
$$\fr_\partial \colon \nu(e|_{\partial \Sigma^\circ}) \to \partial \Sigma^\circ \times~\R^2.$$
First, given a~$\pm$-labelled disk $B_k$,   we frame~$\nu(e|_{\partial B_k })$ using the composition
\begin{equation}
\label{eq:fr}
\fr_{\partial B_k} \colon  \nu(e|_{\partial B_k }) \xrightarrow{=}  \partial B_k  \times \R^2 \xrightarrow{\alpha_k ^{-1} \times \id_{\R^2}} (S^1 \times \R^2)_k   \xrightarrow{\eta^{\pm}} S^1 \times \R^2 \xrightarrow{\alpha_k \times \id_{\R^2}} \partial B_k \times \R^2.
\end{equation}
Identifying the link~$L_i:=\partial B_{2i-1} \cup \partial B_{2i}$ with its image in~a component of~$\partial N^\circ$, we argue that~$\fr_{\partial B_{2i-1}} \cup \fr_{\partial B_{2i}}$ gives $L_i$ the Seifert framing. 
Note that upon applying~$\iota \colon \nu(e) \looparrowright N$
the identity framing~$\nu(e|_{\partial B_k }) \xrightarrow{=}  \partial B_k \times \R^2$ gives the $0$-framing on each component of $L_i$,  since the framing of each component unknot is obtained by restricting the disk framing.
%%Don't delete.
%Disk framing coming from the (image of the) $B_k$
Now observe that under~$\iota$,  the $\overline{\nu}(e|_{B_k})$ are mapped into $U_i$, whose boundary has the opposite orientation from that of the corresponding component of~$\partial N^\circ$.
The definition of~$\eta^\pm$ implies that~$\fr_{\partial B_k}$ corresponds to the~$(\pm 1)$-framing on $L_i \subset \partial U_i$ which is the $(\mp 1)$ framing on $L_i \subset \partial N^\circ$.  This completes the argument,  since the Seifert framing on the components of the~$\pm$-Hopf link is~$(\mp 1)$.

Second,   we have that $\nu(e|_{\partial \Sigma})= \partial \Sigma \times \R^2$ (Construction~\ref{cons:normalbundleforboundary}), and we take the identity framing
\[\fr_{\partial \Sigma} \colon \nu(e|_{\partial \Sigma})= \partial \Sigma \times \R^2 \to \partial \Sigma \times \R^2.\]
Given a normal bundle $(\nu(e),\iota)$ that extends $e|_{\partial \Sigma}$,  upon applying $\iota \colon \nu(e) \looparrowright N$, the framing $\fr_{\partial \Sigma}$ gives the Seifert framing on $e(\partial \Sigma)= K$,  
%or technically,  on $e(\partial \Sigma)=h^{-1}(K)$. 
essentially because of our choice of $\fr_K$ in Construction~\ref{not:nbhdforK} and the use of $\fr_K^{-1}$ in the definition of $\iota|_{\partial \Sigma}$ in Construction~\ref{cons:normalbundleforboundary}.

The $\{ \fr_{\partial B_k}\}_k$ together with~$\fr_{\partial \Sigma}$ lead to the required framing
$$\fr_\partial \colon \nu(e|_{\partial \Sigma^\circ}) \to \partial \Sigma^{\circ} \times \R^2.$$
\end{construction}

The next proposition justifies the fact,  mentioned in Section~\ref{sec:PlumbedManifolds}, that the manifold $W$ is homeomorphic to the neighborhood of an immersed surface.

\begin{proposition}
\label{prop:IdentifyW}
Let $e \colon \Sigma \looparrowright N$ be a $\Z$-immersion that is compatible with $\alpha$. 
\begin{enumerate}
%%AC:  Only fixed $\fr_k$ because that seemed to enough to glue together; didn't need to say anything about $\partial \Sigma$ for that
\item Given a framing $\fr \colon \nu(e|_{\Sigma^\circ}) \to \Sigma^\circ \times \R^2$ that extends the framings $\{\fr_{\partial B_k}\}_k$ from~\eqref{eq:fr}, there exists a homeomorphism
%{AC: This is probably more i.e. an isomorphism of linear $D^2$-bundles. Perhaps morally helpful to point it out. 
%Or to write everything in terms of vector bundles.
%TBD.  
%AM: It seems like the only thing we need to confirm is that our good framing $\fr$ can be chosen to be a linear $D^2$-bundle isomorphism rather than just a homeomorphism.  Is this going to follow from some general nonsense about how vector/ linear $D^2$-bundles work? 
%AC: I think so.
%}
$$\gamma \colon \overline{\nu}(e) \xrightarrow{\cong} E$$
whose restriction 
\begin{itemize}
\item  to~$\overline{\nu}(e|_{\Sigma^\circ})$ agrees with the framing $\fr$,
\item to~$\overline{\nu}(e|_{B_k})=B_k \times D^2$ is the homeomorphism
$$(\alpha_k^{-1} \times \id_{D^2}) \colon \overline{\nu}(e|_{B_k })=B_k \times D^2 \to {(D^2 \times D^2)_k}$$
\end{itemize}
\item  Given a framing $\fr \colon \nu(e|_{\Sigma^\circ}) \to \Sigma^\circ \times \R^2$ that extends the framings $\{\fr_{\partial B_k}\}_k$ from~\eqref{eq:fr},  a family $\mathcal{U}$ of double point charts for $(e, \alpha \colon \bigsqcup_{2c}D^2 \hookrightarrow \Sigma)$ and an immersion $\iota \colon \overline{\nu}(e) \looparrowright N$ as in Proposition~\ref{prop:NormalBundle},  the homeomorphism~$\gamma$ induces a homeomorphism
$$\widehat{\gamma} \colon \overline{\nu}(S):=\iota(\overline{\nu}(e)) \xrightarrow{\cong} W$$
 that fits into the following commutative diagram
\begin{equation}
\label{eq:IdentificationDiagram}
\xymatrix{
\overline{\nu}(e) \ar@{->>}[r]^-{\iota}\ar[d]^{\gamma}_\cong& \overline{\nu}(S) \ar[d]^{\widehat{\gamma}}_\cong \\
E \ar@{->>}[r]^{\operatorname{proj}}& W
}
\end{equation}
and induces a homeomorphism
$$ \widehat{\gamma} \colon \partial \overline{\nu}(S) \setminus \left( \partial N \setminus \intt(\overline{\nu}(K))\right) \xrightarrow{\cong} P.  $$
 % Technically, the domain of $\widehat{\gamma}$ is $\partial \overline{\nu}(S) \setminus \left( \partial N \setminus h^{-1}(\nu(K))\right)$.  But we are generally omitting this distinction from our language wherever psosible.
 \end{enumerate}
\end{proposition}
\begin{proof}
Because of the decomposition~$\overline{\nu}(e)=\overline{\nu}(e|_{\bigsqcup B_k })\cup\overline{\nu}(e|_{\Sigma^\circ})$, in order for the homeomorphism~
\[\bigsqcup_{k=1}^{2c} (\alpha_k ^{-1} \times \id_{D^2})   \colon \overline{\nu}(e|_{\bigsqcup B_k })=\bigsqcup_{k=1}^{2c} (B_k \times D^2)  \to \bigsqcup_{k=1}^{2c} (D^2 \times D^2)_k \]
from Construction~\ref{cons:SeifertFraming} and the framing
\[\fr\colon \overline{\nu}(e|_{\Sigma^\circ}) \to \Sigma^\circ \times D^2\] 
to combine into a homeomorphism~$\gamma \colon \overline{\nu}(e) \to E$, we need only verify that for any $x \in \overline{\nu}(e|_{\partial B_k })$
% we have that 
%%AC: Less of a sentence but gains a line and is more elegant visually
\[ (S^1 \times D^2)_k  \ni (\alpha_k ^{-1} \times \id_{D^2}) (x) \sim_E \fr(x) \in \Sigma^\circ \times D^2.\]
%\color{teal}
Recall the construction of the $D^2$-bundle $E$ as
 $$ E:=\left(\bigsqcup_{k=1}^{2c}\left( D^2 \times D^2 \right)_k  \sqcup \Sigma^\circ \times D^2\right)\Bigg/\sim_E,$$
where~$y \in \partial (D^2 \times D^2)_k $ is identified with~$\left((\alpha_k \times \id_{D^2}) \circ \eta^\pm\right)  (y) \in \partial B_k  \times D^2 \subset \Sigma^\circ \times D^2$.  
%%AC: $x$ was already in use so I changed that one to~$y$.
So,  for~$x \in \overline{\nu}(e|_{\partial B_k })$,  we need only check that
\[ \left((\alpha_k \times \id_{D^2}) \circ \eta^\pm \circ (\alpha_k ^{-1} \times \id_{D^2})\right)(x) = {\fr}(x). 
\]
This follows immediately from the fact that ${\fr}$ extends $\{\fr_{\partial B_k}\}$ and our definition of $\{\fr_{\partial B_k}\}$ in~\eqref{eq:fr}.

Next we assert that~$\gamma \colon \overline{\nu}(e) \xrightarrow{\cong} E$ induces a homeomorphism~$\widehat{\gamma} \colon \iota(\overline{\nu}(e)) \to W$ that makes the diagram in~\eqref{eq:IdentificationDiagram} commute.
Given~$x' \in \overline{\nu}(S)$, there exists~$x \in \overline{\nu}(e)$ so that~$\iota(x)=x'$.
The assertion will follow once we show that
$$\widehat{\gamma}(x'):= \left(\operatorname{proj} \circ \gamma\right)(x)$$
does not depend on the choice of the~$x \in \overline{\nu}(e)$ with $\iota(x)=x'$ and that the resulting map $\widehat{\gamma}$ is injective.
%\color{teal}
This is clear outside of the plumbed regions (as $\iota$ restricts to an embedding on $\overline{\nu}(e|_{\Sigma^\circ})$),  
so we let~$x \in \overline{\nu}(e|_{B_{2i-1}})$ and~$y \in \overline{\nu}(e|_{B_{2i}})$ 
and observe that 
%in order to show that 
%\[\iota(x)=\iota(y) \in \overline{\nu}(S) \iff \left(\proj \circ \gamma\right)(x) =\left( \proj \circ\gamma\right)(y) \in W.\]
\begin{align*}
\iota(x)=\iota(y) & \iff \left(\shrink_{2i-1} \circ(\alpha_{2i-1}^{-1} \times \id_{D^2})\right)(x)=\left(\shrink_{2i}\circ(\alpha_{2i}^{-1} \times \id_{D^2})\right)(y) \\
& \iff \left( \shrink_{2i-1} \ \circ\ \gamma\right)(x)=\left( \shrink_{2i} \ \circ \ \gamma \right)(y) \\
& \iff \gamma(x) \sim_W \gamma(y)\\
& \iff \left(\proj \circ \gamma\right)(x) =\left( \proj \circ\gamma\right)(y).
\end{align*}
Here the first equivalence follows from the fact that~$\iota$ has the effect of plumbing~$\overline{\nu}(e)$, as noted in the second item of Proposition~\ref{prop:NormalBundle};  the second equivalence follows from the fact that~$\gamma$ restricts to~$\alpha_k^{-1} \times  \id$ on the~$\overline{\nu}(e|_{B_k})$ for~$k \in \{2i-1,2i\}$; and the third equivalence follows from the definition of $\sim_W$.
This concludes the proof of the proposition.
\end{proof}

\subsection{Nice framings}
\label{sub:GoodFramings}

Proposition~\ref{prop:IdentifyW} shows that there is some flexibility in defining an identification $\widehat{\gamma} \colon \overline{\nu}(S) \cong W$,  coming from the choice of extension of the Seifert framing $\fr_\partial$ from Construction~\ref{cons:SeifertFraming} to a framing of all of $\nu(e|_{\Sigma^\circ})$.
The next definition describes a specific choice of framing that will be important to ensure the homomorphism
$$\pi_1(P) \to \pi_1(\partial \overline{\nu}(S)) \to \pi_1(N_S) \cong \Z$$
coincides with the one from Section~\ref{sub:CoefficientSystems}.

\medbreak

\begin{definition}(Nice framings.)
\label{def:NiceFraming}
A framing $\fr$ of $(\nu(e|_{\Sigma^\circ}),\iota)$ is \emph{nice} if
\begin{enumerate}
\item it extends the Seifert framing $\fr_{\partial}$ of $\nu(e|_{\partial \Sigma^\circ})$ from Construction~\ref{cons:SeifertFraming};
%%Don't delete.
%In particular it extends the framings $\{\fr_{\partial B_k}\}_k$ so can apply Proposition~\ref{prop:IdentifyW}
\item the homeomorphism $\widehat{\gamma}$ associated to $\fr$ by Proposition~\ref{prop:IdentifyW} satisfies 
$$\widehat{\gamma}^{-1}(c)=0 \in  H_1(N_{e(\Sigma)})$$
for every plumbing loop and genus loop $c \subset P$.
\end{enumerate}
\end{definition}

We describe how, on genus loops, the second condition can be restated using push-offs.
Recall that every~$(z,x) \in \Sigma^\circ \times D^2$ determines an element in~$E$ and that if~$x \in \partial D^2$,  one further obtains an element in~$P$.
Recall furthermore that if~$\eta \subset \Sigma^\circ$ is a curve and~$x \in \partial D^2$, then we write~$\eta \times \{x\}$ for the corresponding loop in~$P$.
The section \emph{associated to a framing~$\fr \colon \nu(e|_{\Sigma^\circ}) \to \Sigma^\circ \times \R^2$ and to an~$x \in \R^2$} is the section~$s_x \colon \Sigma^\circ \to \nu(e|_{\Sigma^\circ}), z \mapsto \fr^{-1}(z,x).$
%Framing because $\pi \circ \fr^{-1}(z,x)=\proj_1(z,x)=x. $

\begin{lemma}
\label{lem:NiceReformulationGenus}
On a genus loop~$\eta \times \{x\} \subset P$,  the second condition in the definition of a nice framing~$\fr$ is equivalent to requiring that for any section~$s_y \colon \Sigma^\circ \to \nu(e|_{\Sigma^\circ})$ associated to~$\fr$ and~$y \in \partial D^2$,  the push-off~$(\iota \circ s_y)(\eta)$ is nullhomologous in~$N_{e(\Sigma)}$:
$$(\iota \circ s_y)(\eta)=0 \in  H_1(N_{e(\Sigma)}).$$
\end{lemma}
\begin{proof}
The outcome of pushing off a curve $\gamma \subset e(\Sigma^\circ)$ using any section $s_y$ with $y \in \partial D^2$ results in isotopic (and in particular homologous) outcomes: a rotation takes one image to the other.
%%Don't delete
%Technically the push off is \iota \circ s_y \circ e^{-1}.
It therefore suffices to prove the lemma for~$y=x \in \partial D^2$.
But now, by definition of $\widehat{\gamma}$ on the surface~$\Sigma^\circ$, the condition that $\widehat{\gamma}^{-1}(\eta \times \{x\})=0$ can be rewritten as
$$0=\widehat{\gamma}^{-1}(\eta \times \{x\})=\iota \circ \fr^{-1}(\eta \times \{x\})=\iota \circ s_x(\eta).$$
\end{proof}

Next,  we formulate the analogue of Lemma~\ref{lem:NiceReformulationGenus} for plumbing loops.
Given a plumbing arc~$\omega^{\operatorname{arc}} \subset \Sigma$ and $x \in S^1$,  the arc $\omega^{\operatorname{arc}}  \times \{ x \} \subset \Sigma^\circ \times S^1$ defines a plumbing loop, denoted~$\omega \times \{ x \}$ in $P$.
Similarly,
$s_x(\omega^{\operatorname{arc}}) \subset \partial \overline{\nu}(e|_{\Sigma^\circ})$ is an arc,  but~$\iota \circ s_x(\omega^{\operatorname{arc}}) \subset N_{e(\Sigma)}$ is a loop.

\begin{lemma}
\label{lem:NiceReformulationPlumbing}
On a plumbing loop~$\omega \times \{x\} \subset P$,  the second condition in the definition of a nice framing~$\fr$ is equivalent to requiring that for any section~$s_y \colon \Sigma^\circ \to \nu(e|_{\Sigma^\circ})$ associated to~$\fr$ and~$y \in \partial D^2$,  the push-off~$\iota \circ s_y(\omega^{\operatorname{arc}})$ is nullhomologous in~$N_{e(\Sigma)}$:
$$(\iota \circ s_y)(\omega^{\operatorname{arc}})=0 \in  H_1(N_{e(\Sigma)}).$$
\end{lemma}
\begin{proof}
As in the proof of Lemma~\ref{lem:NiceReformulationGenus}, it suffices to consider the case $y=x$ and the result now follows by noting that~$0=\widehat{\gamma}^{-1}(\omega \times \{x\})=\iota \circ s_x(\omega^{\operatorname{arc}}).$
\end{proof}

Our aim is now to prove that nice framings exist.
We start with a lemma.

\begin{lemma}
\label{lem:RelativeEulerNumber}
The relative Euler number of~$\nu(e|_{\Sigma^\circ})$  with respect to the Seifert framing~$\fr_\partial$ is
$$ \operatorname{eul}(\nu(e|_{\Sigma^\circ}),\fr_\partial)=Q_N([S],[S])).$$
In particular, if $Q_N([S],[S]))=0$, then $\fr_\partial$ extends over $\Sigma^\circ$.
\end{lemma}
\begin{proof}
Since $e|_{\Sigma^\circ} \colon \Sigma^\circ \hookrightarrow N^\circ$ is an embedding,  the assertion can be equivalently stated and proved with~$S^\circ=e(\Sigma^\circ)$ in place of~$\Sigma^\circ$.
We start with the case where~$\Sigma$  is closed so that the normal Euler number, intersections and self-intersections of the immmersion~$e$ are related by
$$Q_N([S],[S])=2\mu(S)+\operatorname{eul}(\nu(e)).$$
Since the image has $c_+$ positive double points and $c_-$ negative double points, it follows that
\begin{equation}
\label{eq:EulerSelfInt}
\operatorname{eul}(\nu(e))=Q_N([S],[S])-2\mu(F)=Q_N([S],[S])-2(c_+-c_-).
\end{equation}
Next we calculate $\operatorname{eul}(\nu(e))$ in a different way,  by pushing $S$ off itself and counting intersections with the $0$-section. 
The Seifert framing $\fr_\partial$ determines (up to isotopy) a push-off of~$\partial S^\circ$.
% i.e. a section of~$\partial \overline{\nu}(F^\circ)$.
Push~$S$ off itself in a way that extends this push-off on the boundary, and denote the result by~$S'$.
This determines push offs~$D_i':=e(B_i)'$ of~$D_i:=e(B_i)$ and~$(S^\circ)':=e(\Sigma^\circ)'$ of~$S^\circ=e(\Sigma^\circ)$ and we calculate
\begin{align*}
Q_N([S],[S])-2(c_+-c_-)
=\operatorname{eul}(\nu(e)) 
=S \cdot S' 
=S^\circ \cdot (S^\circ )' +\sum_{i=1}^{2c_+} D_i \cdot D_i'+\sum_{i=2c_++1}^{2c_-} D_i \cdot D_i'.
\end{align*}
Using again the definition of the relative Euler number as the count of intersections of a push off with the $0$-section,  note that~$S^\circ \cdot (S^\circ )'=\operatorname{eul}(\nu(e|_{S^\circ}),\fr_\partial)$.
On the other hand, since the disks~$D_i$ are contained in a small $4$-ball, it follows that $D_i \cdot D_i'=\ell k(\partial D_i,\partial D_i')$.
Since our push-off of $F$ extends the (Seifert) framing $\fr_\partial$,  each of these linking numbers equals $\mp 1$, depending on the sign of the disks.
Indeed $\ell k(\partial D_i,\partial D_i')$ equals the Seifert framing of a $\pm$-Hopf link and this is $\mp 1$.
We therefore conclude that
\begin{align*}
Q_N([S],[S])-2(c_+-c_-)
%=\operatorname{eul}(\nu(e|_{S^\circ}),\fr_\partial)+\sum_{i=1}^{2c_+} \operatorname{eul}(D_i,\fr_\partial)+\sum_{i=1}^{2c_-} \operatorname{eul}(D_i,\fr_\partial \\
=\operatorname{eul}(\nu(e|_{S^\circ}),\fr_\partial)-2(c_+-c_-)
\end{align*}
and therefore~$\operatorname{eul}(\nu(e|_{S^\circ}),\fr_\partial)=Q_N([S],[S])$ as required.

We conclude by treating the case when $N$ has boundary $\partial N \cong S^3$ and~$\partial S=K \subset \partial N$ is a knot.
In this case, we cap off~$S$ with  a Seifert surface~$S_0$ for $K$ that we slightly push into $D^4$. 
Write~$\widehat{e}$ for the resulting immersion of a closed surface
and $\widehat{S}$ for its image so that 
$$\operatorname{eul}(\nu(\widehat{e}|_{\widehat{S}^\circ}))=Q_N([\widehat{S}],[\widehat{S}])=Q_N([S],[S]).$$
Reasoning with the $0$-section we have 
$$\operatorname{eul}(\nu(\widehat{e}|_{\widehat{\Sigma}^\circ}))=\operatorname{eul}(\nu(e|_{S^\circ}))+\operatorname{eul}(\nu(S_0),\fr_\partial).$$
As~$K$ is endowed with the Seifert framing,  $\operatorname{eul}(\nu(S_0),\fr_\partial)=0$ and the proposition follows.
\end{proof}

In the next proposition,  a \emph{$\Z$-immersion} refers to an immersion whose image is a $\Z$-surface.

\begin{proposition}
\label{prop:FramingPushoffUsingF}
A $\Z$-immersion $e \colon \Sigma \looparrowright N$ admits a nice framing
$$\fr \colon \nu(e|_{\Sigma^\circ}) \xrightarrow{\cong} \Sigma^\circ \times \R^2 .$$
\end{proposition}
\begin{proof}
 Lemma~\ref{lem:RelativeEulerNumber} ensures that~$\fr_\partial$ extends to a framing $\fr'$ on~$\Sigma^\circ$.
This lemma applies because immersed $\Z$-surfaces are nullhomologous; recall Proposition~\ref{prop:ZSurfaceAreNullhomologous}.
Let $s'_x$ be a nowhere zero section associated to this framing and to~$x \in \partial D^2$. 
The idea is to redefine this section on the genus and plumbing loops, so that the push offs become nullhomologous.

Since $\Sigma^\circ$ deformation retracts onto its $1$-skeleton~$(\Sigma^\circ)^{(1)}$,  the section $s'_x$ is determined by its values on $(\Sigma^\circ)^{(1)}$.
This $1$-skeleton is homotopy equivalent to the union of the genus loops $\eta_i$ together with the $\omega_i^{\operatorname{arc}}$ arcs,  the boundary components $\partial \Sigma^\circ \smallsetminus \partial \Sigma$,  and some arcs joining them.
If~$c$ is one of these loops or arcs,  then~$s'_x(c)$ takes values in~$\mathbb{S}(\nu(\Sigma^\circ)|_{c})$ which is a trivial circle bundle over $c$.
%Triviality is guaranteed by the existence of the framing.

Consider~$\mathbb{S}(\nu(\Sigma^\circ)|_{{\eta}})$ with~${\eta}$ a genus loop.
Using Lemma~\ref{lem:NiceReformulationGenus} we have~$\widehat{\gamma}^{-1}(\eta)=\iota \circ s_x'(\eta).$
%This is a trivial~$D^2$-bundle over a circle and therefore admits a~$\Z$-worth of framings.
The homology of the torus~$\mathbb{S}(\nu(\Sigma^\circ))|_{{\eta}}$ is generated by the~$S^1$-fibre and a longitude.
The longitude is not canonical and is determined by choosing one of the~$\Z$'s worth of framings of the trivial~$\R^2$-bundle~$\nu(\Sigma^\circ)|_{{\eta}}.$
If a section~$s$ is nice, then each of the~$\iota \circ s(\eta)$ has no meridional component in homology.
While this might not be the case for~$s'_x$,  this can be arranged by redefining it on the~$\eta$ by adding multiples of the~$S^1$-fibre to~$s'_x(\eta)$.

Next we consider~$\widehat{\gamma}^{-1}(\omega \times \{x\}) \in H_1(N_S) \cong \Z$ with~$\omega \times \{x\} \subset P$ a plumbing loop.
Using Lemma~\ref{lem:NiceReformulationPlumbing} we have~$(\widehat{\gamma}_*)^{-1}(\omega)=\iota \circ s_x(\omega^{\operatorname{arc}}).$
The framing is fixed on the endpoints because, no matter the framing of~$\Sigma^\circ$,  the remainder of~$\widehat{\gamma}$ is defined using the~$\alpha_k$.
We can nonetheless modify the homology class of~$\widehat{\gamma}^{-1}(\omega \times\{x\})=\iota \circ s_x'(\omega^{\operatorname{arc}})$ by adding multiples of the~$S^1$-fibre to~$s'_x(\omega^{\text{arc}})$.
Since~$H_1(N_{e(\Sigma)})\cong \Z$ is generated by the meridian,  this allows us to modifiy~$(\widehat{\gamma})^{-1}(\omega) \in H_1(N_{e(\Sigma)})$ until it vanishes.

Since~$(\nu(e|_{\Sigma^\circ}),\iota|)$ is an oriented vector bundle,  we can apply a rotation to our newly obtained ``nice section" to obtain a second nowhere zero nice section.
Together, these sections give rise to the required nice framing.
%Implicitly using the lemmas to verify the second condition.
\end{proof}

\begin{remark}
\label{rem:ClosedFraming}
We note that the results of this section also hold for immersions of closed surfaces in closed simply-connected~$4$-manifolds.
The proofs and constructions are entirely analogous (and are in fact simpler since there is no knot $K$) with $E_U,W_U$ and $P_U$ in place of $E,W$ and $P$.

In particular,  nice framings are defined similarly and, given a closed simply-connected~$4$-manifold $X$ and a $\Z$-immersion $e \colon \Sigma_g \looparrowright X$ that is compatible with $\alpha$, Proposition~\ref{prop:IdentifyW} can be adapted to show that under the same assumptions, a nice framing extends to a homeomorphism~$\gamma \colon \overline{\nu}(e) \to E_U$ that, in the presence of a family of double point charts for $(e,\alpha)$ and a normal bundle $(\nu(e),\iota)$ further gives rise to a homeomorphism~$\widehat{\gamma} \colon \overline{\nu}(S) \to W_U$ that then restricts to a homeomorphism $\widehat{\gamma} \colon \partial \overline{\nu}(S) \xrightarrow{\cong} P_U.$
\end{remark}

\section{Immersed $\Z$-surfaces and $\Z$-fillings of $P_K$.}
\label{sec:MainStatement}

The goal of this section is to prove that equivalence classes of immersed $\Z$-surfaces correspond bijectively to fillings of $P_K$ with fundamental group $\Z$.
Section~\ref{sub:StatementBijection} introduces the notation needed to make this statement precise in Theorem~\ref{thm:SurfacesManifolds},  whereas Sections~\ref{sub:ImmersionsManifolds}-\ref{sub:ImmersionsToSurfaces} are concerned with the proof of Theorem~\ref{thm:SurfacesManifolds}.
The modifications needed to obtain the analogues of these results for immersions of closed surfaces will be discussed in Section~\ref{sub:ProofClosed}.

\subsection{The  statement}
\label{sub:StatementBijection}

The goal of this section is to describe a bijection between the set of equivalence  classes rel.\ \ boundary of immersed $\Z$-surfaces and the set of rel. boundary homeomorphism classes of pairs~$(V,f)$, where $V$ is a $\Z$-filling of $P_K$ and $f \colon \partial V \to P_K$ is a homeomorphism. 
We recall the notation needed to make this statement more precise.
%Some of this notation was introduced incrementally in the previous sections but we recall it here for the reader's convenience. 

\begin{notation}
Fix a knot $K \subset S^3$,  a tubular neigbhorhood $\overline{\nu}(K)$ and write $E_K$ for the knot exterior.
We also fix a meridian $\mu_K$,  a $0$-framed longitude $\lambda_K$ of $K$, and
\begin{itemize}
\item non-negative integers~$c,c_+$ and~$c_-$ such that~$c=c_++c_-$,
\item a compact oriented surface~$\Sigma:=\Sigma_{g,1}$ with a single boundary component,
\item an embedding $\alpha \colon \bigsqcup_{2c} D^2 \hookrightarrow \Sigma$ with $\im(\alpha)\subset  \intt(\Sigma)$.
Set $B_i:=\alpha_i(D^2)$ for~$i=1,\ldots,2c$ and label the first~$2c_+$ of these disks with a sign~$+$ and the remaining~$2c_-$ with a sign~$\unaryminus$. 
\end{itemize}
Recall the plumbed $3$-manifold $P$ from Construction~\ref{cons:W} and the homeomorphism $d_K \colon \partial  E_K \to \partial  P$ with $d_K(\mu_K)=\mu$ and $d_K(\lambda_K)=\partial \Sigma \times \lbrace \pt \rbrace.$
As in Construction~\ref{cons:PK}, consider the closed $3$-manifold 
$$P_K:=E_K \cup_{d_K} P.$$
We remind the reader that the manifolds $P:=P_g(c_+,c_-)$ and $P_K:=P_{K,g}(c_+,c_-)$ depend on the genus $g$ of the surface~$\Sigma:=\Sigma_{g,1}$,  the integers $c_+,c_- \geq 0$ and on the fixed embedding $\alpha.$

As proved in Proposition~\ref{prop:HomologyPK}, the group $H_1(P_K)$ sits in an exact sequence
$$0 \to  \Z \xrightarrow{j} H_1(E_K) \oplus H_1(P) \to H_1(P_K) \to 0 $$
where $j(1)=(\mu_K,\mu)$.  
The inclusions therefore lead to an isomorphism
$$H_1(P_K) \cong H_1(P)/\Z\langle \mu \rangle \oplus \Z\langle \mu_K \rangle.$$
Consider the epimorphism from Section~\ref{sub:CoefficientSystems}
\begin{equation}
\label{eq:CoeffSys}
 \varphi \colon H_1(P_K) \to \Z
 \end{equation}
that vanishes on the genus and plumbing loops and satisfies~$\varphi(\mu_K)=1.$
\end{notation}

\begin{notation}
\label{def:Surface(g)RelBoundary}
For a nondegenerate hermitian form~$\lambda$ over $\Z[t^{\pm 1}]$, set
\begin{align*}
 \operatorname{Surf}(g;c_+,c_-)^0_\lambda(N,K)&:= \lbrace [S] \in \Surf^0(g;c_+,c_-) \mid  \lambda_{N_S}\cong \lambda \rbrace
 \\
%%%%%
\mathcal{V}_\lambda^0(P_K)&:=
\frac{\lbrace (V^4,f) \mid \pi_1(V) \cong \Z \text{ and } f \colon \partial V \xrightarrow{\cong} P_K \text{ and } \lambda_{V}\cong \lambda \rbrace}
{\text{homeomorphism rel.~$\partial$}}. 
\end{align*}
Here, when we write $f \colon \partial V \xrightarrow{\cong} P_K$, we require~$\pi_1(P_K) \xrightarrow{f_*} \pi_1(\partial V) \to \pi_1(V) \cong \Z$
% to 
agree with the homomorphism $\varphi \colon \pi_1(P_K) \to \Z$ from~\eqref{eq:CoeffSys}. 
 It is also understood that $f$ is orientation-preserving.
Also, two pairs $(V_1,f_1)$ and $(V_2,f_2)$ are said to be \emph{homeomorphic rel. boundary} if there is a homeomorphism $F \colon V_1 \to V_2$ with $f_2 \circ F|_{\partial V_1} =f_1$.
%The condition involving the Blanchfield form is necessary for these sets to be nonempty.

Finally, we set $\varepsilon:=\ks(N)$ and 
$$\mathcal{V}_\lambda^{0,\varepsilon}(P_K)=\{[V,f] \in \mathcal{V}_\lambda^0(P_K) \mid \ks(V)=\varepsilon \}.$$
Note that when no further condition is imposed on $\lambda$, these sets might be empty. 
\end{notation}

\begin{construction}[The action of $\Homeo_\alpha(\Sigma,\partial)$ on $\mathcal{V}_\lambda^0(P_K)$]
\label{cons:ActionHomeo}
Recall that $\Homeo_\alpha(\Sigma,\partial)$ denotes the group of rel. boundary homeomorphisms of $\theta \colon \Sigma \to \Sigma$ that satisfy $\theta \circ \alpha=\theta$.
Recall that a homeomorphism~$\theta \in \Homeo_\alpha(\Sigma,\partial)$ gives rise to a homeomorphism $\widehat{\theta} \colon W \to W$: consider the homeomorphism~$\theta|_{\Sigma^\circ} \times \id_{D^2}$ of $\Sigma^\circ \times D^2$,  and extend over each $D^2 \times D^2$ by the identity to get a homeomorphism of~$E$. Finally,   verify that this homeomorphism descends to $ \widehat{\theta} \colon W \to W$.
Slightly abusing notation we also write $\widehat{\theta} \colon P_K \to P_K$ for the homeomorphism obtained by extending~$\widehat{\theta}| \colon P \to P$ by the identity on $E_K$.
Define 
$$\theta \cdot (V,f):=(V,\widehat{\theta} \circ f).$$
To verify this is an action, one notes the equalities~$\widehat{\id_{\Sigma}}=\id_W$ and $\widehat{\theta \circ \theta'}=\widehat{\theta} \circ \widehat{\theta}'$.
%%Don't delete.
%\widehat{\theta \circ \theta'}
%=proj_W \circ ((\theta \circ \theta') \times \id_{D^2} \sqcup \id_{D^2 \times D^2})
%=proj_W \circ \theta \times \id_{D^2} \sqcup \id_{D^2 \times D^2}) \circ  proj_W \circ \theta' \times \id_{D^2} \sqcup \id_{D^2 \times D^2}) 
%\widehat{\theta} \circ \widehat{\theta'}
\end{construction}

We are now ready to state the main result of this section,  a more precise version of which will be formulated in Theorem~\ref{thm:SurfacesManifoldsMainProof}.
\begin{theorem}
\label{thm:SurfacesManifolds}
Let~$N$ be a simply-connected~$4$-manifold with boundary~$\partial N \cong S^3$,  let~$K \subset S^3$ be a knot,  let $c_+,c_-$ and $g$ be non-negative integers, and let~$\lambda$ be  a nondegenerate hermitian form with $\lambda(1) \cong Q_N \oplus (0)^{2g+c}$, where $c:=c_++c_-.$

Taking the exterior of an immersed surface (see Construction~\ref{cons:EmbVBijection}) gives rise to a bijection
$$
\operatorname{Surf}_\lambda^0(g;c_+,c_-)(N,K) \xrightarrow{\cong}
\begin{cases}
\mathcal{V}_\lambda^0(P_K)/\operatorname{Homeo}_\alpha(\Sigma,\partial) &\quad \text{ if $\lambda$ is even}
\\
 \mathcal{V}_\lambda^{0,\varepsilon}(P_K)/\operatorname{Homeo}_\alpha(\Sigma,\partial)
 &\quad \text{ if $\lambda$ is odd.}
\end{cases}
$$
\end{theorem}
%{AM: Just to spell this out because I've gotten myself confused later on: what we are most directly defining is an isomorphism 
%\[\Theta \colon \text{Imm}_{\alpha}/ \text{Homeo}_{\alpha} \to \mathcal{V}/ \text{Homeo}_{\alpha}.\] 
% (Add all the appropriate sub/super scripts in.) And then applying our section 3 results that $\Surf \cong \Imm/ \operatorname{Homeo} \cong \Imm_\alpha/ \operatorname{Homeo}_\alpha$. 
%In particular,  this is why when we're defining $\Psi$,  we need to show that $\Psi(V,f)$ lands in $\Imm_\alpha$,  not  just $\Imm$. 
%If this is correct, I think perhaps we should say this?
%{AC: Right. Happy to have this written in.
%Also the first step is missing an $\alpha$.
%I think the end of this section was written before we had a section devoted to immersions so we need to amend it.
%}
%}

The proof shows that if $\operatorname{Surf}_\lambda^0(g;c_+,c_-)(N,K)$ is nonempty, then so is the target of the bijection (without any assumption on~$\lambda$).  
The converse requires the assumption that~$\lambda(1) \cong Q_N \oplus  (0)^{\oplus 2g+c}.$

Theorem~\ref{thm:SurfacesManifolds} will be proved in three steps (here, we focus on the case where $\lambda$ is even for simplicity).
\begin{enumerate}
\item Section~\ref{sub:ImmersionsManifolds} defines a map $\Theta \colon \operatorname{Imm}_\alpha(g;c_+,c_-)_\lambda^0(N,K) \to \mathcal{V}_\lambda^0(P_K)$.
\item Section~\ref{sub:ManifoldsImmersions} proves that the map $\Theta$ is a bijection, by defining an inverse $\Psi$.
\item 
Section~\ref{sub:ImmersionsToSurfaces} argues that $\Theta$ intertwines the $\operatorname{Homeo}_\alpha(\Sigma,\partial)$-actions and thus descends to 
%the bijection claimed in Theorem~\ref{thm:SurfacesManifolds}:
$$\operatorname{Surf}_\lambda^0(g;c_+,c_-)(N,K)  \cong \frac{\operatorname{Imm}_\alpha(g;c_+,c_-)_\lambda^0(N,K)}{\Homeo_\alpha(\Sigma,\partial)} \xrightarrow{\cong} \frac{\mathcal{V}_\lambda^0(P_K)}{\Homeo_\alpha(\Sigma,\partial)}.$$
Here, the first bijection was proved in Proposition~\ref{prop:ImmIsImmalpha}.
\end{enumerate}

\subsection{From immersions to manifolds.}
\label{sub:ImmersionsManifolds}

This section defines the bijection $\Theta$ that will be used to prove Theorem~\ref{thm:SurfacesManifolds}.
During the remainder of this section we fix a homeomorphism~$h \colon \partial N \xrightarrow{\cong} S^3$ and, slightly abusively,  we write $K \subset \partial N$ instead of $h^{-1}(K)$.
Recall that we have a fixed tubular neighborhood $\overline{\nu}(K) \subset S^3$,  which pulls back under $h$ to give a tubular neighbohood $h^{-1}(\overline{\nu}(K)) \subset N$.
We will use~$E_K$ exclusively to denote~$S^3 \smallsetminus \intt(\overline{\nu}(K))$.

%\medbreak
\begin{convention}
Given an immersion $e$, in this section, we use $\nu(e)$ to denote both the normal vector bundle of $e$ (this is an $\R^2$-bundle) and the interior of the disk bundle~$\overline{\nu}(e)$.
We do so to avoid notation such as $\intt(\iota(\overline{\nu}(e))).$
\end{convention}

We begin the construction of the bijection $\Theta$.

\begin{construction}
\label{cons:EmbVBijection}
Given an $\alpha$-compatible immersion $e \colon \Sigma \looparrowright N$, we write $N_{e(\Sigma)}$ for the complement of a yet-to-be-chosen tubular neighborhood of $e(\Sigma)$
and aim to construct a homeomorphism
$$f \colon \partial N_{e(\Sigma)} \to P_K.$$
The pair $(N_{e(\Sigma)},f)$ will depend on several choices, but we will show that the outcome is independent of these choices up to equivalence rel.\  boundary. 
In particular,  we choose:
\begin{itemize}
\item A family $\mathcal{U}$ of double point charts for  $(e,  \alpha)$.
\item A normal bundle $(\nu(e),\iota)$ as in Proposition~\ref{prop:NormalBundle} with
$\nu(e|_{\partial \Sigma})=\partial \Sigma \times \R^2.$
\item A nice framing
$\fr\colon \nu(e|_{\Sigma^\circ}) \xrightarrow{\cong} \Sigma^\circ \times \R^2$.
\end{itemize}

We also record some objects and notation that stem from these choices.
\begin{itemize}
\item The boundary of the surface exterior $N_{e(\Sigma)}=N\setminus \iota(\nu(e))$ decomposes as
%\begin{equation}
%\label{eq:DecompositionBoundarySurfaceExterior}
%\partial N_{e(\Sigma)} \cong \big(\partial N \setminus \nu(K)\big) \cup
% \big(\partial \iota(\overline{\nu}(e))  \cap \intt(N)\big).
%\end{equation}
\begin{equation}
\label{eq:DecompositionBoundarySurfaceExterior}
\partial N_{e(\Sigma)} \cong \big(\partial N \setminus \nu(K)\big) \cup \Big{(}\partial \iota(\overline{\nu}(e))\smallsetminus \left( \iota(\nu(e)) \cap \partial N \right)\Big{)}.
\end{equation}
Here, the first part of this union is homeomorphic to the knot exterior $S^3 \setminus \nu(K)$ via 
$$h| \colon \partial N \smallsetminus \nu(K) \to E_{K} \subset P_K.$$
while the second is homeomorphic to the plumbed $3$-manifold $P$.
%\item Restricting our fixed homeomorphism $h \colon \partial N \xrightarrow{\cong} S^3$ to the knot exterior part in~\eqref{eq:DecompositionBoundarySurfaceExterior}, we obtain the homeomorphism
%$$h| \colon \partial N \smallsetminus \nu(K) \to E_{K} \subset P_K.$$
\item Recall from Proposition~\ref{prop:IdentifyW} that the nice framing $\fr \colon \nu(e|_{\Sigma^\circ}) \cong \Sigma^\circ \times D^2$ extends to a homeomorphism $\gamma \colon \ol{\nu}(e)\to E$ which itself induces a homeomorphism $\widehat{\gamma} \colon \ol{\nu}(S) \to W$ that fits in the commutative diagram
\begin{equation}
\label{eq:IdentificationDiagram}
\xymatrix{
\overline{\nu}(e) \ar@{->>}[r]^-{\iota}\ar[d]^{\gamma}_\cong& \overline{\nu}(S) \ar[d]^{\widehat{\gamma}}_\cong \\
E \ar@{->>}[r]^{\operatorname{proj}}& W.
}
\end{equation}
The construction of $\widehat{\gamma}$ depends on the family of double point charts $\mathcal{U}$.

On the plumbed part of~\eqref{eq:DecompositionBoundarySurfaceExterior},  this homeomorphism restricts to
  \[ \widehat{\gamma}| \colon \Big{(}\partial \iota(\overline{\nu}(e))\smallsetminus \left( \iota(\nu(e)) \cap \partial N \right) \Big{)} \to P \subset P_K.\]
\end{itemize}
We claim that $h|$ and $\widehat{\gamma}|$ can be glued,  yielding the homeomorphism we have been building towards:
\begin{equation}
\label{eq:BoundaryHomeoSurface}
f:=h| \cup \widehat{\gamma}|\colon \partial N_{e(\Sigma)} \to P_K.
 \end{equation}
Recall that $\partial N_{e(\Sigma)} \cong \big(\partial N \setminus \nu(K)\big) \cup \Big{(}\partial \iota(\overline{\nu}(e))\smallsetminus \left( \iota(\nu(e)) \cap \partial N \right)\Big{)}$, glued along  the boundary of tubular neighborhood of $K \subset \partial N$.
On this torus, viewed as a subset of the second part of this union,  Construction~\ref{cons:SeifertFraming} ensures that $\iota= h^{-1} \circ (\fr_{K})^{-1}$,  which is certainly invertible.
Since we also know that on this torus,  $\gamma$ restricts to $\fr_{\partial \Sigma}= \id$, Proposition~\ref{prop:IdentifyW} implies that
\[\widehat{\gamma}|=\gamma \circ \iota^{-1}
= \id \circ (h^{-1} \circ (\fr_{K})^{-1})^{-1}
= \fr_{K} \circ h
= d_K \circ h,
\]
where for the last equality we used that~$\fr_{K}|_{\partial \overline{\nu}(K)}= d_K$. 
This implies that the maps $\widehat{\gamma}$ on $P$ and~$h$ on~$E_K$ glue together to give a well-defined map on $P_K= E_K \cup_{d_K} P$, as claimed.
\end{construction}

%We note that $(N_{e(\Sigma)},f)$ determines an element in $\mathcal{V}^0(P_K)$. 

\begin{remark}
\label{rem:fonpi1}
We argue that the following composition agrees with $\varphi \colon H_1(P_K) \to \Z$:
$$ \pi_1(P_K) \xrightarrow{f_*^{-1}} \pi_1(\partial N_{e(\Sigma)}) \to \pi_1(N_{e(\Sigma)}) \xrightarrow{\cong} \Z.$$
This map factors though $H_1(P_K)$ which, recall from Proposition~\ref{prop:HomologyPK} is freely generated by the meridian of $K$, the genus loops and the plumbing loops.

It therefore suffices to prove that the above homomorphism maps $\mu_K$ to $1$ and the other generators to zero.
For the meridian, $f^{-1}(\mu_K)=d_K^{-1}(\mu_K)=(\{ \pt\} \times \partial D^2)$ is a meridian of $e(\Sigma)$ and is therefore mapped to $1$.
For the genus and plumbing loops, this follows because $\widehat{\gamma}$ was constructed using a nice framing.
\end{remark}
\color{black}

The next two propositions show that the assignment $e \mapsto  (N_{e(\Sigma)},f)$ gives rise to a map
%well defined map
$$\operatorname{Imm}_{\alpha}(g;c_+,c_-)_\lambda^0(N,K) \to \mathcal{V}^0_\lambda(P_K).$$
Note that Remark~\ref{rem:fonpi1} ensures that $(N_{e(\Sigma)},f)$ lies in $\mathcal{V}^0_\lambda(P_K)$.

\begin{proposition}
\label{prop:PhiWellDefined}
Up to equivalence rel.\ \ boundary, the pair $(N_{e(\Sigma)}, f)$ depends neither on the choice of the double point charts $\mathcal{U}:=\{ (U_p,\psi_p) \}$, nor on the choice of the $\mathcal{U}$-adapted immersion~$\iota \colon \nu(e) \looparrowright N$ nor on the nice framing~$\fr \colon \nu(e|_{\Sigma^\circ}) \cong \Sigma^\circ \times \R^2$.
\end{proposition}
\begin{proof}
Assume that~$\mathcal{U}:=\{ (U_p,\psi_p) \}$ and $\mathcal{U}':=\{ (U_p',\psi_p') \}$ are two families of double point charts for $(e, \alpha)$  and that $\iota$ and $\iota'$ are two immersions of $\nu(e)$ that are $\mathcal{U}$ and $\mathcal{U}'$-adapted respectively.
Set 
$$N^\circ:=N \setminus \bigcup_p U_p \ \ \text{and} \ \ (N^\circ)':=N \setminus \bigcup_p U_p'.$$
Choose nice framings~$\fr$ and $\fr'$ (for $\iota$ and $\iota'$ respectively) extending the Seifert framing.
%%Don't delete just yet.
%as in~\eqref{eq:Compatible}.
Lower in the proof, we will write $\widehat{\gamma}$ and $\widehat{\gamma}'$ for the corresponding homeomorphisms into $P$.

Choose a homeomorphism~$H \colon N \to N$  as in Lemma~\ref{lem:Homeo}.
Since $H$ takes the $U_p$ to the $U_p'$, we obtain two normal bundles
 of $\Sigma^\circ$ inside $(N^\circ)'$ namely
\begin{align*}
& H \circ \iota \circ {\fr}^{-1} \colon \Sigma^\circ \times \R^2 \hookrightarrow (N^\circ)'\\
&  \iota'  \circ ({\fr}')^{-1} \colon \Sigma^\circ \times \R^2 \hookrightarrow (N^\circ)'.
\end{align*}
The uniqueness of normal bundles from~\cite[Theorem 9.3]{FreedmanQuinn} ensures the existence of a rel. boundary homeomorphism 
$G_0 \colon (N^\circ)' \to (N^\circ)'$ with 
\begin{equation}
\label{eq:IsotopyPunctured}
G_0 \circ (H \circ \iota \circ {\fr}^{-1})= \iota'  \circ ({\fr}')^{-1}.
\end{equation}
Since $G_0$ restricts to the identity on $\partial (N^\circ)'=\partial N \sqcup \bigsqcup_p \partial U_p'$, it follows that we can extend it to a homeomorphism~$G \colon N\to N$ by defining $G|_{U_p'}:=\id_{U_p'}$ on each of the $U_p'$.
\begin{claim} \label{claim:GPresNbhd}
 $G$ preserves tubular neighbhorhoods, i.e.~$G(H(\iota(\overline{\nu}(e)))=\iota'(\overline{\nu}(e)).$
\end{claim}
\begin{proof}
Recall that the linear disk bundle $\overline{\nu}(e)$ decomposes as $\overline{\nu}(e|_{\Sigma^\circ}) \cup \overline{\nu}(e|_{\bigsqcup_k B_k})$ and that, by virtue of being $\mathcal{U}$-adapted, the immersion $\iota \colon \overline{\nu}(\Sigma) \looparrowright N$ satisfies, for $k \in \{ 2i-1,2i \}$,
$$  \iota(\overline{\nu}(e|_{\Sigma^\circ})) \subset N^\circ
 \ \ \text{and} \ \
\iota(\overline{\nu}(e|_{B_k})) \subset U_i. $$
We will verify the claim on these two parts of $\overline{\nu}(e)$.
We verify that $G(H(\iota(\overline{\nu}(e|_{\Sigma^\circ})))=\iota'(\overline{\nu}(e|_{\Sigma^\circ}))$.
This follows because $G$ was obtained by extending~$G_0$ and by recalling~\eqref{eq:IsotopyPunctured}:
$$
G(H(\iota(\overline{\nu}(e|_{\Sigma^\circ})))
=G \circ H \circ \iota \circ \fr^{-1}(\Sigma^\circ \times D^2)
=\iota' \circ (\fr')^{-1}(\Sigma^\circ \times D^2)
%%Replaced $\fr$ by $\fr'$.
=\iota'(\overline{\nu}(e|_{\Sigma^\circ})).
 $$
We now verify that $G(H(\iota(\overline{\nu}(e|_{B_j})))=\iota'(\overline{\nu}(e|_{B_j})).$
For $k \in \{2i-1,2i \}$, we have that $\iota(\overline{\nu}(e|_{B_k})) \subset U_i$ and $\iota'(\overline{\nu}(e|_{B_k})) \subset U_i'$;  $H|_{U_i}=(\psi_i')^{-1} \circ \psi_i$ takes $U_i$ to $U_i'$; and~$G|_{U_i'}=\id$, so this further reduces to showing that $(\psi_i')^{-1} \circ \psi_i \circ \iota=\iota'$ on $\overline{\nu}(e|_{B_k})=B_k \times D^2$ i.e.  that  $\psi_i \circ \iota=\psi_i' \circ \iota'$ on each $B_k \times D^2$.

By definition of an adapted normal bundle (recall Definition~\ref{def:UAdapted})
% and Proposition~\ref{prop:NormalBundle}) 
we have
$$ \psi_i \circ \iota \circ (\alpha_k \times \id_{D^2})
=\shrink_k
=\psi_i' \circ \iota' \circ (\alpha_k \times \id_{D^2}).$$
As $\alpha_k \times \id_{D^2}$ is a homeomorphism onto $B_k \times D^2$,  the equality $(\psi_i \circ \iota)|_{B_k \times D^2}=(\psi_i' \circ \iota')|_{B_k \times D^2}$~follows.
\end{proof}

\color{black}
 Claim~\ref{claim:GPresNbhd}
 implies that $G \colon N \to N$ restricts to a homeomorphism $N \setminus H(\iota(\nu(e))) \to N \setminus \iota'(\nu(e))$.
The situation we have so far can be summarized in the following diagram:
$$
\xymatrix@C1.5cm{
P_K \ar[d]^= 
& \partial (N \setminus \iota(\nu(e))) \ar[r]^-{\subset}\ar[d]^-{H} \ar[l]_-{h \cup \widehat{\gamma}} 
& N \setminus \iota(\nu(e)) \ar[r]^-{\subset}\ar[d]^-{H}  
& N \ar[d]^-{H} \\
%%%%%
P_K 
\ar[d]^= 
& \partial (N \setminus H(\iota(\nu(e)))) \ar[r]^-{\subset} \ar[d]^-{G}  \ar[l]_-{h \cup (\widehat{\gamma} \circ H^{-1})} 
& N \setminus H(\iota(\nu(e))) \ar[r]^-{\subset} \ar[d]^-{G}
&  N \ar[d]^-{G} \\
%%%%%
P_K 
& \partial (N \setminus \iota'(\nu(e))) \ar[r]^-{\subset}  \ar[l]_-{h \cup \widehat{\gamma}'} 
& N \setminus \iota'(\nu(e)) \ar[r]^-{\subset}
 & N. \\
}
$$
We now will be done if we can show that this diagram commutes.
The only real substance lies in verifying that the bottom left square commutes.
Since $G|_{\partial N}=\id_{\partial N}$, we have $h \circ G=h \colon \partial N \to S^3$ and so this further reduces to proving that
$$G \circ H \circ \widehat{\gamma}^{-1}=\widehat{\gamma}'^{-1}.$$
Recall that the domain of these functions, namely the plumbed $3$-manifold~$P$, is defined as a subset of $\partial W$ where $W=E/\sim_W$ is a quotient of the $D^2$-bundle $E=(\bigsqcup (D^2 \times D^2)_i \sqcup \Sigma^\circ \times D^2)/\sim_E$.
We will therefore verify the claim on (the image of)~$\Sigma^\circ \times D^2$ and then on (the image of)~$D^2 \times D^2$.
Since the functions involved descend to $P$,  our result will follow.
%{AC: It's a bit delicate because these are neither subsets of $E$ nor of~$W$. That's why there's language so as ``on the part coming from $\Sigma^\circ \times D^2$" and ``that arise from the $D^2 \times D^2$". It would be nice to nail this down some more. I also replaced several $B_i \times D^2$ by $D^2 \times D^2$ since $E$ involves the latter and not the former.}
%{AC: One has to double check this makes sense.}
%This could be improved.

On the part coming from $\Sigma^\circ \times D^2$, the homeomorphism $\widehat{\gamma}^{-1}$ is induced from~$\iota \circ {\fr}^{-1}$ (recall Proposition~\ref{prop:IdentifyW}).
As~$G \colon N \to N$ is obtained by extending $G_0 \colon (N^\circ)' \to (N^\circ)'$, the equality follows from~\eqref{eq:IsotopyPunctured} which stated that
$$G_0 \circ (H \circ \iota \circ \fr^{-1})= \iota'  \circ (\fr')^{-1}.$$

We now verify the equality on the subsets of $P$ that arise from the $D^2 \times D^2$.
For $k \in \{2i-1,2i\},$
we have to show that the outer square of the following diagram commutes:
$$
\xymatrix@R0.5cm{
D^2 \times D^2\ar[dd]^=
& B_k \times D^2  \ar[l]_{\widehat{\gamma}} \ar[r]^-{\iota} 
%\ar[dd]^{=}
&U_i \ar[rr]^{H} \ar[rd]^{\psi_i}
&
&U_i' \ar[dd]^{G|_{U_i'}=\id} \ar[ld]^{\psi_i'} \\
%%%%%%%%%%%%%%%%%%%%%
&
&
& D^2\times D^2
& \\
%%%%%%%%%%%%%%%
D^2 \times D^2 
&  B_k \times D^2 \ar[rrr]^{\iota'} \ar[l]_{\widehat{\gamma}'}
&
&
&U_i'.  \ar[ul]^{\psi_i'}
}
$$
Since Lemma~\ref{lem:Homeo} ensures that~$H|_{U_i}= (\psi_i')^{-1} \circ \psi_i$, the upper triangle commutes.  The rightmost triangle clearly commutes,  leaving us to consider the left part of the diagram. 
By Proposition~\ref{prop:IdentifyW} we have that  on $B_k \times D^2 \subset \Sigma \times D^2$,  the homeomorphisms~$\widehat{\gamma}$ and $\widehat{\gamma}'$ both agree with  
 $\alpha_k^{-1} \times\id_{ D^2}$,  independently of $\fr$ and $\fr'$.  
 Combining this with the second condition of 
Definition~\ref{def:UAdapted},  we see that the left hand side of the diagram commutes:  
$$ \psi_i \circ \iota \circ \widehat{\gamma}^{-1}
= \psi_i \circ \iota \circ (\alpha_k \times \id_{D^2})
=\shrink_k
=\psi_i' \circ \iota' \circ (\alpha_k \times \id_{D^2})
=\psi_i' \circ \iota' \circ \widehat{\gamma}'^{-1}.$$
This concludes the proof of the proposition.
\end{proof}

The next proposition concludes the proof that the assignment $e \mapsto  (N_{e(\Sigma)},f)$ gives rise to a well defined map $\operatorname{Imm}_{\alpha}(g;c_+,c_-)_\lambda^0(N,K) \to \mathcal{V}^0_\lambda(P_K)$.

\begin{proposition}
\label{prop:ThetaWellDefEnd}
Up to equivalence rel. boundary, the pair $(N_{e(\Sigma)}, f)$ does not depend on the equivalence rel. boundary type of the immersion~$e$. 
\end{proposition}
\begin{proof}
Assume that~$e,e' \colon \Sigma \looparrowright N$ are immersions compatible with~$\alpha$ that are homeomorphic rel.\ boundary via a homeomorphism~$H \colon N \to N$ with~$H \circ e=e'$.
Pick a family of double point charts~$\mathcal{U}':=\{ (U_p',\psi_p')\}_p$ for~$(e',\alpha)$,  an~$\mathcal{U}'$-adapted normal bundle~$(\nu(e'),\iota')$,  and a nice framing~${\fr}'$ for~$\nu(e'|_{\Sigma^\circ})$.
In what follows, we will define a family of double point charts for~$(e,\alpha)$, an~$\mathcal{U}$-adapted normal bundle~$(\nu(e),\iota)$ for~$e$, and a nice framing for~$\nu(e|_{\Sigma^\circ})$ with respect to~$\iota$.

Consider the family of double point charts~$\mathcal{U}:=\{ (H^{-1}(U_p'),\psi_p' \circ H)\}_p$ for~$(e,\alpha)$,  and the normal bundle~$(\nu(e),\iota)$ for~$e$ given by~$\nu(e):=\nu(e')$, with the same projection onto~$\Sigma$,  and
$$\iota:= H^{-1} \circ \iota' \colon \nu(e)=\nu(e') \looparrowright N.$$
%%%Don't delete
%The reader can verify that if $(\pi' \colon \nu(e') \to \Sigma,\iota')$ is a normal bundle for $e'$, then $(\pi',H^{-1} \circ \iota')$ is a normal bundle for $e=H^{-1} \circ e'$
%%with $0$-section $s_0=s_0'$
%One can verify that if $U'$ a trivialising sets 
%%(NOT THE SAME AS DOUBLE POINT CHARTS despite notation) 
%with $F \colon E|_{U'} \to U' \times \R^2$ for $\pi$, then trivially $U=U'$ and $F=F'$ gives the corresponding data for $\pi=\pi'$.
%The $0$-section satisfies $F \circ  s_0(p)=(p,0)$.
%Thus $s_0=s_0'$ is trivially still a $0$-section.
%To check that $\iota'$ works again note that $\iota' \circ s_0'=e'$ implies $\iota \circ s_0=H^{-1} \circ \iota' \circ s_0'=H^{-1} \circ e'=e$
%%%%%

%\color{teal}
We verify that the normal bundle~$(\nu(e),\iota)$ is $\mathcal{U}$-adapted.
The requirement~$\nu(e|_{B_k})=B_k \times \R^2$ holds because~$\nu(e')=\nu(e)$ is $\mathcal{U}'$-adapted and so we focus on verifying that for $k \in \{2i-1,2i \}$ we have~$\psi_i \circ \iota \circ (\alpha_k \times \id_{D^2})=\shrink_k$.
%%Yes, yes it should be that $\nu(e'),\iota'$ is U'-adapted, but I wrote it this way to emphasis that the total spaces are what matters for this particular verification.
To do so, we use the definitions of $\psi_i$ and $\iota$ and the fact that~$\iota$ is~$\mathcal{U}'$-adapted:
$$\psi_p \circ \iota \circ (\alpha_k \times \id_{D^2})
=(\psi_p' \circ (H^{-1})|_{U_p'}) \circ (  H|_{U_p} \circ \iota) \circ (\alpha_k \times \id_{D^2})
=\psi_p' \circ \iota'  \circ (\alpha_k \times \id_{D^2})
=\shrink_k.$$
%%Don't delete just yet: probably some justification about domains.
%First,  note that for $k \in \{2i-1,2i \}$, we have~$\iota(B_k \times D^2)=H^{-1}(\iota'(B_k \times D^2)) \subset H^{-1}(U_i')=:U_i.$
%Note also that~$\iota(\overline{\nu}(e|_{\partial B_k})) \subset H^{-1} (N')^\circ \subset N^\circ.$
Thus the normal bundle~$(\nu(e),\iota)$ is $\mathcal{U}$-adapted, as required.
\color{black}

Proposition~\ref{prop:IdentifyW} implies that the framings give rise to homeomorphisms~$\gamma \colon \overline{\nu}(e) \xrightarrow{\cong} E$ and~$\gamma' \colon \overline{\nu}(e') \xrightarrow{\cong} E$.
Since $\fr=\fr'$,  we have~$\gamma'=\gamma$.
Proposition~\ref{prop:IdentifyW} also implies that $\fr,\iota$ and~$\fr',\iota'$ give rise to homeomorphisms~$\widehat{\gamma} \colon \overline{\nu}(e) \xrightarrow{\cong} W$ and~$\widehat{\gamma}' \colon \overline{\nu}(e') \xrightarrow{\cong} W$ that satisfy $\widehat{\gamma} \circ \iota=\proj_W \circ \gamma$ and~$\widehat{\gamma}' \circ \iota'=\proj_W \circ \gamma'.$
We deduce that
$$  \widehat{\gamma}' \circ H \circ \iota
=\widehat{\gamma}' \circ H \circ (H^{-1} \circ \iota')
=\widehat{\gamma}' \circ \iota'
=\proj_W \circ \gamma'
=\proj_W \circ \gamma
=\widehat{\gamma} \circ \iota
$$
and therefore the following diagram commutes:
\color{black}
\begin{equation}
\label{eq:DiagramVerificationThetaRelBoundary}
\xymatrix{
W  \ar[d]^=
& \iota(\overline{\nu}(e)) \ar[r]^-{\subset } \ar[d]^{H} \ar[l]_-{\widehat{\gamma},\cong}
& N  \ar[d]^H\\
%%%%%%
W
& \iota'(\overline{\nu}(e')) \ar[r]^-{\subset } \ar[l]_-{\widehat{\gamma}',\cong}
& N.
}
\end{equation}
Since $\fr'$ is a nice framing, the commutativity of this diagram ensures that so is~$\fr$.
It follows that~$(N \setminus \iota(\nu(e)),  f=h| \cup \widehat{\gamma})$ can be used to calculate $\Theta(e).$

We can now conclude the proof of the proposition.
As in Construction~\ref{cons:EmbVBijection}, the choices above lead to boundary homeomorphisms
\begin{align*}
& f=h| \cup \widehat{\gamma}  \colon \partial (N \setminus \iota(\nu(e))) \xrightarrow{\cong} P_K, \\
& f'=h| \cup \widehat{\gamma}' \colon \partial (N \setminus \iota'(\nu(e'))) \xrightarrow{\cong} P_K.
\end{align*}
Using the diagram from~\eqref{eq:DiagramVerificationThetaRelBoundary} and the fact that $H$ is a rel.\ boundary homeomorphism, we deduce that $H|_{ \partial (N \setminus \iota(\nu(e)))}=f'^{-1}\circ f$ and that $H$ restricts to a homeomorphism
%AC: The diagram tells you its id on the \Sigma_{g,1} \times S^1 part also.
$$H|\colon N \setminus \iota(\nu(e)) \to N \setminus \iota'(\nu(e')).$$
We conclude that $(N \setminus \iota(\nu(e)),f)$ agrees with $(N \setminus \iota'(\nu(e)),f')$ in $\mathcal{V}^0_\lambda(P_K)$.
\end{proof}

\subsection{From manifolds to immersions}
\label{sub:ManifoldsImmersions}

This section constructs an inverse~$\Psi$ to the assignment~$\Theta \colon e \mapsto (N_{e(\Sigma)},f)$ from Construction~\ref{cons:EmbVBijection}.
Let~$(V,f)$ be a pair, where~$V$ is a~$4$-manifold with fundamental group~$\pi_1(V)\cong\Z$,  equivariant intersection form~$\lambda_V\cong \lambda$ and,  in the odd case,  Kirby-Siebenmann invariant~$\ks(V)=\ks(N)$, and~$f  \colon \partial V \xrightarrow{\cong} P_K$ is a homeomorphism.
%Here recall that $\varepsilon:=\ks(N).$

\begin{construction}
\label{cons:Psi}
Assume that $\lambda(1) \cong Q_N \oplus (0)^{2g+c}.$
The inverse~$\Psi(V,f)$ is an $\alpha$-compatible immersion $\Sigma  \looparrowright N$ defined as follows.
Glue the plumbed $4$-manifold~$W$ to~$V$ via the homeomorphism~$f^{-1}|_P$.
%a homeomorphism which identifies~$\Sigma\times\partial D^2$ with the~$\Sigma\times S^1$
% in the definition of~$M_{K,g}$
This produces a~$4$-manifold
$$\widehat{V}:=V \cup_{f^{-1}(P) \cong P} W$$
 with boundary~
 \[\partial \widehat{V}=(\partial V \setminus f^{-1}(P)) \cup (\partial \Sigma \times D^2)= f^{-1}(S^3\setminus \nu(K)) \cup (\partial \Sigma \times D^2) \cong S^3.\]
 %{AC: Technically there should probably be a quotient notation here but maybe it's clear enough as is: it's an embedding on $\partial \Sigma \times D^2$}, 
%\partial W=P \cup \partial \Sigma \times D^2 
We outline how the canonical immersion~$[\times \{0\}] \colon \Sigma \looparrowright W\subset \widehat{V}$ from Remark~\ref{rem:SigmaInW} gives rise to an immersion~$\Sigma \looparrowright N$.
\color{black}

We will use the homeomorphism~$f \colon \partial V \to  P_K$ to define a homeomorphism~$f' \colon  \partial \widehat{V} \to \partial N$ and then use  Freedman's classification of compact simply-connected 4-manifolds with~$S^3$ boundary,  to deduce that this homeomorphism extends to a homeomorphism~$F\colon \widehat{V} \to N$. We will then take our immersion to be
$$\Psi(V,f)  :=F\circ  ([\times \lbrace 0 \rbrace])
\colon \Sigma  \looparrowright N.$$
The next paragraphs flesh out the details of this construction.
First we build the homeomorphism~$f' \colon \partial \widehat{V} \to \partial N$ and secondly we argue it extends to a homeomorphism $F \colon \widehat{V} \to N$. 
Finally in Construction~\ref{cons:UConstruction} we will verify that~$e$ is compatible with~$\alpha$ (for any of the choices involved in the construction).

\begin{itemize}
\item  Observe that restricting~$f$  gives a homeomorphism~$f| \colon \partial V \setminus f^{-1}(P) \xrightarrow{\cong} S^3\setminus \nu(K)$.
Recall that the homeomorphism $d_K \colon  \partial \overline{\nu}(K) \to \partial \Sigma \times S^1 $ sends $\lambda_K$ to $\partial\Sigma\times\{\operatorname{pt}\}$ and~$\mu_K$ to~$\{\pt\}\times~\partial D^2$,  where $\lambda_K$ and~$\mu_K$ respectively denote the Seifert longitude and meridian of~$K\subset S^3$.
 Recall from Notation~\ref{not:nbhdforK} that  $d_K^{-1}$ extends to 
\begin{equation}
\label{eq:varphi}
\fr_K^{-1}\colon \partial\Sigma\times D^2\to \overline{\nu}(K).
%%\vartheta\colon \partial\Sigma\times D^2\to \overline{\nu}(K).
\end{equation}
and that $\fr_K^{-1}(\partial \Sigma \times \{0\})= K.$
Since $f|$ is obtained by restricting
\[f \colon \partial V \to P_K=(S^3\setminus \nu(K)) \cup_{d_K} P\]
and $\fr_K^{-1}$ extends~$d_K^{-1}$,the maps~$f$ and~$\fr_K^{-1}$ combine to a homeomorphism
$$ f|\cup\fr_K^{-1} \colon =\partial\widehat{V} \to S^3.$$
Then~$h^{-1} \circ (f|\cup\fr_K^{-1})$ gives the required homeomorphism
$$f':= h|^{-1} \circ (  f|  \cup \fr_K^{-1}) \colon \partial \widehat{V} \to \partial N.$$ Further, we observe that $f'(\partial\Sigma \times \{0\})=(h|^{-1} \circ \fr_K^{-1})(\partial\Sigma  \times \{0\})=K$.
\item To prove that this homeomorphism extends to a homeomorphism~$\widehat{V} \cong N$, we will appeal to Freedman's theorem that for every pair of simply-connected topological~$4$-manifolds with boundary homeomorphic to~$S^3$, the same intersection form, and the same Kirby-Siebenmann invariant, every homeomorphism between the boundaries extends to a homeomorphism between the 4-manifolds~\cite{Freedman}.
%BoyerUniqueness.
 We check that the hypotheses are satisfied.

\color{black}
First, we argue that~$\widehat{V}$ is simply-connected.
The hypothesis that~$(V,f)$ lies in~$\mathcal{V}^0_\lambda(P_K)$ implies that there is an isomorphism~$\widehat{\varphi} \colon \pi_1(V) \xrightarrow{\cong} \Z$ such that~$\varphi=\widehat{\varphi} \,  \circ \, i_*\,  \circ\, (f^{-1})_*$, where~$i_*$ is the inclusion induced map~$\pi_1(\partial V) \to \pi_1(V)$.
Gluing~$W$ to $V$ along~$P$ has the effect of killing~$(i_* \circ (f^{-1})_*)(\mu_K) \in \pi_1(V)$,  which generates $\pi_1(V) \cong \Z$ since we required that~$\varphi(\mu_K)$ generates~$\Z$.  
We conclude that~$\widehat{V}$ is simply-connected as claimed.

Next we must show that~$Q_{\widehat{V}}$ is isometric to~$Q_N$.
\begin{claim*}
There is an isometry~$Q_{\widehat{V}}\oplus (0)^{\oplus 2g+c}\cong Q_V$.
\end{claim*}
\begin{proof}
Since $\widehat{V}$ has boundary $S^3$ and is simply-connected,  duality and the universal coefficient theorem give~$H_3(\widehat{V}) \cong H^1(\widehat{V},\partial \widehat{V})=0$.
The Mayer-Vietoris sequence for the decomposition~$\widehat{V}=V \cup_{P} W$ therefore reduces to
$$ 0 \to H_2(P) \to H_2(V) \oplus H_2(W) \to H_2(\widehat{V}) \to H_1(P) \to H_1(V) \oplus H_1(W) \to 0.$$
Recall from Proposition~\ref{prop:HomologyW} that~$H_1(W)\cong \Z^{2g+c}$ and $H_2(W)=0$.
%Since~$W$ deformation retracts onto a wedge of $2g+c$ circles, one sees that~$H_1(W)\cong \Z^{2g+c}$ and $H_2(W)=0$.
This proposition also ensures that the inclusion induces an isomorphism $H_1(P) \xrightarrow{\cong} H_1(V) \oplus H_1(W)$ and, as a consequence,  the previous exact sequence further simplifies to 
$$ 0 \to H_2(P) \to H_2(V)  \to H_2(\widehat{V}) \to 0.$$
As~$\widehat{V}$ is simply-connected and has boundary $S^3$, we have~$H_2(\widehat{V}) \cong H^2(\widehat{V},\partial \widehat{V}) \cong H_2(\widehat{V},\partial \widehat{V})^*$.
This abelian group is free,  so the exact sequence above splits: $H_2(V) \cong H_2(\widehat{V}) \oplus H_2(P)$.
Since $P \subset \partial V$,  we deduce that~$Q_{V} \cong Q_{\widehat{V}} \oplus (0)^{\oplus (2g+c)}.$
\end{proof}

This claim combined with our assumption on the form~$\lambda$ leads to the isometries
$$   Q_{\widehat{V}} \oplus (0)^{\oplus 2g+c} \cong Q_V \cong \lambda_V(1) \cong \lambda(1) \cong Q_N \oplus  (0)^{\oplus 2g+c}.$$
This implies that~$Q_{\widehat{V}} \cong Q_N$ because both forms are nonsingular (indeed~$\partial \widehat{V}\cong \partial N \cong S^3$).
%AC: The details are in a note but basically write out the isometry as a 2x2 block matrix and then use the fact that it is an isometry to deduce that the upper left block itself preserves the forms.  Proving bijectivity is then ok.

In the even case, we deduce that both~$\widehat{V}$ and~$N$ are spin.
In the odd case, using the additivity of the Kirby-Siebenmann invariant (see e.g.~\cite[Theorem 8.2]{FriedlNagelOrsonPowell}), we have~$\ks(\widehat{V})=\ks(V)=\varepsilon=\ks(N)$.

Therefore~$\widehat{V}$ and~$N$ are simply-connected topological~$4$-manifolds with boundary~$S^3$, with the same intersection form and the same Kirby-Siebenmann invariant. Freedman's classification of simply-connected~$4$-manifolds with boundary~$S^3$ now ensures that the homeomorphism~$f'  \colon \partial   \widehat{V} \to \partial N$ extends to a homeomorphism $F\colon \widehat{V} \to N$ that induces the isometry $Q_{\widehat{V}} \cong Q_N$.
\end{itemize}
As mentioned above,  the desired immersion is now obtained as
\begin{equation}
\label{eq:DefOfPsi}
\Psi(V,f):=\Big{(}e\colon \Sigma  \xrightarrow{[\times \lbrace 0\rbrace]} \widehat{V} \xrightarrow{F,\cong}N\Big{)}.
\end{equation}
We illustrate $\widehat{V}= V \cup W$ and the image of $[\times \{0\}]$ in Figure~\ref{fig:vhat} below.

Finally, we finish by verifying that~$e$ is a~$\Z$-immersion with~$\partial e(\Sigma)=K$.  
The immersed surface~$e(\Sigma)=(F \circ [\times \{0\}])(\Sigma)$ has boundary~$K$ because~$F$ extends~$f'=h^{-1} \cup (f| \cup \fr_K^{-1})$ which satisfies~$f'(\partial \Sigma \times \{0\})=(h^{-1} \circ \fr_K^{-1})(\partial \Sigma \times \{0\})=K.$
The fact that~$\pi_1(N_{e(\Sigma})) \cong \Z$ follows from the homeomorphism~$N \setminus e(\Sigma) \cong V \cup (W \setminus \text{im}([ \times \{0\}])$ and the fact that this latter space deformation retracts onto~$V$.
%Pictures with crosses and radial directions.
  The verification that~$e$ is~$\alpha$-compatible  is deferred to Construction~\ref{cons:UConstruction} below.
\end{construction}

\begin{construction}
\label{cons:UConstruction}
For  any of the choices involved in Construction~\ref{cons:Psi},  we construct a family of double point charts $\mathcal{U}$ for $(e,\alpha)$, so that $e$ is compatible with $\alpha$.
In what follows we set
$$U^+=\left\lbrace (z,w)  \in D^2 \times D^2 \ \Big| \ |z| \leq \frac{1}{2} \text{ or } |w| \leq \frac{1}{2} \right\rbrace.$$
We also fix a homeomorphism~$\psi \colon U^+ \to D^2 \times D^2$ 
\color{black}
with $\psi(r_1, \theta_1, r_2,\theta_2)=(r_1', \theta_1, r_2', \theta_2)$,  
where~$(r_1', r_2')= \psi_{\text{rad}}(r_1,r_2)$, where $\psi_{\text{rad}}$ affinely maps the corresponding colored line segments of Figure~\ref{fig:buildingpsi}.

We record the following properties for later use,  where colors refer to Figure~\ref{fig:buildingpsi}:
\begin{enumerate}
\item (Blue) $\psi_{\text{rad}}$ sends $\{r_1=0\} \cup \{r_2=0\}$ to itself via the identity map. 
%This ensures that $U_i$ is a double point chart for $(e,\alpha)$. 
\item  (Green) $\psi_{\text{rad}}$ sends $\{r_1=1, \, r_2 \leq 1/4\}$ to $\{r_1=1, \, r_2 \leq 1/2\}$  via $\psi_{\text{rad}}(r_1, r_2)= (r_1, 2r_2)$, and similarly sends $\{r_1 \leq 1/4,  r_2=1\}$ to  $\{r_1 \leq 1/2,  r_2=1\}$ via $\psi_{\text{rad}}(r_1,r_2)=(2r_1,r_2)$. 
%This will eventually ensure that $\iota$ agrees on $\partial B_k \times D^2$ and $\partial Sigma_0 \times D^2$. 
\item For  $p=(s_1,\theta_1, s_2,\theta_2) \in U^+$ such that $s_1,s_2 \geq 1/4$ (darker grey on the left) there is a unique $a \in [0,1]$ and $p_L \in \{ r_1=1/4, r_2 \geq 1/4\} \cup \{r_1 \geq 1/4, r_2=1/4\}$
(red on the left) such that $(s_1,s_2)= a\cdot p_L + (1-a) \cdot \psi_{\text{rad}}(p_L)$,  as illustrated on the left of Figure~\ref{fig:buildingpsi}. 
\end{enumerate}

\color{black}
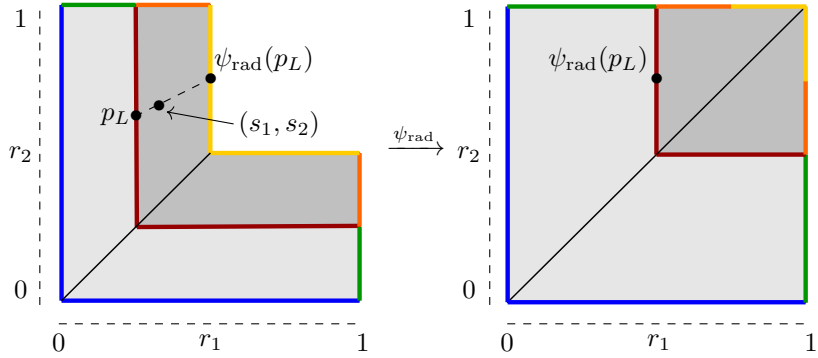
\begin{figure}[h!]
\centering
\begin{tikzpicture}
%\draw[step=1cm,color=gray] (0,0) grid (12,5);Uncomment this to get some helpful grid lines
\node[anchor=south west,inner sep=0] at (0,0){\includegraphics[height=4cm]{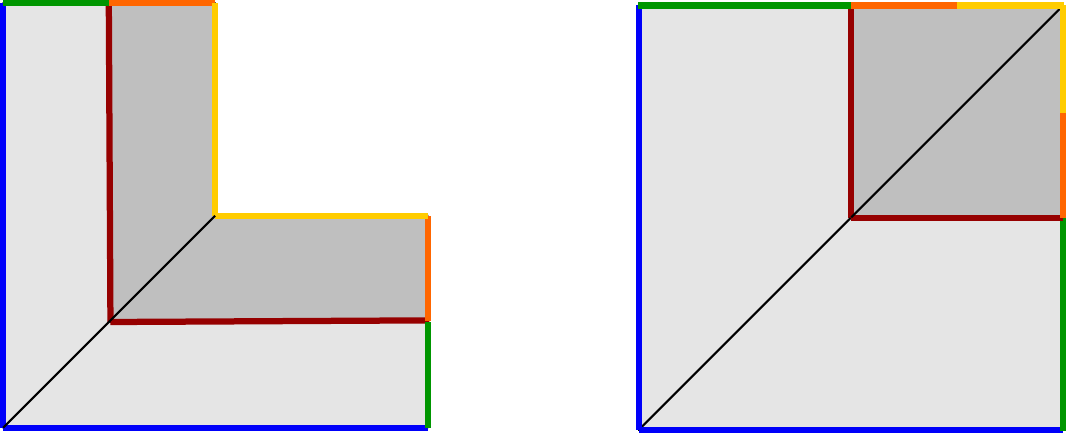}};
%\node at (2.25, -1){$U^+ \subset D^2 \times D^2$};
%\node at (8, -.5){$D^2 \times D^2$};
\node at (2, 3) {$\bullet$};
\node at (2.7, 3.25) {$\psi_{\text{rad}}(p_L)$};
\node at (7.87, 3) {$\bullet$};
\node at (7.1, 3.25) {$\psi_{\text{rad}}(p_L)$};
\node at (1.02,  2.5) {$\bullet$};
\node at (.75,2.5) {$p_L$};
\node at (4.7, 2.2) {$\xrightarrow{\psi_{\text{rad}}}$};
\draw[dashed] (1,2.5)--(2,3);
\draw[dashed] (-.25,  0)--(-.25,4);
\draw[dashed] (0, -.25)--(4, -.25);
\node at (0,  -.5){$0$};\node at (4, -.5){$1$}; \node at (2,-.5){$r_1$};
\node at (-.5,0.2){$0$}; \node at (-.5,3.8){$1$}; \node at (-.5,2){$r_2$};
\draw[dashed] (5.67,  0)--(5.67,4);
\draw[dashed] (5.9, -.25)--(9.9, -.25);
\node at (5.9,  -.5){$0$};\node at (9.9, -.5){$1$}; \node at (7.9,-.5){$r_1$};
\node at (5.4,0.2){$0$}; \node at (5.4,3.8){$1$}; \node at (5.4,2){$r_2$};

\node at (1.32,  2.64){$\bullet$};
%\node at (1.45,  3.05)[rotate=25]{\scriptsize $(s_1,s_2)$};
\draw[<-] (1.42,2.56)--(2.3, 2.4);
\node at(2.9,2.4){$(s_1,s_2)$};
\end{tikzpicture}
\caption{The map $\psi\colon U^+ \to D^2 \times D^2$  leaves angular coordinates unchanged and via $\psi_{\text{rad}}$ converts the radial coordinates of a point in $U^+$ (i.e.  a point in the subset of $[0,1] \times [0,1]$ depicted on the left)  to the radial coordinates of a point in~$D^2 \times D^2$ (i.e.  a point in $[0,1] \times [0,1]$ as depicted on the right). 
Note that on the left $\psi_{\text{rad}}(p_L)$ is redrawn in $U^+$ to illustrate the third property of $\psi_{\text{rad}}$ mentioned in Construction~\ref{cons:UConstruction}.} 
\label{fig:buildingpsi}
\end{figure}
Write $\proj_E \colon (D^2\times D^2)_i \to E$ and $\proj_W \colon E \to W$ for the projections.
  For~$i=1,\ldots,c$, set
$$U_i:=\proj_W \circ \proj_E ((D^2 \times D^2)_{2i-1} \cup (D^2 \times D^2)_{2i}).$$
The~$U_i$ (resp.\ $F(U_i)$) are pairwise disjoint neighborhoods of the double points of the immersion~$[\times \{ 0\}]$  (resp.~$e$) which, recall, are the images of the centers of the disks~$B_j$.
By definition of the shrink map~$\shrink_{2i-1} \cup \shrink_{2i} \colon U_i \to D^2 \times D^2$ induces a homeomorphism onto its image, which is~$U^+.$
%By definition of the shrinks.
By definition of $\sim_W$, the map
$$\psi_i:=\psi \circ (\shrink_{2i-1} \cup \shrink_{2i})$$
descends to $U_i \subset W$.
\medskip

We now verify that $\{ (U_i,\psi_i) \}$ (resp.~$\{ (F(U_i),\psi_i \circ F^{-1})\}$) is a family of double point charts for~$([\times \{0\}],\alpha)$ (resp.~$(e,\alpha)$).
%\color{teal}
We perform the verification for~$(e,\alpha)$ and note that it also yields the result for~$([\times \{0\}],\alpha)$.
This is a consequence of the definition of $e$ and the first condition we imposed on $\psi$.
Indeed, for $x \in D^2$,  and $k \in \{2i-1,2i\}$ we have
\begin{align*}
(\psi \circ (\shrink_{2i-1} \cup \shrink_{2i}) \circ F^{-1}) \circ e \circ \alpha_k(x)
&=\psi \circ (\shrink_{2i-1} \cup \shrink_{2i})  \circ [\times \{0\}] \circ \alpha_k(x) \\
&=\psi \circ (\shrink_{2i-1} \cup \shrink_{2i})  \circ \proj_W \circ \proj_E  (x,0) \\
&=
\begin{cases}
(x,0)  \quad& \text{if } k=2i-1, \\
(0,x)  \quad& \text{if }  k=2i.
\end{cases}
\end{align*}
\color{black}
%As mentioned above, this verification also shows that~$\{(U_i,\psi_i)\}$  is a family of double point charts for~$([\times \{0\}],\alpha)$. 
\end{construction}
\color{black}
We illustrate $\widehat{V}= V \cup W$ and the image of $[\times \{0\}]$ in Figure~\ref{fig:vhat}.  

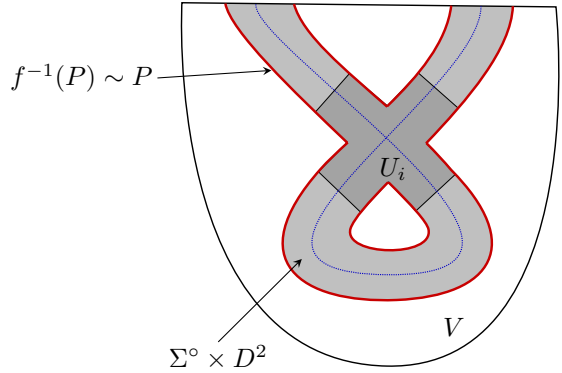
\begin{figure}[htbp!]
\centering
\begin{tikzpicture}
%\draw[step=1cm,color=gray] (0,0) grid (6,5);Uncomment this to get some helpful grid lines
\node[anchor=south west,inner sep=0] at (0,0){\includegraphics[height=5cm]{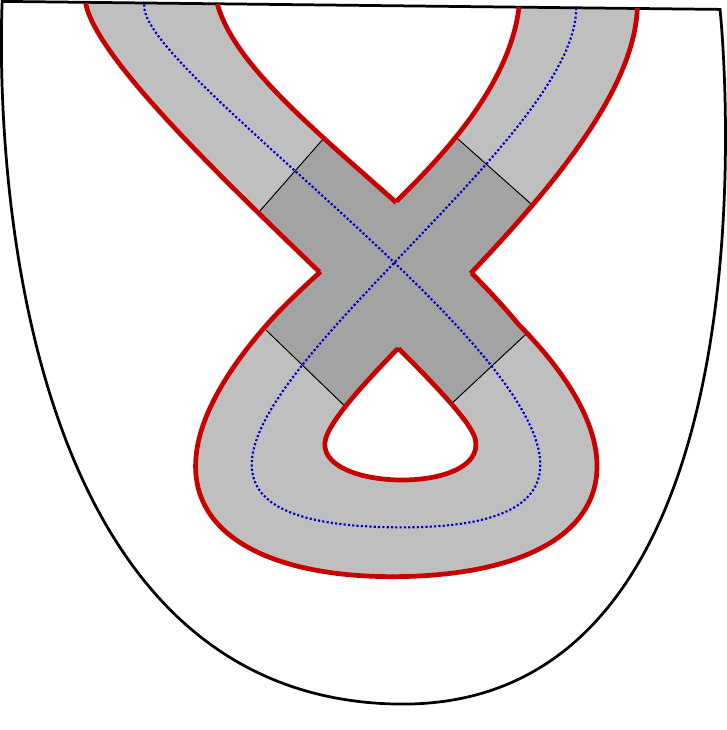}};
\node at (2.8,2.8){$U_i$};
\node at (3.6, .7){$V$};
\draw [-stealth] (.5, .5)--(1.6,1.6);
\draw [-stealth](-.3,  4)--(1.2,4.1);
\node at (-1.3,4){$f^{-1}(P) \sim P$};
\node at (.5, .3){$\Sigma^\circ \times D^2$};
\end{tikzpicture}
\caption{A schematic picture of $\widehat{V}= V \cup W$. The double point chart region $U_i\subset W$ is shaded in dark grey and $\Sigma^{\circ} \times D^2 \subset W$ is shaded in light grey.   The unshaded region is $V$.  } 
\label{fig:vhat}
\end{figure}
\color{black}

The next proposition verifies that if $\lambda$ is even,  this construction gives rise to a map
$$\Psi \colon \mathcal{V}^0_\lambda(P_K)\to\operatorname{Imm}_{\alpha}(g;c_+,c_-)_\lambda^0(N,K)$$
whereas if $\lambda$ is odd, it gives rise to a map 
$$\Psi \colon \mathcal{V}^{0,\varepsilon}_\lambda(P_K)\to\operatorname{Imm}_{\alpha}(g;c_+,c_-)_\lambda^0(N,K).$$
Here, as before, recall that $\varepsilon:=\ks(N).$

\begin{proposition}
\label{prop:PsiWellDef}
Up to homeomorphisms of $N$ rel.\ boundary,  the $\alpha$-compatible immersion $e$ from~\eqref{eq:DefOfPsi} depends neither on the choice of isometry~$Q_{\widehat{V}} \cong Q_N$
% nor the choice of $\vartheta$ from \eqref{eq:varphi}
 nor the homeomorphism~$\widehat{V} \cong N$ extending the boundary homeomorphism $f'$ nor on the homeomorphism rel.\ boundary type of~$(V,f)$.

In particular, $\Psi$ is well defined.
\end{proposition}
\begin{proof}
The proof is nearly identical to the corresponding argument in the embedded case~\cite[proof of Proposition 6.6]{ConwayPiccirilloPowell}. 
  The precise immersion~$e$ depends on the homeomorphism~$\widehat{V} \cong N$ chosen to extend a given~$f'$.
 This homeomorphism in turn depends on the choice of isometry~$Q_{\widehat{V}} \cong Q_N$.
 However for any two choices $F_1$ and $F_2$ of homeomorphisms~$\widehat{V} \cong N$ extending $f'$,  the resulting immersions are equivalent rel. boundary,  as can be seen by composing one choice of homeomorphism with the inverse of the other:
$$
\xymatrix@R0.5cm{
\Sigma \ar[r]^{[\times \{0\}]} \ar[d]^=&\widehat{V} \ar[r]^{F_1}\ar[d]^=& N \ar[d]^{F_2 \circ F_1^{-1}}\\
%%%
\Sigma \ar[r]^{[\times \{0\}]} &\widehat{V} \ar[r]^{F_2}& N.
}
 $$ 
So the rel. boundary equivalence class of the immersion~$\Psi(V,f)$ does not depend on the choice of isometry~$Q_{\widehat{V}} \cong Q_N$ nor on the choice of homeomorphism~$\widehat{V} \cong N$ realizing this isometry and extending $f'$.
Next, we check the independence of the rel.\ boundary homeomorphism type of~$(V,f)$.
If we have~$(V_1,f_1)$ and~$(V_2,f_2)$ that are homeomorphic rel.\ boundary, then there is a homeomorphism~$F \colon V_1 \to V_2$ that satisfies~$f_2 \circ F| =f_1$.
This homeomorphism extends to~$\widehat{F}:=F \cup \id_{W} \colon \widehat{V}_1 \to \widehat{V}_2$ and therefore to a homeomorphism~$N \to N$ that is, by construction rel.\ boundary.
%AC: To see that $F$ extends use $f_2 \circ F| =f_1$ and the fact that $\widehat{V}_i=V_i \cup_{f_i} \Sigma \times D^2$
A formal verification using this latter homeomorphism then shows that the immersions~$\Psi(V_1,f_1)$ and~$\Psi(V_2,f_2)$ are equivalent rel.\ boundary. 
%\qedhere
%AC: I uploaded the verification to the dropbox: Psi well def.pdf
%It's for embeddings, but the proof is similar. A key point/simplification being that $(\Phi \cup \id_W) \circ \times \{0 \}= \times \{0 \}$ because $ \times \{0 \}$ lands in $W$ and then we apply $\id_W.$
%\end{itemize}
\end{proof}

Our next goal is to prove that~$\Theta$ and~$\Psi$ are mutually inverse.
In order to carry out this goal,  we need an explicit $\mathcal{U}$-adapted normal bundle for any immersion  $e:=F \circ [\times \{0\}]$ representing~$\Psi(V,f)$.
Here~$\mathcal{U}$ denotes the family of  double point charts for $(e,\alpha)$ that was described in Construction~\ref{cons:UConstruction}.
The next construction is concerned with describing such a bundle.
\begin{construction}
\label{cons:UAdaptedNormalBundle}
For  any of the choices involved in Construction~\ref{cons:Psi},  we construct an~$\mathcal{U}$-adapted normal bundle $(\nu(e),\iota)$. 
In fact, it will be convenient to simultaneously fix this data for the immersion $[\times \{0\}] \colon \Sigma \looparrowright \widehat{V}$ as this will prove to be useful in Proposition~\ref{prop:EmbVBijections} when we argue that~$\Theta \circ \Psi=\id$.

%We begin the construction our double point charts.
We begin the construction of an $\mathcal{U}$-adapted normal bundle $(\nu(e),\iota)$.
\begin{itemize}
\item Define the vector bundle
$$ \nu(e):=E_\R:=\left(\bigsqcup_{k=1}^{2c}(B_k \times \R^2) \sqcup \Sigma^\circ \times \R^2\right)\Bigg/\sim_{E_\R}.$$
Here $\sim_{E_\R}$ identifies 
\[B_k \times \R^2 \supset \partial B_k \times R^2  \ni x  \sim_{E_{\R}} (\alpha_k \times \id_{\R^2}) \circ \eta_\R^\pm \circ (\alpha_k^{-1} \times \id_{\R^2})(x)
\in \partial B_k \times \R^2 \subset \Sigma^\circ \times \R^2\]
 where the homeomorphism~$\eta_\R^\pm \colon S^1 \times \R^2 \to S^1 \times \R^2$ is given by~$\eta_{\R}^{\pm}(\omega,  (r, \theta))= (\omega,  (r, \theta \pm \omega))$. 
%We then set~$\nu(e):=E_\R$. 

Observe that the $D^2$-bundle  $\overline{\nu}(e)$ of $\nu(e)$ is given by the same formula as $E$ except that we replace~$(D^2 \times D^2)_k$ by~$B_k \times D^2$ in the definition:
$$\overline{\nu}(e)=E':=\left(\bigsqcup_{k=1}^{2c}(B_k \times D^2) \sqcup \Sigma^\circ \times D^2\right)\Bigg/\sim_E'.$$
Here we write~$\sim_E'$ to indicate that we are using the same equivalence relation as in Construction~\ref{cons:E} but with $B_k \times D^2$ in place of $(D^2 \times D^2)_k.$
In other words~$\sim_E'$ identifies $x$ with~$(\alpha_k \times \id) \circ \eta^\pm \circ (\alpha_k^{-1} \times \id)(x)$.
The reason for this minor change is that in a $\mathcal{U}$-adapted normal bundle,  the restriction of the disk bundle to $B_k$ must \emph{equal}~$B_k \times D^2$.

\item We begin by defining $\iota|_{\overline{\nu}(e)} \colon \overline{\nu}(e) \looparrowright N$. 
In fact we will define a map $j\colon \overline{\nu}(e) \to E$ and then obtain $\iota|_{\overline{\nu}(e)}$ as the composition
$$\iota|_{\overline{\nu}(e)} \colon \overline{\nu}(e) \xrightarrow{j} E \xrightarrow{\proj_W} W \to V \cup W \xrightarrow{F,\cong} N.$$
%This will simultaneously lead to a normal bundle $(\nu(e),\iota)$ for $e$ and to a normal bundle~$(\nu(e),F^{-1} \circ \iota)$ for $[\times \{0\}].$
%%Don't delete
%Since we wish to obtain $\psi_i \circ \iota \circ (\alpha_k \times \id)=\shrink_k^\pm$ on each $D^k \times D^k$, this forces us to 

For $k \in \{2i-1,2i\}$,  define $j|_{B_k \times D^2}$ as 
\begin{equation}
\label{eq:jRestrictedToBkxD2}
j|_{B_k \times D^2}:= \psi_i^{-1} \circ \shrink_k \circ (\alpha_k^{-1} \times \id) \colon \overbrace{B_k \times D^2}^{\subset \overline{\nu}(e)} \to  \overbrace{(D^2 \times D^2)_k }^{\subset E}.
\end{equation}

We need to define $j$  on $\Sigma^\circ \times D^2$.  We will define $j$ slightly differently near the boundary of $\Sigma$,  as illustrated on the left of Figure~\ref{fig:imageofiota},  so we begin by fixing a collar of $\partial \Sigma \subset \Sigma$:
$$ \beta \colon \partial \Sigma \times [0,1] \hookrightarrow \Sigma^\circ$$
with $\beta(-,1)=\id_{\partial \Sigma}$.
The exterior of this collar in $\Sigma^\circ$ is denoted~$\Sigma_\beta$:
$$\Sigma_\beta:=\Sigma^\circ \setminus \operatorname{int}(\im(\beta)).$$
We now claim that the $j|_{B_k \times D^2}$ and
\begin{align*}
j|_{\Sigma^\circ \times D^2} \colon \Sigma^\circ \times D^2 &\to N \\
(p,x) & \mapsto
\begin{cases}
(p,\frac{x}{2})  \quad & \text{ if }p \notin \im(\beta) \\
(p,x(1+\frac{s}{2})) \quad & \text{ if } p=\beta(y,s).
\end{cases}
\end{align*}
can be combined in order to obtain a map~$\overline{\nu}(e) \to E$,  whose image in $\widehat{V}= V \cup W$ we illustrate in Figure~\ref{fig:imageofiota}. 
We need to verify that $j|_{B_k \times D^2}$ and~$j|_{\Sigma^\circ \times D^2} $ agree on the image of~$\partial B_k \times D^2$ in $\overline{\nu}(e).$
%\color{teal}
For $(p,z) \in \partial B_k \times D^2$,  using successively the definition of $j|_{B_k \times D^2}$, the definition of $\psi_i$ and the second condition on $\psi$, we obtain
%{AC: Does the $/2$ come from the shrinks or from the $\psi$? AM: Sort of both,  I think. ....  $\psi$ contributes a $/2$, and $\shrink$ and $\shrink^{-1}$ together contribute a net factor of 1.}
\begin{align*} 
j|_{B_k \times D^2}(p,z)
&= \psi_i^{-1} \circ \shrink_k \circ (\alpha_k^{-1} \times \id)(p,z) \\
&= (\psi \circ (\shrink_{2i-1} \cup \shrink_{2i}) )^{-1} \circ \shrink_k \circ (\alpha_k^{-1} \times \id)(p,z) \\
&=\left(\alpha_k^{-1}(p),\frac{z}{2}\right).
\end{align*}
We must therefore show that,  once considered in $E$,  this is equivalent to $j|_{\Sigma^\circ \times D^2}([(p,z)]).$
%This is a formal verification that follows from the fact that $\sim_E'$ is modelled after $\sim_E$.
In order to calculate $j|_{\Sigma^\circ \times D^2}([(p,z)])$, we first note that,  by definition of $\sim_E'$
$$(p,z) \sim_E' (\alpha_k \times \id) \circ \eta^\pm \circ (\alpha_k^{-1} \times \id)(p,z).$$
Since $\eta^{\pm}$ commutes with $(x,y)\mapsto (x,y/2)$, it then follows that
\begin{align*}
j|_{\Sigma^\circ \times D^2}([(p,z)])
&=j|_{\Sigma^\circ \times D^2} \circ (\alpha_k \times \id) \circ \eta^\pm \circ (\alpha_k^{-1} \times \id)(p,z) \\
&=(\alpha_k \times \id) \circ \eta^\pm \circ (\alpha_k^{-1} \times \id)\left(p,\frac{z}{2}\right) \\
&\sim_E \left(\alpha_k^{-1}(p),\frac{z}{2}\right).
\end{align*}
\color{black}

\begin{figure}[htbp!]
\centering
\begin{tikzpicture}
%\draw[step=1cm,color=gray] (0,0) grid (6,5);Uncomment this to get some helpful grid lines
\node[anchor=south west,inner sep=0] at (0,0){\includegraphics[height=5cm]{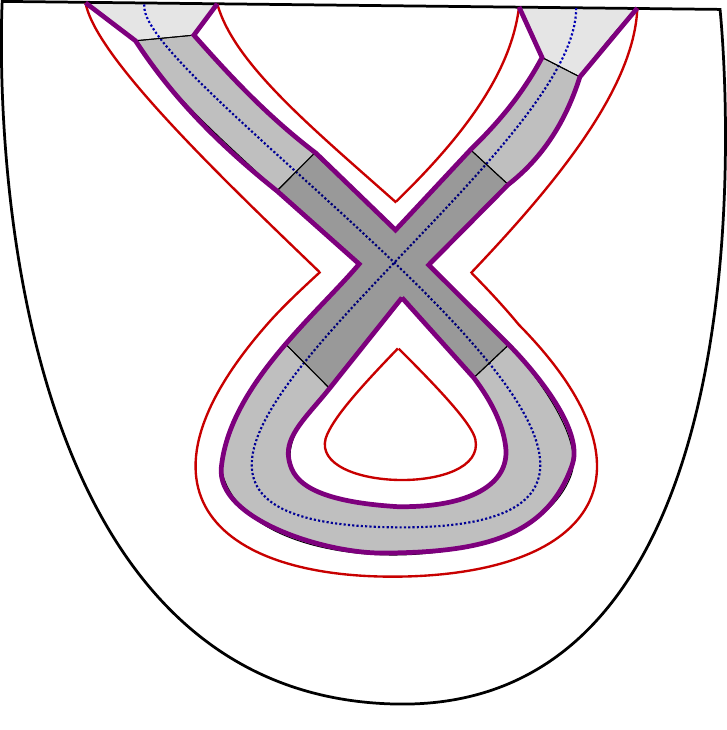}};
\node[anchor=south west,inner sep=0] at (6,0){\includegraphics[height=5cm]{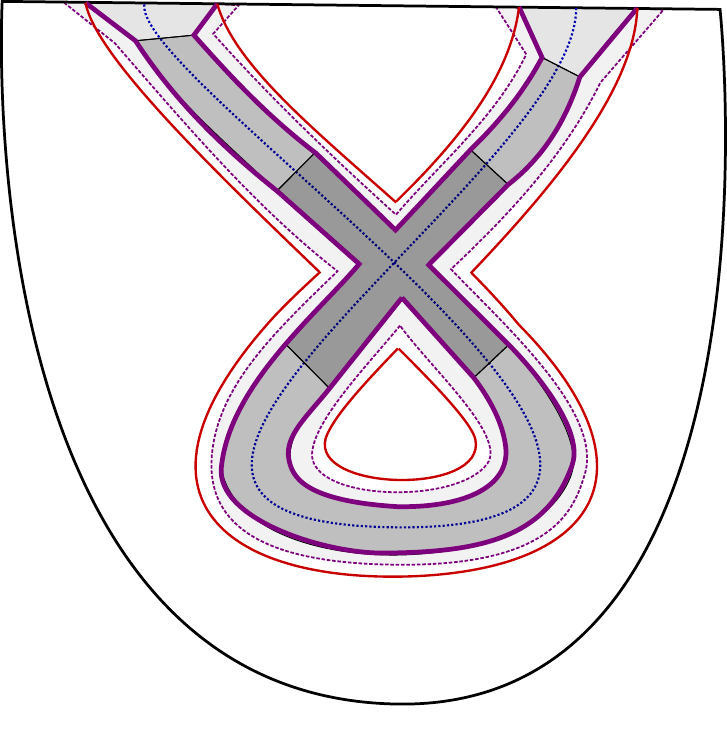}};
\end{tikzpicture}
\caption{A schematic picture of $\text{proj}_W \circ j(\overline{\nu}(e))$ in $\widehat{V}$ (left).  Notice that,  as indicated by the piecewise definition,  this decomposes naturally into three pieces, corresponding to the images of the normal disk bundle restricted to
${\text{im}(\beta)}$ (light grey),~${\Sigma_{\beta}}$ (medium grey),  
and $B_{2i-1} \cup B_{2i}$ (dark grey).   
On the right, we show a schematic picture of the image of all of $\nu(e)$ in $\widehat{V}$,  which is roughly obtained from that of $\overline{\nu}(e)$ by adding a small open collar of $\partial \overline{\nu}(e) \smallsetminus \overline{\nu}(e|_{\partial \Sigma})$. 
} 
\label{fig:imageofiota}
\end{figure}

\item We now briefly describe how to extend $\proj_W \circ j \colon \overline{\nu}(e) \to W \subset \widehat{V}$ to an immersion~$\widehat{j} \colon \nu(e) \looparrowright \widehat{V}$.  
The process is best understood by considering the right hand side of Figure~\ref{fig:imageofiota}.
We include some details for completeness but emphasize that $\widehat{j}$ will not be needed again; we only include this discussion because the existence of such an extension $\widehat{j}$ is required by the definition of a normal bundle. 
Let $s \colon [1,\infty) \to [1/2, 3/4)$ be a homeomorphism.  
We define $\widehat{j}$ on $\Sigma_\beta \times \R^2_{|z| \geq 1}$ by setting~$\widehat{j}(p,z)=\proj_W(p, \frac{s(|z|)}{|z|} z) \in \proj_W(\Sigma_\beta \times D^2_{\frac{1}{2} \leq r \leq \frac{3}{4}})$.  
We can similarly use the extra space in $(D^2 \times D^2)_k$ to have $\widehat{j}$ send $B_k \times \R^2_{|z| \geq 1}$ homeomorphically to $\proj_W(D^2 \times D^2_{\frac{1}{2} \leq r\leq \frac{3}{4}})$. 
 However,  note that $\proj_W \circ j$ sends $\overline{\nu}(e|_{\partial \Sigma)})$ onto all of~$\proj_W(\partial \Sigma \times D^2)$.  
 We therefore need to use a small collar neighborhood of $\partial V$ in order to extend $\proj_W \circ j$ to an embedding~$\widehat{j}$ of all of $\nu(e|_{\text{im}(\beta)})$ into $\widehat{V}$. 
  We leave the remaining details of this to the interested reader. 
%This concludes the description of the immersion $\widehat{j}\colon \nu(e)\looparrowright \widehat{V}$. 

\item As indicated previously,  we define the immersion $\iota \colon {\nu}(e) \looparrowright N$ as the composition
$$\nu(e) \xrightarrow{\widehat{j}} \widehat{V} \xrightarrow{F,\cong} N.$$
We simultaneously verify that $(\nu(e),\iota)$ is a normal bundle for $e$ and that~$(\nu(e),F^{-1} \circ \iota)$ is a normal bundle for $[\times \{0\}]$.
For $[\times \{0\}]$, the $0$-section is $s_0 \colon \Sigma \to \nu(e),x \mapsto [(x,0)]$, whereas for $e$ it is $F \circ s_0$.
In both cases,  the verification that this is a $0$-section reduces to checking that $\proj_W \circ j \circ s_0=[\times \{0\}]$ and this follows from the definitions.
%%Don't delete.
%\iota \circ s_0=e
%F \circ \proj_W \circ j \circ s_0=F \circ [\times \{0 \}]
% \proj_W \circ j \circ s_0= [\times \{0 \}]
%Both have $\alpha^{-1} \times \id$ and then the /2 factor doesn't affect the \times 0 part; the shrink and $\psi_0$ leave the 0 part untouched.
\end{itemize}
We now verify that~$(\nu(e),\iota)$ is a~$\{(F(U_i), \psi_i \circ F^{-1})\}$-adapted normal bundle for $e$. 
By definition of the bundle $\nu(e)$, we have $\overline{\nu}(e|_{B_k})=B_k \times D^2$ so this reduces to using the definitions of $\psi_i, \iota,j$ to check that 
\begin{align*}
\psi_i \circ (F^{-1} \circ \iota) \circ (\alpha_k \times \id_{D^2})
&=(\psi_i   \circ F^{-1}) \circ (F\circ \proj_W \circ j) \circ (\alpha_k \times \id_{D^2}) \\
%%%
&=\psi_i \circ  \proj_W \circ \psi_i^{-1} \circ \shrink_k \\
&=\shrink_k.
%We noted that $\psi_i$ descends to U_i \subset W so I think this makes sense.
\end{align*} 
This argument also shows that~$(\nu(e),F^{-1} \circ \iota)$  is a~$\{(U_i, \psi_i)\}$-adapted normal bundle for $[\times \{0\}]$. 
\end{construction}

Now we prove that the maps~$\Theta$ and~$\Psi$ are mutually inverse.
\begin{proposition}
\label{prop:EmbVBijections}
Let~$N$ be a simply-connected~$4$-manifold with boundary~$\partial N \cong S^3$,  let~$K \subset S^3$ be a knot,  let $c_+,c_-$ and $g$ be non-negative integers, and let~$\lambda$ be  a nondegenerate hermitian form with $\lambda(1) \cong Q_N \oplus (0)^{2g+c}$, where $c:=c_++c_-.$
\begin{enumerate}
\item If~$\lambda$ is even, then the map~$\Theta$ from Construction~\ref{cons:EmbVBijection} determines a bijection
$$\operatorname{Imm}_\alpha(g;c_+,c_-)_\lambda^0(N;K) \to \mathcal{V}^0_\lambda(P_K).$$
\item If~$\lambda$ is odd, then the map~$\Theta$ from Construction~\ref{cons:EmbVBijection} determines a bijection
$$\operatorname{Imm}_\alpha(g;c_+,c_-)_\lambda^0(N;K) \to \mathcal{V}^{0,\varepsilon}_\lambda(P_K).$$
\end{enumerate}
\end{proposition}
\begin{proof}
We prove that the maps $\Theta$ and $\Psi$ satisfy~$\Psi \circ \Theta=\id$ and~$\Theta \circ \Psi=\id$.

First we prove that~$\Psi \circ \Theta=\id$.
Start with an immersion~$e \colon \Sigma \looparrowright N$ that is compatible with~$\alpha$ and write~$\Theta(e)=(N_{e(\Sigma)},f)$ with~$f=h| \cup \widehat{\gamma}| \colon \partial N_{e(\Sigma)} \to P_K$ the homeomorphism described in Construction~\ref{cons:EmbVBijection}.
Then~$\Psi(\Theta(e))$ is an immersion
$$ \Sigma \stackrel{[\times \lbrace 0 \rbrace]}{\looparrowright} N_{e(\Sigma)} \cup_f W \xrightarrow{F,\cong} N.$$
We showed that the equivalence class of this immersion is independent of the homeomorphism~$F$
that extends~$f$.
It suffices to show that we can make choices so that $\Psi(\Theta(e))$ recovers $e$. 
This can be done explicitly as follows.
By definition we have~$f'=h|^{-1} \circ (f| \cup \fr_K^{-1})$ and~$f=h| \cup \widehat{\gamma}$.
Since~$f|$ denotes the restriction of~$f$ to the knot complement, we have~$f|=h$ and we deduce that~$f'=\id_{\partial N \setminus \nu(K)}  \cup (h^{-1} \circ \fr_K^{-1})$.
Since~$\widehat{\gamma}$ extends the Seifert framing~$\fr_\partial=\{ \fr_{\partial B_k} \}_k \sqcup \id$, we can rewrite this as~$f'=\id_{\partial N \setminus \nu(K)} \cup \widehat{\gamma}|^{-1}$.
%%%Don't delete. This is because good framings extend the Seifert framing by definition.
We already know an extension of $f'$, namely $\id_{N_{e(\Sigma)}} \cup \widehat{\gamma}^{-1}$, which we take to be~$F$.
Thus we obtain
$$\Psi(\Theta(e))=\widehat{\gamma}^{-1} \circ [\times \{0\}] \colon \Sigma \looparrowright N.$$ 
We verify that this agrees with the initial immersion $e$.
We must check that~$\widehat{\gamma} \circ e=[\times \{0\}] \colon \Sigma \to W$.
Use $s_0 \colon \Sigma \to \overline{\nu}(e)$ to denote the $0$-section of the linear $D^2$-bundle $\overline{\nu}(e)\to \Sigma$ so that $\iota \circ s_0 = e.$
This way,  it suffices to prove that~$\gamma \circ s_0 =[\times \{0\}]_E \colon \Sigma \to E$.
Indeed, we would then have
$$ \widehat{\gamma} \circ e
=\widehat{\gamma} \circ (\iota \circ s_0)
=[\proj_W \circ \gamma \circ s_0]_W
=[\times \{0\}].$$
We first verify this on $\Sigma^\circ$ and then on the $B_k$.

For $x \in \Sigma^\circ$, note that $s_0(x) \in \pi^{-1}(x)$ is the zero element, as is $(x,0) \in  \{ x \} \times D^2 \subset \Sigma^\circ \times D^2$.
Since the homeomorphism~$\gamma|_{\nu(e|_{\Sigma^\circ})}=\fr$ is induced by an isomorphism of vector bundles,  it therefore indeed satisfies $\gamma \circ s_0(x)=(x,0)$.

For $x \in B_k$,  since the zero section of $\nu(e|_{B_k})=B_k \times D^2 \xrightarrow{\proj_1} B_k$ is $s_0(x)=(x,0)$,  we have 
$$\gamma \circ s_0(x)
=(\alpha_k^{-1} \times \id)(x,0)
=(\alpha_k^{-1}(x),0)
=\eta^\pm(\alpha_k^{-1}(x),0)
\sim_E
(x,0)
=[\times \{0\}]_E(x).
$$
This concludes the verification that $\Psi \circ \Theta=\id.$

\medbreak

Next we prove that~$\Theta \circ \Psi=\id$.
This time we start with a pair~$(V,f)$ consisting of a 4-manifold~$V$ and a homeomorphism~$f \colon  \partial V \to P_K$.
By definition,~$\Psi(V,f)$ is represented by an immersion~$e \colon \Sigma \stackrel{ [\times \lbrace 0 \rbrace]}{\looparrowright} \widehat{V} \xrightarrow{F,\cong} N$.
Recall that we write~$h \colon \partial N \to S^3$ for our preferred homeomorphism and that by construction,  on the boundaries,~$F$ restricts to
$$h|^{-1} \circ (f|  \cup \fr_K^{-1}) \colon \partial \widehat{V} \to \partial N.$$
%where (the isotopy class of)~$\vartheta \colon \partial \Sigma \times D^2 \to \overline{\nu}(K)$ satisfies the properties listed below equation~\eqref{eq:varphi}.

In order to understand the effect of~$\Theta$ on~$e$, we need to fix a nice framing on~$e|_{\Sigma^\circ}$. 
To this effect, we use the family~$\mathcal{U}$ of double point charts for~$(e,\alpha)$ and the~$\mathcal{U}$-adapted normal bundles~$(\nu(e),\iota)$ (for the immersion~$e$) and~$(\nu(e),F^{-1} \circ \iota)$ (for the immersion $[\times \{0\}]$) from Constructions~\ref{cons:UConstruction} and~\ref{cons:UAdaptedNormalBundle}.
We frame both~$[\times \{0\}]$ and~$e$ by noting that 
$$ \nu(e|_{\Sigma^\circ})= \Sigma^\circ \times \R^2.$$
In the notation of Proposition~\ref{prop:FramingPushoffUsingF},  this means that~$\fr=\id.$

Proposition~\ref{prop:IdentifyW} then ensures that this framing gives rise to a homeomorphism $\gamma \colon \overline{\nu}(e) \xrightarrow{\cong}~E$.
Proposition~\ref{prop:IdentifyW}  also shows that the additional data of the normal bundles $(\nu(e),\iota)$ and $(\nu(e),F^{-1} \circ~\iota)$ gives rise to homeomorphisms
\begin{align*}
&\widehat{\gamma}_{[\times \{0\}]} \colon  (F^{-1}\circ\iota)(\overline{\nu}(e)) \xrightarrow{\cong} W,  \qquad \widehat{\gamma} \colon \iota(\overline{\nu}(e)) \xrightarrow{\cong} W.
\end{align*}
We record for later on that~$\widehat{\gamma}=\widehat{\gamma}_{[\times \{0\}]} \circ F^{-1}$: this can be seen from the proof of Proposition~\ref{prop:IdentifyW} combined with the equations~$\widehat{\gamma} \circ \iota= \proj_W \circ \gamma$ and $\widehat{\gamma}_{[\times \{0\}]} \circ (F^{-1}\circ\iota)= \proj_W \circ \gamma$.
%%Don't delete.
 %(the fact that the same $\gamma$ appears in both equations is a consequence of the fact that $\nu(e)$ is the bundle for both immersions.

Pending the verification that our framing of $e$ is nice,  we obtain~$\Theta(\Psi(V,f))=(N_{e(\Sigma)},h| \cup~\widehat{\gamma}|)$, where, as dictated by Construction~\ref{cons:EmbVBijection}, the boundary homeomorphism is~$h| \cup  \widehat{\gamma}| \colon \partial N_{e(\Sigma)} \to P_K$.
Here we are making use of the fact that up to homeomorphism rel. boundary, we can choose any nice framing in the definition of~$\Theta$.

We have to prove that~$(N_{e(\Sigma)},h| \cup \widehat{\gamma}|)$ is homeomorphic rel.\ boundary to~$(V,f)$.
We begin with an intermediate step which consists of considering the~$\alpha$-compatible immersion~$[\times \{0\}] \colon \Sigma \looparrowright~\widehat{V}$.
Recall that the immersion~$[\times \{0\}]$ admits the same compatible pair of double point charts~$\mathcal{U}$ as~$e$ and admits~$(\nu(e),\iota \circ F^{-1})$ as an~$\mathcal{U}$-adapted normal bundle, and write~$\widehat{V}_{[\times \{0\}](\Sigma)}$ for the exterior of~$[\times \{0\}]$.
Our strategy is to first show that~$(N_{e(\Sigma)},h| \cup \widehat{\gamma}|)$ is homeomorphic rel. boundary to~$(\widehat{V}_{[\times \{0\}](\Sigma)},f \cup \widehat{\gamma}_{[\times \{0\}]})$ and to then show that~$(\widehat{V}_{[\times \{0\}](\Sigma)},f \cup \widehat{\gamma}_{[\times \{0\}]})$ is homeomorphic rel. boundary to~$(V,f)$.

\begin{claim}\label{claim:homeo1}
The pair $(N_{e(\Sigma)},h| \cup \widehat{\gamma}|)$ is homeomorphic rel. boundary to~$(\widehat{V}_{[\times \{0\}](\Sigma)},f \cup \widehat{\gamma}_{[\times \{0\}]})$.
\end{claim}
\begin{proof}
Since~$(\nu(e),\iota \circ F^{-1})$ is a normal bundle for the immersion~$[\times \{0\}] \colon \Sigma \looparrowright \widehat{V}$,  the homeomorphism~$F \colon \widehat{V} \to N$ restricts to a homeomorphism~$F| \colon \widehat{V}_{[\times \{0\}](\Sigma)} \to N_{e(\Sigma)}$ between the exteriors.
We therefore obtain the following diagram,  where the two right squares certainly commute:
$$
\xymatrix @C+0.3cm{
P_K \ar[d]^=& \partial \widehat{V}_{[\times \{0\}](\Sigma)} \ar[r]^{\subset} \ar[d]^{F|} \ar[l]_-{f \cup \widehat{\gamma}_{[\times \{0\}]}}&   \widehat{V}_{[\times \{0\}](\Sigma)} \ar[r]^{\subset} \ar[d]^{F|}& \widehat{V} \ar[d]^F \\
%%%%%%%%%%%%%
P_K & \partial N_{e(\Sigma)} \ar[r]^{\subset} \ar[l]_-{h| \cup \widehat{\gamma}} & N_{e(\Sigma)} \ar[r]^{\subset} &N.
}
$$
To see that the left square commutes and complete the proof of this claim,  use~$\widehat{\gamma}=\widehat{\gamma}_{[\times \{0\}]} \circ F^{-1}$ and recall that~$F|$ extends the homeomorphism~$h|^{-1} \circ f$ on~$\partial V \setminus f^{-1}(P)$.
\end{proof}

%This concludes the proof that~$\Theta(\Psi(V,f))=(N_{e(\Sigma)},h| \cup \widehat{\gamma}|)$ is homeomorphic rel. boundary to~$(\widehat{V}_{[\times \{0\}](\Sigma)},f \cup \widehat{\gamma}_{[\times \{0\}]})$

\begin{claim} 
The pair~$(\widehat{V}_{[\times \{0\}](\Sigma)},f \cup \widehat{\gamma}_{[\times \{0\}]})$ is homeomorphic rel. boundary to~$(V,f)$.
\end{claim}
\begin{proof}
We need to build a homeomorphism
$$ G \colon \widehat{V}_{[\times \{0\}](\Sigma)} \to V$$
and then verify that its restriction to the boundary satisfies~$f \circ G|=\widehat{\gamma}_{[\times \{0\}]} \cup f|.$

To construct~$G$,  begin by choosing a boundary collar~$\partial V \times I \subset V$ that satisfies~$\partial V \times \{1\}= \partial V$,  consider~$f^{-1}(P) \times I \subset \partial V \times I$, and define
$$ V_{\text{small}}:=V \setminus (f^{-1}(P) \times (0,1]).$$
Informally,  $G$ can be described as leaving $V_{\text{small}}$ alone,  while homeomorphically compressing  $(f^{-1}(P) \times I) \cup (\widehat{V}_{[\times \{0\}](\Sigma)} \cap W)$ onto $f^{-1}(P) \times I$. However,  it takes significant work to make this precise while ensuring that our boundary identifications correspond as required.

We note the decompositions
\begin{align*}
V&=V_{\text{small}} \cup (f^{-1}(P) \times I), \\
\widehat{V}_{[\times \{0\}](\Sigma)}&=(V \cup W) \setminus j(\nu(e))=V_{\text{small}} \cup (f^{-1}(P) \times I) \cup (W \setminus j(\nu(e))).
\end{align*}
We will define~$G|_{V_{\text{small}}}:=\id$, so the challenge is really to define a homeomorphism 
$$(f^{-1}(P) \times I) \cup (W \setminus j(\nu(e))) \to (f^{-1}(P) \times I).$$

For this,  we will decompose~$W \smallsetminus j(\nu(e))$ and~$f^{-1}(P) \times I$ into three pieces (see Figure~\ref{fig:buildingG}) and define~$G$ on these pieces; these constituents of $G$ will be denoted~$G_1,G_2$ and~$G_3$. 
\color{black}
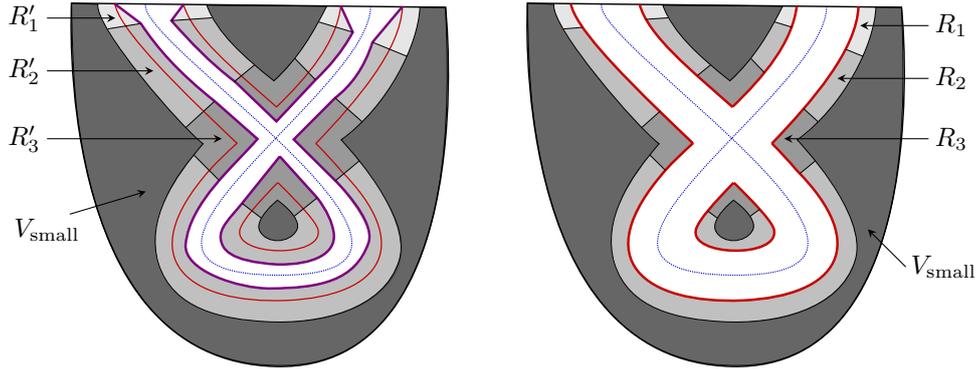
\begin{figure}[h!]
\centering
\begin{tikzpicture}
%\draw[step=1cm,color=gray] (0,0) grid (12,5);Uncomment this to get some helpful grid lines
\node[anchor=south west,inner sep=0] at (0,0){\includegraphics[height=5cm]{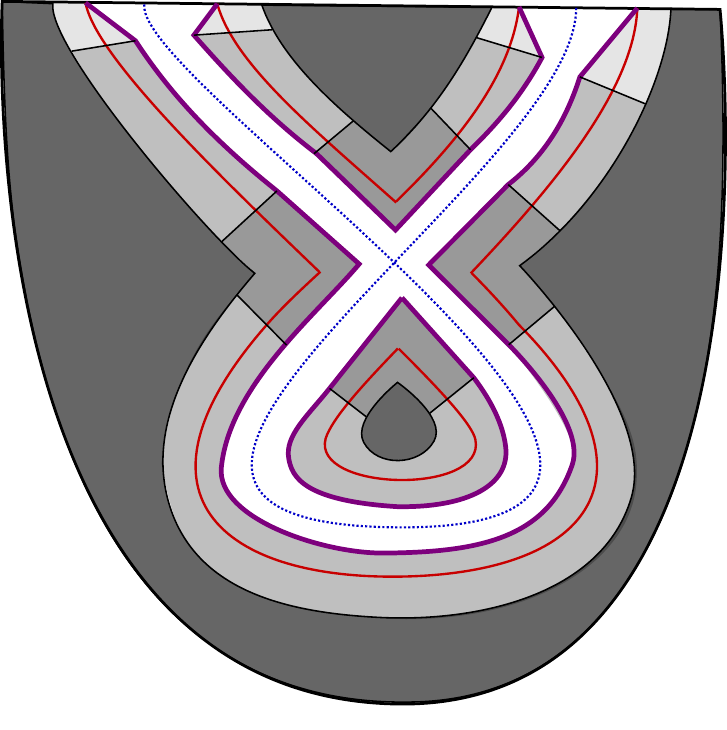}}; 
 \node[anchor=south west,inner sep=0]at (6,0){\includegraphics[height=5cm]{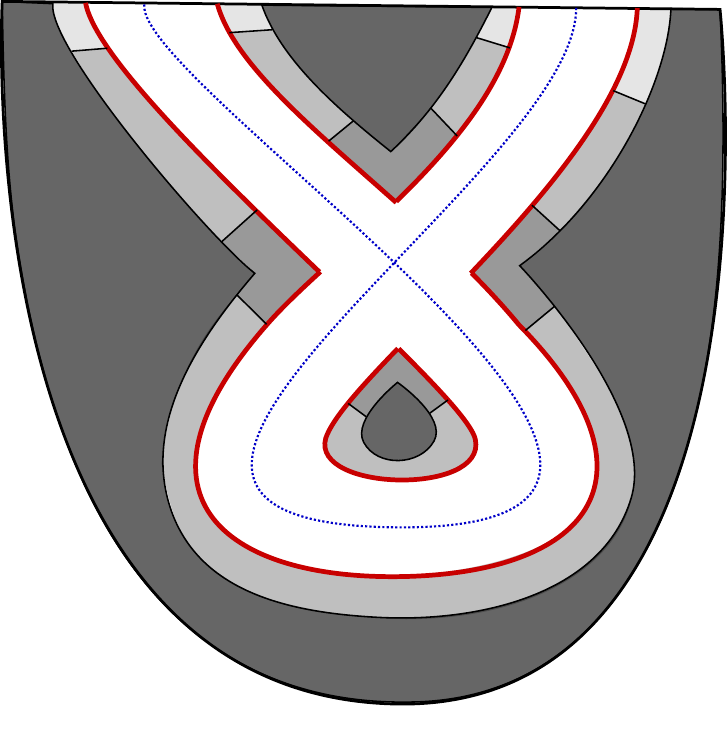}};
 \draw [-stealth](-.3, 4.8)--(.6, 4.8);
 \node at (-.6, 4.8){$R_1'$};
 \draw [-stealth](-.3, 4.1)--(1.1, 4.1);
 \node at (-.6, 4.1){$R_2'$};
 \draw[-stealth] (-.3,  3.2)--(2, 3.2);
 \node at (-.6, 3.2){$R_3'$};
 \draw [-stealth](0, 2.2)--(1, 2.5);
 \node at (-.3, 2){$V_{\text{small}}$};
 \draw [stealth-](10.4, 4.7)--(11.3,4.7);
 \node at (11.6, 4.7){$R_1$};
 \draw [stealth-] (10.1, 4)--(11.3,  4);
 \node at (11.6, 4) {$R_2$};
 \draw [stealth-] (9.4,3.2)--(11.3, 3.2);
 \node at (11.6, 3.2){$R_3$};
  \draw [stealth-](10.5, 2)--(11,1.5);
 \node at (11.5,1.5){$V_{\text{small}}$};
\end{tikzpicture}
\caption{Right: 
the decomposition~$V=R_1 \cup R_2 \cup R_3 \cup V_{\text{small}}$, where $R_1=f^{-1}(\text{im}(\beta) \times S^1) \times I$,  $R_2=f^{-1}(\Sigma_\beta \times S^1) \times I$ and~$R_3=\bigcup_{i=1}^c f^{-1}(P \cap U_i) \times  I$.
Left: the decomposition~$\widehat{V}_{[\times \{0\}](\Sigma)}= R_1' \cup R_2' \cup R_3' \cup V_{\text{small}}$,  where $R_1'=  \left((\text{im}(\beta) \times D^2) \setminus j(\nu(e|_{\im(\beta)}))\right)\cup R_1$, 
  $R_2'=  \left(\Sigma_\beta \times D^2 \setminus j(\nu(e|_{\Sigma_\beta}))\right) \cup R_2$,  
  and $R_3'= \bigcup_{i=1}^c \left(U_i \setminus j(\nu(e|_{B_{2i-1} \cup B_{2i}})) \right) \cup R_3$.  
The goal is to construct a homeomorphisms~$G_i \colon R_i' \to R_i$ for $i=1, 2,3$ that agree on the overlaps.
}
\label{fig:buildingG}
\end{figure}

A verification using the definition of~$j$ gives us the decompositions
\[\begin{tikzcd}[column sep=small]
j(\overline{\nu}(e))  \arrow[r,phantom,"=" description]  \arrow[d, "\subset"] & \left\lbrace \bsm \beta(y,s) \\ x' \esm  \ \Big| \ (y,s) \in \partial \Sigma \times  [0,1] ,|x'| \leq \frac{1+s}{2} \right\rbrace
 \arrow[r,phantom,"\cup" description] 
 \arrow[d, "\subset"]
& \arrow[r,phantom,"\cup" description]  \arrow[d, "\subset"] \Sigma_\beta \times D^2_{\frac{1}{2}}
&  \arrow[d, "\subset"] \bigcup_{k=1}^{2c}D^2 \times D^2_{\frac{1}{2}} \\
W  \arrow[r,phantom,"=" description] & \im(\beta) \times D^2  \arrow[r,phantom,"\cup" description] & \Sigma_\beta \times D^2 \arrow[r,phantom,"\cup" description] & \bigcup_{i=1}^{c} U_i \\
P  \arrow [u,  "\subset"] \arrow[r,phantom,"=" description] & \im(\beta) \times S^1 \arrow [u,  "\subset"]  \arrow[r,phantom,"\cup" description] & \Sigma_\beta \times S^1 \arrow [u,  "\subset"]  \arrow[r,phantom,"\cup" description] & \bigcup_{i=1}^{c} P \cap U_i, \arrow [u,  "\subset"]  \\
\end{tikzcd}
\]
and a decomposition of $f^{-1}(P)$ obtained by applying $f^{-1}$ to each piece of~$P$. 
%$$ f^{-1}(P) = f^{-1}(\im(\beta) \times S^1) \cup f^{-1}(\Sigma_\beta \times S^1 ) \cup  \left( \bigcup_{i=1}^{c} f^{-1}(P \cap U_i) \right).$$
%A verification using the definition of~$j$ shows that
%$$j_{[\times \{0\}]}(\overline{\nu}(e))
%= \overbrace{\left\lbrace \bsm \beta(y,s) \\ x' \esm  \ \Big| \ (y,s) \in \partial \Sigma \times [0,1],x' \in D^2_{\frac{1+s}{2}} \right\rbrace}^{\subset \ \im(\beta) \times D^2}
%\cup (\Sigma_\beta \times D^2_{\frac{1}{2}})
% \cup  \left( \bigcup_{i=1}^{2c}D^2 \times D^2_{\frac{1}{2}} \right).
%~$$
%Omitting the projection maps to~$W$ from the notation, we record the decomposition
%$$ W = \left(\im(\beta) \times D^2\right) \cup (\Sigma_\beta \times D^2) \cup \left( \bigcup_{i=1}^{c} U_i \right)$$
%from which we obtain the decomposition
%$$ f^{-1}(P) = f^{-1}(\im(\beta) \times S^1) \cup f^{-1}(\Sigma_\beta \times S^1 ) \cup  \left( \bigcup_{i=1}^{c} f^{-1}(P \cap U_i) \right).$$

We now construct the constituent homeomorphisms~$G_1,G_2,G_3$ of~$G$.  Note that in each of the domains of $G_1,G_2,G_3$ described below,  the $\cup$ refers to a union, not inclusion. 

The homeomorphism $G_1$ is defined as
\begin{align*}
G_1 \colon R_1'= 
\left( \begin{array}{c}
 (\im(\beta) \times D^2) \setminus j(\nu(e|_{\im(\beta)})) \\
 \cup \\
 f^{-1}(\im(\beta) \times S^1) \times I 
 \end{array} \right) &\to f^{-1}(\im(\beta) \times S^1) \times I=R_1\\
%%%%%
 (\beta(p,s),(r,\theta)) \qquad \qquad &\mapsto \left(f^{-1}(\beta(p,s),\theta),1+ r(s-1)/2 \right) \\
 (f^{-1}(\beta(p,s),\theta), t)\qquad \quad &\mapsto \left(f^{-1}(\beta(p,s),\theta), t(s+1)/2\right).
\end{align*}
Here and in what follows, we use $(r,\theta)$ to denote polar coordinates on $D^2$.

The homeomorphism~$G_2$ is defined as
\begin{align*}
G_2 \colon R_2'= \left( \begin{array}{c}
(\Sigma_\beta \times D^2)  \setminus j(\nu(e|_{\Sigma_\beta})) \\ \cup\\  f^{-1}(\Sigma_\beta \times S^1) \times I\end{array}\right) & \to f^{-1}(\Sigma_\beta \times S^1) \times I=R_2 \\
%%%%%
 (x,(r,\theta))\qquad \qquad &\mapsto \left(f^{-1}(x,\theta),3/2-r\right) \\
(f^{-1}(x,\theta), t) \qquad \quad  &\mapsto \left(f^{-1}(x,\theta), t/2\right).
\end{align*}

The construction of the homeomorphism~$G_3$ is more arduous and requires some additional notation.
An element of~$U_i \setminus j(\nu(e|_{B_{2i-1} \cup B_{2i}}))$ is of the form~$(\shrink_{2i-1} \cup \shrink_{2i})^{-1}(p)$ where, as illustrated in Figure~\ref{fig:buildingpsi},  $p=(r_1, \theta_1, r_2, \theta_2)$ with $(r_1,r_2)=s \cdot p_L +(1-s)\psi(p_L)$ for~$s \in [0,1]$ and~$p_L \in (\partial D^2_{\frac{1}{4}} \times D^2_{\geq \frac{1}{4}}) \cup (D^2_{\geq \frac{1}{4}} \times D^2_{\frac{1}{4}}).$

Using this notation, for each $i=1, \dots, c$ we define the homeomorphism $G_3^i$ as 
\begin{align*}
G_3^i \colon (R_3^i )' = \left(\begin{array}{c} U_i \setminus j(\nu(e|_{B_{2i-1} \cup B_{2i}})) \\ \cup\\  f^{-1}(U_i \cap P) \times I \end{array} \right) &\to f^{-1}(U_i \cap P) \times I = R_3^i\\
%%%%%
 (\shrink_{2i-1} \cup \shrink_{2i})^{-1}(p)  \quad& \mapsto \left( f^{-1} \left( (\shrink_{2i-1} \cup \shrink_{2i})^{-1}(p')\right), (s+1)/2 \right) \\
(f^{-1}(y),t)  \qquad \qquad & \mapsto \left(f^{-1}(y),t/2 \right),
\end{align*}
where $p=(r_1, \theta_1, r_2, \theta_2)$ for $(r_1,r_2)= sp_L+(1-s) \psi_{\operatorname{rad}}(p_L)$ 
for~$p_L \in (\{1/4\} \times [1/4,1]) \cup ([1/4,1] \times \{1/4\})$ and $s \in [0,1]$, and~$p':= (r_1', \theta_1, r_2', \theta_2)$ with $(r_1',r_2')=\psi_{\operatorname{rad}}(p_L)$.
%$(\partial D^2_{\frac{1}{4}} \times D^2_{\geq \frac{1}{4}}) \cup (D^2_{\geq \frac{1}{4}} \times D^2_{\frac{1}{4}})$ as above

A verification shows that the maps $G_1,G_2,  G_3:= \bigcup_{i=1}^c G_3^i$, and $\id_{V_{\text{small}}}$ are well-defined homeomorphisms that agree  where their domains overlap, and give rise to a homeomorphism
$$ G:= G_1 \cup G_2 \cup G_3  \cup \id_{V_{\text{small}}} \cup \colon \widehat{V}_{[\times \{0\}](\Sigma)} \to V.$$
\medskip

We now need to verify that~$f \circ G|=f \cup \widehat{\gamma}_{[\times \{0\}]}|$ on $\partial \widehat{V}_{[\times \{0\}](\Sigma)}$ .
We describe the main steps and leave the ultimate verification to the reader.
The first step is to describe~$\partial (j(\overline{\nu}(e)))$ and~$\widehat{\gamma}_{[\times \{0\}]}|$.
We write~$r_1$ and~$r_2$ for the radial coordinates of~$D^2 \times D^2.$
Using the decomposition of~$j(\overline{\nu}(e))$ from above, and the definition of~$\widehat{\gamma}_{[\times \{0\}]}|$ from Proposition~\ref{prop:IdentifyW},  the homeomorphism
\[ \widehat{\gamma}_{[\times \{0\}]} \colon \partial (j(\overline{\nu}(e))) \setminus \partial \Sigma \times D^2 \to P\]
is given piecewise by
 \begin{align*}
 \begin{array}{ccc}
 \left\lbrace \bsm \beta(y,s) \\ z' \esm  \ \Big| \ (y,s) \in \partial \Sigma \times [0,1],|z'|=\frac{1+s}{2} \right\rbrace &
 \xrightarrow{(\beta(y,s),z') \mapsto (\beta(y,s),z'\cdot \frac{2}{1+s})}& \im(\beta) \times S^1,  \\
& &\\
\Sigma_\beta \times \partial D^2_{\frac{1}{2}}  &\xrightarrow{(p,z)  \mapsto (p,2z)}& \Sigma_\beta \times S^1, \\
& & \\
 \bigcup_{k=1}^{2c} D^2_{\frac{1}{2} \leq r_1} \times \partial D^2_{\frac{1}{2}} &
\xrightarrow{ (z_1, z_2) \mapsto (z_1, 2z_2)} &
\bigcup_{k=1}^{2c} P \cap (D^2 \times D^2)_k.
\end{array}
\end{align*}
One can now verify using the definition of~$G$ that~$f \circ G|=f \cup \widehat{\gamma}_{[\times \{0\}]}|$.
We omit the details, but the point is that one checks the equality on each of the pieces of $\partial (j(\overline{\nu}(e)))$ using the definition of~$G_1,G_2$ and $G_3$.
This concludes the proof of this claim.
% verification that~$(\widehat{V}_{[\times \{0\}](\Sigma)},f \cup \widehat{\gamma}_{[\times \{0\}]})$ is homeomorphic rel. boundary to~$(V,f)$. 
\end{proof} 
The combination of these claims implies that~$(N_{e(\Sigma)},h| \cup \widehat{\gamma}|)$ is homeomorphic rel.\ boundary to~$(V,f)$.
Since~$\pi_1(P_K) \xrightarrow{f,\cong} \pi_1(\partial V) \to \pi_1(V) \cong \Z$ agrees with the epimorphism~$\varphi$, it sends the genus and plumbing loops to zero.
The aforementioned rel.\ boundary homeomorphism therefore ensures that our framing of $e$ is nice.
%%Don't delete
%Because goodness is verifying those curves are nullhomologous.
In turn, this ensures that~$\Theta(\Psi(V,f))=(N_{e(\Sigma)},h| \cup \widehat{\gamma}|)$.
This concludes the proof that $\Theta \circ \Psi=\id$ and the proposition follows.
\end{proof}

\subsection{From immersions to immersed surfaces}
\label{sub:ImmersionsToSurfaces}

The goal of this section is to deduce Theorem~\ref{thm:SurfacesManifolds} from Proposition~\ref{prop:EmbVBijections}.
This will be done by mod-ing out the domain and target of the bijection~$\Theta$ by $\Homeo_\alpha(\Sigma,\partial)$-actions and verifying that the bijection~$\Theta$ intertwines these actions.

\medbreak

Recall that~$\Homeo_\alpha(\Sigma,\partial)$ acts on $\mathcal{V}_\lambda^0(P_K)$ by $\theta \cdot (V,f)=(V,\widehat{\theta} \circ f)$.
The next result proves the main result of this section, 
%namely 
Theorem~\ref{thm:SurfacesManifolds}, which we restate in a more precise way as follows.

\begin{theorem}
\label{thm:SurfacesManifoldsMainProof}
Let~$N$ be a simply-connected~$4$-manifold with boundary~$\partial N \cong S^3$,  let~$K \subset S^3$ be a knot,  let $c_+,c_-$ and $g$ be non-negative integers, and let~$\lambda$ be a nondegenerate hermitian form over~$\Z[t^{\pm 1}]$ with $\lambda(1) \cong Q_N \oplus (0)^{2g+c}$, where $c:=c_++c_-.$

The map~$\Theta$ from Construction~\ref{cons:EmbVBijection} descends to a bijection
$$
\operatorname{Surf}_\lambda^0(g;c_+,c_-)(N,K) \xrightarrow{\cong}
\begin{cases}
\mathcal{V}_\lambda^0(P_K)/\operatorname{Homeo}_\alpha(\Sigma,\partial) &\quad \text{ if $\lambda$ is even}
\\
 \mathcal{V}_\lambda^{0,\varepsilon}(P_K)/\operatorname{Homeo}_\alpha(\Sigma,\partial)
 &\quad \text{ if $\lambda$ is odd.}
\end{cases}
$$
%%Don't delete. Nonempty iff nonempty is implicit.
\end{theorem}
\begin{proof}
By Proposition~\ref{prop:EmbVBijections},  it is enough to check that~$\Theta(\theta \cdot e)=\theta \cdot \Theta(e)$ for~$\theta \in \Homeo_\alpha(\Sigma,\partial)$ and~$e \colon  \Sigma \looparrowright N$ an immersion representing an element of~$\Imm_\alpha(g;c_+,c_-)_\lambda^0(N;K)$.
By definition of~$\Theta$, we know that~$\Theta(\theta \cdot e)$ is~$(N_{e(\theta^{-1}(\Sigma))},f_{e\circ \theta^{-1}})$ and~$\theta \cdot \Theta(e)=(N_{e(\Sigma)},\widehat{\theta} \circ f_{e})$ where~$f_e$ and~$f_{e \circ \theta^{-1}}$ are homeomorphisms from the boundaries of these surface exteriors to~$P_K$,  built as in Construction~\ref{cons:EmbVBijection}, using any choice of double point charts,  immersion of the normal bundle, and nice framing.
We will make choices of these data so that the pairs~$\Theta(\theta \cdot e)=(N_{e(\theta^{-1}(\Sigma))},f_{e \circ \theta^{-1}})$ and~$\theta \cdot \Theta(e)=(N_{e(\Sigma)},\widehat{\theta} \circ f_e)$ are homeomorphic rel.\ boundary.

The reader can verify that if $(\pi \colon \nu(e) \to \Sigma,\iota)$ is a normal bundle for $e$, then $(\theta \circ \pi,\iota)$ is a normal bundle for $e \circ \theta^{-1}$.
%%Don't delete
%with $0$-section $s_0 \circ \theta^{-1}$, 
%One can verify that if $U \subset \Sigma$ a trivialising set with $F \colon E|_U \to U \times \R^2$ for $\pi$, then $\theta(U)$ and $\theta \circ F$ gives the corresponding data for $\theta \circ \pi$.
%Here this makes sense because $E|_{\theta(U)}=\pi^{-1} \theta^{-1}(\theta(U))=\pi^{-1}(U)=E|_U$.
%The $0$-section satisfies $F \circ s_0(p)=(p,0)$.
%To check $s_0 \circ \theta^{-1}$ is a section for the new bundle, note that we have
%$(x \times \id) F \circ s_0 x^{-1}(p)
%=x \times \id (\theta^{-1}(p),0)
%=(p,0).$
%To check that $\iota$ works again note that $\iota \circ s_0=e$ implies $\iota \circ (s_0 \circ \theta^{-1})=e \circ \theta^{-1}$
%\color{black}
%%%%%
The total spaces of these two bundles are identical, and it follows that the exteriors~$N_{e(\Sigma)}:=N\setminus \iota(\nu(e))$ and~$N_{e \circ \theta^{-1}(\Sigma)}:=N\setminus \iota(\nu(e \circ \theta^{-1}))$ are equal.
The reader can also use $\theta \circ \alpha=\alpha$ to verify that if~$\{ (U_p,\psi_p) \}_p$ is a family of double point charts for~$(e,\alpha)$, then it is also a family of double point charts for~$(e \circ \theta^{-1},\alpha)$.
%%Don't delete.
%because $\psi_i \circ (e \circ \theta^{-1}) \circ \alpha=\psi \circ e \circ \alpha=\id$ because $\theta \alpha=\alpha.$
Also, if~${\fr}$ is a nice framing for~$(e,\iota)$, then~$\theta \circ {\fr}$ is nice for~$(e\circ \theta^{-1},\iota)$.
%%Don't delete.
%%%Theta takes genus and plumbings to a mix of these but then these all become nullhomologous anyways.
The equation~$\theta \circ \alpha=\alpha$ implies that $\theta|_{B_k}=\id$ and it follows that in Proposition~\ref{prop:IdentifyW}, the homeomorphism~$\widehat{\gamma}$ for~$e$ becomes~$\widehat{\theta} \circ \widehat{\gamma}$ for~$e \circ \theta^{-1}$.
%{This is ok on the $\Sigma^\circ$ but scared me on the $B_k$ part.  Perhaps it's ok after all because for $x=\alpha_k(y) \in B_k$, we have $\theta(x)=\theta(\alpha_k(y))=\alpha_k(y)=x$ so $\theta|_{B_k}=\id_{B_k}$ so $\widehat{\theta} \circ \widehat{\gamma}$ on $B_k$ is also $\alpha_k \times \id.$ }
Using these choices to construct~$f_{e \circ \theta^{-1}}$, we have~$\Theta(e \circ \theta^{-1})=(N_{e \circ \theta^{-1}(\Sigma)},h| \cup (\widehat{\theta} \circ \widehat{\gamma}))$.
Using these observations and the fact that~$N\setminus \iota(\nu(e))=N\setminus \iota(\nu(e \circ \theta^{-1}))$, we obtain
\begin{align*}
  \Theta(\theta \cdot e)
&=\Theta(e \circ \theta^{-1})
=(N_{e \circ \theta^{-1}(\Sigma)},h| \cup (\widehat{\theta} \circ \widehat{\gamma}|))\\
&=(N_{e \circ \theta^{-1}(\Sigma)},\widehat{\theta} \circ (h| \cup \widehat{\gamma}|))
  =\theta \cdot (N_{e(\Sigma)},f_e)
=\theta \cdot  \Theta(e).
\end{align*}
%AC: I guess that we are using N_{ey{\Sigma)}=N_{e{\Sigma}} for any surface homeomorphism~$y$; just set theoretic.
This proves that the pairs~$\Theta(\theta \cdot e)=(N_{e(\theta^{-1}(\Sigma))},f_{e \circ \theta^{-1}})$ and~$\theta \cdot \Theta(e)=(N_{e(\Sigma)},f_e)$ are homeomorphic rel.\ boundary and thus concludes the proof of the theorem.
\end{proof}

\section{Proof of the classification theorems }
\label{sec:ProofClassifications}

The goal of this section is to prove the main classification results stated in Theorems~\ref{thm:SurfacesRelBoundary} and~\ref{thm:SurfacesClosed}.

\subsection{Proof of the rel. boundary classification}
\label{sub:ProofRelBoundary}

We begin with immersed $\Z$-surfaces that have a nonempty boundary.

\medbreak

We recall some notation from~\cite{ConwayPiccirilloPowell}.
Let $Y$ be a closed $3$-manifold and let $\varphi \colon \pi_1(Y) \twoheadrightarrow \Z$ be an epimorphism.
The main technical result of~\cite{ConwayPiccirilloPowell} states if a nondegenerate hermitian form~$\lambda$ presents~$\Bl_Y$, then the assignment~$(W,f) \mapsto b(W,f)=f_* \circ D_W \circ \partial \varpi$ induces (together with the Kirby-Siebenmann invariant if~$\lambda$ is odd) a bijection
$$
b \colon \mathcal{V}_\lambda^0(Y) 
\xrightarrow{\cong}
\begin{cases}
\Iso(\partial \lambda,\unaryminus \Bl_Y)/\Aut(\lambda) \quad &\text{ if $\lambda$ is even,} \\
( \Iso(\partial \lambda,\unaryminus \Bl_Y)/\Aut(\lambda)) \times \Z_2 \quad &\text{ if $\lambda$ is odd.} 
 \end{cases}
$$
Here recall from Section~\ref{sub:BlanchfieldIntro} that $D_W \colon \lambda_W \cong \unaryminus \Bl_{\partial W}$ is the canonical isometry from Proposition~\ref{prop:PresentsBlanchfield} and that $\partial \varpi \colon \partial  \lambda \cong  \partial \lambda_W$ denotes the boundary of an isometry $\varpi \colon \lambda \cong \lambda_W$,  as described in Remark~\ref{rem:InducedIsometry}.
The bijection $b$ does not depend on the choice of $\varpi.$

%\begin{theorem}
%\label{thm:SurfacesRelBoundaryProof}
\begin{customthm}
{\ref{thm:SurfacesRelBoundary}}
Let~$N$ be a simply-connected~$4$-manifold with boundary~$\partial N \cong S^3$,  let~$K \subset S^3$ be a knot, and let $c_+,c_-$ and $g$ be non-negative integers.
Set $c:=c_++c_-$.
Given a nondegenerate hermitian form~$\lambda$ over~$\Z[t^{\pm 1}]$, the following assertions are equivalent:
\begin{enumerate}
\item
the hermitian form~$\lambda$ presents~$\Bl_{P_{K,g}(c_+,c_-)}$ and satisfies~$\lambda(1)\cong Q_N \oplus  (0)^{\oplus 2g+c}$;
\item the set~$\operatorname{Surf}_\lambda^0(g;c_+,c_-)(N,K)$ is nonempty and there is a bijection
$$\operatorname{Surf}_\lambda^0(g;c_+,c_-)(N,K) \approx \frac{\Iso(\partial \lambda,\unaryminus\Bl_{P_K})}{(\Aut(\lambda)\times \Homeo_\alpha(\Sigma_{g,1},\partial))}.$$
\end{enumerate}
\end{customthm}
%\end{theorem}
\begin{proof}
The implication $(2) \Rightarrow (1)$ follows from Proposition~\ref{prop:NecessaryConditions}, and so we focus on the $(1) \Rightarrow (2)$ implication.
Set $\Sigma:=\Sigma_{g,1}$ and $P_K:=P_{K,g}(c_+,c_-).$
Since $\lambda$ presents $\Bl_{P_K}$, applying~\cite[Theorem 1.1]{ConwayPiccirilloPowell} ensures that~$\mathcal{V}^0_\lambda(P_K)$ is nonempty and in fact that~$\mathcal{V}^{0,\varepsilon}_\lambda(P_K)$ is nonempty when $\lambda$ is odd.
The condition $\lambda(1) \cong Q_N \oplus (0)^{2g+c}$ ensures that the map $\Psi$ is defined (Proposition~\ref{prop:PsiWellDef}) and therefore that $\operatorname{Surf}_\lambda^0(g;c_+,c_-)(N,K)$ is nonempty as well.
Theorem~\ref{thm:SurfacesManifoldsMainProof} then shows that the map~$\Theta$ from Construction~\ref{cons:EmbVBijection} induces a bijection
$$
\operatorname{Surf}_\lambda^0(g;c_+,c_-)(N,K) \xrightarrow{\cong}
\begin{cases}
\mathcal{V}_\lambda^0(P_K)/\operatorname{Homeo}_\alpha(\Sigma,\partial) &\quad \text{ if $\lambda$ is even}
\\
 \mathcal{V}_\lambda^{0,\varepsilon}(P_K)/\operatorname{Homeo}_\alpha(\Sigma,\partial)
 &\quad \text{ if $\lambda$ is odd.}
\end{cases}
$$
Thus the theorem will follow once we show that the map~$b \colon V_\lambda^0(P_K) \to \Iso(\partial \lambda,\unaryminus \Bl_{P_K})/\Aut(\lambda)$ intertwines the~$\Homeo_\alpha(\Sigma,\partial)$-actions, i.e. satisfies~$b_{\theta\cdot (W,f)}=\theta \cdot b_{(W,f)}$ for every~$\theta \in \Homeo^+_\alpha(\Sigma,\partial)$ and for every pair~$(W,f)$ representing an element of~$V_\lambda^0(P_K)$.
Here note that,  as observed above,~$b$ is a bijection thanks to the hypothesis that $\lambda$ presents $\Bl_{P_K}$\cite{ConwayPiccirilloPowell}.

This follows formally from the definitions of the actions:
on the one hand,  for any isometry~$\varpi \colon \lambda \cong  \lambda_W$, we have~$b_{\theta\cdot (W,f)}=b_{(W,\widehat{\theta} \circ f)}=\widehat{\theta} _*\circ  f_* \circ D_W \circ \partial \varpi$; on the other hand,~$\theta \cdot b_{(W,f)}$ is~$\theta \cdot (f_* \circ D_W \circ \partial \varpi)$ and this gives the same result.
This concludes the proof of Theorem~\ref{thm:SurfacesRelBoundary}.
\end{proof}

\subsection{Proof of the closed classification}
\label{sub:ProofClosed}

The goal of this section is to prove Theorem~\ref{thm:SurfacesClosed}, which is our main classification result of closed immersed $\Z$-surfaces up to homeomorphism.

\medbreak

%\begin{theorem}
%\label{thm:SurfacesClosedProof}
\begin{customthm}
{\ref{thm:SurfacesClosed}}
Let~$X$ be a closed simply-connected~$4$-manifold,  and let $c_+,c_-$ and $g$ be non-negative integers.
Set $c:=c_++c_-$.
Given a nondegenerate hermitian form~$\lambda$ over~$\Z[t^{\pm 1}]$, the following assertions are equivalent:
\begin{enumerate}
\item
the hermitian form~$\lambda$ presents~$\Bl_{P_{U,g}(c_+,c_-)}$ and satisfies~$\lambda(1)\cong Q_X \oplus  (0)^{\oplus 2g+c}$;
\item the set~$\operatorname{Surf}_\lambda(g;c_+,c_-)(X)$ is nonempty and there is a bijection
$$\operatorname{Surf}_\lambda(g;c_+,c_-)(X) \approx \frac{\Iso(\partial \lambda,\unaryminus \Bl_{P_{U,g}(c_+,c_-)})}{\Aut(\lambda) \times \Homeo_\alpha(\Sigma_g)}.$$
\end{enumerate}
\end{customthm}
\begin{proof}
The implication $(2) \Rightarrow (1)$ follows from Proposition~\ref{prop:NecessaryConditions},  and so we focus on $(1) \Rightarrow (2)$.
Set~$\Sigma:=\Sigma_g$ and~$P_U:=P_{U,g}(c_+,c_-).$
The proof of the rel. boundary classification adapts to the closed case, as we now explain.
First,  as noted in Remark~\ref{rem:ClosedImmSurfAdapt} one verifies that the proofs of Propositions~\ref{prop:SurfQuotientOfImm} and~\ref{prop:ImmIsImmalpha} can be adapted to closed surfaces, so that mapping an immersion to its image yields a bijection
$$\operatorname{Imm}_\alpha(g;c_+,c_-)(X)/\Homeo_\alpha(\Sigma) \to \operatorname{Surf}(g;c_+,c_-)(X).$$ 
We outline how the analysis of $\operatorname{Imm}_\alpha(g;c_+,c_-)(X)$ from Sections~\ref{sec:PlumbedManifolds}, ~\ref{sec:BoundarIdentifications} and~\ref{sec:MainStatement} adapts to the closed case.
First,  as noted in Remark~\ref{rem:ClosedModels},  the definitions of the model manifolds~$P$ and~$W$ can be adapted to case where $\Sigma$ is closed; we used the notation $P_U$ and $W_U$ for the resulting manifolds.
%We also noted in that remark that the definition of $P$ for $\Sigma$ closed agrees with the definition of $P_K$ with $K=U.$
%Note that $\partial W_U=P_U$.

As noted in Remark~\ref{rem:ClosedFraming}, nice framings can also be defined for closed surfaces and the results of Section~\ref{sec:BoundarIdentifications} adapt to the closed setting,
%(in fact one can omit the homeomorphism $h$) 
yielding~$f=\widehat{\gamma} \colon \partial X_{e(\Sigma)} \xrightarrow{\cong} P_U$.
Setting~$\varepsilon:=\ks(X)$, the exact same verifications as in (the propositions leading up to) Proposition~\ref{prop:EmbVBijections} show that the assignment~$\Theta(e):=(X_{e(\Sigma)},f)$ induces a bijection
$$\Imm_{\alpha}(g;c_+,c_-)_\lambda(X) \to 
\begin{cases}
\mathcal{V}_\lambda^0(P_U)  \quad& \text{if $\lambda$ is even} \\
\mathcal{V}_\lambda^{0,\varepsilon}(P_U)  \quad& \text{if $\lambda$ is odd.}
\end{cases}
$$
Verifications analogous to in the proof of Theorem~\ref{thm:SurfacesManifoldsMainProof} shows that this bijection descends to
$$
\operatorname{Surf}_\lambda(g;c_+,c_-)(X) \xrightarrow{\cong}
\begin{cases}
\mathcal{V}_\lambda^{0}(P_U) / \Homeo_\alpha(\Sigma) &\quad \text{ if $\lambda$ is even}
\\
 \mathcal{V}_\lambda^{0,\varepsilon}(P_U) / \Homeo_\alpha(\Sigma)
 &\quad \text{ if $\lambda$ is odd.}
\end{cases}
$$
Since $\lambda$ presents $\Bl_{P_U}$, applying~\cite[Theorem 1.1]{ConwayPiccirilloPowell} ensures that~$\mathcal{V}^0_\lambda(P_U)$ is nonempty and in fact that~$\mathcal{V}^{0,\varepsilon}_\lambda(P_K)$ is nonempty when $\lambda$ is odd.
The condition $\lambda(1) \cong Q_X \oplus (0)^{2g+c}$ ensures that~$\operatorname{Surf}_\lambda^0(g;c_+,c_-)(N,K)$ is nonempty as well.
By~\cite[Theorem 1.1]{ConwayPiccirilloPowell},  the bijection~$b$ described in Section~\ref{sub:ProofRelBoundary} now gives the required bijection,  thus proving the theorem.
\end{proof}

\section{Open questions}
\label{sec:OpenQuestions}

We conclude with two open questions.
The first asks whether there are genus~$g$ immersed $\Z$-surfaces with distinct equivariant intersection forms.
\begin{question*}
\label{quest:Equivariant}
Are there examples of two genus $g$ immersed $\Z$-surfaces with $c_+$ positive double points and $c_-$ negative double points whose exteriors have distinct equivariant intersection forms?
When the surfaces have nonempty boundary, we assume that their boundary knots coincide.

In particular,  if such a surface is immersed in $S^4$ (or in $D^4$ with boundary an Alexander polynomial one knot), must the equivariant intersection form of its exterior be equivalent to $\lambda_{c_+,c_-} \oplus \mathcal{H}_2^{\oplus g}$?
\end{question*}

The answer to the second question is positive for embedded surfaces in $S^4$ (or in $D^4$ if $\Delta_K =1$) if $g \geq 3$~\cite[Theorem 7.4]{ConwayPowell},  for embedded disks and spheres in $\C P^2$~\cite{ConwayDaiMiller,ConwayOrson},  and for immersed spheres in $S^4$ and disks in~$D^4$ when $c_++c_-=1$; recall Theorem~\ref{thm:g=0c=1BoundaryD4Intro}.
Recent work of Juhasz-Powell shows that most of the known smooth constructions of embedded $\Z$-tori in $S^4$ admit $\mathcal{H}_2$ as their equivariant intersection form~\cite{JuhaszPowell}.

\begin{remark}
Assume that the ambient manifold is $S^4$ or $D^4$; in the latter case further assume the knot on the boundary has Alexander polynomial one.
When $g=0$,  Lemmas~\ref{lem:B=zA} and~\ref{lem:A(1)} show that if $\lambda$ is the equivariant intersection form of such a $\Z$-surface exterior with $(c_+,c_-)$-double points, then~$\lambda=(t-1)(t^{-1}-1)\lambda'$ where $\lambda'$ is nonsingular and satisfies $\lambda'(1) \cong (1)^{\oplus c_+} \oplus (-1)^{\oplus c_-}$.
It follows that equivariant intersection forms for such $\Z$-immersed sphere or disk exteriors correspond bijectively to the set of nonsingular hermitian forms $\lambda'$ that satisfy $\lambda'(1) \cong (1)^{\oplus c_+} \oplus (-1)^{\oplus c_-}$. 
Thus when $g=0$ the answer to Question~\ref{quest:Equivariant} is ``yes"  if and only if this set is a singleton.
%AC: The point being that $\lambda'$ uniquely determines \lambda by $\lambda=(t-1)(t^{-1}-1)\lambda'$.
\end{remark}
\color{black}

\medbreak

Even if two genus $g$ immersed $\Z$-surfaces with $c_+$ double points, $c_-$ double points and boundary a knot $K$ have exteriors with isometric equivariant intersection forms, the surfaces need not be isotopic.
This was first noted for embedded $\Z$-disks in $(\C P^2)^\circ$~\cite{ConwayDaiMiller} and in Theorem~\ref{thm:g=0c=1BoundaryD4Intro} for immersed $\Z$-disks in~$D^4$ with a single double point.
The situation is less clear in the closed case: while closed embedded $\Z$-surfaces are determined up to equivalence by the equivariant intersection form of their exteriors~\cite[Theorem 1.4]{ConwayPowell},  and similarly for immersed $\Z$-surfaces with a single double point (Theorem~\ref{thm:Other4ManifoldsSpheresIntro}),the corresponding question for immersed surfaces remains open in general.

\begin{question*}
If two closed genus $g$ immersed $\Z$-surfaces with the same number of positive and negative double points have exteriors whose equivariant intersection forms are isometric, must the surfaces be equivalent?
\end{question*}

In terms of algebra,  a negative answer to this question amounts to deciding whether there are forms for which the orbit set~$\Aut(\Bl_{P_U})/(\Aut(\lambda) \times \Homeo_\alpha(\Sigma))$ is nontrivial.
As proved in Proposition~\ref{prop:AutBl/AutStandard},  when $\lambda \cong Q_X \oplus \lambda_{c_+,c_-} \oplus \mathcal{H}_2^{\oplus g}$, then this set is trivial and so closed genus $g$ immersed $\Z$-surfaces with this equivariant intersection form are equivalent to the standard one.
\medbreak

\appendix
\section{Concordances to the Hopf link}

Finally we prove our uniqueness result for $\Z^2$-concordances from two-component 
%Alexander polynomial one 
links to the Hopf link.
While this result might be of independent interest, recall that we also made use of it during the proof of Lemma~\ref{lem:Homeo} which was a crucial ingredient needed to establish the bijection~$ \operatorname{Imm}_{\alpha}(g;c_+,c_-)^0(K,N)/\Homeo_\alpha(\Sigma,\partial) \cong \operatorname{Surf}^0(g;c_+,c_-)(K,N). $
In what follows,  links are assumed to be ordered and oriented and the Alexander polynomial refers to the multivariable Alexander polynomial.

We remind the reader that any 2-component link with Alexander polynomial~$1$ is $\Z^2$-concordant to the Hopf link with the same linking number~\cite{Davis},  and in fact it is not hard to show that the converse holds, i.e.\ that any link that is $\Z^2$-concordant to the Hopf link must have Alexander polynomial 1. 

\begin{customthm}
{\ref{thm:HopfConcordancesEquivalent}}
Let~$L$ be a~$2$-component link
with Alexander polynomial~$1$.
Any two~$\Z^2$-concordances between~$L$ and the Hopf link $H$ are equivalent rel.\  boundary.
\end{customthm}

\begin{remark}
We do not know whether the result holds with isotopy in place of equivalence: 
Since~$\operatorname{Mod}(S^3 \times I,\partial)\cong\Z_2$, generated by a Dehn twist~\cite{OrsonPowell},  the homeomorphism from Theorem~\ref{thm:HopfConcordancesEquivalent} need a priori not be isotopic to the identity.
Put differently,  in the setting of Theorem~\ref{thm:HopfConcordancesEquivalent}, we do not know whether a given $\Z^2$-concordance is isotopic to the $\Z^2$-concordance obtained by applying a Dehn twist to $S^3 \times [0,1].$
\end{remark}

The idea of the proof of Theorem~\ref{thm:HopfConcordancesEquivalent} is to observe that the exterior of a $\Z^2$-concordance is a~$K(\Z^2,1)$ and to then consider the surgery exact sequence.
The proof of this second step is near identical to that of~\cite{ConwayDiscs}.
The overall proof strategy is similar to the one from the earlier~\cite{ConwayPowellDiscs} and can also be used to prove that up to equivalence rel. boundary, there is a unique $\Z$-concordance between an Alexander polynomial one knot and the unknot.

\subsection{Algebraic topology of concordance exteriors}

This section establishes that a $\Z^2$-concordance exterior is a $K(\Z^2,1)$ (Proposition~\ref{prop:ConcordanceExterior}) and proves a technical result (Proposition~\ref{prop:Compatible}) needed during the proof of Theorem~\ref{thm:HopfConcordancesEquivalent}. 

\begin{notation}
In what follows,~$H= H_a \cup H_b$ denotes the Hopf link and~$L= L_a \cup L_b$ is a~$2$-component link with~$\Delta_L=1$.
We denote the meridians of $L$ by $\mu_a$ and $\mu_b.$
%{AM: Making it consistent that we always use $a$ and $b$ for the two components, and $0$ and $1$ for two different concordances (though maybe 1 and 2 would be better to avoid vague suggestions of $[0,1]$.)} 
%%AC: Could do 1,2.
We let~$E_H$ and~$E_L$ denote the exteriors of~$H$ and~$L$ respectively.
Finally,  $E_C$ will denote the exterior of a concordance~$C$ between $L$ and $H$,  and we decompose $\partial E_C=E_L \cup \partial_{\intt} E_C\cup E_H$.
All our concordances are of ordered links,  and so consist of a pair of annuli,  one between~$L_a$ and~$H_a$ and  one between~$L_b$ and~$H_b$.
\end{notation}

We begin by calculating the algebraic topology of a $\Z^2$-concordance exterior.

\begin{proposition}
\label{prop:ConcordanceExterior}
Let~$C$ be a~$\Z^2$-concordance between~$L$ and~$H$. 
\begin{itemize}
\item The inclusion induced map~$\pi_1(\partial E_C) \to \pi_1(E_C)$ is surjective.
\item The inclusion induces an isomorphism~$H_*(E_L; \Z) \xrightarrow{\cong} H_*(E_C; \Z).$
%\item The~$\Z[\Z^2]$-homology of~$E_C$ with respect to the coefficient system that maps the~$\mu_i$ to~$t_i$ for~$i=1,2$ satisfies
%$$
%H_i(E_C;\Z[\Z^2])=
%\begin{cases}
%\Z &\quad \text{ for } i=0 \\
%0 &\quad \text{ otherwise.}
%\end{cases}
%$$
\item The concordance exterior~$E_C$ is a~$K(\Z^2,1).$ 
\end{itemize}
\end{proposition}
\begin{proof}
For the first assertion,  note that~$\pi_1(E_C) \cong \Z^2$ is freely generated by meridians.
%and that~$E_C$ is spin because~$S^3 \times I$ is spin.
The second fact follows because the exterior of a concordance is a relative homology cobordism.
The third assertion will follow from  Hurewicz's theorem applied to the universal cover of~$E_C$, once we establish that  the~$\Z[\Z^2]$-homology of~$E_C$ with respect to the coefficient system mapping~$\mu_j$ to~$t_j$ for~$j=a,b$ is
$$
H_i(E_C;\Z[\Z^2])\cong
\begin{cases}
\Z &\quad \text{ for } i=0, \\
0 &\quad \text{ otherwise.}
\end{cases}
$$
We argue that $H_i(E_C; \Z[\Z^2])=0$ for $i=1,2,3.$
Since~$\pi_1(E_C)=\Z^2,$ we have~$H_1(E_C;\Z[\Z^2])=0$.
Note also that since~$\pi_1(\partial E_C)\to \pi_1(E_C)$ is surjective we have~$H_i(E_C,\partial E_C;\Z[\Z^2])=0$ for~$i=0,1$,  which in turn implies that~\[H_3(E_C;\Z[\Z^2])=H^1(E_C,\partial E_C;\Z[\Z^2])=0.\]
Next observe that~$H_2(E_C;\Z[\Z^2])$ is torsion because
$$0=1-2+1-0=\chi(E_C)=\chi^{\Z[\Z^2]}(E_C)=0-0+b_2^{\Z[\Z^2]}-0=b_2^{\Z[\Z^2]}$$
but it is also torsion-free because, using that~$H_i(E_C,\partial E_C;\Z[\Z^2])=0$ for~$i=0,1$,
%that~$\pi_1(M_L)\to \pi_1(E_C)$ is surjective,  
we have
$$H_2(E_C;\Z[\Z^2])\cong H^2(E_C,\partial E_C;\Z[\Z^2]) \cong H_2(E_C,\partial E_C;\Z[\Z^2])^*.$$
It follows that~$H_2(E_C;\Z[\Z^2])=0$,  finishing the proof of the proposition.
% proving the third assertion.
%The fourth assertion now follows from the third together with Hurewicz's theorem applied to the universal cover of~$E_C.$
\end{proof}

The remainder of this section is devoted to proving a preliminary result needed in the proof of Theorem~\ref{thm:HopfConcordancesEquivalent}.
Given $\Z^2$-concordances $C_1$ and $C_2$ between $L$ and $H$, this result,  stated precisely in Proposition~\ref{prop:Compatible},  ensures that there is an isomorphism $\psi \colon \pi_1(E_{C_1}) \to \pi_1(E_{C_2})$ and a homeomorphism $f \colon \partial E_{C_1} \to \partial E_{C_2}$ such that $i_2 \circ f_*=\psi \circ i_1  \colon \pi_1(\partial E_{C_1}) \to \pi_1(E_{C_2})$, where $i_1,i_2$ denote the appropriate inclusion-induced maps.

\begin{construction}[Generators for $H_1(\partial E_C)$]\label{cons:Loopc}
A Mayer-Vietoris argument gives~$H_1(\partial E_C) \cong \Z^3$.
More precisely,~$H_1(\partial E_C)\cong \Z\mu_ a \oplus \Z\mu_b \oplus \Z$, where~$\mu_a,\mu_b$ are the images of the meridians of the link~$L$ (or $H$) and the last summand is obtained by splitting the short exact sequence
$$ 0 \to \Z\mu_a \oplus \Z\mu_b \to H_1(\partial E_C) \to \Z \to 0.$$
Such a splitting is not canonical,  but a generator $\gamma$ for the third summand can be obtained by picking points $a_L, b_L$ on each torus component of $\partial E_C \cap \overline{\nu}(L)$ and connecting them with an arc in~$E_L$; picking $a_H,b_H$ on each torus component of $\partial E_C \cap \overline{\nu}(H)$ and connecting them with an arc in~$E_H$;  and finally connecting $a_L$ to $a_H$ and $b_L$ to $b_H$ with  a pair of arcs,  one in each component of~$\partial_{\text{int}}E_C\cong C \times S^1$.  
%Roughly speaking,  the winding of the arc from $a_L$ to $a_H$ around $\mu_a$  and the winding of the arc from $b_L$ to $b_H$ about $\mu_b$ together determine~$\gamma \in H_1(\partial E_C)$.
%%Helpful intuitively but commented because it seems to imply that the arcs in E_L and E_H don’t matter, which of course they do.
\end{construction}

%However, 
Our framings~$\fr \colon \overline{\nu}(X) \to X \times D^2$ continue to satisfy~$\fr(X)= X \times \{0\}$,  though contrarily to Section~\ref{sec:ImmersionsAndNormalBundles},  we suppress the underlying vector bundles from the notation and work directly with tubular neighborhoods.
%%Don't delete
%%Because induced by framing of vector bundle so automatic.
We also fix tubular neighborhoods $\overline{\nu}(L)$ and $\overline{\nu}(H)$ of $L$ and~$H$ and require that tubular neighorhoods of concordances from $L$ to $H$ extend these fixed ones.

\begin{notation}[The $0$-framings of $L$ and $H$]
Denote the $0$-framings of $L$ and $H$ by
\begin{align*}
\fr_L \colon \overline{\nu}(L) \to L \times D^2 \text{,  } \quad
\fr_H \colon \overline{\nu}(H) \to H \times D^2.
\end{align*}
By the~$0$-framing, we mean that each component of $L$ and $H$ is endowed with the $0$-framing.
Given a concordance $C$ between $L$ and $H$,  this framing extends to a framing of each component of $C$; this is a standard result in knot concordance that can be proved via a simpler version of the proof of Lemma~\ref{lem:RelativeEulerNumber}.
\end{notation}

\begin{construction}[Constructing homeomorphisms~$\partial E_{C_1} \cong \partial E_{C_2}$]
Let $C_1$ and $C_2$ be two concordances between~$L$ and~$H$.
Since both $C_1$ and $C_2$ connect $L_a$ to $H_a$ and $L_b$ to $H_b$,  there exists a homeomorphism $g \colon C_1 \to C_2$ with $g|_{L \sqcup H}=\id_{L \sqcup H}$.  

Given framings  $\fr_i \colon \overline{\nu}(C_i) \to C_i \times D^2$ for $i=1,2$  that each extend the $0$-framing $\fr_L \sqcup \fr_H$,  we obtain a homeomorphism
$$ \id_{E_L} \cup (\fr_2^{-1} \circ (g \times \id_{D^2}) \circ \fr_1)  \cup \id_{E_H}  \colon 
\overbrace{E_L \cup \partial_{\intt} \overline{\nu}(C_1) \cup E_H}^{=\partial E_{C_1}}
\to \overbrace{E_L \cup \partial_{\intt} \overline{\nu}(C_2) \cup E_H}^{=\partial E_{C_2}}.$$
These homeomorphisms glue together because on $\partial \overline{\nu}(L)$,  the homeomorphism $\fr_2^{-1} \circ (g \times \id_{D^2}) \circ \fr_1$ restricts to $\fr_L^{-1} \circ (\id_L \times \id_{D^2}) \circ\fr_L=\id|_{\partial \overline{\nu}(L)}$ and similarly on $\partial \overline{\nu}(H)$.
\end{construction}

\begin{lemma}
\label{lem:CompatibleHomology}
Let $C_1$ and $C_2$ be concordances between $L$ and $H$.
%that connect $L_a$ to $H_a$ and $L_b$ to $H_b$.}\footnote{AM: Open to rephrasing of this.  Also,  I guess this is just a concordance from $L$ to $H$ as an ordered link,  which is what we were/ should have been thinking of the whole time,  so maybe we don't really need to rephrase Theorem~\ref{thm:HopfConcordancesEquivalent}?
There are framings~$\fr_i$ of $C_i$ and
% homeomorphisms~$g_i \colon C_i \to (S^1 \sqcup S^1) \times [0,1]$
a homeomorphism $g \colon C_1  \to C_2$ such that the following diagram commutes:
$$
\xymatrix@C4cm{
H_1(E_{C_1})\ar[r]^{\psi,\cong}&H_1( E_{C_2}) \\
H_1(\partial E_{C_1}) \ar[u]^{i_1}\ar[r]^{(\id_{E_L} \cup (\fr_2^{-1} \circ (g \times \id) \circ \fr_1)  \cup \id_{E_H})_*} &  H_1(\partial E_{C_2}) \ar[u]^{i_2},
}
$$
where $i_1,i_2$ are induced by inclusion and  $\psi \colon H_1(E_{C_1}) \to H_1(E_{C_2})$ sends $\mu_a$ to $\mu_a$ and $\mu_b$ to $\mu_b$.
\end{lemma}
\begin{proof}
Fix a homeomorphism $g \colon C_1 \to C_2$ with $g|_{L \sqcup H}=\id_{L \sqcup H}$,  and a loop $\gamma_1 \in \partial E_{C_1}$ as in Construction~\ref{cons:Loopc}.  
%%Don't delete
%Existence of $g$ is by the same argument in the comments of the proof of Proposition~\ref{prop:ImmIsImmalpha}.
By modifying the arcs in $\gamma_1$ lying on the two components of $\partial_{\intt} E_{C_1}$ by parallel push-offs of $\mu_a$ and $\mu_b$,  we can ensure that $\gamma_1$ is null-homologous in $E_{C_1}$.  Now choose framings~$\fr_1$ of $C_1$  and $\fr_2'$ of $C_2$ that each extend the $0$-framing,  and 
consider the loop
$$\gamma_2':=(\id_{E_L} \cup (\fr_2'^{-1} \circ (g \times \id_{D^2}) \circ \fr_1)  \cup \id_{E_H})(\gamma_1) \subset \partial E_{C_2}.$$
%Observe that $\gamma_2' \cap E_L=\gamma_1 \cap E_L$ and $\gamma_2' \cap E_H=\gamma_1 \cap E_H$.
Depending on our choice of framings,  we may have that  $\gamma_2'$ is nontrivial in $H_1(E_{C_2})$: our goal is to modify $\fr_2'$ (and thus $\gamma_2'$) so that this is not the case.
The subset~$\partial_{\intt}E_{C_2}$ has two components~$\partial_{\intt}^a E_{C_2}$ and~$\partial_{\intt}^bE_{C_2}$,  each of which is homeomorphic to~$S^1 \times S^1 \times [0,1]$.
Changing the framing~$\fr_2'$ of the two components of~$C_2=C_2^a \sqcup C_2^b$ modifies the part of~$\gamma_2'$ within the corresponding thickened torus.
We change the framing of~$C_2^a$ and~$C_2^b$ by adding the relevant number of meridians to~$\gamma_2' \cap \partial_{\intt}^a E_{C_2}$ and~$\gamma_2' \cap \partial_{\intt}^b E_{C_2}$ so that the outcome~$\gamma_2$ has no meridional components; this step is similar to the proof of Proposition~\ref{prop:FramingPushoffUsingF}.
%\footnote{AC: Flagging change of indices.  AM: Good catch thanks.}
In other words, we change the framing~$\fr_{2}'$ into a framing~$\fr_{2}$ so that~$(\id_{E_L} \cup (\fr_2^{-1} \circ (g \times \id_{D^2}) \circ \fr_1)  \cup \id_{E_H})(\gamma_1)$ is nullhomologous in~$E_{C_2}$.
%%%Don't delete:
%How come we proceed differently for \gamma_0 and \gamma_1?
%Aka why can't we write the same for \gamma_1 as for \gamma_0?
%Answer: We could build gamma_1 not using the framing maps, but then it would be less clear that it is mapped to by gamma_0?
%Or maybe even thinking about a gamma_1 is unhelpful? We want to build our bottom map so that right circ bottom sends gamma_0 to something null homologous in E_c_1
%%%

The diagram commutes on meridians; this uses that both $C_1$ and $C_2$ are concordances of ordered links.
%{AM: This is where we use that the two concordances both `match up' the components the same way,  I think.}
Finally, on~$\gamma_1$ both paths in the diagram give zero,  as justified above.  
This finishes the proof,  since from Construction~\ref{cons:Loopc} we have that $H_1(\partial E_{C_1})$ is generated by $\mu_a,\mu_b,$ and $\gamma_1$. 
\end{proof}

The next proposition establishes the main technical result of this section.

\begin{proposition}
\label{prop:Compatible}
Let $C_1$ and $C_2$ be $\Z^2$-concordances between $L$ and $H$.
%There are homeomorphism~$f_i \colon \overline{\nu}(C_i) \to S^1 \times [0,1] \times D^2$ and
There is a homeomorphism $F \colon \overline{\nu}(C_1) \to \overline{\nu}(C_2)$ with~$F(C_1)=C_2$  and $F|_{\overline{\nu}(L) \cup \overline{\nu}(H)}=\id$ such that the following diagram commutes:
$$
\xymatrix@C2.5cm{
\pi_1(E_{C_1})\ar[r]^{\psi,\cong}&\pi_1( E_{C_2}) \\
\pi_1(\partial E_{C_1}) \ar[u]^{i_1}\ar[r]^{(\id_{E_L} \cup F| \cup \id_{E_H})_*} &  \pi_1(\partial E_{C_2}) \ar[u]^{i_2},
}
$$
where~$i_1,i_2$ are induced by inclusion and $\psi \colon \pi_1(E_{C_1}) \to \pi_1(E_{C_2})$ sends~$\mu_a$ to~$\mu_a$ and~$\mu_b$ to~$\mu_b$. 
\end{proposition}
\begin{proof}
Since the fundamental groups of the concordance exteriors are abelian,  it suffices  to verify the statement on homology.
%{AC: Double check this.  AM: I agree with this.}
%\psi i_0(x)=\psi i_0([x])=i_1 blabla ([x])=i_1 \circ blabla(x).
%The first and last equalities are because target abelian means map factors through homology
The result now follows from Lemma~\ref{lem:CompatibleHomology} by setting~$F:=\fr_2^{-1} \circ (g \times \id_{D^2}) \circ \fr_1$.
Indeed
\[F(C_1)= \fr_2^{-1} \circ (g \times \id_{D^2})\circ \fr_1(C_1)= 
 \fr_2^{-1} \circ (g \times \id_{D^2})(C_1 \times \{0\})
 =  \fr_2^{-1}(C_2 \times \{0\})=C_2,
 \]
and~$F|_{\overline{\nu}(L) \cup \overline{\nu}(H)}=\id$ because~$g|_{L \sqcup H}=\id$ and~$\fr_1,\fr_2$ both extend the $0$-framing $\fr_L \sqcup \fr_H$. 
\end{proof}

\subsection{Surgery theory input}

Let $L$ be a~$2$-component link with Alexander polynomial~$1$.
We prove Theorem~\ref{thm:HopfConcordancesEquivalent},  which states that any two~$\Z^2$-concordances between~$L$ and the Hopf link $H$ are equivalent rel.\  boundary.
\begin{proof}[Proof of Theorem~\ref{thm:HopfConcordancesEquivalent}]
Let $C_1$ and $C_2$ be~$\Z^2$-concordances between~$L$ and~$H$.
Proposition~\ref{prop:Compatible} ensures that there is a homeomorphism $F \colon \overline{\nu}(C_1) \to \overline{\nu}(C_2)$ with~$F(C_1)=C_2$ and $F|_{\overline{\nu}(L) \cup \overline{\nu}(H)}=\id$ such that the following diagram commutes
$$
\xymatrix@C2.5cm{
\pi_1(E_{C_1})\ar[r]^{\psi,\cong}&\pi_1( E_{C_2}) \\
\pi_1(\partial E_{C_1}) \ar[u]^{i_1}\ar[r]^{(\id_{E_L} \cup F| \cup \id_{E_H})_*} &  \pi_1(\partial E_{C_2}) \ar[u]^{i_2}
}
$$
where $i_1,i_2$ are induced by inclusion and $\psi \colon \pi_1(E_{C_1}) \to \pi_1(E_{C_2})$ sends $\mu_a$ to $\mu_a$ and $\mu_b$ to $\mu_b$. 

The commutativity of this diagram and the fact that~$E_{C_2}$ is a $K(\Z^2,1)$ (by Proposition~\ref{prop:ConcordanceExterior}) ensures that the obstruction theoretic argument from~\cite[Lemma 2.1(2)]{ConwayPowellDiscs} applies to our setting.
The outcome is that the homeomorphism~$\id_{E_L} \cup F| \cup \id_{E_H} \colon \partial E_{C_1} \to \partial E_{C_2}$ extends to a homotopy equivalence~$E_{C_1} \xrightarrow{\simeq} E_{C_2}$ that induces $\psi$ on $\pi_1$.
Consider the rel.\ boundary surgery exact sequence:
$$ \mathcal{N}(E_{C_2} \times [0,1],\partial (E_{C_2} \times [0,1])) \xrightarrow{\sigma_5}  L_5(\Z[\Z^2]) \dashrightarrow \mathcal{S}(E_{C_2},\partial E_{C_2}) \xrightarrow{\eta} \mathcal{N}(E_{C_2},\partial E_{C_2}) \xrightarrow{\sigma_4} L_4(\Z[\Z^2]).$$
We omitted decorations because~$\Z^2$ has vanishing Whitehead group~\cite{BassHellerSwan}.
The sequence is exact because~$\Z^2$ is a good group.
Since~$E_{C_2}$ is a~$K(\Z^2,1)$,  using algebraic surgery,  for~$i=4,5$ the surgery obstruction map~$\sigma_i$ agrees with the composition
$$\sigma_i \colon H_i(E_{C_2};\mathbf{L}\langle 1 \rangle_\bullet) \xrightarrow{\cong} H_i(E_{C_2};\mathbf{L}_\bullet) \xrightarrow{A_i} L_i(\Z[\Z^2]).$$
Here~$\mathbf{L}_\bullet$ denotes the~$L$-theory spectrum, ~$\mathbf{L}\langle 1\rangle_\bullet$ is its~$1$-connective cover, and~$A_i$ denotes the assembly map in~$L$-theory.
We use these notions as blackboxes and refer to~\cite{LueckMacko} and the references within for background on surgery theory.
Since $B\Z^2 \simeq K(\Z^2,2) \simeq S^1 \times S^1$ admits a~$2$-dimensional CW model and satisfies the Farrell-Jones conjecture (e.g. by virtue of being a surface group see e.g.~\cite[Theorem~15.1]{LueckIsomorphism}), it follows that $\sigma_4$ is injective and $\sigma_5$ is surjective~\cite[Lemma 2.2]{KasprowskiLand}.
\color{black}

%the first map in this composition is an isomorphism because~$\Z^2$ is~$2$-dimensional~\cite[proof of Lemma 2.2]{KasprowskiLand}.
%Finally~$A_i$ denotes the assembly map in~$L$-theory.
%We use these notions as blackboxes and refer to~\cite{LueckMacko} and the references within for background on surgery theory.

%Since~$\Z^2$ is a surface group, it satisfies the Farell-Jones conjecture; see e.g.~\cite[Theorem~15.1]{LueckIsomorphism}.
%Since~$\Z^2$ is~$2$-dimensional, it follows that~$A_4$ is injective and~$A_5$ is surjective; see~\cite[Lemma 2.2]{KasprowskiLand}.
%We deduce that~$\sigma_4$ is injective and~$\sigma_5$ is surjective.
It then follows that the structure set~$ \mathcal{S}(E_{C_2})$ is a singleton and therefore the exteriors are s-cobordant rel. boundary.
Since~$\operatorname{Wh}(\Z^2)=0$ and since~$\Z^2$ is a good group (see e.g.~\cite[Chapter 12]{DET}), the~$5$-dimensional relative~$s$-cobordism theorem implies that $\id_{E_L} \cup F \cup \id_{E_H}$ extends to a homeomorphism $G \colon E_{C_1} \to E_{C_2}$~\cite[Theorem 7.1A]{FreedmanQuinn}.
Combining this homeomorphism with the homeomorphism~$F \colon \overline{\nu}(C_1) \to  \overline{\nu}(C_2)$,  which satisfies $F(C_1)=C_2$ and $F|_{\overline{\nu}(L) \cup \overline{\nu}(H)}=\id$, 
%~$E_{C_0}$ and~$E_{C_1}$ are homeomorphic rel. boundary~\cite[Theorem 7.1A]{FreedmanQuinn}.
%Extending this homeomorphism by the identity over~$((S^1 \sqcup S^1) \times I) \times D^2$ 
yields the desired rel.\ boundary homeomorphism of~$S^3 \times I$ taking one concordance to the other.
% by gluing~$((S^1 \sqcup S^1) \times I) \times D^2$,  to the concordance exteriors, we deduce that the concordances are equivalent rel. boundary.
\end{proof}

\color{black}

\bibliography{BiblioImmersed}
\bibliographystyle{alpha}
\end{document}